\documentclass[twoside,a4paper,11pt]{article}
\usepackage{amsthm,amsmath,amssymb,amscd,amsfonts,stmaryrd}
\usepackage{mathrsfs}
\usepackage[margin=25mm]{geometry}
\usepackage{bigints}
\usepackage{mathtools}
\usepackage{enumitem}
\usepackage{comment}
\usepackage{fancyhdr}
\usepackage{latexsym}
\usepackage{amsfonts}
\usepackage{bm}
\usepackage[d]{esvect}
\usepackage{array}
\usepackage{euscript}

\usepackage[bookmarks=false]{hyperref}

\setlist{topsep = 1.5ex, itemsep = 0.4ex} 

\DeclareMathOperator{\Leb}{Leb}
\DeclareMathOperator{\id}{id}
\DeclareMathOperator{\closure}{cl}
\DeclareMathOperator{\dom}{dom}
 
\DeclareMathOperator{\supp}{supp}

\DeclareMathOperator{\diag}{diag}

\newcommand{\trf}{\operatorname{trf}}

\allowdisplaybreaks[1]
\thepage

\mathtoolsset{showonlyrefs=true}

\numberwithin{equation}{section}

\theoremstyle{definition}
\newtheorem{exm} {Example}[section]
\newtheorem{dfn}[exm] {Definition}
\newtheorem{rem}[exm] {Remark}

\theoremstyle{plane}
\newtheorem{lem}[exm]{Lemma}
\newtheorem{prop}[exm]{Proposition}
\newtheorem{thm}[exm]{Theorem}
\newtheorem{cor}[exm]{Corollary}
\newtheorem{assum}[exm]{Assumption}

\newcommand{\NN}{\mathbb{N}}
\newcommand{\ZN}{\mathbb{Z}}
\newcommand{\ZNp}{\mathbb{Z}_{+}}

\newcommand{\RN}{\mathbb{R}} 
\newcommand{\RNp}{\mathbb{R}_{\geq 0}}
\newcommand{\RNpp}{\mathbb{R}_{>0}}

\newcommand{\bcmAB}{\mathsf{bcm}}
\newcommand{\cadlag}{c\`adl\`ag}

\newcommand{\cX}{\mathcal{X}}
\newcommand{\cY}{\mathcal{Y}}

\newcommand{\Borel}{\mathcal{B}}

\newcommand{\ProcLaw}{\mathcal{L}}

\newcommand{\Compact}[1]{\mathcal{C}_c(#1)}
\newcommand{\Closed}[1]{\mathcal{C}(#1)}
\newcommand{\HausMet}[1]{d_{#1}^H}
\newcommand{\FellMet}[1]{d_{#1}^{F}}

\newcommand{\graphmap}{\mathfrak{g}}
\newcommand{\hatCc}{\widehat{C}_{c}} 
\newcommand{\hatC}{\widehat{C}} 

\newcommand{\hatCcMet}[2]{d^{\widehat{C}_c}_{#1, #2}}
\newcommand{\hatCMet}[2]{d^{\widehat{C}}_{#1, #2}}

\newcommand{\hatCSt}{\tau_{\widehat{C}}}

\newcommand{\upC}{C_{\uparrow}}

\newcommand{\SkorohodMet}[1]{d^{J_1}_{#1}}

\newcommand{\BorMeas}{\mathcal{M}_{\mathrm{Bor}}}
\newcommand{\STOMMeas}{\mathcal{M}_{(*, \mathrm{fin})}}
\newcommand{\STOMMet}[1]{d_{#1}^{\STOMMeas}}
\newcommand{\STOMSt}{\tau_{\STOMMeas}}

\newcommand{\finMeas}{\mathcal{M}_{\mathrm{fin}}} 

\newcommand{\Prob}{\mathcal{P}} 
\newcommand{\ProhMet}[1]{d_{#1}^{P}} 
\newcommand{\ProbSt}{\tau_{\mathcal{P}}}

\newcommand{\Meas}{\mathcal{M}}
 
\newcommand{\VagueMet}[1]{d_{#1}^{V}}
\newcommand{\MeasSt}{\tau_{\mathcal{M}}}
\newcommand{\finMeasSt}{\tau_{\finMeas}}

\newcommand{\SkorohodSt}{\tau_{J_1}} 
\newcommand{\LzeroMet}{d^{L^0}}
\newcommand{\LzeroSt}{\tau_{L^0}}

\newcommand{\HKSt}{\tau_{\mathrm{HK}}}
\newcommand{\ColSt}{\tau_{\mathrm{CM}}}
\newcommand{\SPMSt}{\tau_{\mathrm{SP}}}
\newcommand{\SPMPCAFSt}{\tau_{\mathrm{SP, PCAF}}}
\newcommand{\SPMSTOMSt}{\tau_{\mathrm{SP, STOM}}}
\newcommand{\LzeroSPMSt}{\tau_{L^0\text{-}\mathrm{SP}}}
\newcommand{\LzeroSPMPCAFSt}{\tau_{L^0\text{-}\mathrm{SP}, \mathrm{PCAF}}}
\newcommand{\LzeroSPMSTOMSt}{\tau_{L^0\text{-}\mathrm{SP}, \mathrm{STOM}}}
\newcommand{\LzeroSPMvSTOMSt}{\tau_{L^0\text{-}\mathrm{SP}, \mathrm{v}\text{-}\mathrm{STOM}}}

\newcommand{\rootBCM}{\mathfrak{M}_\bullet}
\newcommand{\rootCM}{\mathfrak{K}_\bullet}

\newcommand{\GFMet}{d_{\rootBCM}}
\newcommand{\GHMet}{d_{\rootCM}}

\newcommand{\rootResisSp}{\mathfrak{R}_\bullet} 
\newcommand{\ColResisSp}{\mathfrak{R}_\bullet^{\mathrm{col}}}

\newcommand{\domain}{\mathcal{D}}
\newcommand{\form}{\mathcal{E}}
\newcommand{\rdomain}{\mathcal{F}}
\newcommand{\filt}[0]{\EuScript{F}}

\newcommand{\ResolDensity}[0]{r}
\newcommand{\Potential}[0]{\mathcal{R}}

\newcommand{\sigalg}{\EuScript{M}}

\newcommand{\apPCAF}{\mathfrak{P}}
\newcommand{\apSTOM}{\mathfrak{S}}

\newcommand{\diagMeas}[1]{#1^{\diag}}

\newcommand{\SPM}{\Upsilon} 
\newcommand{\hatSPMcol}{\hat{\Upsilon}^{\mathrm{col}}} 

\newcommand{\GHSPM}{\cX}
\newcommand{\GHHKSPM}{\cX^{\mathrm{HK}}}
\newcommand{\GHColSPM}{\cX^{\mathrm{col}}}

\newcommand{\lrangle}[1]{\langle #1 \rangle}

\newcommand{\supnorm}[1]{\| #1 \|_{\infty}}

\newcommand{\BExc}{\mathbf{e}}

\newcommand{\exitTime}{\breve{\sigma}}

\title{Convergence of space-time occupation measures of stochastic processes and its application to collisions} 
\date{}
\author{Ryoichiro Noda\thanks{Research Institute for Mathematical Sciences, Kyoto University, Kyoto, 606-8502,
JAPAN. E-mail:sgrndr@kurims.kyoto-u.ac.jp}}

\begin{document}

\maketitle

\begin{abstract}
We introduce a new perspective on positive continuous additive functionals (PCAFs) of Markov processes,
which we call \emph{space--time occupation measures (STOMs)}.
This notion provides a natural generalization of classical occupation times and occupation measures,
and offers a unified framework for studying their convergence.
We analyze STOMs via so-called smooth measures 
associated with PCAFs through the Revuz correspondence.
We establish that if the underlying spaces, 
the processes living on them, 
their heat kernels, and the associated smooth measures converge,
and if the corresponding potentials of these measures satisfy a uniform decay condition,
then the associated PCAFs and STOMs also converge 
in suitable Gromov--Hausdorff-type topologies. 

We then apply this framework to the analysis of collisions of independent stochastic processes. 
Specifically, by exploiting the STOM formulation,
we introduce the notion of \emph{collision measures}, 
which record both the collision sites and times of two processes,
and prove general convergence theorems for these measures. 
The abstract results are further specialized to random walks on electrical networks
via the theory of resistance metric spaces,
leading to concrete scaling limits for collision measures 
of random walks on critical random graphs,
such as critical Galton--Watson trees, critical Erd\H{o}s--R\'enyi random graphs, 
and the uniform spanning tree.
\end{abstract}

\tableofcontents

\section{Introduction} \label{sec: intro}

Positive continuous additive functionals (PCAFs) play a crucial role in the analysis of Markov processes,
appearing in various contexts, such as the Feynman--Kac formula, time-changed processes, and reflected processes.
This paper is concerned with convergence of PCAFs,
which has important applications in a wide range of problems.
\begin{enumerate} [label = \textup{(\roman*)}]
  \item Convergence of PCAFs directly implies convergence of the associated Feynman--Kac semigroups 
    (cf.\ \cite{Albeverio_Blanchard_Ma_91_FKsemigroup}).
    As an application in the PDE context,
    Fan~\cite{Fan_16_Discrete} constructed a discrete approximation for solutions of the heat equation with general Robin boundary conditions,
    by proving the convergence of the boundary local times of random walks in a certain Euclidean domain
    to those of the reflected diffusion.
  \item Convergence of PCAFs is a crucial step in deriving convergence of corresponding time-changed processes, 
    such as Liouville Brownian motion and the Bouchaud trap model 
    \cite{Croydon_Hambly_Kumagai_17_Time-changes,Garban_Rhodes_Vargas_16_Liouville,Ooi_25_Convergence}.
  \item Convergence of local times can be used to deduce convergence of cover times of random walks 
    \cite{Andriopoulos_etc_pre_cover_on_CRT,Croydon_pre_CoverTimeOnBianaryTree}. 
\end{enumerate}

A fundamental result in the theory of PCAFs is the Revuz correspondence \cite{Revuz_70_Mesures},
which establishes a one-to-one correspondence between PCAFs and so-called smooth measures.
Recent studies \cite{Nishimori_Tomisaki_Tsuchida_Uemura_pre_On,Noda_pre_Continuity,Ooi_pre_Homeo} 
have shown that this correspondence is continuous, in the sense that
if smooth measures converge, then the associated PCAFs also converge,
with the mode of convergence specified suitably.
These works, however, have focused on convergence of PCAFs of a fixed Markov process.
In this paper, we establish convergence results for PCAFs of Markov processes living on spaces that vary along a sequence, 
where the convergence of the underlying spaces is treated within the framework of Gromov--Hausdorff-type topologies 
(see Section~\ref{sec: GH-type topologies} below).
This extension of the framework makes the theory applicable to a wider range of settings, 
such as discrete approximations of spaces and scaling limits of random graphs.

In this paper, we also introduce a new perspective on PCAFs, 
viewing them as random measures on space and time.
Let $X=(X_t)_{t \geq 0}$ be a standard process on a suitable topological space $S$
(see \cite{Blumenthal_Getoor_68_Markov,Chen_Fukushima_12_Symmetric} for the definition of standard processes, for example).
For a PCAF $A = (A_t)_{t \geq 0}$ of $X$ (see Definition~\ref{dfn: PCAF} below),
we define a random measure on $S \times [0,\infty)$ by 
\begin{equation} \label{eq: dfn of STOM in intro}
  \Pi(\cdot) \coloneqq \int_0^\infty \mathbf{1}_{(X_t, t)}(\cdot)\, dA_t,
\end{equation}
where $\mathbf{1}_{\cdot}$ denotes the indicator function.
We refer to $\Pi$ as the \emph{space-time occupation measure (STOM)} associated with $A$.
This notion extends the classical occupation time and occupation measure in the following sense.
If one takes $A_t = t$, $t \geq 0$, then, for any Borel set $E \subseteq S$ and $t \geq 0$,  
\begin{equation}
  \Pi(E \times [0,t]) = \int_0^t \mathbf{1}_E(X_s)\, ds.
\end{equation}
Thus, fixing $E$ and varying $t$ recovers the occupation time,
while fixing $t$ and varying $E$ recovers the occupation measure.

The introduction of STOMs is not merely a reformulation of PCAFs.  
It provides a new framework that we apply to the study of collisions of stochastic processes.  
As a particular class of STOMs, we introduce the notion of \emph{collision measures} for two processes,
which simultaneously record both the locations and the times of collisions.  
We establish convergence results for these measures,
and present several concrete examples to illustrate their applications.

The contributions of this paper can be summarized as follows. 
\begin{enumerate}[label = \textup{(\roman*)}]
  \item We prove that if processes, heat kernels, and smooth measures converge,
    then, under suitable assumptions on the latter two, 
    the associated PCAFs and STOMs also converge.
  \item We introduce the STOM framework as a new tool for the study of collisions of stochastic processes.
    Within this framework, we establish convergence results for collision measures
    and apply them to several concrete examples.  
\end{enumerate}
In the following subsections, we elaborate on these contributions in more detail.

\subsection{An overview of the framework} \label{sec: An overview of the framework}

Here we present an overview of the framework underlying our main results,
mainly focusing on the topologies used to discuss convergence of PCAFs and STOMs.

\medskip
We first recall some notion concerning PCAFs and smooth measures, 
which play a central role in this work.
For details, see Section~\ref{sec: Dual processes}.
Let $(S, d_S)$ be a boundedly-compact metric space (a \emph{$\bcmAB$ space}, for short),
that is, every bounded closed subset of $S$ is compact. 
Let 
\begin{equation}
  X = \bigl((X_t)_{t \ge 0}, (P^x)_{x \in S}\bigr)
\end{equation}
be a conservative standard process on $S$ satisfying the duality and absolute continuity condition (Assumption~\ref{assum: dual hypothesis}).
In particular, we assume that it admits a transition density (heat kernel)
$p \colon (0, \infty) \times S \times S \to [0, \infty]$
with respect to a certain invariant reference measure of $X$. 
Let $\mu$ be a Radon measure on $S$, 
i.e., a Borel measure such that its restriction $\mu|_K$ to every compact subset $K$ is finite.
We define its \emph{$\alpha$-potential} by 
\begin{equation}
  \Potential_p^\alpha \mu(x)
  \coloneqq 
  \int_S \ResolDensity^\alpha(x, y)\, \mu(dy),
  \quad x \in S,
\end{equation}
where $\ResolDensity^\alpha$ denotes the $\alpha$-resolvent density associated with $p$,
namely,
\begin{equation}
  \ResolDensity^\alpha(x,y) \coloneqq 
  \int_0^\infty e^{-\alpha t} p(t,x,y)\, dt.
\end{equation}
The measure $\mu$ is said to belong to the \emph{local Kato class} 
if it does not charge any semipolar set (see Definition~\ref{dfn: exceptional set} below) and, for every compact subset $K \subseteq S$,
\begin{equation}
  \|\Potential_p^\alpha (\mu|_K)\|_\infty 
  = 
  \sup_{x \in S} \int_K \ResolDensity^\alpha(x,y)\, \mu(dy)  
  \xrightarrow[\alpha \to \infty]{} 0.
\end{equation}
If $\mu$ is in the local Kato class,
then it is a smooth measure of $X$,
and by the Revuz correspondence 
there exists a unique PCAF $A = (A_t)_{t \ge 0}$ of $X$ satisfying,
for all non-negative Borel measurable $f$ on $S$ and all $\alpha > 0$,
\begin{equation}
  E^x\!\left[ \int_0^\infty e^{-\alpha t} f(X_t)\, dA_t \right]
  = 
  \int_S \ResolDensity^\alpha(x,y)\, f(y)\, \mu(dy).
\end{equation}
The associated STOM is defined by \eqref{eq: dfn of STOM in intro}.

\medskip
\textbf{Spaces for PCAFs and STOMs.}
We next introduce the spaces used to describe convergence of PCAFs and STOMs.
We write $\upC(\RNp, \RNp)$ for the space of non-decreasing continuous functions from $\RNp$ to itself,
equipped with the compact-convergence topology, i.e., uniform convergence on compacts.
The conservativeness of $X$ ensures that each PCAF is finite at every finite time,
so that convergence of PCAFs can be discussed within $\upC(\RNp, \RNp)$.
For the same reason, each STOM is a Radon measure on $S \times \RNp$,
and one may consider convergence of STOMs under the vague topology.
However, the vague topology is, in general, too weak to capture convergence of the associated PCAFs.
Indeed, if $\Pi$ is the STOM associated with a PCAF $A$, then
\begin{equation} \label{eq: relation of STOM and PCAF}
  \Pi(S \times [0,t]) = A_t, \quad t \geq 0.
\end{equation}
Thus, even if $\Pi_n \to \Pi$ vaguely,
the above relation does not guarantee the convergence of the corresponding PCAFs when $S$ is non-compact.
To overcome this difficulty, 
we introduce a space $\STOMMeas(S \times \RNp)$
containing all STOMs,
and equip it with a Polish topology 
such that convergence in this space implies convergence of the associated PCAFs in $\upC(\RNp, \RNp)$.
(See Section~\ref{sec: The space for STOMs} for details.)
In the remainder of this introduction, 
we mainly focus on convergence of STOMs, 
as convergence of PCAFs in this setting follows from that of the corresponding STOMs.
(Nevertheless, both PCAFs and STOMs will be treated in parallel in the main sections.)

\begin{rem}
  We make a few comments on the conservativeness assumption.
  \begin{enumerate}[label=\textup{(\alph*)}]
    \item  
    Without conservativeness, 
    a PCAF may blow up in finite time,
    in which case the space $\upC(\RNp, \RNp)$ is no longer suitable for treating PCAFs. 
    By instead considering the associated STOMs,
    one can still discuss convergence in a natural topology 
    without assuming conservativeness. 
    This demonstrates the generality of the STOM framework.
    See Section~\ref{sec: approx of PCAFs and STOMs} 
    (in particular, Remark~\ref{rem: flexibility of STOM framework}) 
    for further details.

    \item 
    Conservativeness is also, in a sense, a natural assumption when discussing convergence of stochastic processes in a general setting.
    In the presence of explosion, usual topologies such as the $J_1$-Skorohod topology
    are overly sensitive to the behavior of paths near explosion times,
    and one must carefully choose an appropriate topology.
    See \cite{Gradinaru_Haugomat_18_LocalSkorohod} for recent developments in this direction.
  \end{enumerate}
\end{rem}

\medskip
\textbf{Law maps and their continuity.}
We write $D(\RNp, S)$ for the space of $S$-valued \cadlag\ functions on $\RNp$,
equipped with the usual $J_1$-Skorohod topology.
Fix a Radon measure $\mu$ in the local Kato class, and denote the associated STOM by $\Pi$.
For each $x \in S$, set 
\begin{equation}
  \ProcLaw_X(x) \coloneqq P^x(X \in \cdot),
  \quad 
  \ProcLaw_{(X, \Pi)}(x) \coloneqq P^x((X, \Pi) \in \cdot),
\end{equation}
which are probability measures on $D(\RNp, S)$ and on 
$D(\RNp, S) \times \STOMMeas(S \times \RNp)$, respectively.
We regard these as maps from $S$, and refer to them as \emph{law maps}.

For a wide class of processes (for example, Feller processes),
one can verify that the distribution of $X$ depends continuously on the starting point,
that is, the map 
\begin{equation}
  \ProcLaw_X \colon S \to \Prob(D(\RNp, S))
\end{equation}
is continuous, where $\Prob(D(\RNp, S))$ denotes the space of probability measures on $D(\RNp, S)$
endowed with the weak topology.
In such cases, Theorem~\ref{thm: R map STOM approx for local Kato} below shows that
$\ProcLaw_{(X, \Pi)}$ inherits the continuity of $\ProcLaw_X$
under an additional joint continuity assumption on the heat kernel.
This result serves as a cornerstone for our main theorems,
as convergence of PCAFs and STOMs will be formulated in terms of convergence of the corresponding law maps.
(See Remark~\ref{rem: why law map conv} below for an explanation of this.)
The proof is based on an approximation argument,
which also forms the essential backbone of the proofs of our main results.
A brief outline of this approximation method is given in the next subsection.

\begin{rem} \label{rem: why law map conv}
  Distributional convergence of Markov processes is usually formulated 
  for a fixed starting point (or a fixed sequence of starting points).
  However, in many cases the proof implicitly yields convergence 
  whenever the starting points themselves converge.
  Indeed, a standard argument based on tightness and convergence of semigroups typically suffices.
  This mode of convergence is equivalent to compact convergence of their law maps,
  provided that the maps are continuous.
  Although this viewpoint may seem less conventional,
  it is particularly natural in the Gromov--Hausdorff-type setting,
  where processes are defined on different spaces and there may be no canonical fixed starting point.
  Accordingly, throughout this paper, 
  convergence of processes and of their associated PCAFs/STOMs 
  is formulated as convergence of the corresponding law maps.
\end{rem}

\subsection{Main results} \label{sec: STOMs and their convergence}

Here we present the main results of this paper together with a sketch of their proofs.

\medskip
For each $n \ge 1$, let $(S_n, d_{S_n})$ be a $\bcmAB$ space equipped with a distinguished element $\rho_n$, 
called the \emph{root} of $S_n$.
Let $X_n$ be a conservative standard process on $S_n$ 
satisfying the duality and absolute continuity condition (Assumption~\ref{assum: dual hypothesis}),
and denote by $p_n$ its jointly continuous heat kernel.
We assume that the law map $\ProcLaw_{X_n}$ is continuous.
Fix a Radon measure $\mu_n$ on $S_n$ belonging to the local Kato class of $X_n$,
and let $\Pi_n$ be the associated STOM.
As discussed in the previous subsection, 
the joint law map $\ProcLaw_{(X_n, \Pi_n)}$ is also continuous.

To describe the limit objects, 
let $(S, d_S, \rho, X, p, \mu, \Pi)$ 
be a tuple satisfying the same conditions as $(S_n, d_{S_n}, \rho_n, X_n, p_n, \mu_n, \Pi_n)$.
In particular, $(S, d_S, \rho)$ is a rooted $\bcmAB$ space,
$X$ is a conservative standard process satisfying the same assumptions as $X_n$,
$p$ is its jointly continuous heat kernel,
$\mu$ is a Radon measure in the local Kato class of $X$,
and $\Pi$ is the corresponding STOM.

The following theorem is one of the main results of this paper.
The notation $\rootBCM(\cdot)$ denotes a Gromov--Hausdorff-type space,
as introduced in Section~\ref{sec: GH-type topologies} below.
For each $r > 0$ and $n \geq 1$, we write 
\begin{equation}
  S_n^{(r)} \coloneqq \{x \in S_n \mid d_{S_n}(\rho_n, x) \leq r \},
\end{equation}
and $\mu_n^{(r)}$ for the restriction of $\mu_n$ to $S_n^{(r)}$.

\begin{thm} \label{thm: intro STOM dtm}
  Assume that the following conditions are satisfied.
  \begin{enumerate}[label=\textup{(\roman*)}]
    \item It holds that 
      \begin{equation}
        (S_n, d_{S_n}, \rho_n, \mu_n, p_n, \ProcLaw_{X_n})
        \to
        (S, d_S, \rho, \mu, p, \ProcLaw_X)
      \end{equation}
      in $\rootBCM\bigl(\MeasSt \times \HKSt \times \SPMSt\bigr)$.
    \item
      For each $r > 0$,
      \begin{equation}
        \lim_{\alpha \to \infty}
        \limsup_{n \to \infty}
        \bigl\| \Potential_{p_n}^\alpha \mu_n^{(r)} \bigr\|_\infty
        = 0,
      \end{equation}
      or equivalently,
      \begin{equation}
        \lim_{\delta \to 0} 
        \limsup_{n \to \infty} 
        \sup_{x \in S_n^{(r)}}
        \int_0^\delta \int_{S_n^{(r)}} p_n(t, x, y)\, \mu_n(dy)\, dt
        = 0.
      \end{equation}
  \end{enumerate}
  Then
  \begin{equation}
    \bigl(S_n, d_{S_n}, \rho_n, \mu_n, a_n, p_n, \ProcLaw_{(X_n, \Pi_n)}\bigr)
    \to
    \bigl(S, d_S, \rho, \mu, a, p, \ProcLaw_{(X, \Pi)}\bigr)
  \end{equation}
  in $\rootBCM\bigl(\MeasSt \times \HKSt \times \SPMSTOMSt\bigr)$.
\end{thm}

We also obtain a corresponding result for random environments.
That is, assume that the tuples 
$\mathcal{S}_n = (S_n, d_{S_n}, \rho_n, \mu_n)$ and $\mathcal{S} = (S, d_S, \rho, \mu)$
are random rooted-and-measured $\bcmAB$ spaces.
For each realization of $\mathcal{S}_n$,
we consider a process $X_n$ with jointly continuous heat kernel $p_n$, 
satisfying the same conditions as in the deterministic case,
a Radon measure $\mu_n$ in the local Kato class of $X_n$,
and the associated STOM $\Pi_n$.
Similarly, for each realization of $\mathcal{S}$,
we have a process $X$ with jointly continuous heat kernel $p$,
a Radon measure $\mu$ in the local Kato class of $X$,
and the associated STOM $\Pi$.
(We also assume suitable measurability of these objects, 
so that we can discuss their laws in Gromov--Hausdorff-type spaces;
see Section~\ref{sec: PCAF conv for rdm spaces} for the precise setting.)
Under a random analogue of the above assumptions, 
we establish distributional convergence of the law maps $\ProcLaw_{(X_n, \Pi_n)}$.
Below, $\mathbf{E}_n$ denotes the expectation with respect to the randomness of the environment.

\begin{thm} \label{thm: intro STOM rdm}
   Assume that the following conditions are satisfied.
  \begin{enumerate}[label=\textup{(\roman*)}]
    \item It holds that 
      \begin{equation}
        (S_n, d_{S_n}, \rho_n, \mu_n, p_n, \ProcLaw_{X_n})
        \xrightarrow[]{\mathrm{d}}
        (S, d_S, \rho, \mu, p, \ProcLaw_X)
      \end{equation}
      as random elements of $\rootBCM\bigl(\MeasSt \times \HKSt \times \SPMSt\bigr)$.
    \item
      For each $r > 0$,
      \begin{equation}
        \lim_{\alpha \to \infty}
        \limsup_{n \to \infty}
        \mathbf{E}_n
        \Bigl[
          \bigl\| \Potential_{p_n}^\alpha \mu_n^{(r)} \bigr\|_\infty \wedge 1
        \Bigr]
        = 0,
      \end{equation}
      or equivalently,
      \begin{equation}
        \lim_{\delta \to 0} 
        \limsup_{n \to \infty} 
        \mathbf{E}_n\!   
        \left[
          \left(
            \sup_{x \in S_n^{(r)}}
            \int_0^\delta \int_{S_n^{(r)}} p_n(t, x, y)\, \mu_n(dy)\, dt
          \right)
          \wedge 1
        \right]
        = 0.
      \end{equation}
  \end{enumerate}
  Then
  \begin{equation}
    \bigl(S_n, d_{S_n}, \rho_n, \mu_n, a_n, p_n, \ProcLaw_{(X_n, \Pi_n)}\bigr)
    \xrightarrow[]{\mathrm{d}}
    \bigl(S, d_S, \rho, \mu, a, p, \ProcLaw_{(X, \Pi)}\bigr)
  \end{equation}
  as random elements of $\rootBCM\bigl(\MeasSt \times \HKSt \times \SPMSTOMSt\bigr)$.
\end{thm}

The above theorems follow from Theorems~\ref{thm: PCAF/STOM dtm result} and \ref{thm: PCAF/STOM rdm result}
(see also Remark~\ref{rem: replace L^0 by J^1 for dtm main result}), respectively,
which are formulated in a slightly more general setting where the convergence of processes is assumed in a weaker topology.
See Remark~\ref{rem: why L^0} for further discussions on this point.

\medskip
\noindent
\textbf{Sketch of the proofs.}
The main proof idea is to approximate STOMs using heat kernels as mollifiers.
For each $\delta > 0$, set
\begin{equation}
  \Pi_n^{(\delta, *)}(dz\,dt)
  \coloneqq
  \frac{1}{\delta}
  \int_\delta^{2\delta}
  p_n(u, X_n(t), z)\, du\, \mu_n(dz)\, dt.
\end{equation}
From the convergence of $\mu_n$, $p_n$, and $\ProcLaw_{X_n}$,
we deduce convergence of the law maps $\ProcLaw_{(X_n, \Pi_n^{(\delta, *)})}$.
Using moment estimates for differences of PCAFs established in \cite{Noda_pre_Continuity},
we bound the expected distance between $\Pi_n^{(\delta, *)}$ and $\Pi_n$
by the potential norm $\|\Potential_{p_n}^\alpha \mu_n^{(r)}\|_\infty$.
Thus, applying condition~(ii) in Theorem~\ref{thm: intro STOM dtm} (or Theorem~\ref{thm: intro STOM rdm}),
we obtain that $\ProcLaw_{(X_n, \Pi_n^{(\delta_n, *)})} \to \ProcLaw_{(X, \Pi)}$ as $\delta_n \to 0$
uniformly in $n$,
which yields the desired convergence.

\begin{rem} \label{rem: why L^0}
  Our fundamental assumption concerns the convergence of processes.
  To impose the weakest possible requirement, 
  we employ the \emph{$L^0$ topology} on $D(\RNp, S)$,
  induced by the metric
  \begin{equation}
    \LzeroMet_S(f, g)
    \coloneqq
    \int_0^\infty e^{-t} (1 \wedge d_S(f_t, g_t))\, dt,
    \quad f, g \in D(\RNp, S)
  \end{equation}
  (see Section~\ref{sec: Topologies on the space of cadlag functions} for details of this topology).
  This topology is substantially weaker than the compact convergence, $J_1$-Skorohod, or $M_1$ topologies,
  thereby broadening the scope of the framework.
  Accordingly, in our main results, Theorems~\ref{thm: PCAF/STOM dtm result} and \ref{thm: PCAF/STOM rdm result},
  the convergence of the joint laws of processes and STOMs 
  is formulated so that the process components converge with respect to the $L^0$ topology.
  Nevertheless, when convergence of processes holds under a stronger Polish topology,
  the same conclusions automatically extend to that topology
  by a simple tightness argument 
  (see Appendix~\ref{appendix: Replacement of the L^0 topology by stronger topologies}).
  Thus the use of the $L^0$ topology does not restrict generality,
  while keeping the framework maximally flexible.
  Indeed, this level of generality is particularly important for applications to the study of collision, 
  see Remark~\ref{rem: L^0 is good for col}.
\end{rem}

\subsection{Application to the study of collisions} \label{sec: Application to the study of collisions}

Let $(X^1_t)_{t \geq 0}$ and $(X^2_t)_{t \geq 0}$ be two (independent) stochastic processes 
taking values in a common state space $S$.
Given a realization of such processes,
we say that $t \in [0,\infty)$ is a \emph{collision time} if $X^1_t = X^2_t = x$ for some $x \in S$,
in which case we refer to $x$ as the \emph{collision site} (at time $t$).
Based on the framework of STOMs,
we introduce the notion of \emph{collision measures} 
for a broad class of pairs of stochastic processes.
These measures capture collision times and sites simultaneously, 
allowing us to treat collisions in a unified manner. 
By applying the main convergence theorems for STOMs, 
we obtain convergence results for the corresponding collision measures.

\medskip

\textbf{Background.}
The study of collisions of independent stochastic processes 
can be traced back to Pólya’s classical work~\cite{Polya_1921_recurrence} 
on the recurrence of simple random walks (SRWs) on Euclidean lattices.
On such lattices, 
recurrence is equivalent to the \emph{infinite collision property (ICP)}, 
namely the property that two independent SRWs collide infinitely many times with probability one.
However, this equivalence fails on general graphs~\cite{Krishnapur_Peres_04_Comb},
and the ICP has been extensively studied for various (random) graph models 
(cf.\ \cite{Barlow_Peres_Sousi_12_Collision,Croydon_Ambroggio_24_TripleCollision,
Halberstam_Hutchcroft_22_Collision,Hutchcroft_Peres_15_Collision,Watanabe_23_ICPfor3dUST}).
In the continuous setting, 
most studies have focused on the Hausdorff dimension of the set of collision times or collision sites 
(cf.\ \cite{Jain_Pruitt_69_Collisions,Knopova_Schilling_15_CollisionofFeller,Shiozawa_Wang_24_HausdorffDim}).

The scaling limit of collision times and sites has, to the best of our knowledge, 
barely been investigated in the existing literature.
Even in a very simple case, the one-dimensional lattice $\mathbb{Z}$, 
no explicit results were known until the recent work by Nguyen~\cite{Nguyen_23_Collision}. 
He showed that the point process consisting of pairs of collision times and sites of two SRWs on $\mathbb{Z}$
converges, under Brownian scaling, 
to a non-trivial random measure.
However, his approach essentially relies on the scaling limit of the partition function 
for a directed polymer model 
and thus does not extend easily to more general settings. 
Moreover, the limiting measure was characterized only via moment identities, 
leaving its pathwise connection with the collisions of the limiting Brownian motions unclear.

In contrast, our formulation provides a direct, pathwise definition of collision measures 
for a wide class of stochastic processes, 
together with a robust convergence theorem derived from the STOM framework.

\medskip
\textbf{Collision measure.} 
Here we define collision measures. 
See Section~\ref{sec: collision measure} for details.
Let $S$ be a $\bcmAB$ space. 
For each $i \in \{1,2\}$, 
let $X^i$ be a conservative standard process on $S$ 
satisfying the duality and absolute continuity condition 
(Assumption~\ref{assum: dual hypothesis}) 
with heat kernel $p^i$ and reference measure $m^i$. 
We define the associated \emph{product process} 
$\hat{X} = ((\hat{X}_t)_{t \geq 0}, (\hat{P}^{\bm{x}})_{\bm{x} \in S \times S})$ by
\begin{alignat}{3}
  \hat{X}_t &\coloneqq& (X^1_t, X^2_t), &\quad& t &\geq 0,\\
  \hat{P}^{\bm{x}} &\coloneqq& P_1^{x_1} \otimes P_2^{x_2}, &\quad& \bm{x} &= (x_1, x_2) \in S \times S.
\end{alignat}
Then $\hat{X}$ is again a conservative standard process 
satisfying the duality and absolute continuity condition. 
In particular, the STOM framework applies to this product process.

Define the \emph{diagonal map} $\diag \colon S \to S \times S$ by $\diag(x) \coloneqq (x, x)$. 
The collision of $X^1$ and $X^2$ can be rephrased as the event that $\hat{X}$ hits the diagonal (cf.\ \cite{Shiozawa_Wang_24_HausdorffDim}). 
Specifically, for any $t \geq 0$ and any subset $B \subseteq S$, 
\begin{equation}
  X^1_t = X^2_t \in B
  \quad \Longleftrightarrow \quad
  \hat{X}_t \in \diag(B).
\end{equation}
This observation naturally leads to the following definition of collision measures.

Fix a Radon measure $\mu$ on $S$. 
We write $\diagMeas{\mu} \coloneqq \mu \circ \diag^{-1}$ 
for the pushforward of $\mu$ by $\diag$. 
Assume that $\diagMeas{\mu}$ is a smooth measure of $\hat{X}$, 
and let $\hat{\Pi}$ denote the associated STOM of $\hat{X}$. 
We then define a random measure $\Pi$ on $S \times [0, \infty)$ by
\begin{equation} \label{eq: intro dfn of col meas}
  \Pi \coloneqq \hat{\Pi} \circ (\diag \times \id_{\RNp}),
\end{equation}
where $\id_{\RNp}$ denotes the identity map on $\RNp = [0, \infty)$. 
We refer to $\Pi$ as the \emph{collision measure} of $X^1$ and $X^2$ associated with $\mu$.
The measure $\mu(dx)$ plays a role in determining the contribution of collisions at each site $x$; 
we therefore refer to $\mu$ as the \emph{weighting measure} (see \eqref{eq: intro discrete col meas}).
In Remarks~\ref{rem: coincidence with Nguyen} and \ref{rem: more comment about Nguyen} below, 
we verify that the limiting random measure obtained by Nguyen~\cite{Nguyen_23_Collision} 
can be interpreted as a collision measure within our framework.

For example, when $S$ is discrete, 
\begin{equation} \label{eq: intro discrete col meas}
  \Pi(dx\, dt)
  =
  \frac{1}{m^1(\{x\})\, m^2(\{x\})}\,
  \mathbf{1}_{\{X^1_t = X^2_t = x\}}\,
  \mu(dx)\, dt,
\end{equation}
(see Proposition~\ref{prop eq: collision process for discrete space}).
This expression suggests that the natural choice of weighting measure $\mu$
is the pointwise product of the reference measures,
that is,
\begin{equation} \label{eq: intro canonical weighting}
  \mu(\{x\}) = (m^1 \cdot m^2)(\{x\}) \coloneqq m^1(\{x\}) m^2(\{x\}), \quad x \in S.
\end{equation}
Indeed, the collision measure $\Pi$ associated with this weighting measure
assigns a unit contribution to each collision at every site.
Namely, we have
\begin{equation} \label{eq: intro uniform col meas}
  \Pi(dx\, dt)
  = \sum_{y \in S} 
    \mathbf{1}_{\{X^1_t = X^2_t = y\}}\,
    \delta_y(dx)\, dt,
\end{equation}
where $\delta_y$ denotes the Dirac measure at $y$.
In particular, for each $x \in S$ and $t \geq 0$,
\begin{equation}
  \Pi(\{x\} \times [0,t]) = \int_0^t \mathbf{1}_{\{X^1_s = X^2_s = x\}}\, ds,
\end{equation}
which represents the total amount of time up to $t$ 
during which $X^1$ and $X^2$ collide at $x$.
We refer to $\Pi$ in \eqref{eq: intro uniform col meas} as the \emph{uniform collision measure},
and refer to the associated weighting measure $\mu = m^1 \cdot m^2$ as the \emph{canonical weighting measure}.

\medskip
\textbf{Convergence of collision measures.} 
Although the definition \eqref{eq: intro dfn of col meas} is not 
in the direct form to which the continuous mapping theorem applies,
an inverse version of the continuous mapping theorem 
(see Appendix~\ref{appendix: An inverse version of the continuous mapping theorem})
can be employed to handle this situation. 
In particular, the convergence of collision measures can be reduced 
to that of the associated STOMs.
Hence we do not restate the full setup and theorems here.
Instead, we provide a brief outline for clarity.
See Theorems~\ref{thm: col dtm result} and \ref{thm: col rdm result} for precise formulations.

For each $n \geq 1$,
suppose that
$X_n^1$ and $X_n^2$ are independent conservative standard processes
living on a common rooted $\bcmAB$ space $(S_n, d_{S_n}, \rho_n)$
with heat kernels $p_n^1$ and $p_n^2$, respectively,
and let $\mu_n$ be a weighting measure.
Write $\Pi_n$ for the collision measure of $X_n^1$ and $X_n^2$ associated with $\mu_n$.
By Theorem~\ref{thm: intro STOM dtm},
we deduce the following:
if the rooted metric spaces $(S_n, d_{S_n}, \rho_n)$,
the weighting measures $\mu_n$,
the heat kernels $p_n^1$ and $p_n^2$,
and the law maps $\ProcLaw_{X^1_n}$ and $\ProcLaw_{X^2_n}$
converge in a Gromov--Hausdorff type topology,
and if 
\begin{equation} \label{eq: intro heat cond}
  \lim_{\delta \to 0} 
  \limsup_{n \to \infty} 
  \sup_{x_1, x_2 \in S_n^{(r)}}
  \int_0^\delta \int_{S_n^{(r)}} 
    p_n^1(t, x_1, y)\, p_n^2(t, x_2, y)\, 
    \mu_n(dy)\, dt
  = 0,
  \quad 
  \forall r > 0,
\end{equation}
then the joint law maps $\ProcLaw_{(X_n^1, X_n^2, \Pi_n)}$
converge to the limiting one.
An analogous result also holds for random spaces,
after suitable modification of the assumptions, 
as in the case of Theorem~\ref{thm: intro STOM rdm}.

\medskip
\textbf{Integration with the framework of resistance metric spaces.} 
Collisions of independent stochastic processes are phenomena
that typically occur in spaces of low spectral dimension.
The framework of \emph{resistance metric spaces} due to Kigami~\cite{Kigami_01_Analysis,Kigami_12_Resistance}
is particularly useful for analyzing such low-dimensional behaviors,
where point recurrence typically holds.
It was originally developed for analysis on fractals,
and recent studies have revealed that this framework is also highly effective 
in describing scaling limits of random walks and related objects on critical random graphs
\cite{Croydon_Hambly_Kumagai_17_Time-changes,Croydon_18_Scaling,Noda_pre_Convergence,Noda_pre_Aging,Noda_pre_Scaling}.

The key advantage of working with measured resistance metric spaces lies in the fact that
each such space naturally carries a symmetric Markov process 
through the theory of Dirichlet forms~\cite{Fukushima_Oshima_Takeda_11_Dirichlet},
and the associated process admits a jointly continuous heat kernel.
Moreover, convergence of measured resistance metric spaces
in Gromov--Hausdorff-type topologies
implies convergence of the corresponding processes and their heat kernels,
under the so-called non-explosion condition (see Section~\ref{sec: Preliminary results on stoch proc on resis sp}).
Consequently, once the convergence of the underlying spaces is established,
the convergence of the associated collision measures 
reduces essentially to verifying two ingredients:
\begin{enumerate} [label = \textup{(\Alph*)}]
  \item \label{item: intro conv of weighting meas}
    the convergence of weighting measures,
  \item \label{item: intro heat cond}
    the heat-kernel condition~\eqref{eq: intro heat cond}.
\end{enumerate}
The latter can be checked by heat kernel estimates based on volume estimates.

\medskip
\textbf{Collisions of variable-speed random walks (VSRWs).}
In Section~\ref{sec: conv of col meas of VSRW}, we develop general convergence theorems (Theorems~\ref{thm: dtm col meas for VSRW} and \ref{thm: rdm col meas for VSRW})
that yields, from scaling limits of electrical networks,
scaling limits of uniform collision measures of two independent variable-speed random walks (VSRWs) living on those networks,
where a VSRW refers to a continuous-time random walk whose jump rate between two vertices is proportional to the network conductance
(see Section~\ref{sec: resistance preliminary} below).
A key point is that, although the canonical weighting measure is defined as the squared invariant measure~\eqref{eq: intro canonical weighting}, 
the invariant measure for VSRWs is the counting measure.
Hence, squaring does not alter it,
and condition~\ref{item: intro conv of weighting meas} automatically follows 
from the convergence of the underlying measured spaces.
This also simplifies condition~\eqref{eq: intro heat cond} via the Chapman--Kolmogorov equation
to the following:
\begin{equation} \label{eq: intro simple heat cond}
  \lim_{\delta \to 0} 
  \limsup_{n \to \infty} 
  \sup_{x_1, x_2 \in S_n^{(r)}}
  \int_0^\delta
  p_n(2t, x_1, y)\, 
  dt
  = 0,
  \quad 
  \forall r > 0.
\end{equation}
(Note that $p_n^1 = p_n^2$ in this setting, since we consider collisions of independent and identically distributed (i.i.d.)\ random walks.)
In Theorems~\ref{thm: dtm col meas for VSRW} and \ref{thm: rdm col meas for VSRW} below,
we provide simple sufficient conditions, expressed in terms of lower volume estimates, to verify the above (and its random version) condition.
The results apply to a broad range of critical random graphs, including 
critical Galton--Watson trees, 
the uniform spanning tree on $\mathbb{Z}^d$ with $d = 2,3$, and on high-dimensional tori, 
the critical Erd\H{o}s--R\'enyi graph, and the critical configuration model
(see Example~\ref{exm: rdm VSRW}).

\medskip
\textbf{Collisions of constant-speed random walks (CSRWs).}
The situation for \emph{constant-speed random walks} (CSRWs) is markedly subtler,
where a CSRW refers to a continuous-time random walk with a constant jump rate equal to $1$.
Here, the invariant measure is the conductance measure 
(which reduces to the degree measure under unit conductances),
so the canonical weighting measure becomes the \emph{squared conductance measure},
whose convergence is genuinely nontrivial and must be analyzed on a model-by-model basis.
Under the assumption that the squared conductance measures converge under 
the same scaling as the invariant measures,
Theorems~\ref{thm: dtm col meas for CSRW} and \ref{thm: rdm col meas for CSRW} below
provide sufficient conditions, again in terms of lower volume estimates, to verify condition~\ref{item: intro heat cond}.
In Section~\ref{sec: Critical Galton--Watson trees},
we apply this result to critical Galton--Watson trees conditioned on size
whose offspring distribution has a finite $(3+\varepsilon)$-moment for some $\varepsilon > 0$,
establishing the convergence of the squared conductance measures 
and, consequently, the scaling limit of uniform collision measures 
of two independent CSRWs on those trees.

CSRWs are closely related to discrete-time random walks,
since they share the same jump chain and hence are expected to exhibit identical scaling behavior.
In this sense, scaling limits of CSRW collision measures naturally indicate
the corresponding limits for discrete-time random walks,
although making this connection rigorous requires additional analysis
because collision measures involve occupation-time quantities
(see~\cite{Noda_pre_Scaling} for a bridging argument transferring the CSRW result 
to its discrete-time counterpart in the case of local times).

\medskip
\textbf{Outlook:\ persistence of inhomogeneities and future directions.}
The study of squared invariant measures appears to be new,
and it provides an intriguing perspective for future research.
For instance, in a version of one-dimensional Bouchaud trap models \cite{Arous_Cerny_06_Dynamics,Bouchaud_Cugliandolo_etc_97_Out,Fontes_Isopi_Newman_02_Random}
where the tail of the trap distribution lies in a regime with finite mean but infinite second moment,
the invariant measure converges to the Lebesgue measure,
whereas the squared measure converges to a limit that is singular with respect to the Lebesgue measure;
namely, the invariant measure of the Fontes--Isopi--Newman diffusion~\cite{Fontes_Isopi_Newman_02_Random}.
This indicates that, while each individual walk homogenizes to Brownian motion, 
the collision sites may still reflect microscopic irregularities of the environment,
accumulating near deep traps.
(A full verification of the heat-kernel condition in this setting remains open,
though recent work~\cite{Andres_Croydon_Kumagai_24_HKEon1dBTM} provides partial evidence towards its validity.)
Similar behavior is also expected in other settings,
such as critical Galton--Watson trees with heavy-tailed offspring distributions
or random graphs with strongly inhomogeneous degree sequences.
Such a phenomenon---where the underlying processes lose their inhomogeneity in the limit
but their collisions retain it---offers a new viewpoint on interactions in disordered systems.
In this sense, collision measures provide a novel observable 
that can detect microscopic irregularities even when homogenization occurs at the process level.
These observations suggest rich possibilities 
for exploring how local inhomogeneities may persist through scaling limits in multi-particle systems.

\subsection{Organization of the paper and notational conventions} \label{sec: notation}

The remainder of the paper is organized as follows.
Section~\ref{sec: Preliminaries on topologies} summarizes the topologies and their metrizations that will be used throughout.
Section~\ref{sec: GH-type topologies} develops a general Gromov--Hausdorff-type framework for metric spaces equipped with additional structures, 
which serves as the basis for formulating the convergence results of the paper.
Section~\ref{sec: Dual processes, PCAFs, and smooth measures} reviews the notion of dual processes, smooth measures, PCAFs, and STOMs, 
together with the Revuz correspondence linking them.
Section~\ref{sec: approx of PCAFs and STOMs} introduces approximation schemes for PCAFs and STOMs 
and establishes continuity of their joint laws with the underlying processes.
Section~\ref{sec: conv of PCAFs and STOMs} presents the main convergence theorems for PCAFs and STOMs on both deterministic and random spaces.
Section~\ref{sec: Collision measures and their convergence} 
applies these results to define collision measures and to prove their convergence.
Section~\ref{sec: Preliminary results on stoch proc on resis sp} and Section~\ref{sec: col of proc on resis sp} 
further specialize the general theory to stochastic processes on resistance metric spaces, 
culminating in concrete scaling limit results for collision measures of variable-speed and constant-speed random walks on critical random graphs.
Finally, several auxiliary results are collected in the appendices, 
including an inverse version of the continuous mapping theorem.

Below we list the notational conventions used throughout this paper.
\begin{enumerate} [label = \textup{(\arabic*)}]
  \item \label{note: integers} 
    We denote by $\ZN$, $\ZNp$, and $\NN$,
    the sets of integers, non-negative integers, and positive integers, respectively.
  \item \label{note: real number}
    We write $\RNp \coloneqq [0, \infty)$ and $\RNpp \coloneqq (0, \infty)$,
    and equip each of these spaces with the usual Euclidean metric.
    The space $[0, \infty]$ is the usual one-point compactification of $[0, \infty)$.
  \item \label{note: max min}
    For $a, b \in \RN \cup \{ \pm \infty \}$,
    we write $a \wedge b \coloneqq \min\{a, b\}$ and $a \vee b \coloneqq \max\{a, b\}$.
  \item \label{note: sup norm}
    For a non-empty set $S$ and $f \colon S \to [-\infty, \infty]$, 
    we write $\|f\|_\infty \coloneqq \sup\{|f(x)| \mid x \in S\}$.
  \item \label{note: Lebesgue}
    We denote by $\Leb$ the Lebesgue measure on $\RN^d$, 
    where the dimension is clear from the context.
  \item \label{note: indicator function}
    For a set $S$ and a subset $S_0 \subseteq S$,
    $\mathbf{1}_{S_0}$ denotes the indicator function of $S_0$,
    i.e., $\mathbf{1}_{S_0}(x) = 1$ for $x \in S_0$, and $\mathbf{1}_{S_0}(x) = 0$ for $x \in S \setminus S_0$.
  \item \label{note: Bore subsets}
    For a topological space $S$, we write $\Borel(S)$ for the set of Borel subsets of $S$.
  \item \label{note: metric space is non-empty}
    When we say that $S$ is a metric space, we always assume that $S$ is nonempty 
    and that its associated metric is denoted by $d_S$, unless otherwise specified.
  \item \label{note: balls in metric space}
    Given a metric space $S$, we write 
    \begin{gather}
      B_S(x, r) = B_{d_S}(x,r) 
      \coloneqq
      \{ y \in S \mid d_S(x,y) < r \},\\
      D_S(x,r) = D_{d_S}(x,r)
      \coloneqq 
      \{ y \in S \mid d_S(x,y) \leq r \}.
    \end{gather}
  \item \label{note: closure}
    Given a topological space $S$ and a subset $A$ of $S$,
    we denote by $\closure(A) = \closure_S(A)$ the  closure of $A$ in $S$.
  \item \label{note: support}
    Let $S$ be a topological space.
    Given a function $f \colon S \to \RN$,
    we write $\supp(f)$ for the (closed) support of $f$.
    Given a Borel measure $\mu$ on $S$,
    we write $\supp(\mu)$ for the (topological) support of $\mu$.
  \item \label{note: identity}
    For a set $A$,
    we denote the identity map from $A$ to itself by $\id_{A}$.
  \item \label{note: product map}
    Given maps $f_i \colon X_i \to Y_i$, $i = 1,2$, 
    we define $f_1 \times f_2 \colon X_1 \times X_2 \to Y_1 \times Y_2$ by setting $(f_1 \times f_2)(x_1, x_2) \coloneqq (f_1(x_1), f_2(x_2))$.
  \item \label{note: product space}
    When $X$ and $Y$ are topological spaces, we always equip $X \times Y$ with the product topology.
    Moreover, if $X$ and $Y$ are metric spaces, then we always equip $X \times Y$ with the \emph{max product metric} defined as follows: 
    \begin{equation}
      d_{X \times Y}((x_1, y_1), (x_2, y_2)) \coloneqq d_X(x_1, x_2) \vee d_Y(y_1, y_2).
    \end{equation}
  \item \label{note: topological embedding}
    We say that a map $f \colon X \to Y$ between topological spaces is a \emph{topological embedding} 
    if and only if it is a homeomorphism onto its image with the relative topology.
  \item \label{note: isometric embedding}
    We say that a map $f \colon X \to Y$ between metric spaces is an \emph{isometric embedding} (resp.\ \emph{isometry})
    if and only if it is distance-preserving (resp.\ distance-preserving and bijective).

  \item \label{note: bcm}
    A metric space is said to be boundedly compact if every closed and bounded subset is compact.
    We refer to such a space as a \emph{$\bcmAB$ space}.
\end{enumerate}

\section{Preliminaries on topologies and metrizations} \label{sec: Preliminaries on topologies}

This paper deals with various objects such as measures, stochastic processes, and heat kernels. 
The aim of this section is to summarize the topologies on these objects and their metrizations,
which will be used throughout the paper.
This section is based on \cite[Section~3]{Noda_pre_Metrization}.

\subsection{The Hausdorff and Fell topologies} \label{sec: The Hausdorff and Fell topologies}

In this subsection, we recall the Hausdorff and Fell topologies, 
which are standard topologies on the collections of compact and closed subsets, respectively. 
For detailed properties of these topologies as metric spaces, see \cite[Section~2.1]{Noda_pre_Metrization}, 
and for further topological background see \cite[Appendix~C]{Molchanov_17_Theory}.
In this subsection, we fix a $\bcmAB$ space $(S, d_S)$.

\begin{dfn} \label{dfn: space of compact and closed subsets}
  We define $\Closed{S}$ to be the set of closed subsets of $S$.
  We denote by $\Compact{S}$ the subset of $\Closed{S}$ consisting of compact subsets.
  (NB.\ Both sets include the empty set.)
\end{dfn}

We equip $\Compact{S}$ with the Hausdorff metric $\HausMet{S}$.
To recall it,
we write, for each subset $A \subseteq S$ and $\varepsilon \geq 0$,
\begin{equation}  \label{eq: e-neighborhood}
  A^{\lrangle{\varepsilon}}
  \coloneqq
  \{
    x \in S \mid \exists y \in A\ \text{such that}\ d_S(x,y) \leq \varepsilon
  \},
\end{equation}
which is the \textit{(closed) $\varepsilon$-neighborhood} of $A$ in $S$.
The Hausdorff metric $\HausMet{S}$ on $\Compact{S}$ is then defined by
\begin{equation} \label{eq: dfn of Hausdorff metric}
  \HausMet{S}(A,B)
  \coloneqq
  \inf \!\left\{
    \varepsilon \geq 0 \mid A \subseteq B^{\lrangle{\varepsilon}},\, B \subseteq A^{\lrangle{\varepsilon}}
  \right\},
\end{equation} 
where we set the infimum over the empty set to be $\infty$.
The function $\HausMet{S}$ is indeed a complete and separable (extended) metric on $\Compact{S}$
(see \cite[Section~17.6]{Cech_69_Point}).
(Note that the distance between the empty set and a non-empty set is always infinite.)
We call the topology on $\Compact{S}$ induced by $\HausMet{S}$ the \textit{Hausdorff topology}
(also known as the \emph{Vietoris topology}).

We now fix a distinguished element $\rho \in S$, which we call the \emph{root} of $S$.
To extend $\HausMet{S}$ to a metric on $\Closed{S}$,
we write, for each $A \in \Closed{S}$ and $r > 0$, 
\begin{equation}
  A|_\rho^{(r)} \coloneqq A \cap D_S(\rho, r).
\end{equation}
When the root is clear from the context, we write $A^{(r)} \coloneqq A|_\rho^{(r)}$.
We then define, for each $A, B \in \Closed{S}$,  
\begin{equation}
  \FellMet{S, \rho}(A, B) 
  \coloneqq 
  \int_0^\infty e^{-r} \bigl(1 \wedge \HausMet{S}(A|_\rho^{(r)}, B|_\rho^{(r)}) \bigr)\, dr.
\end{equation}
The function $\FellMet{S, \rho}$ is a metric on $\Closed{S}$ (\cite[Theorem~3.8]{Noda_pre_Metrization}) 
and induces the Fell topology
(see \cite{Molchanov_17_Theory} for the Fell topology, for example).
In particular, the topology is compact and only depends on the topology on $S$.

\subsection{The weak and vague topologies} \label{sec: The weak and vague topologies}

We use the weak and vague topologies to discuss convergence of probability measures and Radon measures, respectively. 
In this subsection, we recall the definitions and metrizations of these topologies. 
For background, see \cite{Billingsley_99_Convergence} for the weak topology and \cite{Kallenberg_17_Random} for the vague topology.
We fix a separable metric space $(S, d_S)$.

\begin{dfn} 
  We define $\Meas(S)$ to be the set of Radon measures $\mu$ on $S$,
  that is, 
  $\mu$ is a Borel measure on $S$ such that $\mu(K) < \infty$ for every compact subset $K$.
  We denote by $\finMeas(S)$ (resp.\ $\Prob(S)$) 
  the subset of $\Meas(S)$ consisting of finite Borel measures (resp.\ probability measures).
\end{dfn}

A commonly used metric on $\finMeas(S)$ is the \emph{Prohorov metric},
given as follows:
for $\mu, \nu \in \finMeas(S)$,
\begin{equation}
  \ProhMet{S}(\mu, \nu)
  \coloneqq
  \inf
  \left\{
    \varepsilon > 0 \mid
    \mu(A) \leq \nu(A^{\lrangle{\varepsilon}}) + \varepsilon,\,
    \nu(A) \leq \mu(A^{\lrangle{\varepsilon}}) + \varepsilon
    \ \text{for all Borel subsets } A \subseteq S
  \right\}.
\end{equation} 
The Prohorov metric $\ProhMet{S}$ induces the weak topology on $\finMeas(S)$,
that is, finite Borel measures $\mu_n$ converge to $\mu$ in $\finMeas(S)$ if and only if,
for all bounded continuous function $f \colon S \to \RN$,
\begin{equation}
  \lim_{n \to \infty} 
  \int_S f(x)\, \mu_n(dx) 
  = 
  \int_S f(x)\, \mu(dx).
\end{equation}
We also equip $\Prob(S)$ with the Prohorov metric. 

The following lemma is useful for estimating the Prohorov distance 
between the laws of random elements defined on a common probability space.

\begin{lem} [{\cite[Lemma~6.6]{Noda_pre_Aging}}] \label{lem: Prohorov estimate}
  Let $X$ and $Y$ be random elements of $S$ defined on a common probability space with probability measure $P$.
  Suppose that on an event $E$, we have $d_S(X, Y) \leq \varepsilon$ almost surely.
  Then it holds that 
  \begin{equation}
    \ProhMet{S}
    \bigl(
      P(X \in \cdot), P(Y \in \cdot)
    \bigr)
    \leq 
    P(E^{c}) + \varepsilon.
  \end{equation}
\end{lem}

We now assume that $(S, d_S)$ is boundedly compact 
and fix an element $\rho \in S$ as the root of $S$. 
We extend the Prohorov metric to $\Meas(S)$ in a way analogous to the metric $\FellMet{S,\rho}$ for the Fell topology.
We write, for each $\mu \in \Meas(S)$ and $r > 0$, 
\begin{equation} \label{eq: restriction of meas}
  \mu|_\rho^{(r)}(\cdot) \coloneqq \mu(D_S(\rho, r) \cap \cdot).
\end{equation}
When the root is clear from the context, we write $\mu^{(r)} \coloneqq \mu|_\rho^{(r)}$.
We then define, for each $\mu, \nu \in \Meas(S)$,  
\begin{equation} \label{eq: dfn of vague metric}
  \VagueMet{S, \rho}(\mu, \nu) 
  \coloneqq 
  \int_0^\infty e^{-r} \bigl(1 \wedge \ProhMet{S}(\mu|_\rho^{(r)}, \nu|_\rho^{(r)}) \bigr)\, dr.
\end{equation}
The function $\VagueMet{S,\rho}$ is a complete and separable metric on $\Meas(S)$ and induces the vague topology, that is,
Radon measures $\mu_n$ converge to $\mu$ with respect to $\VagueMet{S,\rho}$ if and only if,
for all compactly-supported continuous function $f \colon S \to \RN$, 
\begin{equation}
  \lim_{n \to \infty} 
  \int_S f(x)\, \mu_n(dx) 
  = 
  \int_S f(x)\, \mu(dx)
\end{equation}
(see \cite[Theorems~3.9 and 3.10]{Noda_pre_Metrization}).

We introduce the smooth truncation of measures, which will be used later. 
For each measure $\mu \in \Meas(S)$ and $R > 0$, we define 
\begin{equation} \label{eq: smooth truncation}
  \tilde{\mu}^{(R)}(dx) 
  \coloneqq 
  \int_{R-1}^{R} \mathbf{1}_{S^{(r)}}(x)\, dr\, \mu(dx).
\end{equation}
In general, the restriction of measures in the form of \eqref{eq: restriction of meas} may fail to be continuous 
with respect to the vague convergence at some radii. 
In contrast, the above truncation is smooth in the following sense.

\begin{lem} \label{lem: smooth truncation of meas}
  The map $\Meas(S) \times (1, \infty) \ni (\mu, R) \mapsto \tilde{\mu}^{(R)} \in \finMeas(S)$ is continuous.
\end{lem}

\begin{proof}
  For each $R > 1$, we set 
  \begin{equation} \label{lem pr: smooth truncation of meas. 1}
    \chi^{(R)}(x) \coloneqq \int_{R-1}^R \mathbf{1}_{S^{(r)}}(x)\, dr, 
    \qquad x \in S.
  \end{equation}
  It is straightforward to check that each $\chi^{(R)}$ is continuous.
  Indeed, if $x_n \to x$ in $S$, then $\mathbf{1}_{S^{(r)}}(x_n) \to \mathbf{1}_{S^{(r)}}(x)$ 
  for all $r > 0$ except possibly for $r = d_S(\rho, x)$.
  Moreover, we readily verify that 
  \begin{equation}
    \sup_{x \in S} \bigl| \chi^{(R)}(x) - \chi^{(R')}(x) \bigr| 
    \leq 2 |R - R'|.
  \end{equation}

  Suppose that a sequence $(\mu_n, R_n)_{n \geq 1}$ converges to an element $(\mu, R)$ 
  in $\Meas(S) \times (1, \infty)$.
  Fix a bounded and continuous function $f$ on $S$.
  The above inequality implies that $\chi^{(R_n)} \to \chi^{(R)}$ uniformly on $S$.
  Using this, one readily verifies that 
  \begin{equation}
    \lim_{n \to \infty} \int_S f(x)\, \tilde{\mu}_n^{(R_n)}(dx)
    =
    \int_S f(x)\, \tilde{\mu}^{(R)}(dx),
  \end{equation}
  which yields the desired conclusion.
  (The above convergence also follows from Lemma~\ref{lem: vague convergence and hatC topology} below.)
\end{proof}

\subsection{The compact-convergence topology with variable domains} \label{sec: The compact-convergence topology with variable domains}

Fix a $\bcmAB$ space $(S, d_S)$ and a Polish space $\Xi$. 
We write $C(S, \Xi)$ for the space of continuous functions $f \colon S \to \Xi$, equipped with the compact-convergence topology. 
Recall that convergence in this topology is characterized as follows: 
$f_n \to f$ if and only if $f_n \to f$ uniformly on every compact subset of $S$. 
The purpose of this subsection is to introduce an extension of this topology,
which enables us to handle continuous functions defined on possibly different domains.
Such an extension is needed for discussing convergence of heat kernels of stochastic processes on different spaces, 
as will become clear in our main results in Section~\ref{sec: conv of PCAFs and STOMs}. 
For further details of this extended topology, see \cite[Section~3.3]{Noda_pre_Metrization}.

We first introduce the spaces of interest. 
Recall from Definition~\ref{dfn: space of compact and closed subsets} that 
$\Closed{S}$ denotes the collection of closed subsets of $S$, including the empty set.

\begin{dfn}
  We define
  \begin{equation}
    \hatC(S,\Xi)
    \coloneqq
    \bigcup_{F \in \Closed{S}} C(F, \Xi).
  \end{equation}
  Note that $\hatC(S,\Xi)$ contains the empty map 
  $\emptyset_{\Xi} \colon \emptyset \to \Xi$. 
  For a function $f$, we write $\dom(f)$ for its domain. 
  We then define $\hatCc(S,\Xi)$ to be the subset of $\hatC(S,\Xi)$ 
  consisting of those functions $f$ whose domain $\dom(f)$ is compact.
\end{dfn}

Below, we introduce metrics on $\hatCc(S, \Xi)$ and $\hatC(S, \Xi)$.
For each function $f$,
we write $\graphmap(f)$ for its graph, i.e.,
\begin{equation} \label{eq: def of graphmap}
  \graphmap(f) \coloneqq \{(x, f(x)) \in S \times \Xi \mid x \in \dom(f)\}.
\end{equation}
We then define
\begin{equation}
  \hatCcMet{S}{\Xi}(f, g)
  \coloneqq
  \HausMet{S \times \Xi}(\graphmap(f), \graphmap(g)),
  \quad
  f, g \in \hatCc(S,\Xi).
\end{equation}
Here, $\HausMet{S \times \Xi}$ denotes the Hausdorff metric on $\Compact{S \times \Xi}$,
the set of compact subsets of $S \times \Xi$.
Fix an element $\rho \in S$ as the root of $S$.
For each $f \in \hatC(S, \Xi)$ and $r > 0$,
we define $f|_\rho^{(r)} \in \hatCc(S, \Xi)$ by setting 
\begin{equation}
  f|_\rho^{(r)} \coloneqq f|_{D_S(\rho, r)}.
\end{equation}
When the root is clear from the context, we write $f^{(r)} \coloneqq f|_\rho^{(r)}$.
We then define, for each $f, g \in \hatC(S, \Xi)$,
\begin{equation} \label{eq: dfn of hatC metric}
  \hatCMet{(S, \rho)}{\Xi}(f, g)
  \coloneqq 
  \int_0^\infty e^{-r} \bigl(1 \wedge \hatCcMet{S}{\Xi}(f|_\rho^{(r)}, g|_\rho^{(r)})\bigr)\, dr.
\end{equation}

The functions $\hatCcMet{S}{\Xi}$ and $\hatCMet{(S, \rho)}{\Xi}$ are metrics on $\hatCc(S, \Xi)$ and $\hatC(S, \Xi)$, respectively, 
and they induce Polish topologies. 
(NB.\ Neither metric is necessarily complete.) 
The topology on $\hatC(S, \Xi)$ induced by $\hatCMet{(S, \rho)}{\Xi}$ is independent of the choice of $\rho$.
(See \cite[Section~3.3]{Noda_pre_Metrization} for these results). 

\begin{lem} [{\cite[Theorems~3.30 and 3.36]{Noda_pre_Metrization}}] \label{lem: conv in hatC}
  Let $f, f_1, f_2, \dots$ be elements of $\hatCc(S, \Xi)$ (resp.\ $\hatC(S, \Xi)$).
  Then the following statements are equivalent.
  \begin{enumerate} [label = \textup{(\roman*)}, leftmargin = *]
    \item \label{lem item: 1. conv in hatC} 
      The functions $f_{n}$ converge to $f$ in $\hatCc(S, \Xi)$ (resp.\ $\hatC(S, \Xi)$).
    \item \label{lem item: 2. conv in hatC}
      The sets $\dom(f_n) \to \dom(f)$ in the Hausdorff (resp.\ Fell) topology as subsets of $S$,
      and, for any $x_n \in \dom(f_n)$ and $x \in \dom(f)$ such that $x_n \to x$ in $S$,
      it holds that $f_n(x_n) \to f(x)$ in $\Xi$.
    \item \label{lem item: 3. conv in hatC}
      The sets $\dom(f_{n})$ converge to $\dom(f)$ in the Hausdorff (resp.\ Fell) topology as subsets of $S$,
      and there exist functions $g_{n}, g \in C(S, \Xi)$ such that $g_{n}|_{\dom(f_{n})} = f_{n}$, $g|_{\dom(f)} = f$
      and $g_{n} \to g$ in the compact-convergence topology.
  \end{enumerate}
\end{lem}

By the above result, the inclusion map from $C(S, \Xi)$ into $\hatC(S, \Xi)$ is a topological embedding, 
which shows that the topology on $C(S, \Xi)$ is a natural extension of the compact-convergence topology 
(see \cite[Corollary~3.37]{Noda_pre_Metrization}).

The following result extends the classical fact that uniform convergence preserves continuity of functions.

\begin{prop} \label{prop: conti is preserved in hatC}
  Let $(f_n)_{n \geq 1}$ be a sequence in $\hatC(S, \Xi)$
  and $f$ be a function from a closed subset $\dom(f) \subseteq S$ to $\Xi$.
  Assume that the following conditions are satisfied.
  \begin{enumerate} [label = \textup{(\roman*)}]
    \item \label{prop item: 1. conti is preserved in hatC}
      The domains $\dom(f_n)$ converge to $\dom(f)$ in the Fell topology.
    \item \label{prop item: 2. conti is preserved in hatC}
      If elements $x_n \in \dom(f_n)$ converge to $x \in \dom(f)$,
      then $f_n(x_n) \to f(x)$ in $\Xi$.
  \end{enumerate}
  Then the function $f$ is continuous and $f_n \to f$ in $\hatC(S, \Xi)$.
\end{prop}

\begin{proof}
  By Lemma~\ref{lem: conv in hatC}, it suffices to show that $f$ is continuous. 
  Suppose that a sequence $(x_k)_{k \geq 1} \subseteq \dom(f)$ converges to some $x \in \dom(f)$. 
  By condition~\ref{prop item: 1. conti is preserved in hatC},
  for each $k$, there exist points $y_n \in \dom(f_n)$ converging to $x_k$. 
  For such $y_n$, condition~\ref{prop item: 2. conti is preserved in hatC} yields 
  $f_n(y_n) \to f(x_k)$ as $n \to \infty$. 
  Hence we can choose a subsequence $(n_k)_{k \geq 1}$ and points $y_k \in \dom(f_{n_k})$ such that 
  \begin{equation}
    d_S(x_k, y_k) \,\vee\, d_\Xi(f(x_k), f_{n_k}(y_k)) \leq k^{-1}.
  \end{equation}
  Since $x_k \to x$, it follows that $y_k \to x$. 
  Again by condition~\ref{prop item: 2. conti is preserved in hatC}, we then have $f_{n_k}(y_k) \to f(x)$. 
  Combining this with the above inequality gives
  \begin{equation}
    \limsup_{k \to \infty} d_\Xi(f(x_k), f(x)) 
    \leq \limsup_{k \to \infty} \bigl\{ d_\Xi(f(x_k), f_{n_k}(y_k)) + d_\Xi(f_{n_k}(y_k), f(x)) \bigr\}   
    = 0,
  \end{equation}
  which shows that $f$ is continuous. 
  This completes the proof.
\end{proof}

For later use, we prove the following two technical lemmas.
The first lemma provides a quantitative estimate for the distance between two functions in $\hatC(S, \Xi)$ that are close on a ball.

\begin{lem} \label{lem: hatC metric simple estimate}
  Let $f_1, f_2 \in \hatC(S, \Xi)$ and $r>0$ be such that 
  $\dom(f_1^{(r)}) = \dom(f_2^{(r)}) \eqqcolon K$.
  It then holds that 
  \begin{equation}
    \hatCMet{(S, \rho)}{\Xi}(f_1, f_2)
    \leq    
    e^{-r} + \sup_{x \in K} d_\Xi(f_1(x), f_2(x)).
  \end{equation}
\end{lem}

\begin{proof}
  For any $s \leq r$,
  we have that $\dom(f_1^{(s)}) = \dom(f_2^{(s)}) = K^{(s)}$.
  The Hausdorff distance between the graphs of two functions 
  is always bounded above by the uniform distance between them.
  Thus, we have that 
  \begin{equation}
    \hatCcMet{S}{\Xi}(f_1^{(s)}, f_2^{(s)}) 
    \leq     
    \sup_{x \in K^{(s)}} d_\Xi(f_1(x), f_2(x))
    \leq   
    \sup_{x \in K} d_\Xi(f_1(x), f_2(x)),
  \end{equation}
  where the last inequality is a trivial one. 
  It then follows that 
  \begin{align}
    \hatCMet{(S, \rho)}{\Xi}(f_1, f_2)
    &\leq     
    \int_r^{\infty} e^{-s}\, ds   
    + 
    \sup_{x \in K} d_\Xi(f_1(x), f_2(x)) 
    \cdot  
    \int_0^r e^{-s}\, ds 
    \\
    &\leq    
    e^{-r} + \sup_{x \in K} d_\Xi(f_1(x), f_2(x)),
  \end{align}
  which completes the proof.
\end{proof}

The second lemma concerns convergence of integrals 
for convergent functions and measures.

\begin{lem} \label{lem: vague convergence and hatC topology}
  Let $\mu, \mu_1, \mu_2, \ldots \in \finMeas(S)$ 
  with $\mu_n \to \mu$ weakly, 
  and let $f, f_{1}, f_{2}, \ldots \in \hatC(S, \RN)$ with $f_n \to f$ in $\hatC(S, \RN)$. 
  Suppose further that there exists a compact subset $K \subseteq S$ such that 
  the support of every $\mu_n$ and of $\mu$ is contained in $K$. 
  Then 
  \begin{equation}
    \lim_{n \to \infty} \int_{\dom(f_n)} f_n(x)\, \mu_n(dx) 
    = \int_{\dom(f)} f(x)\, \mu(dx).
  \end{equation}
\end{lem}

\begin{proof}
  By Lemma~\ref{lem: conv in hatC}\ref{lem item: 3. conv in hatC}, 
  the domains of $f_n$ and $f$ may be extended continuously to $S$ so that 
  $f_n \to f$ in the compact-convergence topology. 
  Let $g \colon S \to [0,1]$ be a compactly-supported continuous function
  such that $g \equiv 1$ on $K$.
  Since the support of $\mu_n$ is contained in $K$,
  it holds that  
  \begin{equation}
    \int_{\dom(f_n)} f_n(x)\, \mu_n(dx) 
    =
    \int_{\dom(f_n) \cap K} f_n(x)\, \mu_n(dx)
    =
    \int_S f_n(x) g(x)\, \mu_n(dx),
  \end{equation}
  and, similarly 
  \begin{equation}
    \int_{\dom(f)} f(x)\, \mu(dx) 
    =
    \int_S f(x) g(x)\, \mu(dx).
  \end{equation}
  Hence 
  \begin{align}
    &\left|\int_{\dom(f_n)} f_n(x)\, \mu_n(dx) - \int_{\dom(f)} f(x)\, \mu(dx)\right| \\
    &\leq
    \int_S |f_n(x) - f(x)|\, g(x)\, \mu_n(dx)
    + \left| \int_S f(x) g(x)\, \mu_n(dx) - \int_S f(x) g(x)\, \mu(dx) \right| \\
    &\leq
    \mu_n(K)\, \sup_{x \in K} |f_n(x) - f(x)|
    + \left| \int_S f(x) g(x)\, \mu_n(dx) - \int_S f(x) g(x)\, \mu(dx) \right|.
  \end{align}
  The first term tends to $0$ by the uniform convergence $f_n \to f$ on $K$, 
  and the second by the vague convergence $\mu_n \to \mu$. 
  This proves the claim.
\end{proof}
\subsection{Topologies on the space of c\`adl\`ag functions} \label{sec: Topologies on the space of cadlag functions}

In studying convergence of sample paths of stochastic processes, 
a standard choice of topology is the $J_1$-Skorohod topology (see, for example, \cite{Billingsley_99_Convergence}). 
As already mentioned in Remark~\ref{rem: why L^0},
in this paper, we also employ another topology, which we call the $L^0$ topology.
The purpose of this subsection is to introduce these topologies.
Throughout this subsection, we fix a $\bcmAB$ space $(S, d_S)$.

We write $D(\RNp, S)$ for the set of \cadlag\ functions $\xi \colon \RNp \to S$, 
and equip $D(\RNp, S)$ with the Skorohod metric $\SkorohodMet{S}$. 
To recall its definition, let $\Lambda^{J_1}_t$ denote the set of continuous bijections $\lambda \colon [0,t] \to [0,t]$. 
For \cadlag\ functions $\xi, \eta \colon [0,t] \to S$, set 
\begin{equation}
  d^{J_1, t}_S(\xi, \eta) 
  \coloneqq 
  \inf_{\lambda \in \Lambda^{J_1}_t}\!
  \left\{ \sup_{0 \leq s \leq t} |\lambda(s)-s| 
    \,\vee\, \sup_{0 \leq s \leq t} d_S(\eta(s), \xi \circ \lambda(s)) \right\}.
\end{equation}
We then define the Skorohod metric by
\begin{equation} \label{eq: dfn of Skorohod metric}
  \SkorohodMet{S}(\eta, \xi)  
  \coloneqq 
  \int_0^\infty e^{-t} \bigl(1 \wedge d^{J_1, t}_S(\xi|_{[0,t]}, \eta|_{[0,t]}) \bigr)\, dt,
  \quad \xi, \eta \in D(\RNp, S).
\end{equation}
The topology induced by $\SkorohodMet{S}$ is Polish and is called the \emph{$J_1$-Skorohod} topology; 
see \cite[Theorem~2.6]{Whitt_80_Some}. 
(NB.\ The metric $\SkorohodMet{S}$ itself is not complete.)

We write $L^0(\RNp, S)$ for the set of Borel measurable functions $\xi \colon \RNp \to S$.
We adopt the convention that any functions $\xi$ and $\eta$ in $L^0(\RNp, S)$ are identified whenever $\xi(t) = \eta(t)$
for Lebesgue-almost every $t \in \RNp$. 
We define a metric on $L^0(\RNp, S)$ as follows: for each $\xi, \eta \in L^0(\RNp, S)$,
\begin{equation} \label{eq: dfn of L^0 metric}
  \LzeroMet_S(\xi, \eta) \coloneqq \int_0^\infty e^{-t} \bigl( 1 \wedge d_S(\xi(t), \eta(t)) \bigr)\, dt.
\end{equation}
By \cite[Proposition~8.32]{Noda_pre_Metrization},
$\LzeroMet_S$ is a complete and separable metric on $L^0(\RNp, S)$.
We refer to the topology on $L^0(\RNp, S)$ induced by $\LzeroMet_S$ as the \emph{$L^0$ topology}, 
the name being derived from the $L^0$ space endowed with the topology of convergence in measure (cf.\ \cite{Bogachev_07_Measure}).

By \cite[Lemma~5.2]{Kallenberg_21_Foundations},
convergence with respect to $\LzeroMet_S$ coincides with convergence in probability
(with respect to the probability measure $e^{-t}\, dt$).
The following is immediate from the characterization of convergence in probability via almost-sure convergence.

\begin{lem}[{\cite[Lemma~5.2]{Kallenberg_21_Foundations}}] \label{lem: conv in L^0}
  Let $\xi, \xi_1, \xi_2, \dots$ be elements of $L^0(\RNp, S)$.
  Then $\xi_n \to \xi$ in the $L^0$ topology if and only if the following holds:
  \begin{enumerate}[label = \textup{($L^0$)}]
    \item for any subsequence $(n_k)_{k \geq 1}$, 
      there exists a further subsequence $(n_{k(l)})_{l \geq 1}$
      such that, for Lebesgue-almost every $t \in \RNp$, 
      $\xi_{n_{k(l)}}(t) \to \xi(t)$.
  \end{enumerate}
  As a consequence, 
  the $L^0$ topology is independent of the choice of the metric $d_S$ and depends only on the topology on $S$.
\end{lem}

The following lemma is an immediate consequence of Prohorov's theorem, 
which states that tightness of probability measures is equivalent to their relative compactness in the weak topology 
(see \cite[Section~5]{Billingsley_99_Convergence}). 
In what follows we only use one direction, namely, that relative compactness implies tightness.

\begin{lem} \label{lem: range compactness wrt L^0 topology}
  Let $(\xi_n)_{n \geq 1}$ be a sequence converging to a function $\xi$ in $L^0(\RNp, S)$. 
  Fix $T > 0$. 
  Then, for any $\varepsilon > 0$, there exists a compact subset $K \subseteq S$ such that 
  \begin{equation}
    \sup_{n \geq 1} \Leb\!\left( \{ t \in [0, T] \mid \xi_n(t) \notin K \} \right) < \varepsilon,
  \end{equation}
  where we denote by $\Leb$ the Lebesgue measure.
\end{lem}

Since $D(\RNp, S)$ can be viewed as a subset of $L^0(\RNp, S)$,
the $L^0$ topology is also defined on $D(\RNp, S)$ as a relative topology.
Thus, we have two topologies on $D(\RNp, S)$, the $J_1$-Skorohod topology 
and the $L^0$ topology.
Henceforth, to distinguish these two topological structures,
we write
\begin{enumerate} [label = \textup{(D\arabic*)}]
  \item $D_{J_1}(\RNp, S)$ for $D(\RNp, S)$ endowed with the $J_1$-Skorohod topology,
  \item $D_{L^0}(\RNp, S)$ for $D(\RNp, S)$ endowed with the $L^0$ topology.
\end{enumerate}
Below, we verify that the two spaces coincide as measurable spaces.

\begin{prop} \label{prop: Borel algebra by L^0 top}
  The Borel $\sigma$-algebras on $D(\RNp, S)$ generated by the $J_1$-Skorohod and $L^0$ topologies
  are the same.
\end{prop}

\begin{proof}
  We first note that the Borel $\sigma$-algebra generated by the $J_1$-Skorohod topology 
   coincides with the $\sigma$-algebra generated by the evaluation maps 
   $\pi_t \colon D(\RNp, S) \to S$, $t \geq 0$, defined by $\pi_t(f) \coloneqq f(t)$
  (see \cite[Lemma~2.7]{Whitt_80_Some}).
  By Lemma~\ref{lem: conv in L^0}, the $J_1$-Skorohod topology is finer than the $L^0$ topology.
  Thus, it suffices to show that each $\pi_t$ is measurable with respect to the $L^0$ topology.
  Fix $t \geq 0$.
  For each $x \in S$, we define $h_{x,t} \colon D(\RNp, S) \to \RNp$ by setting 
  \begin{equation}
    h_{x,t}(\xi) \coloneqq 1 \wedge d_S(x, \xi(t)).
  \end{equation}
  We approximate $h_{x,t}$ by functions $h_{x,t}^{(n)} \colon D(\RNp, S) \to \RNp$, $n \in \NN$, given by
  \begin{equation}
    h_{x,t}^{(n)}(\xi) \coloneqq n \int_t^{t + n^{-1}} 1 \wedge d_S(x, \xi(s))\, ds,
    \quad 
    \xi \in D(\RNp, S).
  \end{equation}
  By the right-continuity of $\xi$ and the dominated convergence theorem,
  we have $h_{x,t}^{(n)}(\xi) \to h_{x,t}(\xi)$ as $n \to \infty$ for each $\xi \in D(\RNp, S)$.
  Moreover, it is not difficult to see that $h_{x,t}^{(n)}$ is continuous with respect to the $L^0$ topology.
  Thus, $h_{x,t}$ is measurable with respect to the $L^0$ topology.
  Fix an open subset $U$ of $S$.
  Since $S$ is second countable,
  we can find countable subsets $\{x_n\}_{n \geq 1}$ of $U$ and $\{\varepsilon_n\}_{n \geq 1}$ of $[0, 1)$ such that 
  $U = \bigcup_{n \geq 1} B_S(x_n, \varepsilon_n)$.
  We then have that 
  \begin{equation}
    (\pi_t)^{-1}(U) 
    = 
    \bigcup_{n \geq 1} \pi_t^{-1}(B_S(x_n, \varepsilon_n)) \\
    = 
    \bigcup_{n \geq 1} h_{x_n,t}^{-1}([0, \varepsilon_n)).
  \end{equation}
  Therefore, $\pi_t$ is measurable with respect to the $L^0$ topology.
\end{proof}

One might expect that it suffices to work within $D_{L^0}(\RNp, S)$ 
when discussing convergence of \cadlag\ functions in the $L^0$ topology. 
However, since $D_{L^0}(\RNp, S)$ is not Polish, 
it is more convenient to carry out probabilistic arguments in the larger space $L^0(\RNp, S)$. 
The following result ensures that, for stochastic processes with c\`adl\`ag paths, 
weak convergence in $L^0(\RNp, S)$ remains valid when restricted to $D_{L^0}(\RNp, S)$.

\begin{prop} \label{prop: D is Borel in L^0}
  The set $D(\RNp, S)$ is a Borel subset of $L^0(\RNp, S)$.
\end{prop}

This result is readily verified by the Lusin--Souslin theorem, 
which we recall below. 
The theorem will also be used in later arguments.

\begin{lem} [{\cite[Corollary~15.2]{Kechris_95_Classical}}] \label{lem: Lousin--Souslin}
  Let $\Xi_1$ and $\Xi_2$ be Polish spaces.
  Fix a Borel subset $B$ of $\Xi_1$.
  If a map $f \colon B \to \Xi_2$ is Borel measurable and injective,
  then the image $f(B)$ is a Borel subset of $\Xi_2$.
\end{lem}

\begin{proof} [{Proof of Proposition~\ref{prop: D is Borel in L^0}}]
  Define a map $F \colon D_{J_1}(\RNp, S) \to L^0(\RNp, S)$ by setting $F(\xi) \coloneqq \xi$.
  Then $F$ is injective and continuous.
  Since both spaces $D_{J_1}(\RNp, S)$ and $L^0(\RNp, S)$ are Polish,
  the above lemma implies that the image of $F$, namely $D(\RNp, S)$, 
  is a Borel subset of $L^0(\RNp, S)$.
  This completes the proof.
\end{proof}

The following lemma,
which concerns different notions of a topology on the space of product processes,
will be used in Section~\ref{sec: Collision measures and their convergence}.

\begin{lem} \label{lem: L^0 and product}
  The following map is a homeomorphism:
  \begin{equation}
    D_{L^0}(\RNp, S \times S) \ni \xi = (\xi_1, \xi_2) \mapsto (\xi_1, \xi_2) \in D_{L^0}(\RNp, S) \times D_{L^0}(\RNp, S). 
  \end{equation}
\end{lem}

\begin{proof}
  This follows directly from Lemma~\ref{lem: conv in L^0}.
\end{proof}

\subsection{The spaces for PCAFs and STOMs} \label{sec: The space for STOMs}

Fix a $\bcmAB$ space $S$.
Recall from Section~\ref{sec: intro} (in particular, from \eqref{eq: dfn of STOM in intro})
that a space-time occupation measure (STOM) associated with a PCAF is a random measure on $S \times \RNp$. 
As noted in Section~\ref{sec: An overview of the framework},
the vague topology is a natural choice for discussing convergence of STOMs, 
but vague convergence of STOMs does not recover the local uniform convergence of the corresponding PCAFs. 
The aim of this subsection is to introduce spaces for PCAFs and STOMs and define a suitable topology.

We first introduce a space that will be used to discuss the convergence of PCAFs in our main results.

\begin{dfn} \label{dfn: space for PCAF}
  We define $\upC(\RNp, \RNp)$ to be the collection of functions $f \in C(\RNp, \RNp)$ that are non-decreasing. 
  We equip this set with the following metric:
  for each $\varphi_1, \varphi_2 \in \upC(\RNp, \RNp)$,
  \begin{equation} \label{dfn eq: metric on upC}
    d_{\upC(\RNp, \RNp)}(\varphi_1, \varphi_2) 
    \coloneqq 
    \sum_{n = 1}^\infty 2^{-n} \bigl(1 \wedge \sup_{0 \leq t \leq n}|\varphi_1(t) - \varphi_2(t)| \bigr).
  \end{equation}
  In particular, the topology on $\upC(\RNp, \RNp)$ is the compact-convergence topology.
\end{dfn}

\begin{rem}
  Note that convergence in $\upC(\RNp, \RNp)$ is equivalent to pointwise convergence (cf.\ \cite[Chapter~0.1]{Resnick_08_Extrem}).
\end{rem}

\begin{rem}
  Not all PCAFs belong to $\upC(\RNp, \RNp)$.
  Indeed, a PCAF may explode in finite time.
  Therefore, to study the convergence of PCAFs in the space $\upC(\RNp, \RNp)$,
  certain restrictions on the underlying process or on the PCAFs under consideration are required.
  This issue will be discussed in more detail in Section~\ref{sec: The approximation scheme for PCAFs} later.
\end{rem}

We next introduce a space in which the convergence of STOMs will be discussed.

\begin{dfn} \label{dfn: space for STOM}
  We define $\STOMMeas(S \times \RNp)$ to be the set of Borel measures $\Sigma$ on $S \times \RNp$
  such that, for each $t \in \RNp$, $\Sigma|_{S \times [0,t]}$ is a finite Borel measure.
\end{dfn}

Recall from~\eqref{eq: dfn of STOM in intro} that, 
given $(\xi, \varphi) \in D(\RNp, S) \times \upC(\RNp, \RNp)$, 
representing a process and a PCAF respectively, 
the associated STOM is defined by
\begin{equation} \label{eq: general STOM dfn}
  \Psi(\xi, \varphi)(\cdot) \coloneqq \int_0^\infty \mathbf{1}_{(\xi_t, t)}(\cdot)\, d\varphi_t.
\end{equation}
It follows immediately that $\Sigma = \Psi(\xi, \varphi)$ belongs to $\STOMMeas(S \times \RNp)$.
The function $\varphi$ can be recovered from $\Sigma$ via
\begin{equation}
  \varphi(t) = \Sigma(S \times [0,t]) = \int_{S \times \RNp} \mathbf{1}_{S \times [0,t]}(x, s)\, \Sigma(dx\, ds).
\end{equation}
From this observation, one sees that the vague topology cannot recover the convergence of PCAFs,
since the test function $\mathbf{1}_{S \times [0,t]}$ is not compactly supported.

To resolve this issue, 
we introduce a metric that differs from the vague metric.
For each $\Sigma_1, \Sigma_2 \in \STOMMeas(S \times \RNp)$, define
\begin{equation} \label{eq: dfn of stom met}
  \STOMMet{S \times \RNp}(\Sigma_1, \Sigma_2) 
  \coloneqq 
  \int_0^\infty e^{-t} 
  \bigl(1 \wedge \ProhMet{S \times \RNp}(\Sigma_1|_{S \times [0,t]}, \Sigma_2|_{S \times [0,t]}) \bigr)\, dt.
\end{equation}

\begin{prop} \label{prop: polishness of stom met}
  The function $\STOMMet{S \times \RNp}$ defines a complete, seprable metric on $\STOMMeas(S \times \RNp)$.
\end{prop}

\begin{proof}
  This can be shown in exactly the same way as the proof that the vague metric~\eqref{eq: dfn of vague metric} is a complete and seprable metric, 
  and so the proof is omitted. 
  See \cite[Theorem~3.19]{Noda_pre_Metrization} for details.
\end{proof}

For later use, we prove the following lemma, 
which states that the above-defined distance between elements 
is preserved under isometric embeddings of the underlying spaces.
This property is particularly important for the metrization of 
Gromov--Hausdorff-type topologies, 
discussed in Section~\ref{sec: GH-type topologies}.

\begin{lem} \label{lem: preserving property of stom met}
  Let $S_1$ and $S_2$ be $\bcmAB$ spaces, 
  and let $f \colon S_1 \to S_2$ be an isometric embedding.
  Then the following pushforward map is distance preserving 
  with respect to the above-defined metrics:
  \begin{equation}
    \STOMMeas(S_1 \times \RNp) \to \STOMMeas(S_2 \times \RNp), \quad 
    \Sigma \mapsto \Sigma \circ (f \times \id_{\RNp})^{-1}.
  \end{equation}
\end{lem}

\begin{proof}
  This property follows immediately from 
  the corresponding result for the vague metric 
  (see \cite[Proposition~3.24]{Noda_pre_Metrization}),
  and hence the details are omitted.
\end{proof}

Henceforth, we equip $\STOMMeas(S \times \RNp)$ with the topology induced by the metric $\STOMMet{S \times \RNp}$.
The following provides characterizations of convergence in this space.
In particular, one can see that the topology is stronger than the vague topology.

\begin{prop} \label{prop: conv in stom sp}
  Let $\Sigma, \Sigma_1, \Sigma_2, \dots$ be elements of $\STOMMeas(S \times \RNp)$.
  The following statements are equivalent with each other.
  \begin{enumerate} [label = \textup{(\roman*)}]
    \item \label{prop item: conv in stom sp. 1}
      The measures $\Sigma_n$ converge to $\Sigma$ in $\STOMMeas(S \times \RNp)$.
    \item \label{prop item: conv in stom sp. 2}
      For all but countably many $t \in \RNp$, 
      the finite Borel measures $\Sigma_n|_{S \times [0,t]}$ converge weakly to $\Sigma|_{S \times [0,t]}$.
    \item \label{prop item: conv in stom sp. 3}
      For any bounded and continuous function $f \colon S \times \RNp \to \RN$
      such that $\supp(f) \subseteq S \times [0,T]$ for some $T \in \RNp$,
      it holds that 
      \begin{equation}
        \lim_{n \to \infty}\int_{S \times \RNp} f(x, t)\, \Sigma_n(dx\, dt) = \int_{S \times \RNp} f(x,t)\, \Sigma(dx\, dt).
      \end{equation}
  \end{enumerate}
\end{prop}

\begin{proof}
  This can be shown in exactly the same way as the characterization of the vague topology given in \cite[Theorem~3.20]{Noda_pre_Metrization}, 
  so we will omit the proof.
\end{proof}

\begin{rem} \label{rem: Kallenberg framework for stom sp}
  Recall that the vague topology we adopt in this paper is defined by using continuous functions with compact support as test functions. 
  A variant of this notion, in which the test functions are bounded continuous functions whose supports are bounded subsets, 
  has been studied in detail in~\cite{Kallenberg_17_Random}. 
  We refer to the resulting topology as the \emph{vague topology with bounded support}. 
  From Proposition~\ref{prop: conv in stom sp},
  one sees that the topology of $\STOMMeas(S \times \RNp)$
  coincides with the vague topology with bounded support
  when the underlying metric $d_S$ on $S$ is replaced by a bounded, topologically equivalent one,
  for instance $d_S \wedge 1$.
  Hence, all the results of~\cite{Kallenberg_17_Random} concerning this topology are applicable to our setting. 
\end{rem}

\begin{prop} \label{prop: precompact in stom meas sp}
  Fix a countable family $\{\Pi_n\}_{n \geq 1}$ in $\STOMMeas(S \times \RNp)$.
  The family is precompact if and only if the following two conditions are satisfied:
  \begin{enumerate} [label = \textup{(\roman*)}]
    \item $\displaystyle \limsup_{n \to \infty} \Pi_n(S \times [0,t]) < \infty$ for all $t > 0$,
    \item $\displaystyle \limsup_{R \to \infty} \limsup_{n \to \infty} \Pi_n\bigl( (S \setminus S^{(R)})\times [0, t]\bigr) = 0$ for all $t > 0$.
  \end{enumerate}
\end{prop}

\begin{proof}
  By Remark~\ref{rem: Kallenberg framework for stom sp}, 
  the result is immediate from \cite[Theorem~4.2]{Kallenberg_17_Random}.
\end{proof}

In the following proposition, 
we clarify the relationship between the convergence of PCAF-type functions 
and that of the associated STOM-type measures.
In particular, it shows that 
convergence of STOMs in $\STOMMeas(S \times \RNp)$ 
implies convergence of the corresponding PCAFs in $\upC(\RNp, \RNp)$.

\begin{prop} \label{prop: PCAF to STOM map}
  Define a map
  \begin{equation}
    \apSTOM \colon 
    D(\RNp, S) \times \upC(\RNp, \RNp) 
    \ni (\xi, \varphi) \longmapsto (\xi, \Psi(\xi, \varphi))
    \in D(\RNp, S) \times \STOMMeas(S \times \RNp),
  \end{equation}
  where $\Psi(\xi, \varphi)$ is defined in \eqref{eq: general STOM dfn}.
  Then the following statements hold.
  \begin{enumerate}[label=\textup{(\roman*)}]
    \item \label{prop item: PCAF to STOM map. 1}
      If both copies of $D(\RNp, S)$ in the domain and the codomain are equipped with the $J_1$-Skorohod topology, 
      then the map is a topological embedding, i.e., a homeomorphism onto its image.
    \item \label{prop item: PCAF to STOM map. 2}
      If both copies of $D(\RNp, S)$ are equipped with the $L^0$ topology,
      then the map is measurable and its inverse is continuous,
      whereas the map itself is not necessarily continuous.
  \end{enumerate}
\end{prop}

\begin{proof}
  \ref{prop item: PCAF to STOM map. 1}.
  It is clear that the map is injective.
  We first prove its continuity.
  Let $(\xi_n, \varphi_n)_{n \geq 1}$ be a sequence converging to $(\xi, \varphi)$ in $D_{J_1}(\RNp, S) \times \upC(\RNp, \RNp)$.
  Fix a continuity point $T > 0$ of $\xi$.
  By the $J_1$-Skorohod convergence of $\xi_n$ to $\xi$,
  we can find a compact subset $K$ of $S$ such that 
  all the trajectories of $\xi_n$, $n \geq 1$, up to time $T$ are contained in $K$
  (cf.\ \cite[Theorem~A.5.4]{Kallenberg_21_Foundations}).
  Moreover, there exist time-change functions $\lambda_n \in \Lambda_T^{J_1}$, $n \geq 1$, such that 
  \begin{equation}
    \lim_{n \to \infty}
    \varepsilon_n 
    = 0,\quad 
    \text{where}\quad 
    \varepsilon_n \coloneqq
    \sup_{0 \leq t \leq T} |\lambda_n(t)-t| 
    \vee \sup_{0 \leq t \leq T} d_S(\xi_n(t), \xi \circ \lambda_n(t)).
  \end{equation}
  Fix a bounded and continuous function $f \colon S \times \RNp \to \RNp$
  such that $\supp(f) \subseteq S \times [0,T]$ for some $T > 0$.
  By the uniform continuity of $f$ on $K \times [0,T]$,
  if we set, for each $\varepsilon > 0$,
  \begin{equation}
    w_f(\varepsilon) 
    \coloneqq \sup\!\left\{
      |f(x,s) - f(y,t)| \mid (x,s), (y,t) \in K \times [0, T]\ \text{such that}\ d_S(x,y) \vee |s-t| \leq \varepsilon
    \right\},
  \end{equation}
  then $w_f(\varepsilon) \to 0$ as $\varepsilon \to 0$.
  Write $\Sigma_n \coloneqq \Psi(\xi_n, \varphi_n)$, $n \geq 1$, and $\Sigma \coloneqq \Psi(\xi, \varphi)$.
  We deduce that 
  \begin{align}
    &\left| \int_{S \times \RNp} f(x,t)\, \Sigma_n^{(*, T)}(dx\, dt) - \int_{S \times \RNp} f(x,t)\, \Sigma^{(*, T)}(dx\, dt)\right|\\
    &=
    \left| \int_0^T f(\xi_n(t), t)\, \varphi_n(dt) - \int_0^T f(\xi(t), t)\, \varphi(dt) \right|\\
    &\leq   
    \int_0^T \left| f(\xi_n(t), t) - f(\xi \circ \lambda_n(t), t) \right|\, \varphi_n(dt)
    +\left| \int_0^T f(\xi \circ \lambda_n(t), t)\, \varphi_n(dt) - \int_0^T f(\xi(t), t)\, \varphi(dt)\right|\\
    &\leq    
    w_f(\varepsilon_n)\,\varphi_n(T)
    + 
    \left| \int_0^T f(\xi \circ \lambda_n(t), t)\, \varphi_n(dt) - \int_0^T f(\xi(t), t)\, \varphi(dt)\right|.
    \label{thm pr: PCAF conv implies STOM conv. 1}
  \end{align}
  Since $\varphi_n(T)$ is bounded uniformly in $n$ by the convergence of $\varphi_n$ to $\varphi$,
  the first term in the last inequality converges to $0$ as $n \to \infty$.
  Thus, it suffices to show that the second term also vanishes.
  Define the right-continuous inverse of $\varphi_n$ by 
  \begin{equation}
    \varphi_n^{-1}(t) \coloneqq \inf\{s \geq 0 \mid \varphi_n(s) > t\}.
  \end{equation}
  Similarly, define $\varphi^{-1}$ to be the right-continuous inverse of $\varphi$.
  By a change-of-variables formula (cf.\ \cite[Lemma~A.3.7]{Chen_Fukushima_12_Symmetric}),
  \begin{align}
    &\left| \int_0^T f(\xi \circ \lambda_n(t), t)\, \varphi_n(dt) - \int_0^T f(\xi(t), t)\, \varphi(dt)\right|\\
    &=
    \left| \int_0^{\varphi_n(T)} f\bigl( \xi \circ \lambda_n \circ \varphi_n^{-1}(t), \varphi_n^{-1}(t) \bigr)\, dt 
      - \int_0^{\varphi(T)} f\bigl( \xi \circ \varphi^{-1}(t), \varphi^{-1}(t) \bigr)\, dt\right|.
    \label{thm pr: PCAF conv implies STOM conv. 2}
  \end{align}
  By \cite[Theorem~12.5.1 and Corollary~13.6.5]{Whitt_02_Stochastic}, we have $\varphi_n^{-1}(t) \to \varphi^{-1}(t)$
  for Lebesgue-almost every $t > 0$.
  For such $t$, it also holds that $\lambda_n \circ \varphi_n^{-1}(t) \to \varphi^{-1}(t)$.
  Let $N_1$ denote the negligible set of $t \geq 0$ 
  for which the convergence $\lambda_n \circ \varphi_n^{-1}(t) \to \varphi^{-1}(t)$ fails.
  Let $N_2$ denote the set of $t \geq 0$ such that $\varphi^{-1}(t)$ is a discontinuity point of $\xi$.
  Since $\xi$ is continuous except at countably many points and $\varphi^{-1}$ is strictly increasing,
  the set $N_2$ is countable.
  Hence $\Leb(N_1 \cup N_2) = 0$.
  For any $t \in \RNp \setminus (N_1 \cup N_2)$,
  we have $\lambda_n \circ \varphi_n^{-1}(t) \to \varphi^{-1}(t)$ and $\varphi^{-1}(t)$ is a continuity point of $\xi$,
  which implies $\xi(\lambda_n \circ \varphi_n^{-1}(t)) \to \xi(\varphi^{-1}(t))$.
  Therefore, by the dominated convergence theorem,
  the value of \eqref{thm pr: PCAF conv implies STOM conv. 2} converges to $0$ as $n \to \infty$.
  This establishes the desired continuity.

  It remains to prove that the inverse map is continuous.
  Let $(\xi_n, \varphi_n)_{n \geq 1}$ and $(\xi, \varphi)$ be elements of $D_{J_1}(\RNp, S) \times \upC(\RNp, \RNp)$.
  Write $\Sigma_n \coloneqq \Psi(\xi_n, \varphi_n)$, $n \geq 1$, and $\Sigma \coloneqq \Psi(\xi, \varphi)$.
  Assume that $(\xi_n, \Sigma_n) \to (\xi, \Sigma)$.
  It suffices to show that $\varphi_n \to \varphi$ in the compact-convergence topology.
  By the convergence of $\Sigma_n$ to $\Sigma$ and Proposition~\ref{prop: conv in stom sp},
  we deduce that $d\varphi_n \to d\varphi$ vaguely.
  Since $\varphi$ is continuous, $d\varphi$ has no atoms.
  In particular, for any $t \geq 0$, $[0,t]$ is a continuity set of $d\varphi$.
  It then follows that $\varphi_n(t) \to \varphi(t)$ for each $t \geq 0$.
  This completes the proof.

  \ref{prop item: PCAF to STOM map. 2}.
  By Proposition~\ref{prop: Borel algebra by L^0 top} and \ref{prop item: PCAF to STOM map. 1},
  the map is measurable.
  The continuity of the inverse map follows by the same argument as before.
  To see that the map fails to be continuous in general,
  assume that $S$ consists of two distinct elements $v_1$ and $v_2$,
  and take a function $\varphi \in \upC(\RNp,\RNp)$ such that $d\varphi([0,1]) = 1$, $\supp(d\varphi) \subseteq [0,1]$,
  and $\Leb(\supp(d\varphi)) = 0$.
  (For example, the Cantor function suffices.)
  From the last condition,
  for each $n \geq 1$,
  we can find a Borel set $U_n$ which is a disjoint union of finitely many half-open intervals of the form $[a,b)$
  and satisfies $\supp(d\varphi) \subseteq U_n$ and $\Leb(U_n) < 2^{-n}$.
  Define $\eta_n \in D(\RNp,S)$ by setting $\eta_n(t) = v_1$ if $t \in U_n$ and $\eta_n(t) = v_2$ otherwise.
  Then $\eta_n$ converges to the constant function $v_2$ in the $L^0$ topology.
  On the other hand, $\Psi(\eta_n,\varphi)(\{v_1\} \times [0,1]) = d\varphi([0,1]) = 1$ for all $n \geq 1$,
  whereas for the limit $\eta \equiv v_2$ we have $\Psi(\eta,\varphi)(\{v_1\} \times [0,1]) = 0$.
  Hence $\Psi(\eta_n,\varphi)$ does not converge to $\Psi(\eta,\varphi)$.
\end{proof}

\begin{rem} \label{rem: PCAF to STOM map}
  We record some remarks concerning Proposition~\ref{prop: PCAF to STOM map}.
  \begin{enumerate} [label = \textup{(\alph*)}]
    \item 
    The continuity of the inverse of $\apSTOM$ implies that 
    the convergence of STOM-type measures in $\STOMMeas(S \times \RNp)$ 
    entails the convergence of the associated PCAF-type functions.

    \item 
    By Proposition~\ref{prop: PCAF to STOM map}\ref{prop item: PCAF to STOM map. 1},
    when one works in the $J_1$-Skorohod topology (or other relatively fine topologies),
    the joint convergence of processes and PCAFs is equivalent to that of processes and STOMs.
    However, in the main results of this paper, we mainly work under the $L^0$ topology 
    in order to allow greater generality of assumptions.
    Consequently, the convergence of PCAFs does not necessarily imply that of STOMs,
    which makes the technical arguments considerably more involved.
  \end{enumerate}
\end{rem}

\section{Topology on the space of metric spaces equipped with additional structures} \label{sec: GH-type topologies}

The Gromov--Hausdorff metric, introduced by Gromov~\cite{Gromov_07_Metric}, 
is a modification of the Hausdorff metric and measures the distance between compact metric spaces.
It induces a Polish topology on the set of (equivalence classes of) compact metric spaces, 
called the \emph{Gromov--Hausdorff topology}, 
and provides a natural setting for discussing convergence of metric spaces. 
In many situations, metric spaces are equipped with additional structures such as measures or stochastic processes, 
and it is then natural to consider convergence not only of the underlying metric spaces but also of the additional structures.
For this purpose, several extensions of the Gromov--Hausdorff topology have been studied 
(cf.\ \cite{Abraham_Delmas_Hoscheit_13_A_note,Athreya_Lohr_Winter_16_The_gap,Khezeli_20_Metrization}), 
and general frameworks have recently been established in \cite{Khezeli_23_A_unified,Noda_pre_Metrization}. 
Here, following \cite{Noda_pre_Metrization}, we introduce the topological framework needed for our main results.
(The differences between \cite{Khezeli_23_A_unified} and \cite{Noda_pre_Metrization} are discussed in \cite[Section~1]{Noda_pre_Metrization}.)

\subsection{General framework} \label{sec: GH general framework}

In this subsection we explain how to formalize the addition of structures to metric spaces 
and how to define a generalized Gromov--Hausdorff metric. 
We summarize only the part of the framework needed for the discussion in this paper. 
Although the use of category-theoretic language would clarify the framework, 
we avoid this approach in order to keep the presentation self-contained and accessible. 
For further details see \cite[Section~6]{Noda_pre_Metrization}.

We first define a rule that assigns an additional structure to each metric space. 
Recall that if a map $f \colon S_1 \to S_2$ between metric spaces is distance-preserving, 
then $f$ is called an \emph{isometric} embedding. 
An isometric embedding that is surjective (and hence bijective) is called an \emph{isometry}.

\begin{dfn}[Structure] \label{dfn: structure}
  We call $\tau$ a \textit{structure} (on $\bcmAB$ spaces) if it satisfies the following conditions.
  \begin{enumerate}[label = (\roman*), series = structure]
    \item 
      For every $\bcmAB$ space $S$, 
      there exists a metrizable topological space $\tau(S)$.
    \item 
      For every isometric embedding $f \colon S_1 \to S_2$ between $\bcmAB$ spaces, 
      there exists a topological embedding $\tau_f \colon \tau(S_1) \to \tau(S_2)$, 
      i.e., $\tau_f$ is a homeomorphism onto its image.
    \item 
      Let $S_1, S_2, S_3$ be $\bcmAB$ spaces. 
      For any isometric embeddings $f \colon S_1 \to S_2$ and $g \colon S_2 \to S_3$, 
      one has $\tau_{g \circ f} = \tau_g \circ \tau_f$.
    \item 
      For any $\bcmAB$ space $S$, 
      one has $\tau_{\id_S} = \id_{\tau(S)}$.
  \end{enumerate}
  We say that $\tau$ is \emph{separable} if each $\tau(S)$ is separable.
\end{dfn}

Readers who wish to keep a specific structure in mind may refer to Section~\ref{sec: structures used in the present paper} 
for a list of example structures.

Throughout this subsection, we fix a structure $\tau$. 
Given $\mathcal{S}_i = (S_i, d_{S_i}, \rho_{S_i}, a_{S_i})$, $i=1,2$, where 
$(S_i, d_{S_i}, \rho_{S_i})$ is a rooted $\bcmAB$ space and $a_{S_i} \in \tau(S_i)$, 
we say that $\mathcal{S}_1$ and $\mathcal{S}_2$ are \emph{rooted-$\tau$-isometric} 
if there exists a root-preserving isometry $f \colon S_1 \to S_2$ such that $\tau_f(a_{S_1}) = a_{S_2}$. 
Recall that $f$ is \emph{root-preserving} if $f(\rho_{S_1}) = \rho_{S_2}$. 
We then define $\rootBCM(\tau)$ as the collection of rooted-$\tau$-isometric equivalence classes.

\begin{rem} \label{rem: note on set of metric spaces}
  From a rigorous set-theoretic point of view, 
  one cannot directly take rooted-$\tau$-isometric equivalence classes, 
  since the collection of all tuples $(S, d_S, \rho_S, a_S)$ is not a legitimate set. 
  Nevertheless, it is possible to regard $\rootBCM(\tau)$ as a set. 
  Indeed, one can construct a legitimate set $\mathscr{M}(\tau)$ 
  such that every $(S, d_S, \rho_S, a_S)$ is rooted-$\tau$-isometric to a unique element of $\mathscr{M}(\tau)$. 
  For details, see \cite[Proposition~6.2]{Noda_pre_Metrization}. 
  In this paper, we tacitly fix such a representative system and
  thereafter treat $\rootBCM(\tau)$ as a set,
  so that we can proceed without repeatedly referring to these set-theoretic formalities.
\end{rem}

To define a metric on $\rootBCM(\tau)$, we introduce the notion of a metrization of a structure.

\begin{dfn} \label{dfn: metrization of structure}
  We say that $\tau$ admits a \emph{metrization} if and only if, for every rooted $\bcmAB$ space $(S, \rho_S)$,
  there exists a metric $d^\tau_{S, \rho_S}$ on $\tau(S)$ inducing the given topology such that the following condition holds:
  \begin{enumerate}[resume* = structure]
    \item \label{dfn item: metrization of st}
      Let $(S_1, \rho_{S_1})$ and $(S_2, \rho_{S_2})$ be rooted $\bcmAB$ spaces. 
      For every root-preserving isometric embedding $f \colon S_1 \to S_2$, 
      the map $\tau_f \colon \tau(S_1) \to \tau(S_2)$ is distance-preserving with respect to 
      $d^\tau_{S_1, \rho_{S_1}}$ and $d^\tau_{S_2, \rho_{S_2}}$.
  \end{enumerate}
  If each $d^\tau_{S, \rho_S}$ is complete, we say that $\tau$ admits a \emph{complete metrization}.
\end{dfn}

Henceforth we assume that $\tau$ admits a metrization. 
We now define a distance between elements of $\rootBCM(\tau)$.

\begin{dfn} \label{dfn: GF-type metric}
  For $\mathcal{S}_1 = (S_1, d_{S_1}, \rho_{S_1}, a_{S_1})$ and 
  $\mathcal{S}_2 = (S_2, d_{S_2}, \rho_{S_2}, a_{S_2})$ in $\rootBCM(\tau)$, set
  \begin{equation}
    \GFMet^\tau(\mathcal{S}_1, \mathcal{S}_2)
    \coloneqq 
    \inf_{f,g,M}
    \Bigl\{
      \FellMet{M,\rho_M}(f(S_1), g(S_2)) 
      \,\vee\,  
      d^{\tau}_{M, \rho_M}\bigl(\tau_{f}(a_{S_1}), \tau_{g}(a_{S_2})\bigr)  
    \Bigr\},
  \end{equation}
  where the infimum is taken 
  over all rooted $\bcmAB$ spaces $(M, \rho_M)$ 
  and root-preserving isometric embeddings $f \colon  S_1 \to M$, $g \colon S_2 \to M$.
\end{dfn}

To ensure that the above defined function is a metric,
we consider the following continuity condition on $\tau$ regarding embedding maps.

\begin{dfn} \label{dfn: embedding-continuity}
  We say that $\tau$ is \emph{embedding-continuous} if and only if the following condition holds.
  \begin{enumerate} [label = \textup{(EC)}]
    \item Fix $\bcmAB$ spaces $S_1$ and $S_2$.
      Let $f, f_1, f_2, \dots$ be continuous maps from $S_1$ to $S_2$.
      If $f_n \to f$ in the compact-convergence topology,
      then $\tau_{f_n}(a) \to \tau_f(a)$ in $\tau(S_2)$ for all $a \in \tau(S_1)$.
  \end{enumerate}
\end{dfn}

\begin{lem} [{\cite[Theorem~6.19]{Noda_pre_Metrization}}]
  If $\tau$ is embedding-continuous, then $\GFMet^\tau$ is a metric on $\rootBCM(\tau)$.
\end{lem}

Henceforth we assume that $\tau$ is embedding-continuous and endow $\rootBCM(\tau)$ with the topology induced by $\GFMet^\tau$. 
Below we provide a characterization of convergence in $\rootBCM(\tau)$ that is useful in applications. 
In particular, it shows that the resulting topology on $\rootBCM(\tau)$ is independent of the choice of a metrization of $\tau$.

\begin{thm}[Convergence] \label{thm: conv in M(tau)}
  Let $\mathcal{S}_n = (S_n, d_{S_n}, \rho_{S_n}, a_{S_n})$, $n \in \NN \cup \{\infty\}$, be elements of $\rootBCM(\tau)$. 
  The following conditions are equivalent:
  \begin{enumerate}[label = \textup{(\roman*)}]
    \item \label{thm item: 1. conv in M(tau)}
      $\mathcal{S}_n \to \mathcal{S}_\infty$ in $\rootBCM(\tau)$;
    \item \label{thm item: 2. conv in M(tau)}
      there exist a rooted $\bcmAB$ space $(M, \rho_M)$
      and root-preserving isometric embeddings $f_n \colon S_n \to M$, $n \in \NN \cup \{\infty\}$,
      such that $f_n(S_n) \to f_\infty(S_\infty)$ in the Fell topology as closed subsets of $M$, 
      and $\tau_{f_n}(a_{S_n}) \to \tau_{f_\infty}(a_{S_\infty})$ in $\tau(M)$.
  \end{enumerate}
\end{thm}

\begin{rem}
  In Theorem~\ref{thm: conv in M(tau)}\ref{thm item: 2. conv in M(tau)}, 
  the roots $\rho_{S_n}$ of $S_n$ are mapped to a common root $\rho_M$ of $M$. 
  By relaxing this requirement, one can introduce another notion of convergence in $\rootBCM(\tau)$. 
  Namely, we say that $\mathcal{S}_n$ converges to $\mathcal{S}_\infty$ if and only if 
  \begin{equation} \label{eq: Khezeli conv in non-cpt case}
    \begin{minipage}[c]{0.9\linewidth}
      \textit{
      there exist a $\bcmAB$ space $M$ and isometric embeddings $f_n \colon S_n \to M$ 
      such that $f_n(S_n) \to f_\infty(S_\infty)$ in the Fell topology, 
      $f_n(\rho_{S_n}) \to f_\infty(\rho_{S_\infty})$ in $M$, 
      and $\tau_{f_n}(a_{S_n}) \to \tau_{f_\infty}(a_{S_\infty})$ in $\tau(M)$.}
    \end{minipage}
  \end{equation}
  In general, this convergence is weaker than the convergence induced by $\GFMet^\tau$ defined above. 
  However, for most structures of interest the two notions of convergence coincide. 
  See \cite[Section~6.3]{Noda_pre_Metrization} for details.
\end{rem}

The following observation is immediate from the definition of $\GFMet^\tau$ 
and will be useful for the approximation arguments in our main results.

\begin{lem} \label{lem: simple estimate of GH distance}
  For $\mathcal{S} = (S, d_S, \rho_S, a_S)$ and 
  $\mathcal{S}' = (S, d_S, \rho_S, b_S)$ in $\rootBCM(\tau)$, 
  it holds that 
  \begin{equation}
    \GFMet^\tau(\mathcal{S}, \mathcal{S}')
    \leq 
    d^\tau_{S, \rho_S}(a_S, b_S).
  \end{equation}
\end{lem}

To discuss Polishness, we introduce another continuity condition on $\tau$. 
For notational convenience, we adopt the convention that if a $\bcmAB$ space $S$ is isometrically embedded into another $\bcmAB$ space $M$ via an isometric embedding $f$,  
then we regard $\tau(S)$ as a subspace of $\tau(M)$ through the topological embedding $\tau_f \colon \tau(S) \to \tau(M)$.

\begin{assum} \label{assum: semicontinuity}
  Let $S_n$, $n \in \mathbb{N} \cup \{\infty\}$, be $\bcmAB$ spaces
  that are isometrically embedded into a common $\bcmAB$ space $M$
  such that $S_n \to S_\infty$ in the Fell topology as closed subsets of $M$.
  \begin{enumerate}[label=\textup{(\roman*)}, leftmargin=*]
    \item \label{assum item: upper semicontinuity}
      If a sequence $a_n \in \tau(S_n)$ converges to some $a \in \tau(M)$, then $a \in \tau(S_\infty)$.
    \item \label{assum item: lower semicontinuity}
      For every $a \in \tau(S_\infty)$, there exist a subsequence $(n_k)_{k \geq 1}$ and elements $a_k \in \tau(S_{n_k})$ 
      such that $a_k \to a$ in $\tau(M)$. 
  \end{enumerate}
\end{assum}

\begin{dfn}[Continuity] \label{dfn: continuity of structure} 
  We say that $\tau$ is \emph{upper} (resp.\ \emph{lower}) \emph{semicontinuous} 
  if and only if it satisfies Assumption~\ref{assum: semicontinuity}\ref{assum item: upper semicontinuity} (resp.\ \ref{assum item: lower semicontinuity}). 
  We say that $\tau$ is \emph{continuous} 
  if and only if it is embedding-continuous and both upper and lower semicontinuous.
\end{dfn}

\begin{thm} [{\cite[Theorems~6.40 and 6.41]{Noda_pre_Metrization}}]
  Assume that $\tau$ is separable and continuous, and admits a complete metrization.
  Then the induced metric $\GFMet^\tau$ is a complete and separable metric on $\rootBCM(\tau)$.
\end{thm}

By the above theorem, 
$\rootBCM(\tau)$ is Polish for many structures $\tau$. 
On the other hand, some important structures do not satisfy semicontinuity, and hence the theorem does not apply. 
Even in such cases, however, the Polishness of $\rootBCM(\tau)$ can still be established 
via Alexandrov's theorem (cf.\ \cite[Theorem~2.2.1]{Srivastava_98_A_Course}), 
which asserts that any topological space that is topologically embedded into a Polish space as a $G_\delta$-subset is itself Polish. 
To this end, we introduce several additional notions on structure,
which are required to formulate the notion of a \emph{Polish structure} (Definition~\ref{dfn: Polish structure}),
but they will not be used in the subsequent arguments.
The essential result needed for later purposes is Theorem~\ref{thm: Polishness of GH topology}.

\begin{dfn}[{\cite[Definition~5.12]{Noda_pre_Metrization}}] \label{dfn: topological embedding}
  Recall that we have a fixed structure $\tau$.
  Let $\tilde{\tau}$ be another structure. 
  A \emph{topological embedding} $\eta \colon \tau \Rightarrow \tilde{\tau}$ is a family 
  $\{\eta_S \colon \tau(S) \to \tilde{\tau}(S)\}_{S}$ indexed by $\bcmAB$ spaces $S$, satisfying the following conditions:
  \begin{enumerate}[label = \textup{(TE\arabic*)}, leftmargin = *]
    \item each $\eta_S \colon \tau(S) \to \tilde{\tau}(S)$ is a topological embedding;
    \item for any $\bcmAB$ spaces $S_1$ and $S_2$ and isometric embedding $f \colon S_1 \to S_2$, 
      $\tilde{\tau}_f \circ \eta_{S_1} = \eta_{S_2} \circ \tau_f$.
  \end{enumerate}
\end{dfn}

As the terminology suggests, 
if there exists a topological embedding $\eta \colon \tau \Rightarrow \tilde{\tau}$, 
then the following map is itself a topological embedding:
\begin{equation}
  \rootBCM(\tau) \ni (S, d_S, \rho_S, a_S) 
  \longmapsto (S, d_S, \rho_S, \eta_S(a_S)) \in \rootBCM(\tilde{\tau}),
\end{equation}
see \cite[Lemma~6.44]{Noda_pre_Metrization}.

We are now ready to introduce the notion of Polishness for $\tau$.

\begin{dfn}[{\cite[Definition~6.45]{Noda_pre_Metrization}}] \label{dfn: Polish structure}
  We say that $\tau$ is \emph{Polish} if there exist
  another structure $\tilde{\tau}$, a topological embedding $\eta \colon \tau \Rightarrow \tilde{\tau}$,
  and, for each rooted $\bcmAB$ space $(S, \rho_S)$,
  a sequence $(\tilde{\tau}_k(S,\rho_S))_{k \geq 1}$ of open subsets of $\tilde{\tau}(S)$ satisfying the following conditions:
  \begin{enumerate}[label = \textup{(P\arabic*)}]
    \item \label{dfn item: 1. Polish structure}
      The structure $\tilde{\tau}$ is continuous and separable, and admits a complete metrization.
    \item \label{dfn item: 2. Polish structure}
      Let $(S_1, \rho_{S_1})$ and $(S_2, \rho_{S_2})$ be rooted $\bcmAB$ spaces.
      For every root-preserving isometric embedding $f \colon S_1 \to S_2$, 
      one has $\tilde{\tau}_f^{-1}(\tilde{\tau}_k(S_2, \rho_{S_2})) = \tilde{\tau}_k(S_1, \rho_{S_1})$ for all $k \geq 1$.
    \item \label{dfn item: 3. Polish structure}
      For each rooted $\bcmAB$ space $(S, \rho_S)$, 
      $\eta_S(\tau(S)) = \bigcap_{k \geq 1} \tilde{\tau}_k(S, \rho_S)$.
  \end{enumerate}
\end{dfn}

Condition~\ref{dfn item: 1. Polish structure} ensures that $\rootBCM(\tilde{\tau})$ is Polish, 
while Conditions~\ref{dfn item: 2. Polish structure} and \ref{dfn item: 3. Polish structure} guarantee that 
$\rootBCM(\tau)$ is topologically embedded into $\rootBCM(\tilde{\tau})$ as a $G_\delta$-subset. 
Hence, by Alexandrov's theorem, we obtain the following.

\begin{thm}[{\cite[Theorem~6.46]{Noda_pre_Metrization}}] \label{thm: Polishness of GH topology}
  If $\tau$ is Polish, then the topology on $\rootBCM(\tau)$ is Polish. 
  (NB.\ The metric $\GFMet^\tau$ is not necessarily complete.)
\end{thm}

\medskip
So far, we have considered $\bcmAB$ spaces. 
When one considers only compact underlying spaces, 
it is more natural to use the Hausdorff topology 
rather than the Fell topology 
to describe the convergence of the underlying spaces. 
In what follows, we briefly introduce a corresponding modification of the above framework. 
See \cite[Appendix~B]{Noda_pre_Metrization} for details.

Recall that we have a fixed embedding-continuous structure $\tau$ that admits a metrization.
We define $\rootCM(\tau)$ as the collection of $(S, d_S, \rho_S, a_S) \in \rootBCM(\tau)$ 
such that $S$ is compact.
We define a metric on $\rootCM(\tau)$ as follows:
for $\mathcal{S}_1 = (S_1, d_{S_1}, \rho_{S_1}, a_{S_1})$ and 
$\mathcal{S}_2 = (S_2, d_{S_2}, \rho_{S_2}, a_{S_2})$ in $\rootCM(\tau)$, set
\begin{equation}
  \GHMet^\tau(\mathcal{S}_1, \mathcal{S}_2)
  \coloneqq 
  \inf_{f,g,M}
  \Bigl\{
    \HausMet{M}(f(S_1), g(S_2)) 
    \,\vee\,  
    d^{\tau}_{M, \rho_M}\bigl(\tau_{f}(a_{S_1}), \tau_{g}(a_{S_2})\bigr)  
  \Bigr\},
\end{equation}
where the infimum is taken 
over all rooted compact metric spaces $(M, \rho_M)$ 
and all root-preserving isometric embeddings 
$f \colon  S_1 \to M$ and $g \colon S_2 \to M$.

In the above definition, the Hausdorff metric $\HausMet{M}$ is employed
instead of the Fell metric.
Accordingly, the characterization of convergence in $\rootCM(\tau)$
is given in the same manner as in 
statement~\ref{thm item: 2. conv in M(tau)} of Theorem~\ref{thm: conv in M(tau)},
with the Fell topology replaced by the Hausdorff topology.
It should also be noted that, if $\tau$ is Polish, 
then the resulting topology on $\rootCM(\tau)$ is again Polish.

Although the main convergence results of this paper are stated for $\rootBCM(\tau)$,
there are many situations where the compact setting $\rootCM(\tau)$
is more natural or convenient for discussion.
However, convergence in $\rootBCM(\tau)$ 
can easily be converted into convergence in $\rootCM(\tau)$, 
thanks to the following proposition.

\begin{prop} \label{prop: conv in BCM and CM}
  Let $\mathcal{S}_n = (S_n, d_{S_n}, \rho_{S_n}, a_{S_n})$, $n \in \NN \cup \{\infty\}$, be elements of $\rootCM(\tau)$. 
  Then $\mathcal{S}_n \to \mathcal{S}_\infty$ in $\rootCM(\tau)$ if and only if 
  $\mathcal{S}_n \to \mathcal{S}_\infty$ in $\rootBCM(\tau)$ and 
  \begin{equation} \label{prop item: conv in BCM and CM}
    \sup_{n \geq 1} \sup_{x \in S_n} d_{S_n}(\rho_{S_n}, x) < \infty.
  \end{equation}
\end{prop}

\begin{proof}
  Suppose $\mathcal{S}_n \to \mathcal{S}_\infty$ in $\rootBCM(\tau)$.
  Then we can embed all the spaces $S_n$, $n \in \NN \cup \{\infty\}$, 
  isometrically into a common rooted $\bcmAB$ space $(M, \rho_M)$
  so that the statement~\ref{thm item: 2. conv in M(tau)} of Theorem~\ref{thm: conv in M(tau)} holds.
  In addition, if \eqref{prop item: conv in BCM and CM} is satisfied,
  then, by the precompactness criterion for the Hausdorff topology 
  (cf.\ \cite[Lemma~3.5]{Cech_69_Point}),
  we deduce that $S_n \to S$ in the Hausdorff topology as subsets of $M$.
  This implies that $\mathcal{S}_n \to \mathcal{S}_\infty$ in $\rootCM(\tau)$.
  The converse direction is immediate since the Hausdorff topology is finer than the Fell topology.
  This completes the proof.
\end{proof}

\subsection{Product and composition}

In this subsection, we introduce two operations on structures, namely product and composition, 
which allow us to treat a wide variety of additional structures. 
For further details, see \cite[Section~7]{Noda_pre_Metrization}.

We first define the product of structures.

\begin{dfn}[{\cite[Definition~5.12]{Noda_pre_Metrization}}]
  Fix $N \in \NN$ and structures $(\tau^{(n)})_{n=1}^N$. 
  The product structure $\tau = \prod_{n=1}^N \tau^{(n)}$ is defined as follows:
  \begin{enumerate}[label = \textup{(\roman*)}]
    \item 
      For each $\bcmAB$ space $S$, 
      set $\tau(S) \coloneqq \prod_{n=1}^N \tau^{(n)}(S)$, equipped with the product topology. 
    \item 
      For each isometric embedding $f \colon S_1 \to S_2$ between $\bcmAB$ spaces, 
      define $\tau_f \coloneqq \prod_{n=1}^N \tau^{(n)}_f$, that is,
      \begin{equation}
        \tau_f\bigl( (a_n)_{n=1}^N \bigr) 
        \coloneqq \bigl(\tau^{(n)}_f(a_n)\bigr)_{n=1}^N.
      \end{equation} 
  \end{enumerate}
  When $\tau^{(1)} = \cdots = \tau^{(n)} \eqqcolon \sigma$, we write $\tau = \sigma^{\otimes n}$.
  If each $\tau^{(n)}$ admits a metrization, 
  then we define a metrization of $\tau$ by equipping $\tau(S)$, for each rooted $\bcmAB$ space $S$, 
  with the associated max product metric.
\end{dfn}

The following result shows that the product structure inherits the Polishness of its component structures.

\begin{thm} [{\cite[Theorem~7.5]{Noda_pre_Metrization}}] \label{thm: product is Polish}
  Fix $N \in \NN$ and structures $(\tau^{(n)})_{n = 1}^N$. 
  If each $\tau_n$ is Polish, then so is $\prod_{n=1}^N \tau^{(n)}$.
\end{thm}

It is not always possible to compose two structures on $\bcmAB$ spaces. 
In Definition~\ref{dfn: space transformation} below we introduce a variant of structures 
that allows for such compositions. 
Consider a map $\Psi$ satisfying the following conditions.
\begin{enumerate}[label = (ST\arabic*), leftmargin = *, series = space transformation]
  \item \label{item: 1. ST}
    For every $\bcmAB$ space $S$, 
    there exists a $\bcmAB$ space $\Psi(S)$.
  \item 
    For every isometric embedding $f \colon S_1 \to S_2$ between $\bcmAB$ spaces, 
    there exists an isometric embedding $\Psi_f \colon \Psi(S_1) \to \Psi(S_2)$.
  \item 
    Let $S_1, S_2, S_3$ be $\bcmAB$ spaces. 
    For any isometric embeddings $f \colon S_1 \to S_2$ and $g \colon S_2 \to S_3$, 
    one has $\Psi_{g \circ f} = \Psi_g \circ \Psi_f$.
  \item \label{item: 4. ST}
    For any $\bcmAB$ space $S$, 
    one has $\Psi_{\id_S} = \id_{\Psi(S)}$.
\end{enumerate}
By forgetting the metric on $\Psi(S)$ for each $S$, 
we may regard $\Psi$ as a structure on $\bcmAB$ spaces in the sense of Definition~\ref{dfn: structure}. 
Hence the notion of continuity introduced in Definition~\ref{dfn: continuity of structure} applies to $\Psi$ as well.

\begin{dfn}[{\cite[Definition~7.6]{Noda_pre_Metrization}}] \label{dfn: space transformation}
  Let $\Psi$ satisfy \ref{item: 1. ST}--\ref{item: 4. ST}.
  We call $\Psi$ a \emph{space transformation} if it is continuous in the sense of Definition~\ref{dfn: continuity of structure},
  and if it satisfies the following additional condition:
  \begin{enumerate}[resume* = space transformation]
    \item There exists a collection $\mathfrak{r} = \{\mathfrak{r}_S \colon S \to \Psi(S)\}_S$ of topological embeddings indexed by $\bcmAB$ spaces $S$ 
      such that, for any isometric embedding $f \colon S_1 \to S_2$ between $\bcmAB$ spaces,
      $\Psi_f \circ \mathfrak{r}_{S_1} = \mathfrak{r}_{S_2} \circ f$.
  \end{enumerate}
  We refer to $\mathfrak{r}$ as a \emph{rooting system} of $\Psi$.
\end{dfn}

The terminology “rooting system’’ reflects the role of $\mathfrak{r}$ in assigning a canonical root to each transformed space.
Specifically, when $S$ is a rooted space with root $\rho_S$,
we define the root of $\Psi(S)$ by $\mathfrak{r}_S(\rho_S)$.

The following lemma is an immediate consequence of the lower semicontinuity of $\Psi$
(see \cite[Remark~6.36]{Noda_pre_Metrization}),
and will be used later.

\begin{lem} \label{lem: ST preserves Fell conv}
  Let $\Psi$ be a space transformation.
  Let $S, S_1, S_2, \dots$ be $\bcmAB$ spaces that are isometrically embedded into a common $\bcmAB$ space $M$
  in such a way that $S_n \to S$ in the Fell topology as closed subsets of $M$.
  Then $\Psi(S_n) \to \Psi(S)$ in the Fell topology as closed subsets of $\Psi(M)$.
\end{lem}

We now define the composition of a structure with a space transformation.

\begin{dfn} \label{dfn: composition}
  Let $\tau$ be a structure and $\Psi$ a space transformation with rooting system $\mathfrak{r}$.
  The composition $\tau \circ \Psi = \tau(\Psi)$ is defined as follows:
  \begin{enumerate}[label = \textup{(\roman*)}]
    \item 
      For each $\bcmAB$ space $S$, 
      set $(\tau \circ \Psi)(S) \coloneqq \tau(\Psi(S))$. 
    \item 
      For each isometric embedding $f \colon S_1 \to S_2$ between $\bcmAB$ spaces, 
      define $(\tau \circ \Psi)_f \coloneqq \tau_{\Psi_f}$.
  \end{enumerate}
  If $\tau$ admits a metrization, then for each rooted $\bcmAB$ space $(S, \rho_S)$, 
  we equip $\tau(\Psi(S))$ with the metric $d^\tau_{\Psi(S), \mathfrak{r}_S(\rho_S)}$.
\end{dfn}

Similarly to the case of product structures, 
composition also preserves Polishness.
\begin{thm} [{\cite[Theorem~7.14]{Noda_pre_Metrization}}] \label{thm: composition is Polish} 
  Let $\tau$ be a structure and $\Psi$ be a space transformation. If $\tau$ is Polish, then so is $\tau \circ \Psi$. 
\end{thm}

\begin{exm} \label{exm: space transformation}
  Here, we provide space transformations, used in this paper.
  Fix $k \in \NN$ and a rooted $\bcmAB$ space $(\Xi, \rho_\Xi)$.
  We define a space transformation $\Psi = \Psi_{\id^k \times \Xi}$ as follows.
  \begin{itemize}
    \item 
      For each $\bcmAB$ space $S$, 
      define $\Psi(S) \coloneqq S^k \times \Xi$ equipped with the max product metric. 
    \item 
      For each isometric embedding $f \colon S_1 \to S_2$ between $\bcmAB$ spaces, 
      define 
      \begin{equation}
        \Psi_f \coloneqq \underbrace{f \times \cdots \times f}_{k\ \text{times}} \times \id_\Xi.
      \end{equation}
  \end{itemize}
  Its rooting system $\mathfrak{r}$ is given by $\mathfrak{r}_X(x) = (x, \dots, x, \rho_\Xi) \in X^k \times \Xi$.
  Similarly, we define space transformations $\Psi_{\Xi \times \id^k}$ and $\Psi_{\id^k}$ such that
  $\Psi_{\Xi \times \id^k}(S) \coloneqq \Xi \times S^k$ and $\Psi_{\id^k}(S) = S^k$ for each $\bcmAB$ space $S$.
  When $k = 1$, we simply write $\Psi_{\id \times \Xi} = \Psi_{\id^1 \times \Xi}$, 
  $\Psi_{\Xi \times \id} = \Psi_{\Xi \times \id^1}$, and $\Psi_{\id} = \Psi_{\id^1}$.
\end{exm}

\subsection{Structures used in the present paper} \label{sec: structures used in the present paper}

Here we list the structures that will be used in the present paper.
We note that all the structures, except for $\STOMSt$ defined in \ref{st item: stom st} below,
have already been introduced in \cite{Noda_pre_Metrization}.

\begin{enumerate} [label = (S\arabic*), leftmargin = *]
  \item \label{item: fixed structure}
    \textbf{Fixed structure.}
    Fix a Polish space $\Xi$. We define a structure $\tau = \tau_\Xi$ as follows.
    \begin{itemize}
      \item 
        For each $\bcmAB$ space $S$,  
        define $\tau(S) \coloneqq \Xi$. 
      \item 
        For each isometric embedding $f \colon S_1 \to S_2$ between $\bcmAB$ spaces,
        define $\tau_{f} \coloneqq \id_{\Xi}$.
    \end{itemize}
    Let $d_\Xi$ be a metric on $\Xi$ inducing the given topology.
    We define a metrization of $\tau$ by equipping $\tau(S) = \Xi$ with the metric $d_\Xi$.
  \item \textbf{Point.}
    We define a structure $\tau = \tau_{\id}$ as follows.
    \begin{itemize}
      \item 
        For each $\bcmAB$ space $S$,
        define $\tau(S) \coloneqq S$.
      \item 
        For each isometric embedding $f \colon S_1 \to S_2$ between $\bcmAB$ spaces,
        define $\tau_f \coloneqq f$.
    \end{itemize}
    We define a metrization of $\tau$ by equipping $\tau(S) = S$ with the associated metric $d_S$.
  \item \textbf{Measure.}
    We define a structure $\tau = \MeasSt$ as follows. 
    \begin{itemize} 
      \item 
        For each $\bcmAB$ space $S$, 
        set $\tau(S) \coloneqq \Meas(S)$ equipped with the vague topology. 
      \item 
        For each isometric embedding $f \colon S_1 \to S_2$ between $\bcmAB$ spaces, 
        set $\tau_f(\mu) \coloneqq \mu \circ f^{-1}$ for each $\mu \in \Meas(S_1)$, 
        i.e., $\tau_f(\mu)$ is the pushforward of $\mu$ by $f$.
    \end{itemize}
    We define a metrization of $\tau$ by equipping, 
    for each rooted $\bcmAB$ space $(S, \rho_S)$, $\tau(S) = \Meas(S)$ with the metric $\VagueMet{S, \rho_S}$,
    as recalled from \eqref{eq: dfn of vague metric}.
  \item \label{st item: stom st}
    \textbf{STOM structure.}
    We define a structure $\tau = \STOMSt$ as follows.
    \begin{itemize}
      \item 
        For each $\bcmAB$ space $S$,
        define $\tau(S) \coloneqq \STOMMeas(S \times \RNp)$.
      \item 
        For each isometric embedding $f \colon S_1 \to S_2$ between $\bcmAB$ spaces, 
        set $\tau_f(\Pi) \coloneqq \Pi \circ (f \times \id_{\RNp})^{-1}$ for each $\Pi \in \STOMMeas(S_1 \times \RNp)$, 
        i.e., $\tau_f(\Pi)$ is the pushforward of $\Pi$ by $f \times \id_{\RNp}$.
    \end{itemize}
    Checking conditions in Definition~\ref{dfn: structure} is straightforward,
    and thus $\STOMSt$ is indeed a structure.
    We define a metrization of $\tau$ by equipping $\tau(S) = \STOMMeas(S \times \RNp)$ with the metric $\STOMMet{S \times \RNp}$,
    as recalled from \eqref{eq: dfn of stom met}.
    We note that, by Lemma~\ref{lem: preserving property of stom met},
    this indeed defines a metrization in the sense of Definition~\ref{dfn: metrization of structure}.
  \item \textbf{C\`adl\`ag functions}
    We define a structure $\tau = \SkorohodSt$ as follows. 
    \begin{itemize} 
      \item   
        For each $\bcmAB$ space $S$, 
        set $\tau(S) \coloneqq D_{J_1}(\RNp, S)$ equipped with the $J_1$-Skorohod topology. 
      \item 
        For each isometric embedding $f \colon S_1 \to S_2$ between $\bcmAB$ spaces, 
        set $\tau_f(\xi) \coloneqq f \circ \xi$ for each $\xi \in D_{J_1}(\RNp, S_1)$.
    \end{itemize}
    We define a metrization of $\tau$ by equipping $\tau(S) = D_{J_1}(\RNp, S)$ with the Skorohod metric $\SkorohodMet{S}$,
    as recalled from \eqref{eq: dfn of Skorohod metric}.
  \item \textbf{Measurable functions.}
    We define a structure $\tau = \LzeroSt$ as follows.
    \begin{itemize}
      \item 
        For each $\bcmAB$ space $S$,
        define $\tau(S) \coloneqq L^0(\RNp, S)$.
      \item 
        For each isometric embedding $f \colon S_1 \to S_2$ between $\bcmAB$ spaces,
        define $\tau_f(\xi) \coloneqq f \circ \xi$ for each $\xi \in L^0(\RNp, S_1)$.
    \end{itemize}
    We define a metrization of $\tau$ by equipping $\tau(S) = L^0(\RNp, S)$ with the metric $\LzeroMet_S$,
    as recalled from \eqref{eq: dfn of L^0 metric}.
  \item \textbf{Continuous functions with space-dependent domains.}
    Let $\Psi$ be space transformation and $\sigma$ be a Polish structure.
    We define a structure $\tau = \hatCSt(\Psi, \sigma)$ as follows.
    \begin{itemize}
      \item 
        For each $\bcmAB$ space $S$, 
        define $\tau(S) \coloneqq \hatC(\Psi(S), \sigma(S))$. 
      \item 
        For each isometric embedding $f \colon S_1 \to S_2$ between $\bcmAB$ spaces, 
        define $\tau_f(F) \coloneqq \sigma_f \circ F \circ \Psi_f^{-1}$ with $\dom(\tau_f(F)) \coloneqq \Psi_f(\dom(F))$ 
        for $F \in \tau(S_1)$.
    \end{itemize}
    We define a metrization of $\tau$ by equipping, for each rooted $\bcmAB$ space $(S, \rho_S)$,
    $\tau(S) = \hatC(\Psi(S), \sigma(S))$ 
    with the metric $\hatCMet{(\Psi(S), \mathfrak{r}(\rho_S))}{\Xi}$,
    as recalled from \eqref{eq: dfn of hatC metric}.
  \item \textbf{Law of structure.}
    Let $\sigma$ be a Polish structure.
    We define a structure $\tau = \ProbSt(\sigma)$ as follows.
    \begin{itemize}
      \item 
        For each $\bcmAB$ space $S$,
        define $\tau(S) \coloneqq \Prob(\sigma(S))$,
        i.e., the set of probability measures on $\sigma(S)$ equipped with the weak topology. 
      \item 
        For each isometric embedding $f \colon S_1 \to S_2$ between $\bcmAB$ spaces, 
        define $\tau_f(P) \coloneqq P \circ \sigma_f^{-1}$, i.e.,  the pushforward of $P$ by $\sigma_f$, for each $P \in \Prob(\sigma(S_1))$.
    \end{itemize}
    We define a metrization of $\tau$ by equipping, 
    for each rooted $\bcmAB$ space $(S, \rho_S)$,
    $\tau(S) = \Prob(\sigma(S))$ with the Prohorov metric 
    induced by the metric $d^\tau_{S, \rho_S}$ on $\sigma(S)$.
\end{enumerate}

\begin{prop}
  All the structures introduced above are Polish.
\end{prop}

\begin{proof}
  It suffices to show that $\STOMSt$ is Polish, 
  since the Polishness of the other structures was established in \cite[Section~8]{Noda_pre_Metrization}.
  The continuity of the embedding can be readily verified by the continuous mapping theorem.
  Moreover, the semicontinuity property can be proved 
  by following the argument used for $\MeasSt$ 
  in \cite[Theorem~8.9]{Noda_pre_Metrization}.
  Combining these facts with Proposition~\ref{prop: polishness of stom met},
  we deduce that $\STOMSt$ is separable and continuous,
  and the associated metrization is complete.
  Hence $\STOMSt$ is Polish, as desired.
\end{proof}

%
\section{Dual processes, PCAFs, and smooth measures} \label{sec: Dual processes, PCAFs, and smooth measures}

In this section we introduce the class of processes studied in this paper 
and summarize basic results on their positive continuous additive functionals (PCAFs) 
in relation to smooth measures.
Throughout this section,
we fix a locally compact separable metrizable topological space $S$.
We write $S_{\Delta} = S \cup \{\Delta\}$ for the one-point compactification of $S$.
(NB. If $S$ is compact, then we add $\Delta$ to $S$ as an isolated point.)
Any $[-\infty,\infty]$-valued function $f$ defined on $S$ is regarded as a function on $S_{\Delta}$ by setting $f(\Delta) \coloneqq 0$.

\subsection{Dual processes} \label{sec: Dual processes}

Here we introduce standard processes admitting dual processes and heat kernels,
together with related notation used throughout this paper.
For further background, see \cite{Blumenthal_Getoor_68_Markov}, for example.

Let 
\begin{equation}
  X = (\Omega^X, \sigalg^X, (X_{t})_{t \in [0, \infty)}, (P^x_X)_{x \in S_{\Delta}}, (\theta^X_t)_{t \in [0,\infty)})
\end{equation}
be a standard process on $S$
(see \cite[Definition~9.2 in Chapter~I]{Blumenthal_Getoor_68_Markov} or \cite[Definition~A.1.23(i)]{Chen_Fukushima_12_Symmetric}).
Here $(\Omega^X, \sigalg^X)$ denotes the underlying measurable space, 
and $\theta^X_t$ the shift operator,
that is, the map $\theta^X_t \colon \Omega^X \to \Omega^X$ satisfying $X_s \circ \theta^X_t = X_{s+t}$ for all $s \in [0, \infty]$.
Note that $X_{\infty}(\omega) = \Delta$ for every $\omega \in \Omega^X$.
We write $\filt^X_{*} = (\filt^X_{t})_{t \in [0,\infty]}$ 
for the minimal augmented admissible filtration (see \cite[p.\ 397]{Chen_Fukushima_12_Symmetric}),
and set 
\begin{equation}
  \zeta^X \coloneqq \inf\{t \in [0, \infty] \mid X_{t} = \Delta\}
\end{equation}
for the \emph{lifetime} of $X$.
We say that $X$ is \emph{conservative} if 
\begin{equation} \label{eq: conservativeness}
  P^x(\zeta^X = \infty) = 1, \qquad \forall x \in S.
\end{equation}
We often suppress the subscript or superscript $X$ when it is clear from the context.

Throughout this section,
we assume the existence of a dual process and a heat kernel, as specified below.
We will often refer to this assumption as the \emph{duality and absolute continuity condition}
(\emph{DAC condition}, for short).

\begin{assum} [{The DAC condition}]\label{assum: dual hypothesis} \leavevmode 
  There exist a standard process $\check{X}$, a $\sigma$-finite Borel measure $m$ on $S$,
  and a Borel measurable function $p = p_X \colon (0, \infty) \times S \times S \to [0,\infty]$ satisfying the following.
  \begin{enumerate} [label = \textup{(HK\arabic*)}, leftmargin = *]
    \item For all $t > 0$ and $x \in S$, $P^x(X_t \in dy) = p(t, x, y)\, m(dy)$.
    \item For all $t > 0$ and $y \in S$, $P^y(\check{X}_t \in dx) = \check{p}(t, y, x)\, m(dx)$,
    where $\check{p}(t, y, x) \coloneqq p(t, x, y)$.
    \item For any $s, t > 0$ and $x, y \in S$, 
      \begin{equation} \label{assum item eq: C-K equation}
        p(t+s, x, y) = \int_{S} p(t, x, z) p(s, z, y)\, m(dz).
      \end{equation}
  \end{enumerate}
\end{assum}

We recall from \cite[Corollary~1.12 and Remark~1.13, Chapter~VI]{Blumenthal_Getoor_68_Markov}
that $m$ is a $0$-excessive reference measure of $X$
(in the sense of \cite[Definitions~1.1, Chapter~V and 1.10, Chapter~VI]{Blumenthal_Getoor_68_Markov}).
Hence, under Assumption~\ref{assum: dual hypothesis},
the setting of \cite{Revuz_70_Mesures} applies,
and we may invoke the results therein.

\begin{rem} \label{rem: dual process}
  The process $\check{X}$ is called a \emph{dual process} of $X$, as it satisfies the following:
  for any non-negative measurable functions $f$ and $g$ on $S$, and any $t \geq 0$,
  \begin{equation} \label{rem eq: dual process}
    \int_S E_x[f(X_t)]\, g(x)\, m(dx) = \int_S f(x) E_x[g(\check{X}_t)]\, m(dx).
  \end{equation}
  It is known that, if the above identify holds and the semigroups of $X$ and $\check{X}$ are absolutely continuous 
  with respect to a common $\sigma$-finite Borel measure,
  then the heat kernel exists (see \cite[Theorem~1]{Yan_88_A_formula}).
\end{rem}

We refer to $p$ as the \emph{heat kernel} of $X$ with respect to $m$.
For each $\alpha \ge 0$, we then define the \emph{$\alpha$-potential (or resolvent) density}
$\ResolDensity^\alpha = \ResolDensity_p^\alpha \colon S \times S \to [0,\infty]$ by
\begin{equation} \label{2. eq: potential density}
  \ResolDensity^\alpha(x,y)
  \coloneqq
  \int_0^\infty e^{-\alpha t} p(t,x,y), dt.
\end{equation}
Given a Borel measure $\nu$ on $S$,
we define, the \emph{$\alpha$-potential} of $\nu$ by setting 
\begin{equation} \label{eq: dfn of potential}
  \Potential^\alpha \nu(x)
  =
  \Potential_p^\alpha\nu(x) 
  \coloneqq 
  \int_{S} \ResolDensity^\alpha(x,y)\, \nu(dy),
  \quad 
  x \in S.
\end{equation}

For a Borel subset $A \in \Borel(S_\Delta)$, 
we denote by $\sigma_{A} = \sigma^{X}_{A}$ and $\breve{\sigma}_A = \breve{\sigma}^X_A$ 
the first hitting and exit times of $X$ to $A$, respectively,
i.e.,
\begin{gather} 
  \sigma_{A}
  = 
  \sigma^{X}_{A} 
  \coloneqq 
  \inf\{t > 0 \mid X_{t} \in A\},
  \label{eq: dfn of hitting time}
  \\
  \breve{\sigma}_{A}
  = 
  \breve{\sigma}^{X}_{A} 
  \coloneqq 
  \inf\{t > 0 \mid X_{t} \notin A\}.
  \label{eq: dfn of exit time}
\end{gather}
Below, we recall some classes of exceptional sets.

\begin{dfn} [{\cite[Definition~3.1 in Chapter~II]{Blumenthal_Getoor_68_Markov}}] \label{dfn: exceptional set}
  Let $A$ be a Borel subset of $S$.
  \begin{enumerate} [label = \textup{(\roman*)}]
    \item The set $A$ is called \emph{polar} (for $X$) if and only if $P^x(\sigma_A < \infty) = 0$ for all $x \in S$.
    \item The set $A$ is called \emph{thin} (for $X$)  if and only if $P^x(\sigma_A = 0) = 0$ for all $x \in S$.
    \item The set $A$ is called \emph{semipolar} (for $X$)  if and only if it is contained in a countable union of  thin sets.
  \end{enumerate}
\end{dfn}

When $T$ is a separable metric space and $Y \colon (\Omega, \filt_\infty) \to (T, \mathcal{T})$ is measurable,
we define the map $\ProcLaw_Y \colon S \to \Prob(T)$ by 
\begin{equation} \label{eq: notation for law with processes}
  \ProcLaw_Y(x)(\cdot) \coloneqq P^x(Y \in \cdot).
\end{equation}
We refer to this as the \emph{law map} of $Y$.
For example, $\ProcLaw_X(x)$ denotes the law of $X$ starting from $x$.

\subsection{Smooth measures} \label{sec: smooth measure}

In this subsection, we first recall the definition of smooth measures,
and then introduce several classes of smooth measures that play an important role in our study.
We also present some auxiliary results on smooth measures that will be used later.
Our exposition mainly follows Revuz's framework \cite{Revuz_70_Mesures}.
For a more potential-theoretic perspective, see for example 
\cite{Beznea_Boboc_04_Potential,Chen_Fukushima_12_Symmetric,Fukushima_Oshima_Takeda_11_Dirichlet}.
We work within the framework introduced in the previous subsection.
In particular, the DAC condition (Assumption~\ref{assum: dual hypothesis}) is in force.

\begin{dfn} [{Smooth measure, \cite[Theorem~VI.1]{Revuz_70_Mesures}}] \label{dfn: smooth measure}
  A Borel measure $\mu$ on $S$ is said to be a \textit{smooth measure} of $X$ if it satisfies the following conditions:
  \begin{enumerate} [label = (\roman*)]
    \item \label{dfn item: 1. smooth measure}
      $\mu$ charges no semipolar sets, i.e., 
      $\mu(E) = 0$ for any semipolar Borel subset $E$ of $S$;
    \item \label{dfn item: 2. smooth measure}
      there exists an increasing sequence $(E_{n})_{n \geq 1}$ of Borel subsets of $S$,
      called a \emph{nest} of $\mu$,
      such that $\mu(E_{n}) < \infty$ for each $n$,
      \begin{equation}
        \|\Potential^1 (\mu|_{E_n})\|_\infty     
        =
        \sup_{x \in S} \int_{E_n} \ResolDensity^1(x,y)\, \mu(dy) < \infty, \quad \forall n \geq 1,
      \end{equation}
      and 
      $P^x( \lim_{n \to \infty} \breve{\sigma}_{E_n} \geq \zeta) = 1$ for all $x \in S$.
  \end{enumerate}
\end{dfn}

In many examples, the following condition, known as \emph{Hunt's hypothesis}, holds.
\begin{enumerate}[label = \textup{(H)}]
  \item \label{item: Hunt hypo} 
    Every semipolar set is polar.
\end{enumerate}
For instance, if $X$ is a Hunt process associated with a symmetric regular Dirichlet form, 
then condition~\ref{item: Hunt hypo} holds 
(see \cite[Theorem~4.1.3]{Fukushima_Oshima_Takeda_11_Dirichlet}).
Under \ref{item: Hunt hypo}, condition~\ref{dfn item: 1. smooth measure} 
in Definition~\ref{dfn: smooth measure} is implied by condition~\ref{dfn item: 2. smooth measure}.
This follows from the next result.
The proof is given in Appendix~\ref{appendix: proof of potential result}.

\begin{prop} \label{prop: bounded potential and polarity}
  Let $\mu$ be a Borel measure on $S$.
  If $\|\Potential^\alpha \mu\|_\infty < \infty$ for some $\alpha > 0$, then $\mu$ charges no polar sets.
\end{prop}

\begin{rem}
  The above is well known under the assumption 
  that all $\alpha$-excessive functions are lower semicontinuous 
  (see the comments following Proposition~4.3 in Chapter~VI of \cite{Blumenthal_Getoor_68_Markov}).
\end{rem}

Below, we introduce two important classes of smooth measures.

\begin{dfn} \label{dfn: Kato class}
  A finite Borel measure $\mu$ on $S$ is said to be \emph{in the Kato class (of $X$)}
  if and only if the following conditions are satisfied.
  \begin{enumerate} [label = \textup{(\roman*)}]
    \item \label{dfn item: 1. Dynkin and Kato}
     The Radon measure $\mu$ charges no semipolar sets.
    \item \label{dfn item: 2. Dynkin and Kato}
      It holds that $\displaystyle \lim_{\alpha \to \infty} \|\Potential^\alpha \mu\|_\infty =0$.
  \end{enumerate}
  A Borel measure $\mu$ on $S$ is said to be \emph{in the local Kato class (of $X$)}
  if, for all compact subset $K$ of $S$, the restriction $\mu|_K$ is in the Kato class.
  (In particular, $\mu$ is a Radon measure.)
\end{dfn}

Clearly, the local Kato class is wider than the Kato class.
It is easy to verify that a Radon measure in the local Kato class is a smooth measure
(cf.\ \cite[Proposition~2.6]{Mori_21_Kato_Sobolev}).
These classes impose conditions on the decay of $\alpha$-potentials as $\alpha \to \infty$. 
This corresponds to imposing a constraint on the short-time behavior of the heat kernel with respect to the measure in question. 
More precisely, the following result holds.

\begin{prop} \label{prop: Potential and hk behavior}
  Let $\mu$ be a Borel measure on $S$.
  For any open subset $D$ of $S$ and $t > 0$,
  \begin{align}
    \frac{1}{e}
    \sup_{x \in S} \int_0^t \int_D p(s, x, y)\, \mu(dy)\, ds
    &\leq \|\Potential^{1/t}(\mu|_D)\|_\infty\\
    &\leq 
    \frac{1}{1 - e^{-1}}
    \sup_{x \in \closure(D)} \int_0^t \int_D p(s, x, y)\, \mu(dy)\, ds,
    \label{prop eq: Potential and hk behavior}
  \end{align}
  where $\closure(D)$ denotes the closure of $D$ in $S$.
  As a consequence, 
  $\mu$ belongs to the local Kato class 
  if and only if it charges no semipolar sets 
  and,  for all compact subsets $K$ of $S$,
  \begin{equation}
    \lim_{t \to 0} \sup_{x \in K}  \int_0^t \int_K p(s, x, y)\, \mu(dy)\, ds = 0.
  \end{equation}
\end{prop}

To prove the above result, we use the following lemma,
which is an immediate consequence of the Dynkin--Hunt formula in its heat-kernel version.

\begin{lem} \label{lem: Dynkin--Hunt formula}
  Let $D$ be an open subset of $S$.
  It then holds that, for each $x \in S$, $y \in D$, and $t > 0$, 
  \begin{equation}
    p(t, x, y) = E^x\!\left[ p(t - \sigma_D, X_{\sigma_D}, y) \cdot 1_{\{\sigma_D < t\}} \right].
  \end{equation}
\end{lem}

\begin{proof}
  We apply \cite[Corollary~3.14]{Getoor_Sharpe_82_Excursions}, 
  taking the hitting time $\sigma_D$ as the terminal time in that corollary.
  Since $D$ is open, every point of $D$ is regular for $\sigma_D$, 
  i.e., $P^x(\sigma_D = 0) = 1$ for all $x \in D$.
  Therefore, by \cite[Corollary~3.14 and Remark~3.15]{Getoor_Sharpe_82_Excursions}, 
  we obtain the desired result.
\end{proof}

\begin{proof}[{Proof of Proposition~\ref{prop: Potential and hk behavior}}]
  Following the argument in \cite[Equation~(2.10)]{Mori_21_Kato_Sobolev},
  one can verify that 
  \begin{equation} \label{prop pr: 1. Potential and hk behavior}
    \|\Potential^{1/t}(\mu|_D)\|_\infty 
    \leq \frac{1}{1 - e^{-1}} \sup_{x \in S} \int_0^t \int_D p(s, x, y)\, \mu(dy)\, ds.
  \end{equation}
  It follows from Lemma~\ref{lem: Dynkin--Hunt formula} that 
  \begin{align}
    \int_0^t \int_D p(s, x,y)\, \mu(dy)\, ds 
    &= 
    E_x\!\left[ \int_0^t \int_D p(s - \sigma_D, X_{\sigma_D}, y) \cdot 1_{\{\sigma_D < s\}}\, \mu(dy)\, ds \right]\\
    &=
    E_x\!\left[ \int_0^{t - \sigma_D} \int_D p(s, X_{\sigma_D}, y)\, \mu(dy)\, ds \right]\\
    &\leq 
    \sup_{z \in \closure(D)}
    \int_0^t \int_D p(s, z, y)\, \mu(dy)\, ds.
  \end{align}
  This, combined with \eqref{prop pr: 1. Potential and hk behavior}, 
  yields that the second inequality in \eqref{prop eq: Potential and hk behavior} holds.
  The other inequality follows from the calculation
  \begin{align} \label{prop pr eq: 1. Potential and hk behavior}
    \int_0^t \int_D p(s, x, y)\, \mu(dy)\, ds
    & 
    \leq    
    e \int_0^t \int_D e^{-s/t} p(s, x, y)\, \mu(dy)\, ds
    \leq     
    e \|\Potential^{1/t} (\mu|_D)\|_\infty.
  \end{align}
\end{proof}

In the discrete space setting,
any Radon measure belongs to the local Kato class, as shown below.

\begin{prop} \label{prop: Kato meas for discrete space}
  Suppose that $S$ is finite or countably infinite, endowed with the discrete topology.
  Then any Radon measure $\mu$ on $S$ is in the local Kato class of $X$.
\end{prop}

\begin{proof}
  By the right-continuity of $X$, 
  it follows that $P^x(\sigma_{\{x\}} = 0) = 1$ for any $x \in S$.
  Hence the empty set is the only semipolar set,
  and thus $\mu$ charges no semipolar set.

  Observe that, for each $x,y \in S$ and $t > 0$, 
  \begin{equation}
    p(t, x, y) 
    = 
    \frac{P^x(X_t = y)}{m(\{y\})} 
    \leq     
    \frac{1}{m(\{y\})}.
  \end{equation}
  (Note that $m(\{y\}) > 0$ since $m$ is a reference measure.)
  This yields
  \begin{equation}
    \sup_{x \in S} \ResolDensity^\alpha(x,y) \leq \frac{1}{\alpha\, m(\{y\})},
    \quad \forall y \in S.
  \end{equation}
  Therefore, for any compact subset $K$ of $S$,
  \begin{equation}
    \|\Potential^\alpha(\mu|_K)\|_\infty    
    \leq     
    \frac{1}{\alpha} \sum_{y \in K} \frac{\mu(\{y\})}{m(\{y\})}.
  \end{equation}
  Because $K$ is finite, the right-hand side converges to $0$ as $\alpha \to \infty$,
  which establishes the desired result.
\end{proof}

\subsection{PCAFs, STOMs, and Revuz correspondence} \label{sec: PCAFs, STOMs and Revuz correspondence}

In this subsection, we introduce PCAFs and associated STOMs,
and then recall the Revuz correspondence \cite{Revuz_70_Mesures}.
Some basic results on STOMs are also presented.
For a comprehensive theory of PCAFs, see \cite{Blumenthal_Getoor_68_Markov}.
For a potential-theoretic study of them in relation to symmetric Dirichlet forms, 
see \cite{Chen_Fukushima_12_Symmetric,Fukushima_Oshima_Takeda_11_Dirichlet}.
We proceed within the same framework as in the previous subsection.

Since several versions of the definitions of PCAFs exist,
we clarify the definitions adopted in this paper, below.

\begin{dfn} [{PCAF, cf.\ \cite[p.\ 222 and 235]{Fukushima_Oshima_Takeda_11_Dirichlet}}] \label{dfn: PCAF}
  Let $A = (A_{t})_{t \geq 0}$ be an $\filt_{*}$-adapted non-negative stochastic process. 
  It is called a \emph{positive continuous additive functional (PCAF)} of $X$ if 
  there exists a set $\Lambda \in \filt_{\infty}$, called a \emph{defining set} of $A$, 
  satisfying the following.
  \begin{enumerate} [label = \textup{(\roman*)}]
    \item \label{dfn item: 1. PCAF}
      It holds that $P^x(\Lambda) = 1$ for all $x \in S_\Delta$.
    \item \label{dfn item: 2. PCAF}
      For every $t \in [0, \infty]$, $\theta_t(\Lambda) \subseteq \Lambda$.
    \item \label{dfn item: 3. PCAF} 
      For every $\omega \in \Lambda$,
      $A_{0}(\omega) = 0$,
      the function $[0, \infty) \ni t \mapsto A_{t}(\omega) \in [0, \infty]$ is continuous,
      $A_{t}(\omega) < \infty$ for all $t < \zeta(\omega)$, 
      $A_{t}(\omega) = A_{\zeta(\omega)}(\omega)$ for all $t \geq \zeta(\omega)$,
      and $A_{t+s}(\omega) = A_{t}(\omega) + A_{s}(\theta_{t}\omega)$
      for all $s, t \in [0, \infty)$.
  \end{enumerate} 
  We say that two PCAFs $A = (A_{t})_{t \geq 0}$ and $B = (B_{t})_{t \geq 0}$ are \emph{equivalent}
  if there exists a common defining set $\Lambda \in \filt_{\infty}$ 
  such that $P^x(\Lambda)=1$ for all $x \in S_\Delta$  
  and $A_t(\omega) = B_t(\omega)$ for all $t \geq 0$ and $\omega \in \Lambda$.
\end{dfn}

\begin{rem}
  By \cite[Remark~2.5]{Kajino_Noda_pre_Generalized},  
  for any PCAF $A$ with a defining set $\Lambda$,  
  there exists an equivalent PCAF $A'$ with defining set $\Lambda'$ such that  
  $A_t(\omega) = 0$ for all $t \geq 0$ and $\omega \notin \Lambda'$,  
  and $\{\zeta = 0\} \subseteq \Lambda'$.  
  Thus, in this paper, whenever we consider a PCAF $A$ with a defining set $\Lambda$,  
  we shall always assume that $A$ and $\Lambda$ satisfy this property.
\end{rem}

To define STOMs, we introduce a measurable space consisting of Borel measures.
Let $\BorMeas(S \times \RNp)$ be the set of Borel measures on $S \times \RNp$.
We equip $\BorMeas(S \times \RNp)$ with the $\sigma$-algebra generated by the family of evaluation maps 
$\pi_E \colon \BorMeas(S \times \RNp) \to \RNp$, $E \in \Borel(S \times \RNp)$, defined by $\pi_E(\Pi) \coloneqq \Pi(E)$. 

\begin{dfn} [{STOM}]
  For a PCAF $A = (A_t)_{t \geq 0}$, we define a measurable map $\Pi \colon (\Omega, \filt_*) \to \BorMeas(S \times \RNp)$ by 
  \begin{equation}
    \Pi(\omega)(\cdot) \coloneqq \int_0^\infty \mathbf{1}_{(X_t(\omega), t)}(\cdot)\, dA_t(\omega).
  \end{equation}
  We refer to $\Pi$ as the \emph{space-time occupation measure} associated with $A$.
\end{dfn}

\begin{rem}
  When $X$ is conservative, i.e., $P^x(\zeta = \infty) = 1$ for all $x \in S$,
  we may assume that $\Pi$ is a measurable map from $\Omega$ to $\STOMMeas(S \times \RNp)$.
\end{rem}

We now recall a fundamental result, known as the \emph{Revuz correspondence},
which provides a one-to-one correspondence between PCAFs and smooth measures.

\begin{thm}[{\cite[Proposition~V.1 and Theorem~VI.1]{Revuz_70_Mesures}}] 
  \label{thm: Revuz correspondence}
  There exists a one-to-one correspondence between the set of smooth measures 
  and the set of equivalence classes of PCAFs. 
  A smooth measure $\mu$ and a PCAF $A = (A_t)_{t \ge 0}$ correspond to each other 
  if and only if for every $\alpha > 0$ and every non-negative Borel measurable function $f$ on $S$, 
  \begin{equation} \label{eq: Revuz correspondence}
    E^x \!\left[
      \int_0^\infty e^{-\alpha t} f(X_t)\, dA_t
    \right]
    =
    \int_S \ResolDensity^\alpha(x,y) f(y)\, \mu(dy),
    \quad 
    \forall x \in S.
  \end{equation}
\end{thm}

Whenever we say that a PCAF $A$ is associated with a smooth measure $\mu$,  
we mean that they are associated via the Revuz correspondence.  
Through this correspondence, the corresponding notions for smooth measures also apply to PCAFs, and vice versa.  
For example, if $\mu$ belongs to the (local) Kato class, then we also say that $A$ belongs to the (local) Kato class,  
and the STOM $\Pi$ associated with $A$ is then referred to as the STOM associated with $\mu$.

\begin{rem}
  Revuz's definition of PCAFs \cite[Section~I.3]{Revuz_70_Mesures} is somewhat weak, 
  and in particular does not guarantee the existence of a defining set. 
  However, it is not difficult to modify his PCAF so that it satisfies the conditions of our definition. 
  Indeed, by \cite[Theorem~2.1 in Chapter~V]{Blumenthal_Getoor_68_Markov}, his PCAF can be assumed to be a so-called perfect PCAF. 
  The definition of perfect PCAFs consists of conditions~\ref{dfn item: 1. PCAF} and~\ref{dfn item: 3. PCAF} of Definition~\ref{dfn: PCAF} 
  (see \cite[Definition~1.3 in Chapter~IV]{Blumenthal_Getoor_68_Markov} for the precise formulation). 
  Moreover, one can easily verify that the defining set of a perfect PCAF, 
  as constructed in the proof of \cite[Theorem~3.16 in Chapter~IV]{Blumenthal_Getoor_68_Markov}, 
  also satisfies condition~\ref{dfn item: 2. PCAF} of Definition~\ref{dfn: PCAF}. 
  Thus, a perfect PCAF may be regarded as a PCAF in our sense.
\end{rem}

\begin{exm} \label{exm: example of PCAF}
  For any bounded Borel measurable function $f \colon S \to \RNp$,
  the stochastic process $A_t = \int_0^t f(X_s)\, ds$ is a PCAF and its corresponding smooth measure is $f(x)\, m(dx)$
  (cf.\ \cite[Theorem A.3.5(iii)]{Chen_Fukushima_12_Symmetric}).
  When the Dirac measure $\delta_x$ at $x$ is smooth, then the corresponding PCAF is a local time of $X$ at $x$
  in the sense of \cite[Definition~3.12 in Chapter~V]{Blumenthal_Getoor_68_Markov}.
\end{exm}

Below, we present some results (or applications) of \cite{Kajino_Noda_pre_Generalized,Noda_pre_Continuity}.
We note that, although the scope of these is restricted to Hunt processes associated with symmetric regular Dirichlet forms,
the arguments still apply in the present setting.

The following estimate of the difference between PCAFs in terms of their potentials is crucial in the proof of our main results.

\begin{lem} [{\cite[Theorem 1.6]{Noda_pre_Continuity}}]\label{lem: estimate of difference of PCAFs}
  Let $\mu$ and $\nu$ be finite smooth measures 
  such that $\|\Potential^1\mu\|_\infty < \infty$ and $\|\Potential^1\nu\|_\infty < \infty$.
  Let $A$ and $B$ be the associated PCAFs, respectively.
  It then holds that, for any $\alpha, T>0$,
  \begin{align}
    \sup_{x \in S}
    E_{x} 
    \left[\sup_{0 \leq t \leq T} |A_{t} - B_{t}|^{2} \right]
    & \leq 
    18 (\supnorm{\Potential^\alpha \mu} + \|\Potential^\alpha \nu\|_{\infty}) 
    \|\Potential^\alpha  \mu - \Potential^\alpha \nu\|_{\infty} \\
    & \quad
    + 
    4e^{2T}(1- e^{-\alpha T})( \supnorm{\Potential^{1}\mu}^{2} + \|\Potential^{1}\nu\|_{\infty}^{2}).
  \end{align}
\end{lem}

The following is an application of a generalized Kac's moment formula for PCAFs established in \cite{Kajino_Noda_pre_Generalized}.

\begin{prop} \label{prop: moment formula for STOM}
  Let $\mu$ be a smooth measure and $\Pi$ be the associated STOM.
  Fix an arbitrary Borel measurable function $f \colon S \times \RNp \to \RNp$ and $x \in S_\Delta$.
  For notational convenience, write 
  \begin{gather} \label{prop eq: moment formula for STOM. 1}
    \Pi(f) \coloneqq \int_0^\infty \int_S f(y, t)\, \Pi(dy\, dt) = \int_0^\infty f(X_t, t)\, dA_t,
  \end{gather}
  and $I_{\scriptscriptstyle \nearrow}^k = \{(t_1, \dots, t_k) \mid 0 < t_1 < \cdots < t_k < \infty\}$ for each $k \in \NN$.
  Then it holds that, for any $k \in \NN$ and $x_0 \in S$,
  \begin{equation}
      E^{x_0}\!\left[ \Pi(f)^k \right]
      = 
      k!
      \int_{I_{\nearrow}^k} \int_{S^k} \prod_{i=1}^k \bigl( p(t_i - t_{i-1}, x_{i-1}, x_i)\, f(x_i, t_i) \bigr)\, 
      \mu^{\otimes k}(dx)\, dt_1 \cdots dt_k,
  \end{equation}
  where we set $t_0 \coloneqq 0$.
  In particular
  \begin{equation}
    E^{x_0}\!\left[ e^{\Pi(f)} \right]
    = 
    1 
    + 
    \sum_{k = 1}^\infty \int_{I_{\nearrow}^k} \int_{S^k} \prod_{i=1}^k \bigl( p(t_i - t_{i-1}, x_{i-1}, x_i)\, f(x_i, t_i) \bigr)\, 
    \mu^{\otimes k}(dx)\, dt_1 \cdots dt_k.
  \end{equation}
\end{prop}

\begin{proof}
  This is a straightforward application of \cite[Theorem~2.7]{Kajino_Noda_pre_Generalized}.
  Indeed, since we have 
  \begin{equation}
    E^{x_0}\!\left[ \Pi(f)^k \right]
    = 
    k!\, E^{x_0}\!\left[ 
        \int_{I_{\nearrow}^k} \prod_{i=1}^k f(X_{t_i}, t_i)\, dA_{t_1} \cdots dA_{t_k}
    \right],
  \end{equation}
  we can follow the same line as in the proof of \cite[Theorem~2.9]{Kajino_Noda_pre_Generalized}
  to obtain \eqref{prop eq: moment formula for STOM. 1}.
  We thus omit the details.
\end{proof}

The above proposition yields the following corollary on the support of $\Pi$ (cf.\ \cite[Corollary~2.5]{Kallenberg_17_Random}).
This can be also obtained from the fact that a PCAF increases only when the process hits the support of the associated smooth measure
(cf.\ \cite[Theorem~3.8 in Chapter~V]{Blumenthal_Getoor_68_Markov}).

\begin{cor} \label{cor: support of STOM}
  Let $\mu$ be a smooth measure and $\Pi$ be the associated STOM.
  For each $x \in S_\Delta$, the STOM $\Pi$ is supported on $\supp(\mu) \times \RNp$, $P^x$-a.s.
  That is, 
  \begin{equation}
    \supp(\Pi) \subseteq \supp(\mu) \times \RNp, 
    \quad P^x\text{-a.s.}
  \end{equation}
  Equivalently, for all $x \in S_\Delta$,
  \begin{equation}
    \int_0^t \mathbf{1}_{\supp(\mu)}(X_s)\, dA_s = A_t,\quad \forall t \geq 0,\quad   P^x \text{-a.s.}
  \end{equation}
\end{cor}

When the state space $S$ is discrete,  
the STOM associated with a Radon measure admits an explicit description, as follows.  
Recall from Proposition~\ref{prop: Kato meas for discrete space} that,  
in this case, every Radon measure belongs to the local Kato class,  
so the associated STOM is well defined.

\begin{prop} \label{prop: representation of STOM for discrete space}
  Suppose that $S$ is finite or countable, endowed with the discrete topology.
  Fix a Radon measure $\mu$ on $S$.
  Then the STOM $\Pi$ associated with $\mu$ can be represented as follows:
  \begin{equation} \label{prop eq: representation of STOM for discrete space}
    \Pi(dx\, dt) 
    = 
    \frac{1}{m(\{x\})} 
    \mathbf{1}_{\{X_t = x\}}\,
    \mu(dx)\, dt.
  \end{equation}
  (NB. Since $m$ is a reference measure and $\{x\}$ is open, the denominator $m(\{x\})$ is strictly positive.)
\end{prop}

\begin{proof}
  If we set $f(x) \coloneqq \mu(\{x\})/m(\{x\})$ for each $x \in S$,
  then $\mu(dx) = f(x)\, m(dx)$. 
  Thus, the PCAF associated with $\mu$ is given by 
  \begin{equation}
    A_t \coloneqq 
    \int_0^t \frac{\mu(\{X_s\})}{m(\{X_s\})}\, ds 
    = \sum_{x \in S} \int_0^t \frac{\mu(\{x\})}{m(\{x\})}\, \mathbf{1}_{\{X_s = x\}}\, ds
  \end{equation}
  (see Example~\ref{exm: example of PCAF}).
  By the definition of STOMs, we obtain the desired result.
\end{proof}

\section{Approximation of PCAFs and STOMs} \label{sec: approx of PCAFs and STOMs}

In Sections~\ref{sec: The approximation scheme for PCAFs} and \ref{sec: The approximation scheme for STOMs}, 
we establish methods for approximating PCAFs and STOMs by those that are absolutely continuous.
This construction plays a crucial role in the proofs of our main results.
In the following two subsections, we apply these approximation results 
to show that the joint law of the process and its PCAFs/STOMs is continuous in the starting point,
provided that the law of the process itself is continuous.

We work in the same setting as in Section~\ref{sec: smooth measure}.
Throughout this section,
we fix a boundedly compact metric $d_S$ on $S$ that induces the given topology
(see \cite[Theorem~1]{Williamson_Janos_87_Construction} for the existence of such a metric),
and we fix an element $\rho \in S$ as the root of $S$.
In particular, $(S, d_S, \rho)$ is a rooted $\bcmAB$ space.

\subsection{The approximation scheme for PCAFs} \label{sec: The approximation scheme for PCAFs}

We first present an approximation result for PCAFs in the Kato class, see Proposition~\ref{prop: delta approx of PCAF in Kato} below,
and then extend the method to PCAFs in the local Kato class, see Propositions~\ref{prop: R approx in local Kato} and \ref{prop: delta approx for local Kato}.

We start by considering PCAFs in the Kato class.
The approximation is constructed by using the heat kernel $p$ as a mollifier,
which is inspired by approximations for mutual intersection measures \cite[Definition~1.3]{Mori_20_LargeDeviations}.
The validity of this approximation will be justified by an application of Lemma~\ref{lem: estimate of difference of PCAFs}.

\begin{prop} \label{prop: delta approx of PCAF in Kato}
  Let $\mu$ be a finite Borel measure in the Kato class and let $A = (A_t)_{t \geq 0}$ be the associated PCAF.
  For each $\delta > 0$, define 
  \begin{equation} \label{prop eq: 1. simple approx of PCAF in Kato}
    A^{(\delta, *)}_t 
    = 
    \frac{1}{\delta} \int_0^t \int_S \int_\delta^{2\delta} p(u, X_s, z)\, du\, \mu(dz)\, ds,
    \quad t \geq 0.
  \end{equation}
  Then there exist universal constants $C_1, C_2 > 0$ such that the following holds:
  for any $T \in [0, \infty)$,
  $\alpha \in (0,1)$, and $\delta \in (0,1)$,
  \begin{align}
    &\sup_{x \in S} 
    E^x\!
    \left[
      \sup_{0 \leq t \leq T}
      \bigl| A^{(\delta, *)}_t - A_t \bigr|^2
    \right]
    \\
    &\leq    
    C_1 \bigl\| \Potential^\alpha \mu \bigr\|_\infty
    \Bigl(
      \delta \alpha 
      \bigl\| \Potential^\alpha \mu \bigr\|_\infty
      + 
      \bigl\| \Potential^{1/\delta}\!\mu \bigr\|_\infty
    \Bigr)
    +
    C_2 e^{2T} (1 - e^{-\alpha T}) \bigl\| \Potential^1 \mu \bigr\|_\infty^2.
    \label{prop eq: 2. simple approx of PCAF in Kato}
  \end{align}
  In particular, for all $T > 0$,
  \begin{equation} \label{prop eq: 3. simple approx of PCAF in Kato}
    \lim_{\delta \to 0}
    \sup_{x \in S} 
    E^x\!
    \left[
      \sup_{0 \leq t \leq T}
      \bigl| A^{(\delta, *)}_t - A_t \bigr|^2
    \right]
    = 0.
  \end{equation}
  Consequently, $A^{(\delta, *)} \xrightarrow[\delta \to 0]{\mathrm{p}} A$ 
  in $\upC(\RNp, \RNp)$ under $P^x$ for each $x \in S$.
\end{prop}

\begin{proof}
  Define 
  \begin{equation}
    f(y) \coloneqq \frac{1}{\delta} \int_S \int_\delta^{2\delta} p(s, y, z)\, ds\, \mu(dz),
    \quad y \in S.
  \end{equation}
  Since $\mu$ is in the Kato class,
  Proposition~\ref{prop: Potential and hk behavior} yields that $f$ is bounded.
  Hence, by Example~\ref{exm: example of PCAF}, 
  we deduce that $A^{(\delta, *)}$ is a PCAF and the associated smooth measure is 
  \begin{equation}
    \mu_\delta(dy)
    \coloneqq 
    f(y)\, m(dy)
    =
    \left(
      \frac{1}{\delta} \int_S \int_\delta^{2\delta} p(s, y, z)\, ds\, \mu(dz)
    \right)\!
    m(dy).
  \end{equation}
  Using the Chapman--Kolmogorov equation \eqref{assum item eq: C-K equation},
  we obtain that 
  \begin{align}
    \Potential^\alpha \mu_\delta(x)
    &=
    \int_S r^\alpha(x, y)\, \mu_\delta(dy)
    \\
    &= 
    \int_S \int_0^\infty e^{-\alpha t}\, p(t, x, y)\, dt 
    \left(
      \frac{1}{\delta} \int_S \int_\delta^{2\delta} p(s, y, z)\, ds\, \mu(dz)
    \right)\!
    m(d y)
    \\
    &= 
    \frac{1}{\delta} 
    \int_\delta^{2\delta} \int_S \int_0^\infty e^{-\alpha t}\, p(t+ s, x, z)\, dt\, \mu(dz)\, ds
    \\
    &=
    \frac{1}{\delta} 
    \int_\delta^{2\delta} e^{\alpha s}\int_S \int_s^\infty e^{-\alpha t}\, p(t, x, z)\, dt\, \mu(dz)\, ds.
  \end{align}
  By definition,
  \begin{equation}
    \int_S \int_s^\infty e^{-\alpha t}\, p(t, x, z)\, dt\, \mu(dz)
    = 
    \Potential^\alpha \mu(x)
    - 
    \int_S \int_0^s e^{-\alpha t}\, p(t, x, z)\, dt\, \mu(dz).
  \end{equation}
  Hence,
  \begin{align} \label{pr eq: simple approx of PCAF in Kato}
    \Potential^\alpha \mu_\delta(x) 
    =
    \frac{e^{2\delta \alpha} - e^{\delta \alpha}}{\delta \alpha} \Potential^\alpha \mu(x) 
    -
    \frac{1}{\delta} 
    \int_\delta^{2\delta} e^{\alpha s}\int_S \int_0^s e^{-\alpha t}\, p(t, x, z)\, dt\, \mu(dz)\, ds.
  \end{align}
  It follows that 
  \begin{align}
    &\bigl\|
      \Potential^\alpha \mu_\delta
      -
      \Potential^\alpha \mu
    \bigr\|_\infty
    \\
    &\leq 
    \left| \frac{e^{2\delta \alpha} - e^{\delta \alpha}}{\delta \alpha} - 1 \right|
    \bigl\| \Potential^\alpha \mu \bigr\|_\infty
    + 
    \frac{e^{2\delta \alpha} - e^{\delta \alpha}}{\delta \alpha}
    \sup_{x \in S}
    \int_S \int_0^{2\delta} p(t, x, z)\, dt\, \mu(dz).
  \end{align}
  Using the Taylor expansion of $e^t$,
  one can verify that there exists constants $c_1, c_2 > 0$ such that 
  \begin{equation}
    \left| \frac{e^{2t} - e^{t}}{t} - 1 \right| \leq c_1t,\quad 
    \left| \frac{e^{2t} - e^{t}}{t} \right| \leq c_2,
    \quad 
    \forall t \in [0,1].
  \end{equation}
  Moreover, following the proof of Proposition~\ref{prop: Potential and hk behavior},
  we deduce that 
  \begin{equation}
    \sup_{x \in S}
    \int_S \int_0^{2\delta} p(t, x, z)\, dt\, \mu(dz)
    \leq
    e^2 \bigl\| \Potential^{1/\delta} \mu \bigr\|_\infty.
  \end{equation}
  Therefore, we obtain that 
  \begin{equation}
    \bigl\|
      \Potential^\alpha \mu_\delta
      -
      \Potential^\alpha \mu
    \bigr\|_\infty
    \leq
    c_1\delta \alpha \bigl\| \Potential^\alpha \mu \bigr\|_\infty
    + 
    c_2 e^2 \bigl\| \Potential^{1/\delta}\!\mu \bigr\|_\infty.
  \end{equation}
  From \eqref{pr eq: simple approx of PCAF in Kato}, we also obtain that 
  \begin{equation}
    \bigl\|\Potential^\alpha \mu_\delta\bigr\|_\infty
    \leq    
    c_2
    \bigl\|\Potential^\alpha \mu\bigr\|_\infty,
    \quad \forall \alpha > 0.
  \end{equation}
  Therefore,
  applying Lemma~\ref{lem: estimate of difference of PCAFs},
  we deduce that 
  \begin{align}
    \sup_{x \in S} 
    E^x
    \left[
      \sup_{0 \leq t \leq T}
      \bigl| A^{(\delta, *)}_t - A_t \bigr|^2
    \right]
    &\leq     
    18 
    \bigl(
      c_2 \bigl\|\Potential^\alpha \mu\bigr\|_\infty
      + 
      \bigl\| \Potential^\alpha \mu \bigr\|_\infty
    \bigr)
    \Bigl(
      c_1\delta \alpha 
      \bigl\| \Potential^\alpha \mu \bigr\|_\infty
      + 
      c_2 e^2
      \bigl\| \Potential^{1/\delta}\!\mu \bigr\|_\infty
    \Bigr)
    \\
    &\qquad \qquad
    +
    4 e^{2T} (1 - e^{-\alpha T}) 
    \bigl(
      c_2^2 \bigl\| \Potential^1 \mu \bigr\|_\infty^2
      + 
      \bigl\| \Potential^1 \mu \bigr\|_\infty^2
    \bigr)
    \\
    &\leq    
    18(c_2+1) \bigl\| \Potential^\alpha \mu \bigr\|_\infty
    \Bigl(
      c_1\delta \alpha 
      \bigl\| \Potential^\alpha \mu \bigr\|_\infty
      + 
      c_2 e^2
      \bigl\| \Potential^{1/\delta}\!\mu \bigr\|_\infty
    \Bigr)
    \\
    &\qquad \qquad
    +
    4 e^{2T} (1 - e^{-\alpha T})(c_2^2 + 1) \bigl\| \Potential^1 \mu \bigr\|_\infty^2.
  \end{align}
  Thus, \eqref{prop eq: 2. simple approx of PCAF in Kato} holds.
  Since $\mu$ is in the Kato class,
  it follows from \eqref{prop eq: 2. simple approx of PCAF in Kato}
  \begin{equation}
    \limsup_{\delta \to 0}
    \sup_{x \in S} 
    E^x\!
    \left[
      \sup_{0 \leq t \leq T}
      \bigl| A^{(\delta, *)}_t - A_t \bigr|^2
    \right]
    \leq    
    C_2 e^{2T} (1 - e^{-\alpha T}) \bigl\| \Potential^1 \mu \bigr\|_\infty^2.
  \end{equation}
  Letting $\alpha \to 0$ in the above inequality, we obtain \eqref{prop eq: 3. simple approx of PCAF in Kato}.
\end{proof}

We next consider PCAFs in the local Kato class and develop suitable approximation schemes.  
Here, we note that to continue discussing the approximation of PCAFs in $\upC(\RNp, \RNp)$, 
we need to impose additional assumptions. 
Indeed, PCAFs in the local Kato class may explode in finite time, 
which makes it impossible to discuss convergence in the space $\upC(\RNp, \RNp)$. 
This observation leads to the following points.
\begin{enumerate}[label=\textup{(\alph*)}]
  \item  \label{item: issue for general PCAF conv. 1}
    To discuss convergence in the space $\upC(\RNp, \RNp)$, 
    additional assumptions on the process $X$ or on the PCAFs are required.
  \item  \label{item: issue for general PCAF conv. 2}
    A different framework is needed to handle convergence of general PCAFs.
\end{enumerate}

One possible approach in the direction of~\ref{item: issue for general PCAF conv. 1} 
is to assume that the process $X$ is conservative (recall this property from \eqref{eq: conservativeness}).
Under this assumption, the definition of PCAFs immediately implies that 
each PCAF belongs to $\upC(\RNp, \RNp)$ almost surely. 
In this paper, we continue to discuss convergence of general PCAFs in $\upC(\RNp, \RNp)$ under this conservativeness assumption.  

Nevertheless, it is worth noting that one can also develop a framework for studying 
convergence of PCAFs without imposing any additional assumptions, 
in accordance with~\ref{item: issue for general PCAF conv. 2}. 
The key idea is to handle STOMs within a suitably chosen topology. 
See Remark~\ref{rem: flexibility of STOM framework} in the next subsection for further details.

We now return to our main line of argument.
To establish an approximation method for PCAFs in the local Kato class, 
we introduce, in Definitions~\ref{dfn: restriction of PCAF} and \ref{dfn: apPCAF} below,  
two auxiliary maps that will serve as building blocks for the construction of the approximations.
Although these maps are formulated in rather general forms, independent of the process or PCAFs themselves,  
this level of generality will be useful for the main results discussed in Section~\ref{sec: conv of PCAFs and STOMs}.

The first map is used to truncate PCAFs
so that associated smooth measures are restricted onto a compact ball.

\begin{dfn} \label{dfn: restriction of PCAF}
  For each $R > 1$
  we define a map 
  \begin{equation}
    \apPCAF^{(*, R)} \colon 
    D_{L^0}(\RNp, S) \times \upC(\RNp, \RNp) 
    \to
    \upC(\RNp, \RNp)
  \end{equation}
  by 
  \begin{equation}
    (\eta, \varphi) 
    \mapsto 
    \left(
      t 
      \mapsto
      \int_0^t \int_{R-1}^R \mathbf{1}_{S^{(r)}}(\eta_s)\, dr\, d\varphi_s
    \right).
  \end{equation}
\end{dfn}

\begin{prop} \label{prop: R approx in local Kato}
  Assume that $X$ is conservative.
  Let $\mu$ be a smooth measure, and let $A = (A_t)_{t \geq 0}$ denote the associated PCAF.
  Fix $R > 1$.
  Set $A^{(*, R)} \coloneqq \apPCAF^{(*, R)}(X, A)$.
  Then $A^{(*, R)}$ is a PCAF, and its associated smooth measure is $\tilde{\mu}^{(R)}$,
  which is defined in \eqref{eq: smooth truncation}
  Moreover, we have 
  \begin{equation}  \label{prop eq: R approx in local Kato. 1}
    A_t^{(*,R)} = A_t\quad \text{for all $t \in [0,T]$ on the event}\quad \left\{ \exitTime_{S^{(R-1)}} > T \right\}.
  \end{equation}
  In particular, for all $x \in S$ and $T > 0$,
  \begin{equation}   \label{prop eq: R approx in local Kato. 2}
    \lim_{R \to \infty}
    E^x\!\left[ \left( \sup_{0 \leq t \leq T} |A^{(*, R)}_t - A_t| \right) \wedge 1 \right]
    = 0.
  \end{equation}
  Consequently, $A^{(*,R)} \xrightarrow[R \to \infty]{\mathrm{p}} A$ in $\upC(\RNp, \RNp)$ under $P^x$, for each $x \in S$.
\end{prop}

\begin{proof}
  The first assertion follows from \cite[Theorem~A.3.5]{Chen_Fukushima_12_Symmetric}.
  The assertion \eqref{prop eq: R approx in local Kato. 1} is immediate from the definition of $A^{(*,R)}$.
  Using the assertion, we deduce that, for any $x \in S$ and $T > 0$,
  \begin{equation}
    E^x\!\left[ \left( \sup_{0 \leq t \leq T} |A^{(*, R)}_t - A_t| \right) \wedge 1 \right]
    \leq
    P^x\!\left( \exitTime_{S^{(R-1)}} \leq T \right).
  \end{equation}
  The right-hand side converges to $0$ as $R \to \infty$ by the conservativeness of $X$,
  which shows \eqref{prop eq: R approx in local Kato. 2}.
  The last assertion is elementary.
\end{proof}

After the above truncation,
we can approximate the truncated PCAF $A^{(*, R)}$ by absolutely continuous PCAFs,
constructed from the following map.
This map is indeed modeled after the approximation in \eqref{prop eq: 1. simple approx of PCAF in Kato}.

\begin{dfn} \label{dfn: apPCAF}
  For each $\delta > 0$ and $R > 1$,
  we define a map 
  \begin{equation}
    \apPCAF^{(\delta, R)} \colon 
    C(\RNpp \times S \times S, \RNp) \times \Meas(S) \times D_{L^0}(\RNp, S) 
    \to 
    \upC(\RNp, \RNp)
  \end{equation}
  by 
  \begin{equation}
    (q, \nu, \eta) \mapsto 
    \left(
      t 
      \mapsto
      \frac{1}{\delta}
      \int_0^t \int_S \int_\delta^{2\delta} q(u, \eta_s, x)\, du\, \tilde{\nu}^{(R)}(dx)\, ds
    \right),
  \end{equation}
  where the smooth truncation $\tilde{\nu}^{(R)}$ is recalled from \eqref{eq: smooth truncation}.
\end{dfn}

The following is an immediate consequence of Proposition~\ref{prop: delta approx of PCAF in Kato}.

\begin{prop} \label{prop: delta approx for local Kato}
  Let $\mu$ be a Radon measure in the local Kato class and let $A = (A_t)_{t \geq 0}$ be the associated PCAF.
  For each $\delta > 0$ and $R > 1$, set 
  \begin{equation} \label{prop eq: 1. delta approx for local Kato}
    A^{(\delta, R)} \coloneqq \apPCAF^{(\delta, R)}(p, \mu, X),
    \qquad 
    A^{(*, R)} \coloneqq \apPCAF^{(*,R)}(X, A).
  \end{equation}
  Let $C_1, C_2 > 0$ be the universal constants appearing in the assertion of Proposition~\ref{prop: delta approx of PCAF in Kato}.
  Then, for any $T \in [0, \infty)$,
  $\alpha \in (0,1)$, and $\delta \in (0,1)$,
  \begin{align}
    \sup_{x \in S} 
    E^x\!
    \left[
      \sup_{0 \leq t \leq T}
      \bigl| A^{(\delta, R)}_t - A^{(*,R)}_t \bigr|^2
    \right]
    &\leq    
    C_1 \bigl\| \Potential^\alpha \mu^{(R)} \bigr\|_\infty
    \Bigl(
      \delta \alpha 
      \bigl\| \Potential^\alpha \mu^{(R)} \bigr\|_\infty
      + 
      \bigl\| \Potential^{1/\delta}\!\mu^{(R)} \bigr\|_\infty
    \Bigr)\\
    &\quad
    +
    C_2 e^{2T} (1 - e^{-\alpha T}) \bigl\| \Potential^1 \mu^{(R)} \bigr\|_\infty^2.
    \label{lem eq: 2. unif approx of PCAF in local Kato}
  \end{align}
  In particular, for each $R > 1$,
  \begin{equation} \label{lem eq: 3. unif approx of PCAF in local Kato}
    \lim_{\delta \to 0}
    \sup_{x \in S} 
    E^x\!
    \left[
      \sup_{0 \leq t \leq T}
      \bigl| A^{(\delta, R)}_t - A^{(*,R)}_t \bigr|^2
    \right]
    = 0.
  \end{equation}
\end{prop}

\begin{proof}
  Recall from Proposition~\ref{prop: R approx in local Kato} that $A^{(*,R)}$ is a PCAF associated with the Kato class measure $\tilde{\mu}^{(R)}$.
  Thus, applying Proposition~\ref{prop: delta approx of PCAF in Kato} to the PCAF $A^{(*,R)}$,
  we deduce that 
  \begin{align}
    &\sup_{x \in S} 
    E^x\!
    \left[
      \sup_{0 \leq t \leq T}
      \bigl| A^{(\delta, *)}_t - A_t \bigr|^2
    \right]
    \\
    &\leq    
    C_1 \bigl\| \Potential^\alpha \tilde{\mu}^{(R)} \bigr\|_\infty
    \Bigl(
      \delta \alpha 
      \bigl\| \Potential^\alpha \tilde{\mu}^{(R)} \bigr\|_\infty
      + 
      \bigl\| \Potential^{1/\delta}\!\tilde{\mu}^{(R)} \bigr\|_\infty
    \Bigr)
    +
    C_2 e^{2T} (1 - e^{-\alpha T}) \bigl\| \Potential^1 \tilde{\mu}^{(R)} \bigr\|_\infty^2
  \end{align}
  for all $T > 0$, $\alpha \in (0,1)$, and $\delta \in (0,1)$.
  By the definition of $\tilde{\mu}^{(R)}$ (see \eqref{eq: smooth truncation}), we have 
  \begin{equation}
    \bigl\| \Potential^\alpha \tilde{\mu}^{(R)} \bigr\|_\infty \leq \bigl\| \Potential^\alpha \mu^{(R)} \bigr\|_\infty,
  \end{equation}
  which yields the desired result.
\end{proof}

Combining Propositions~\ref{prop: R approx in local Kato} and \ref{prop: delta approx for local Kato}, 
we see that PCAFs in the local Kato class can be approximated in two stages as follows.

\begin{cor} \label{cor: approx in local Kato}
  Assume that $X$ is conservative.
  Let $\mu$ be a smooth measure in the local Kato class, and let $A = (A_t)_{t \geq 0}$ denote the associated PCAF.
  For each $\delta > 0$ and $R > 1$, define 
  \begin{equation} 
    A^{(\delta, R)} \coloneqq \apPCAF^{(\delta, R)}(p, \mu, X),
    \qquad 
    A^{(*, R)} \coloneqq \apPCAF^{(*,R)}(X, A).
  \end{equation}
  Then, for each $x \in S$, the following convergences hold in $\upC(\RNp, \RNp)$ under $P^x$:
  \begin{align}
    A^{(\delta, R)} &\xrightarrow[\delta \to 0]{\mathrm{p}} A^{(*,R)}, \qquad \forall R > 1,\\
    A^{(*, R)} &\xrightarrow[R \to \infty]{\mathrm{p}} A.
  \end{align}
\end{cor}

\subsection{The approximation scheme for STOMs} \label{sec: The approximation scheme for STOMs}

In this subsection,
we extend the approximation method for PCAFs introduced in the previous section to STOMs.
The approximation for STOMs in the Kato class is presented in Proposition~\ref{prop: delta STOM approx in Kato} below,
and for STOMs in the local Kato class is discussed in Propositions~\ref{prop: R STOM approx in local Kato} and \ref{prop: delta STOM approx for local Kato}.

We first establish an approximation for STOMs in the Kato class,
which corresponds to Proposition~\ref{prop: delta approx of PCAF in Kato}.

\begin{prop} \label{prop: delta STOM approx in Kato}
  Let $\mu$ be a finite Borel measure in the Kato class and let $\Pi$ be the associated STOM.
  For each $\delta > 0$, define 
  \begin{equation} \label{prop eq: 1. delta STOM approx in Kato}
    \Pi^{(\delta, *)}(dz\, dt)
    = 
    \frac{1}{\delta} \int_\delta^{2\delta} p(u, X_t, z)\, du\, \mu(dz) dt.
  \end{equation}
  Then there exist universal constants $C_3, C_4 > 0$ such that the following holds:
  for any $T \in [0, \infty)$, $E \in \Borel(S)$,
  $\alpha \in (0,1)$, and $\delta \in (0,1)$,
  \begin{align}
    &\sup_{x \in S} 
    E^x\!
    \left[
      \sup_{0 \leq s < t \leq T}
      \bigl| \Pi^{(\delta, R)}(E \times [s,t)) - \Pi^{(*,R)}(E \times [s, t)) \bigr|^2
    \right]
    \\
    &\leq    
    C_3 \bigl\| \Potential^\alpha \mu^{(R)} \bigr\|_\infty
    \Bigl(
      \delta \alpha 
      \bigl\| \Potential^\alpha \mu^{(R)} \bigr\|_\infty
      + 
      \bigl\| \Potential^{1/\delta}\!\mu^{(R)} \bigr\|_\infty
    \Bigr)
    +
    C_4 e^{2T} (1 - e^{-\alpha T}) \bigl\| \Potential^1 \mu^{(R)} \bigr\|_\infty^2.
    \label{prop eq: 2. delta STOM approx in Kato}
  \end{align}
  In particular, for each $R > 1$, $E \in \Borel(S)$, and $T > 0$
  \begin{equation} \label{prop eq: 3. delta STOM approx in Kato}
    \lim_{\delta \to 0}
    \sup_{x \in S} 
    E^x\!
    \left[
      \sup_{0 \leq s < t \leq T}
      \bigl| \Pi^{(\delta, R)}(E \times [s,t)) - \Pi^{(*,R)}(E \times [s, t)) \bigr|^2
    \right]
    = 0.
  \end{equation}
  Consequently, $\Pi^{(\delta, *)} \xrightarrow[\delta \to 0]{\mathrm{p}} \Pi$ 
  in $\STOMMeas(S \times \RNp)$ under $P^x$ for each $x \in S$.
\end{prop}

\begin{proof}
  Using the simple estimate that $(a+b)^2 \leq 2(a^2 + b^2)$ for any $a, b \geq 0$,
  we have 
  \begin{align}
    &\sup_{x \in S}E^x\!
    \left[
      \sup_{0 \leq s < t \leq T}
      \bigl| \Pi^{(\delta, *)}(E \times [s,t)) - \Pi(E \times [s, t)) \bigr|^2
    \right]\\
    &\leq
    4\sup_{x \in S} 
    E^x\!
    \left[
      \sup_{0 \leq t \leq T}
      \bigl| \Pi^{(\delta, *)}(E \times [0,t)) - \Pi(E \times [0, t)) \bigr|^2
    \right].
  \end{align}
  By the definition of the STOM $\Pi$,
  the process $B_t \coloneqq \Pi(E \times [0,t))$, $t \geq 0$, is a PCAF associated with the smooth measure $\mu|_E$.
  If we let $B^{(\delta, *)} = (B^{(\delta, *)}_t)_{t \geq 0}$ be the approximation of $B$ defined as \eqref{prop eq: 1. simple approx of PCAF in Kato},
  then 
  \begin{equation}
    B^{(\delta,*)}_t 
    =
    \frac{1}{\delta} \int_0^t \int_S \int_\delta^{2\delta} p(u, X_s, z)\, du\, \mu|_E(dz) ds
    = 
    \Pi^{(\delta, *)}(E \times [0,t)),
    \quad 
    t \geq 0.
  \end{equation}
  Hence, \eqref{prop eq: 2. delta STOM approx in Kato} follows from Proposition~\ref{prop: delta approx for local Kato}.
  In the inequality, letting $\delta \to 0$ and then $\alpha \to 0$ yields \eqref{prop eq: 3. delta STOM approx in Kato}.
  The last assertion is now immediate from \cite[Corollary~4.9(iii)]{Kallenberg_17_Random}
  (see also the proof of Theorem~\ref{thm: R map STOM approx for local Kato} below).
\end{proof}

We next consider STOMs in the local Kato class.
Recall that, in Definitions~\ref{dfn: restriction of PCAF} and \ref{dfn: apPCAF},
we introduced two maps $\apPCAF^{(*, R)}$ and $\apPCAF^{(\delta, R)}$ to approximate PCAFs.
Below, we introduce STOM versions of these maps.

\begin{dfn} \label{dfn: truncation of STOM}
  For each $R > 1$
  we define a map 
  \begin{equation}
    \apSTOM^{(*, R)} \colon 
    \STOMMeas(S \times \RNp)
    \to
    \STOMMeas(S \times \RNp)
  \end{equation}
  by 
  \begin{equation}
    \Sigma
    \mapsto 
    \left( \int_{R-1}^R \mathbf{1}_{S^{(r)}}(x)\,dr \right) \Sigma(dx\, dt).
  \end{equation}
\end{dfn}

\begin{dfn} \label{dfn: apSTOM}
  For each $\delta > 0$ and $R > 1$,
  we define a map 
  \begin{equation}
    \apSTOM^{(\delta, R)} \colon 
    C(\RNpp \times S \times S, \RNp) \times \Meas(S) \times D_{L^0}(\RNp, S) 
    \to 
    \STOMMeas(S \times \RNp)
  \end{equation}
  by 
  \begin{equation}
    (q, \nu, \eta) 
    \mapsto
    \frac{1}{\delta}
    \int_\delta^{2\delta} q(u, \eta_t, z)\, du\, \tilde{\nu}^{(R)}(dz)\, dt.
  \end{equation}
\end{dfn}

Below, we present STOM versions of Propositions~\ref{prop: R approx in local Kato} and \ref{prop: delta approx for local Kato}.

\begin{prop} \label{prop: R STOM approx in local Kato}
  Assume that $X$ is conservative.
  Let $\mu$ be a smooth measure, and let $\Pi$ denote the associated STOM.
  For each $R > 1$, set $\Pi^{(*, R)} \coloneqq \apSTOM^{(*, R)}(\Pi)$.
  Then $\Pi^{(*, R)}$ is a STOM associated with the smooth measure $\tilde{\mu}^{(R)}$.
  Moreover, it holds that, for any $T > 0$, 
  \begin{equation}  \label{prop eq: R STOM approx in local Kato. 1}
    \Pi^{(*,R)}|_{S \times [0, T]} = \Pi|_{S \times [0, T]}\quad \text{on the event}\quad \left\{ \exitTime_{S^{(R-1)}} > T \right\}.
  \end{equation}
  In particular, 
  \begin{equation}   \label{prop eq: R STOM approx in local Kato. 2}
    \lim_{R \to \infty}
    E^x\!\left[ \STOMMet{S \times \RNp}(\Pi^{(*,R)}, \Pi) \right]
    = 0,
    \quad \forall x \in S,
  \end{equation}
  where the metric  $\STOMMet{S \times \RNp}$ is recalled from \eqref{eq: dfn of stom met}.
  Consequently, $\Pi^{(*,R)} \xrightarrow[R \to \infty]{\mathrm{p}} \Pi$ in $\STOMMeas(S \times \RNp)$ under $P^x$, for each $x \in S$.
\end{prop}

\begin{proof}
  Write $A = (A_t)_{t \geq 0}$ for the PCAF associated with $\mu$.
  By Proposition~\ref{prop: R approx in local Kato}, 
  $A^{(*,R)} \coloneqq \apPCAF^{(*, R)}(X, A)$ is the PCAF associated with $\tilde{\mu}^{(R)}$,
  and hence the associated STOM $\tilde{\Pi}^{(*,R)}$ is given by 
  \begin{equation}
    \tilde{\Pi}^{(*, R)}(E \times [0,t]) = \int_0^t \mathbf{1}_E(\eta_s) \int_{R-1}^R \mathbf{1}_{S^{(r)}}(\eta_s)\, dr\, dA_s,
    \quad 
    E \in \Borel(S),\ t \geq 0.
  \end{equation} 
  By definition, we have, for any $E \in \Borel(S)$ and $t \geq 0$,
  \begin{equation}
    \Pi^{(*, R)}(E \times [0, t])
    = 
    \int_0^t \mathbf{1}_E(\eta_s) \int_{R-1}^R \mathbf{1}_{S^{(r)}}(\eta_s)\, dr\, dA_s.
  \end{equation}
  Hence, $\tilde{\Pi}^{(*,R)} = \Pi^{(*,R)}$, which proves the first assertion.
  The second assertion \eqref{prop eq: R STOM approx in local Kato. 1} is straightforward from the definition of $\Pi^{(*,R)}$.
  Then, by the definition of $\STOMMet{S \times \RNp}$, 
  we have that, on the event $\left\{ \exitTime_{S^{(R-1)}} > T \right\}$,
  \begin{align} 
    \STOMMet{S \times \RNp}(\Pi^{(*,R)}, \Pi)
    &= 
    \int_0^\infty e^{-t} \bigl(1 \wedge \ProhMet{S \times \RNp}(\Pi^{(*,R)}|_{S \times [0, t]}, \Pi|_{S \times [0,t]}) \bigr)\, dt\\
    &\leq    
    e^{-T}
    \label{prop pr: R STOM approx in local Kato. 1}
  \end{align}
  Since the metric does not exceed $1$,
  it follows that, for any $x \in S$ and $T > 0$,
   \begin{equation}  
    E^x\!\left[ \STOMMet{S \times \RNp}(\Pi^{(*,R)}, \Pi) \right]
    \leq    
    e^{-T} + P^x\!\left( \exitTime_{S^{(R-1)}} > T \right).
  \end{equation}
  Letting $R \to \infty$ and then $T \to \infty$ in the above inequality,
  and using the conservativeness of $X$.
  we deduce \eqref{prop eq: R STOM approx in local Kato. 2}.
  The last assertion is elementary (cf.\ \cite[Lemma~5.2]{Kallenberg_21_Foundations}).
\end{proof}

\begin{prop} \label{prop: delta STOM approx for local Kato}
  Let $\mu$ be a Radon measure in the local Kato class and let $\Pi$ be the associated STOM.
  Fix $R > 1$.
  For each $\delta > 0$, set 
  \begin{equation} \label{prop eq: delta STOM approx for local Kato. 1}
    \Pi^{(\delta, R)} \coloneqq \apSTOM^{(\delta, R)}(p, \mu, X), 
    \qquad 
    \Pi^{(*, R)} \coloneqq \apSTOM^{(*, R)}(\Pi).
  \end{equation}
  Then there exist universal constants $C_3, C_4 > 0$ such that the following holds:
  for any $T \in [0, \infty)$, $E \in \Borel(S)$,
  $\alpha \in (0,1)$, and $\delta \in (0,1)$,
  \begin{align}
    &\sup_{x \in S} 
    E^x\!
    \left[
      \sup_{0 \leq s < t \leq T}
      \bigl| \Pi^{(\delta, R)}(E \times [s,t)) - \Pi^{(*,R)}(E \times [s, t)) \bigr|^2
    \right]
    \\
    &\leq    
    C_3 \bigl\| \Potential^\alpha \mu^{(R)} \bigr\|_\infty
    \Bigl(
      \delta \alpha 
      \bigl\| \Potential^\alpha \mu^{(R)} \bigr\|_\infty
      + 
      \bigl\| \Potential^{1/\delta}\!\mu^{(R)} \bigr\|_\infty
    \Bigr)
    +
    C_4 e^{2T} (1 - e^{-\alpha T}) \bigl\| \Potential^1 \mu^{(R)} \bigr\|_\infty^2.
    \label{prop eq: delta STOM approx for local Kato. 2}
  \end{align}
  In particular, 
  \begin{equation}
    \lim_{\delta \to 0}
    \sup_{x \in S} 
    E^x\!
    \left[
      \sup_{0 \leq s < t \leq T}
      \bigl| \Pi^{(\delta, R)}(E \times [s,t)) - \Pi^{(*,R)}(E \times [s, t)) \bigr|^2
    \right]
    = 0.
  \end{equation}
  Consequently, 
  $\Pi^{(\delta, R)} \xrightarrow[\delta \to 0]{\mathrm{p}} \Pi^{(*,R)}$ 
  in $\STOMMeas(S \times \RNp)$ under $P^x$, for each $x \in S$.
\end{prop}

\begin{proof}
  By Proposition~\ref{prop: R STOM approx in local Kato},
  $\Pi^{(*,R)}$ is a STOM associated with $\tilde{\mu}^{(R)}$.
  On the other hand, by definition, we have 
  \begin{equation}
    \Pi^{(\delta, R)}(dz\, dt) 
    =
    \frac{1}{\delta}
    \int_\delta^{2\delta} q(u, \eta_t, z)\, du\, \tilde{\nu}^{(R)}(dz)\, dt.
  \end{equation}
  Hence, we can apply Proposition~\ref{prop: delta STOM approx in Kato} to the Kato class measure $\tilde{\mu}^{(R)}$.
  Noting that $\tilde{\mu}^{(R)} \leq \mu^{(R)}$,
  we obtain \eqref{prop eq: delta STOM approx for local Kato. 2}.
  The remaining assertions are verified in the same manner as the corresponding assertions of Proposition~\ref{prop: delta STOM approx in Kato},
  and therefore we omit the details.
\end{proof}

Combining the above two propositions, 
similarly to the PCAF case,
we see that STOMs in the local Kato class can be approximated in two stages as follows.

\begin{cor} \label{cor: STOM approx in local Kato}
  Assume that $X$ is conservative.
  Let $\mu$ be a smooth measure in the local Kato class, and let $\Pi$ denote the associated PCAF.
  For each $\delta > 0$ and $R > 1$, define 
  \begin{equation} 
    \Pi^{(\delta, R)} \coloneqq \apSTOM^{(\delta, R)}(p, \mu, X),
    \qquad 
    \Pi^{(*, R)} \coloneqq \apSTOM^{(*,R)}(\Pi).
  \end{equation}
  Then, for each $x \in S$, the following convergences hold in $\STOMMeas(S \times \RNp)$ under $P^x$:
  \begin{align}
    \Pi^{(\delta, R)} &\xrightarrow[\delta \to 0]{\mathrm{p}} \Pi^{(*,R)}, \qquad \forall R > 1,\\
    \Pi^{(*, R)} &\xrightarrow[R \to \infty]{\mathrm{p}} \Pi.
  \end{align}
\end{cor}

\begin{rem} \label{rem: flexibility of STOM framework}
  Here, we provide several reasons why STOMs offer a more flexible framework than PCAFs. 
  \begin{enumerate} [label = \textup{(\alph*)}]
    \item \label{rem item: flexibility of STOM framework. 1}
      In the previous discussion, to approximate PCAFs in the local Kato class within the space $\upC(\RNp, \RNp)$, 
      we imposed the conservativeness of $X$. 
      In fact, by working with STOMs instead, one can study such approximations 
      \emph{without} assuming conservativeness.
      Let us explain this in a bit more detail. 
      Suppose that $\Pi$ is a STOM in the local Kato class. 
      Since the associated PCAF may explode in finite time, 
      $\Pi$ may fail to belong to $\STOMMeas(S \times \RNp)$. 
      However, by Proposition~\ref{prop: moment formula for STOM}, 
      we know that $\Pi$ is a Radon measure on $S \times \RNp$ with probability one. 
      Therefore, although the topology becomes coarser, 
      we can still discuss its convergence within $\Meas(S \times \RNp)$ endowed with the vague topology. 
      See Proposition~\ref{prop: R STOM vague approx in local Kato} below for an application of this idea.
    
    \item \label{rem item: flexibility of STOM framework. 2}
      Convergence of STOMs can also be discussed for general smooth measures. 
      The idea is to introduce a variant of the vague topology that is adapted to the nest of the smooth measure. 
      Let $\mu$ be a general smooth measure and let $\Pi$ denote the corresponding STOM. 
      Fix an nest $\mathfrak{E} = (E_N)_{N \geq 1}$ of $\mu$ (recall it from Definition~\ref{dfn: smooth measure}).
      Define $\Meas_\mathfrak{E}(S \times \RNp)$ to be the collection of measures on $S \times \RNp$ 
      that are finite on each $E_N \times [0, N]$. 
      Then, by an argument similar to the above using Proposition~\ref{prop: moment formula for STOM}, 
      we see that $\Pi$ belongs to $\Meas_\mathfrak{E}(S \times \RNp)$ with probability one. 
      By endowing $\Meas_\mathfrak{E}(S \times \RNp)$ with a suitable modification of the vague topology 
      adapted to the underlying nest, 
      we can discuss the convergence of STOMs. 
      Such nest-dependent modifications of the vague topology 
      were recently introduced in~\cite{Ooi_Tsuchida_Uemura_25_Nest} 
      for the purpose of studying the convergence of general smooth measures 
      (not necessarily Radon measures).
  \end{enumerate}
\end{rem}

As pointed out in Remark~\ref{rem: flexibility of STOM framework}\ref{rem item: flexibility of STOM framework. 1},
we provide a result which shows the flexibility of the STOM framework.
Namely, we present an approximation result for STOMs in local Kato class in the vague topology,
without assuming conservativeness.

\begin{prop} \label{prop: R STOM vague approx in local Kato}
  Let $\mu$ be a smooth measure, and let $\Pi$ denote the associated STOM.
  For each $R > 1$, set $\Pi^{(*, R)} \coloneqq \apSTOM^{(*, R)}(\Pi)$.
  For any $x \in S$, it holds $P^x$-a.s.\ that
  \begin{equation}  \label{prop eq: R STOM vague approx in local Kato. 1}
    \VagueMet{S \times \RNp}(\Pi^{(*,R)}, \Pi) \leq e^{-R+1}, \quad \forall R > 1,
  \end{equation}
  where $\VagueMet{S \times \RNp}$ denotes the vague metric as recalled from \eqref{eq: dfn of vague metric}.
  Consequently, $\Pi^{(*,R)} \xrightarrow[R \to \infty]{\mathrm{a.s.}} \Pi$ vaguely under $P^x$, for each $x \in S$.
\end{prop}

\begin{proof}
 By definition, if $r < R-1$, then we have, for any $E \in \Borel(S)$ and $t \geq 0$,
  \begin{align}
    \Pi^{(*, R)}|_{S^{(r)} \times [0, r]}(E \times [0, t])
    &= 
    \int_0^{t \wedge r} \mathbf{1}_{E \cap S^{(r)}}(\eta_s) \int_{R-1}^R \mathbf{1}_{S^{(u)}}(\eta_s)\, du\, dA_s\\
    &= 
    \int_0^{t \wedge r} \mathbf{1}_{E \cap S^{(r)}}(\eta_s)\, dA_s\\
    &= 
    \Pi|_{S^{(r)} \times [0,r]}(E \times [0,t]).
  \end{align}
  This implies $\Pi^{(*,R)}|_{S^{(r)} \times [0, r]} = \Pi|_{S^{(r)} \times [0, r]}$ for all $r < R-1$.
  It follows that 
  \begin{equation}
    \VagueMet{S \times \RNp}(\Pi^{(*,R)}, \Pi) 
    = 
    \int_0^\infty e^{-r} 
    \bigl\{1 \wedge \ProhMet{S \times \RNp}\bigl( \Pi^{(*,R)}|_{S^{(r)} \times [0, r]}, \Pi|_{S^{(r)} \times [0,r]} \bigr) \bigr\}\, dr
    \leq
    e^{-R+1},
  \end{equation}
  which yields \eqref{prop eq: R STOM vague approx in local Kato. 1}.
  This completes the proof.
\end{proof}

\subsection{Continuity of joint laws with PCAFs} \label{sec: Continuity of joint laws with PCAFs}

In a wide class of stochastic processes, such as Feller processes, 
the law of the process depends continuously on the starting point. 
Here, we employ the approximation introduced in the previous subsection 
to show that the joint law with a PCAF also enjoys this continuity;
see Theorems~\ref{thm: delta map approx for Kato} and \ref{thm: R map approx for local Kato} below.

To formulate the results in a suitable topological framework,
we consider the following conditions for the underlying process $X$ and its heat kernel.
We recall the notation $\ProcLaw_\cdot$ from \eqref{eq: notation for law with processes}. 

\begin{assum}[The weak $L^0$-continuity condition] \label{assum: weak L^0 continuity} \leavevmode
  \begin{enumerate}[label=\textup{(\roman*)}, leftmargin=*]
    \item \label{assum item: weak L^0 continuity. 1} 
      The process $X$ is conservative,
      and the law map 
      $\ProcLaw_X \colon S \to \Prob(D_{L^0}(\RNp, S))$ 
      is continuous.
    \item \label{assum item: weak L^0 continuity. 2}
      The heat kernel $p(t,x,y)$ takes finite values and is continuous in $(t,x,y) \in \RNpp \times S \times S$.
  \end{enumerate}
\end{assum}

\begin{assum}[The spatial tightness condition] \label{assum: spatial tightness} 
  For all $r > 0$ and $T > 0$, 
  \begin{equation}
    \lim_{R \to \infty}
    \sup_{x \in S^{(r)}}
    P^x\!\left( \exitTime_{S^{(R)}} < T\right)
    = 0.
  \end{equation}
\end{assum}

If the law map $\ProcLaw_X$ is weakly continuous 
with respect to a topology that captures the path behavior more precisely than the $L^0$ topology,
such as the $J_1$-Skorohod topology,
then the spatial tightness condition is automatically satisfied.
For instance, the following condition suffices.

\begin{assum}[The weak $J_1$-continuity condition] \label{assum: weak J_1 continuity} \leavevmode
  \begin{enumerate}[label=\textup{(\roman*)}, leftmargin=*]
    \item \label{assum item: weak J_1 continuity. 1} 
      The process $X$ is conservative, and the map 
      $\ProcLaw_X \colon S \to \Prob(D_{J_1}(\RNp, S))$
      is continuous.
    \item \label{assum item: weak J_1 continuity. 2}
      The heat kernel $p(t,x,y)$ takes finite values and is continuous in $(t,x,y) \in \RNpp \times S \times S$.
  \end{enumerate}
\end{assum}

\begin{lem} \label{lem: J_1 tight implies spatial tight}
  Condition~\ref{assum item: weak J_1 continuity. 1} of Assumption~\ref{assum: weak J_1 continuity}
  implies Assumption~\ref{assum: spatial tightness}.
\end{lem}

\begin{proof}
  Condition~\ref{assum item: weak J_1 continuity. 1} of Assumption~\ref{assum: weak J_1 continuity}
  implies that, for each $r > 0$, the family $\{\ProcLaw_X(x)\}_{x \in S^{(r)}}$ 
  is tight as a collection of probability measures on $D_{J_1}(\RNp, S)$.
  Recall that $J_1$-tightness ensures that, for every $T>0$,
  the trajectories $\{X_t \mid t \le T\}$ starting from points in $S^{(r)}$
  remain, with high probability, within a common compact subset of $S$
  (see \cite[Theorem~23.8]{Kallenberg_21_Foundations}).
  Hence, Assumption~\ref{assum: spatial tightness} follows.
  This completes the proof.
\end{proof}

\begin{rem}
  Condition~\ref{assum item: weak J_1 continuity. 1} of Assumption~\ref{assum: weak J_1 continuity} 
  holds for a wide class of processes.
  For instance, any conservative Feller process enjoys this property.
  (This follows from the Aldous tightness criterion for the $J_1$-Skorohod topology;
  see \cite[Theorem~23.11]{Kallenberg_21_Foundations}.)
\end{rem}

\begin{rem}
  The joint continuity of the heat kernel is more delicate 
  than the weak continuity of the law map.
  There is no simple sufficient condition for this in general,
  but sufficient conditions are known in several important frameworks.
  \begin{enumerate}[label=\textup{(\alph*)}]
    \item When $X$ is a L\'{e}vy process, 
      the integrability of its characteristic function is sufficient 
      (see \cite[Proposition~2.5(xii)]{Sato_99_Levy}).
    \item When $X$ is associated with a regular strongly local symmetric Dirichlet form,
      the parabolic Harnack inequality is sufficient 
      \cite{Barlow_Grigoryan_Kumagai_12_PHIandHKE}.
    \item Any process associated with a resistance form 
      (see Section~\ref{sec: Preliminary results on stoch proc on resis sp} below)
      admits a unique jointly continuous heat kernel.
  \end{enumerate}
\end{rem}

In the theorem below,
we show that, 
under the weak $L^0$-continuity condition (Assumption~\ref{assum: weak L^0 continuity}),
the continuity of $\ProcLaw_X$ is inherited by the joint law $\ProcLaw_{(X, A)}$ for any PCAF $A$ in the Kato class.

\begin{thm} \label{thm: delta map approx for Kato}
  Suppose that Assumption~\ref{assum: weak L^0 continuity} is satisfied.
  Let $\mu$ be a finite Borel measure in the Kato class and let $A = (A_t)_{t \geq 0}$ be the associated PCAF.
  For each $\delta > 0$, let $A^{(\delta, *)}$ denote the PCAF defined in \eqref{prop: delta approx of PCAF in Kato}.
  Then the map
  \begin{equation}   \label{thm eq: 1. approx of PCAF in Kato}
    \ProcLaw_{(X, A^{(\delta, *)})} \colon S \to \Prob\bigl( D_{L^0}(\RNp, S) \times \upC(\RNp, \RNp) \bigr)
  \end{equation}
  is continuous for each $\delta > 0$.
  Likewise, $\ProcLaw_{(X, A)}$ is continuous.
  Moreover,
  \begin{equation}  \label{thm eq: 2. approx of PCAF in Kato}
    \ProcLaw_{(X, A^{(\delta, *)})} \xlongrightarrow[\delta \to 0]{} \ProcLaw_{(X, A)}
    \quad \text{in}\quad 
    C\bigl(S, \Prob\bigl( D_{L^0}(\RNp, S) \times \upC(\RNp, \RNp) \bigr)\bigr).
  \end{equation}
\end{thm}

\begin{proof}
  The proof proceeds as follows.
  Using the joint continuity of the heat kernel, we first verify the continuity of $\ProcLaw_{(X, A^{(\delta, *)})}$.
  Then, by Proposition~\ref{prop: delta approx for local Kato},
  we deduce that $\ProcLaw_{(X, A^{(\delta, *)})}$ converges to $\ProcLaw_{(X, A)}$ as $\delta \to 0$ locally uniformly,
  which yields the desired conclusion.
  The details are omitted here, 
  since the argument follows the same line as in the case of PCAFs in the local Kato class, 
  which will be treated next.
\end{proof}

We next consider PCAFs in the local Kato class.
We first state basic properties of the approximation map $\apPCAF^{(\delta, R)}$ introduced in Definition~\ref{dfn: apPCAF}.
The proof is deferred to Appendix~\ref{appendix: Lemma of PCAF/STOM approx}.

\begin{lem} \label{lem: apPCAF continuity}
  Fix $\delta > 0$ and $R > 1$.
  \begin{enumerate} [label = \textup{(\roman*)}]
    \item \label{lem item: apPCAF continuity. 1}
      The map $\apPCAF^{(\delta, R)}$ is Borel measurable.
    \item \label{lem item: apPCAF continuity. 2}
      Fix an arbitrary sequence $(q_n, \nu_n, \eta_n)_{n \geq 1}$ converging to $(q, \nu, \eta)$ in the domain of $\apPCAF^{(\delta, R)}$.
      Assume that
      \begin{equation}  \label{lem eq: apPCAF continuity}
        \sup_{n \geq 1} 
        \sup_{y \in S} \int_{\delta}^{2\delta} \int_S q_n(u, y, z)\, \nu_n^{(R)}(dz)\, du < \infty.
      \end{equation}
      Then $\apPCAF^{(\delta, R)}(q_n, \nu_n, \eta_n)$ converges to $\apPCAF^{(\delta, R)}(q, \nu, \eta)$ in $\upC(\RNp, \RNp)$.
  \end{enumerate}
\end{lem}

The above lemma immediately implies the continuity of the approximation $\ProcLaw_{(X, A^{(\delta, R)})}$, as stated below.

\begin{lem} \label{lem: continuity of delta map for local Kato}
  Suppose that Assumption~\ref{assum: weak L^0 continuity} is satisfied.
  Let $\mu$ be a Radon measure in the local Kato class and let $A = (A_t)_{t \geq 0}$ be the associated PCAF.
  Fix $\delta > 0$ and $R > 1$, and set 
  \begin{equation}
    A^{(\delta, R)} \coloneqq \apPCAF^{(\delta, R)}(p, \mu, X).
  \end{equation}
  Then the map
  \begin{equation}  
    \ProcLaw_{(X, A^{(\delta, R)})} \colon S \to \Prob\bigl( D_{L^0}(\RNp, S) \times \upC(\RNp, \RNp) \bigr)
  \end{equation}
  is continuous.
\end{lem}

\begin{proof}
  This is an immediate consequence of the assumption, Lemma~\ref{lem: apPCAF continuity}\ref{lem item: apPCAF continuity. 2},
  and the continuous mapping theorem.
\end{proof}

Using the approximation of $\ProcLaw_{(X, A^{(*, R)})}$ by $\ProcLaw_{(X, A^{(\delta, R)})}$ 
established in Proposition~\ref{prop: delta approx for local Kato}, 
we show that $\ProcLaw_{(X, A^{(*, R)})}$ is also continuous.

\begin{prop} \label{prop: delta map approx for local Kato}
  Suppose that Assumption~\ref{assum: weak L^0 continuity} is satisfied.
  Let $\mu$ be a Radon measure in the local Kato class, and let $A = (A_t)_{t \geq 0}$ be the associated PCAF.
  Fix $R > 1$.
  For each $\delta > 0$, define 
  \begin{equation} \label{prop eq: delta map approx for local Kato. 1}
    A^{(\delta, R)} \coloneqq \apPCAF^{(\delta, R)}(p, \mu, X),
    \qquad 
    A^{(*, R)} \coloneqq \apPCAF^{(*, R)}(X, A).
  \end{equation}
  Then the following statements hold.
  \begin{enumerate} [label = \textup{(\roman*)}]
    \item \label{prop item: delta map approx for local Kato. 1}
      The map $\ProcLaw_{(X, A^{(*, R)})}$ is continuous.
    \item \label{prop item: delta map approx for local Kato. 2}
      Let $C_1, C_2 > 0$ be the universal constants appearing in Proposition~\ref{prop: delta approx of PCAF in Kato}.
      Then, for any $N \in \NN$, $\alpha \in (0,1)$, and $\delta \in (0,1)$, we have
      \begin{align}
        &\hatCMet{S}{\Prob(L^0(\RNp, S) \times \upC(\RNp, \RNp))}\!
        \left(\ProcLaw_{(X, A^{(\delta, R)})}, \ProcLaw_{(X, A^{(*,R)})}\right)\\
        &\leq
        2^{-N+1} 
        +
        2^{2N}
        C_1 \bigl\| \Potential^\alpha \mu^{(R)} \bigr\|_\infty
        \Bigl(
          \delta \alpha 
          \bigl\| \Potential^\alpha \mu^{(R)} \bigr\|_\infty
          + 
          \bigl\| \Potential^{1/\delta}\!\mu^{(R)} \bigr\|_\infty
        \Bigr)\\
        &\quad
        +
        C_2 (2e)^{2N} (1 - e^{-\alpha N}) \bigl\| \Potential^1 \mu^{(R)} \bigr\|_\infty^2.
        \label{prop eq: delta map approx for local Kato. 2}
      \end{align}
    \item \label{prop item: delta map approx for local Kato. 3}
      It holds that 
      \begin{equation}  \label{thm eq: 2. apprx of PCAF in local Kato}
        \ProcLaw_{(X, A^{(\delta, R)})} \xlongrightarrow[\delta \to 0]{} \ProcLaw_{(X, A^{(*, R)})}
        \quad \text{in}\quad 
        C\bigl(S, \Prob\bigl( D_{L^0}(\RNp, S) \times \upC(\RNp, \RNp) \bigr)\bigr).
      \end{equation}
  \end{enumerate}
\end{prop}

\begin{proof}
  We prove the assertions simultaneously.
  For each $N \in \NN$, define 
  \begin{equation}
    E_N \coloneqq \left\{\sup_{0 \leq t \leq N} |A^{(\delta, R)}_t - A^{(*,R)}_t| \leq 2^{-N} \right\}.
  \end{equation}
  By Lemma~\ref{prop: delta approx for local Kato} and Markov's inequality, we have
  \begin{align}
    \sup_{x \in S}P^x(E_N^c) 
    &\leq 
    2^{2N}
    C_1 \bigl\| \Potential^\alpha \mu^{(R)} \bigr\|_\infty
    \Bigl(
      \delta \alpha 
      \bigl\| \Potential^\alpha \mu^{(R)} \bigr\|_\infty
      + 
      \bigl\| \Potential^{1/\delta}\!\mu^{(R)} \bigr\|_\infty
    \Bigr)\\
    &\quad
    +
    C_2 (2e)^{2N} (1 - e^{-\alpha N}) \bigl\| \Potential^1 \mu^{(R)} \bigr\|_\infty^2.
  \end{align}
  On the event $E_N$, we have 
  \begin{equation}
    d_{\upC(\RNp, \RNp)}(A^{(\delta, R)}, A^{(*, R)}) \leq 2^{-N+1},
  \end{equation}
  where we recall the metric from \eqref{dfn eq: metric on upC}.
  Therefore, by Lemma~\ref{lem: Prohorov estimate}, we obtain 
  \begin{align} 
    & \sup_{x \in S} \ProhMet{L^0(\RNp, S) \times \upC(\RNp, \RNp)}\!
    \left(\ProcLaw_{(X, A^{(\delta, R)})}(x), \ProcLaw_{(X, A^{(*,R)})}(x)\right)\\
    &\leq
    2^{-N+1} 
    +
    2^{2N}
    C_1 \bigl\| \Potential^\alpha \mu^{(R)} \bigr\|_\infty
    \Bigl(
      \delta \alpha 
      \bigl\| \Potential^\alpha \mu^{(R)} \bigr\|_\infty
      + 
      \bigl\| \Potential^{1/\delta}\!\mu^{(R)} \bigr\|_\infty
    \Bigr)\\
    &\quad
    +
    C_2 (2e)^{2N} (1 - e^{-\alpha N}) \bigl\| \Potential^1 \mu^{(R)} \bigr\|_\infty^2.
    \label{thm pr: delta map approx for local Kato. 1}
  \end{align}
  Letting $\delta \to 0$, then $\alpha \to 0$, and finally $N \to \infty$ in the above inequality yields 
  \begin{equation} 
    \lim_{\delta \to 0}
    \sup_{x \in S} \ProhMet{L^0(\RNp, S) \times \upC(\RNp, \RNp)}\!
    \left(\ProcLaw_{(X, A^{(\delta, R)})}(x), \ProcLaw_{(X, A^{(*,R)})}(x)\right)
    = 0.
  \end{equation}
  This, combined with Lemma~\ref{lem: continuity of delta map for local Kato}, 
  yields \ref{prop item: delta map approx for local Kato. 1} and \ref{prop item: delta map approx for local Kato. 3}. 
  Finally, applying Lemma~\ref{lem: hatC metric simple estimate} to~\eqref{thm pr: delta map approx for local Kato. 1}
  gives \ref{prop item: delta map approx for local Kato. 2}.
\end{proof}

Finally, using Proposition~\ref{prop: R approx in local Kato} and the spatial tightness condition, 
we verify that the continuity of $\ProcLaw_{(X, A^{(*, R)})}$ established above 
is inherited by $\ProcLaw_{(X, A)}$, as stated below.

\begin{thm} \label{thm: R map approx for local Kato}
  Suppose that Assumptions~\ref{assum: weak L^0 continuity} and \ref{assum: spatial tightness} are satisfied.
  Let $\mu$ be a Radon measure in the local Kato class and let $A = (A_t)_{t \geq 0}$ be the associated PCAF.
  For each $R > 1$, set 
  \begin{equation} \label{thm eq: R map approx for local Kato. 1}
    A^{(*, R)} \coloneqq \apPCAF^{(*, R)}(X, A).
  \end{equation}
  Then the following statements hold.
  \begin{enumerate} [label = \textup{(\roman*)}]
    \item \label{thm item: R map approx for local Kato. 1}
      The map $\ProcLaw_{(X, A)}$ is continuous.
    \item \label{thm item: R map approx for local Kato. 2}
      It holds that, for any $r > 0$, $N \in \NN$, and $R > 1$,
      \begin{align}
        &\hatCMet{S}{\Prob(L^0(\RNp, S) \times \upC(\RNp, \RNp))}\!\left(\ProcLaw_{(X, A^{(*, R)})}, \ProcLaw_{(X, A)}\right)\\
        &\leq
        e^{-r}
        +
        2^{-N} + 
        \sup_{x \in S^{(r)}}
        P^x\!\left( \exitTime_{S^{(R-1)}} \leq N \right).
        \label{thm eq: R map approx for local Kato. 2}
      \end{align}
    \item \label{thm item: R map approx for local Kato. 3}
      It holds that 
      \begin{equation} \label{thm eq: R map approx for local Kato. 3}
        \ProcLaw_{(X, A^{(*, R)})} \xlongrightarrow[R \to \infty]{} \ProcLaw_{(X, A)}
        \quad \text{in}\quad 
        C\bigl(S, \Prob\bigl( D_{L^0}(\RNp, S) \times \upC(\RNp, \RNp) \bigr)\bigr).
      \end{equation}
  \end{enumerate}
\end{thm}

\begin{proof}
  The proof follows the same line as that of the previous proposition.
  For each $N \in \NN$ and $R > 1$, define 
  \begin{equation}
    E_{N, R} \coloneqq \left\{ \exitTime_{S^{(R-1)}} > N \right\}.
  \end{equation}
  On this event, we have $A^{(*,R)}_t = A_t$ for all $t \leq N$, which implies that 
  \begin{equation}
    d_{\upC(\RNp, \RNp)}(A^{(*, R)}, A) \leq 2^{-N}.
  \end{equation}
  Thus, by Lemma~\ref{lem: Prohorov estimate},
  we obtain, for each $r > 0$,
  \begin{align} 
    &\sup_{x \in S^{(r)}} \ProhMet{L^0(\RNp, S) \times \upC(\RNp, \RNp)}\!
    \left(\ProcLaw_{(X, A^{(*, R)})}(x), \ProcLaw_{(X, A^{(R)})}(x)\right)\\
    &\leq
    2^{-N} 
    + 
    \sup_{x \in S^{(r)}}
    P^x\!\left( \exitTime_{S^{(R-1)}} \leq N \right).
    \label{thm pr: R map approx for local Kato. 1}
  \end{align}
  Letting $R \to \infty$ and then $N \to \infty$ in the above inequality, 
  the spatial tightness condition (Assumption~\ref{assum: spatial tightness}) yields
  \begin{equation}
    \lim_{R \to \infty}
    \sup_{x \in S^{(r)}} \ProhMet{L^0(\RNp, S) \times \upC(\RNp, \RNp)}\!
    \left(\ProcLaw_{(X, A^{(*, R)})}(x), \ProcLaw_{(X, A^{(R)})}(x)\right)
    = 0,
    \quad \forall r > 0.
  \end{equation}
  This, together with Proposition~\ref{prop: delta map approx for local Kato}\ref{prop item: delta map approx for local Kato. 1},
  establishes \ref{thm item: R map approx for local Kato. 1} and \ref{thm item: R map approx for local Kato. 3}.
  Finally, applying Lemma~\ref{lem: hatC metric simple estimate} to~\eqref{thm pr: R map approx for local Kato. 1}
  gives \ref{thm item: R map approx for local Kato. 2}.
\end{proof}

\subsection{Continuity of joint laws with STOMs} \label{sec: Continuity of joint laws with STOMs}

The aim of this subsection is to extend the results of the previous subsection to STOMs. 
The desired continuity of the joint laws of the process and STOMs 
in the Kato and local Kato classes is established in 
Theorems~\ref{thm: delta map STOM approx for Kato} and \ref{thm: R map STOM approx for local Kato} below.

The following is the STOM counterpart of Theorem~\ref{thm: delta map approx for Kato}.
Again, we omit the proof, as it is essentially the same as that for STOMs in the local Kato class discussed below.

\begin{thm} \label{thm: delta map STOM approx for Kato}
  Suppose that Assumption~\ref{assum: weak L^0 continuity} is satisfied.
  Let $\mu$ be a finite Borel measure in the Kato class and let $\Pi$ be the associated STOM.
  For each $\delta > 0$, let $\Pi^{(\delta, *)}$ denote the PCAF defined in \eqref{prop eq: 1. delta STOM approx in Kato}.
  Then the map
  \begin{equation}   
    \ProcLaw_{(X, \Pi^{(\delta, *)})} \colon S \to \Prob\bigl( D_{L^0}(\RNp, S) \times \STOMMeas(S \times \RNp) \bigr)
  \end{equation}
  is continuous for each $\delta > 0$.
  Likewise, $\ProcLaw_{(X, \Pi)}$ is continuous.
  Moreover,
  \begin{equation} 
    \ProcLaw_{(X, \Pi^{(\delta, *)})} \xlongrightarrow[\delta \to 0]{} \ProcLaw_{(X, \Pi)}
    \quad \text{in}\quad 
    C\bigl(S, \Prob\bigl( D_{L^0}(\RNp, S) \times \STOMMeas(S \times \RNp) \bigr)\bigr).
  \end{equation}
\end{thm}

Recall the approximation map $\apSTOM^{(\delta, R)}$ from Definition~\ref{dfn: apSTOM}.
The following lemma corresponds to Lemma~\ref{lem: apPCAF continuity}.
Again, the proof is deferred to Appendix~\ref{appendix: Lemma of PCAF/STOM approx}.

\begin{lem} \label{lem: apSTOM continuity}
  Fix $\delta > 0$ and $R > 1$.
  \begin{enumerate} [label = \textup{(\roman*)}]
    \item \label{lem item: apSTOM continuity. 1}
      The map $\apSTOM^{(\delta, R)}$ is Borel measurable.
    \item \label{lem item: apSTOM continuity. 2}
      Fix an arbitrary sequence $(q_n, \nu_n, \eta_n)_{n \geq 1}$ converging to $(q, \nu, \eta)$ in the domain of $\apSTOM^{(\delta, R)}$.
      Assume that
      \begin{equation}  \label{lem eq: apSTOM continuity}
        \sup_{n \geq 1} 
        \sup_{y \in S} \int_{\delta}^{2\delta} \int_S q_n(u, y, z)\, \nu_n^{(R)}(dz)\, du < \infty.
      \end{equation}
      Then $\apSTOM^{(\delta, R)}(q_n, \nu_n, \eta_n)$ converges to $\apSTOM^{(\delta, R)}(q, \nu, \eta)$ in $\STOMMeas(S \times \RNp)$.
  \end{enumerate}
\end{lem}

\begin{lem}
  Suppose that Assumption~\ref{assum: weak L^0 continuity} is satisfied.
  Let $\mu$ be a Radon measure in the local Kato class.
  Fix $\delta > 0$ and $R > 1$, and set 
  \begin{equation}
    \Pi^{(\delta, R)} \coloneqq \apSTOM^{(\delta, R)}(p, \mu, X).
  \end{equation}
  Then the map
  \begin{equation}  
    \ProcLaw_{(X, \Pi^{(\delta, R)})} \colon S \to \Prob\bigl( D_{L^0}(\RNp, S) \times \STOMMeas(S \times \RNp) \bigr)
  \end{equation}
  is continuous.
\end{lem}

\begin{proof}
  This is an immediate consequence of the assumption, Lemma~\ref{lem: apSTOM continuity}\ref{lem item: apPCAF continuity. 2},
  and the continuous mapping theorem.
\end{proof}

The following proposition corresponds to Proposition~\ref{prop: delta map approx for local Kato}. 
However, unlike the case of PCAFs, 
we cannot directly estimate the distance between the truncated STOM $\Pi^{(*, R)}$ 
and its approximation $\Pi^{(\delta, R)}$ in $\STOMMeas(S \times \RNp)$. 
Instead, we establish the validity of the approximation by means of a coupling argument.

\begin{prop} \label{prop: delta map STOM approx for local Kato}
  Suppose that Assumption~\ref{assum: weak L^0 continuity} is satisfied.
  Let $\mu$ be a Radon measure in the local Kato class, and let $A = (A_t)_{t \geq 0}$ be the associated PCAF.
  Fix $R > 1$.
  For each $\delta > 0$, define 
  \begin{equation} \label{prop eq: delta map STOM approx for local Kato. 1}
    \Pi^{(\delta, R)} \coloneqq \apSTOM^{(\delta, R)}(p, \mu, X), 
    \qquad 
    \Pi^{(*, R)} \coloneqq \apSTOM^{(*, R)}(\Pi).
  \end{equation}
  Then the following statements hold.
  \begin{enumerate} [label = \textup{(\roman*)}]
    \item The map $\ProcLaw_{(X, \Pi^{(*, R)})}$ is continuous.
    \item We have 
      \begin{equation}  \label{prop eq: delta map STOM approx for local Kato. 3}
        \ProcLaw_{(X, \Pi^{(\delta, R)})} \xlongrightarrow[\delta \to 0]{} \ProcLaw_{(X, \Pi^{(*, R)})}
        \quad \text{in}\quad 
        C\bigl(S, \Prob\bigl( D_{L^0}(\RNp, S) \times \STOMMeas(S \times \RNp) \bigr)\bigr).
      \end{equation}
  \end{enumerate}
\end{prop}

\begin{proof}
  We prove both assertions simultaneously.
  Fix a sequence $(\delta_n)_{n \geq 1}$ in $(0,1)$ converging to $0$,
  and a sequence $(x_{n})_{n \geq 1}$ converging to $x$ in $S$.
  By Proposition~\ref{prop: conti is preserved in hatC},
  it suffices to show that 
  \begin{equation}  \label{thm pr: delta map STOM approx for local Kato. 1}
    P^{x_n}\bigl( (X, \Pi^{(\delta_n,R)}) \in \cdot \bigr) 
    \to 
    P^x\bigl( (X, \Pi^{(*,R)}) \in \cdot\bigr)
  \end{equation}
  weakly as probability measures on $D_{L^0}(\RNp, S) \times \STOMMeas(S \times \RNp)$.
  By Assumption~\ref{assum: weak L^0 continuity}\ref{assum item: weak L^0 continuity. 1},
  we have $\ProcLaw_X(x_n) \to \ProcLaw_X(x)$.
  Hence, by the Skorohod representation theorem,
  there exist random elements $(Y_n, \Sigma_n, \Sigma_n^*)$, $n \geq 1$, and $(Y, \Sigma)$ 
  defined on a common probability space with probability measure $P$ 
  such that 
  \begin{gather}
    P\bigl( (Y_n, \Sigma_n^{\delta_n}, \Sigma_n^*) \in \cdot \bigr) 
    = P^{x_n}\bigl( (X, \Pi^{(\delta_n, R)}, \Pi^{(*,R)}) \in \cdot \bigr),
    \quad \forall n \geq 1,
     \label{thm pr: delta map STOM approx for local Kato. 2}
    \\
    P\bigl( (Y, \Sigma^*) \in \cdot \bigr) = P^{x}\bigl( (X, \Pi^{(*,R)}) \in \cdot \bigr),
    \label{thm pr: delta map STOM approx for local Kato. 3}
    \\ 
    Y_n \to Y, \quad \text{almost surely}.
     \label{thm pr: delta map STOM approx for local Kato. 4}
  \end{gather}

  The convergence in~\eqref{thm pr: delta map STOM approx for local Kato. 1} 
  follows once we show that 
  $\Sigma_{n}^{\delta_n} \xrightarrow{\mathrm{p}} \Sigma$ in $\STOMMeas(S \times \RNp)$.
  Recall from Remark~\ref{rem: Kallenberg framework for stom sp} that
  if $S$ is equipped with the bounded metric $1 \wedge d_S$,
  then the topology on $\STOMMeas(S \times \RNp)$
  coincides with the vague topology with bounded support.
  Hence, in what follows, we replace the metric $d_S$ on $S$ by $1 \wedge d_S$
  and prove the desired convergence 
  by applying a characterization of this topology
  given in~\cite[Corollary~4.9(iii)]{Kallenberg_17_Random}.

  Define a collection $\mathcal{I}$ of subsets of $S \times \RNp$ by 
  \begin{equation}
    \mathcal{I} 
    \coloneqq 
    \{E \times [s, t) 
      \mid 
      0 \leq s < t < \infty,\  
      E \in \Borel(M) \text{ such that } \mu(\partial E) = 0\}.
  \end{equation}
  It is straightforward to check that $\mathcal{I}$ is a dissecting semi-ring in the sense of \cite{Kallenberg_17_Random}, 
  i.e., it satisfies the following properties.
  \begin{enumerate} [label = (DSR\arabic*), leftmargin = *]
    \item \label{pr item: DSR, 1}
      Every open subset of $S \times \RNp$ is a countable union of elements of $\mathcal{I}$.
    \item \label{pr item: DSR, 2}
      Every bounded Borel subset of $S \times \RNp$ is covered by finitely many elements of $\mathcal{I}$.
    \item \label{pr item: DSR, 3}
      The collection $\mathcal{I}$ is closed under finite intersections.
    \item \label{pr item: DSR, 4}
      The difference between any two sets in $\mathcal{I}$ is a finite disjoint union of sets in $\mathcal{I}$.
  \end{enumerate}
  (The first two properties mean that $\mathcal{I}$ is dissecting, 
  and the last two mean that it is a semi-ring.)
  By~\cite[Corollary 4.9(iii)]{Kallenberg_17_Random},
  it is enough to show that, 
  for each $H \in \mathcal{I}$,
  $\Sigma_n^{\delta_n}(H) \xrightarrow{\mathrm{p}} \Sigma^*(H)$
  as random variables.
  Fix such a set $H = E \times [s, t)$.
  By Proposition~\ref{prop: delta STOM approx for local Kato},
  we have 
  \begin{equation} \label{thm pr: delta map STOM approx for local Kato. 5}
    \lim_{\delta \to 0}
    \sup_{y \in S}
    E^y\Bigl[
      \bigl| \Pi^{(\delta, R)}(H) - \Pi^{(*,R)}(H) \bigr| \wedge 1
    \Bigr]
    = 0.
  \end{equation}
  The triangle inequality then yields, for any $\delta \in (0,1)$,
  \begin{align}
    &E\bigl[
      |\Sigma_n^{\delta_n}(H) - \Sigma^*(H)| \wedge 1
    \bigr]
    \\
    &\leq 
    E\Bigl[
      \bigl| \Sigma_n^{\delta_n}(H) - \Sigma_n^*(H) \bigr| \wedge 1
    \Bigr] 
    + 
    E\Bigl[
      \bigl| \Sigma_n^*(H) - \apSTOM^{(\delta, R)}(p, \mu, Y_n)(H) \bigr| \wedge 1
    \Bigr]\\
    &\quad 
    + 
    E\Bigl[
      \bigl| \apSTOM^{(\delta, R)}(p, \mu, Y_n)(H) - \apSTOM^{(\delta, R)}(p, \mu, Y)(H) \bigr| \wedge 1
    \Bigr]\\
    &\quad 
    +
    E\Bigl[
      \bigl| \apSTOM^{(\delta, R)}(p, \mu, Y)(H) - \Sigma^*(H)\bigr| \wedge 1
    \Bigr]\\ 
    &= 
    E^{x_n}\Bigl[
      \bigl| \Pi_n^{(\delta_n, R)}(H) - \Pi_n^{(*, R)}(H) \bigr| \wedge 1
    \Bigr]
    + 
    E^{x_n}\Bigl[
      \bigl| \Pi^{(*, R)}(H) - \Pi^{(\delta, R)}(H) \bigr| \wedge 1
    \Bigr]\\
    &\quad 
    + 
    E\Bigl[
      \bigl| \apSTOM^{(\delta, R)}(p, \mu, Y_n)(H) - \apSTOM^{(\delta, R)}(p, \mu, Y)(H) \bigr| \wedge 1
    \Bigr]
    +
    E^x\Bigl[
      \bigl| \Pi^{(\delta, R)}(H) - \Pi(H) \bigr| \wedge 1
    \Bigr].
    \label{pr eq: 4, continuity of collision map}
  \end{align}
  Letting $n \to \infty$ in the last display, 
  the first term converges to $0$ by~\eqref{thm pr: delta map STOM approx for local Kato. 5}, 
  and the third term converges to $0$ by~\eqref{thm pr: delta map STOM approx for local Kato. 4} 
  and Lemma~\ref{lem: apSTOM continuity}\ref{lem item: apSTOM continuity. 2}.
  Hence,
  \begin{equation}
    \limsup_{n \to \infty}
    E\bigl[
      |\Sigma_n^{\delta_n}(H) - \Sigma^*(H)| \wedge 1
    \bigr]
    \leq  
    2\sup_{y \in S}
    E^y\Bigl[
      \bigl| \Pi^{(\delta, R)}(H) - \Pi^{(*,R)}(H) \bigr| \wedge 1
    \Bigr].
  \end{equation}
  Letting $\delta \to 0$ and using~\eqref{thm pr: delta map STOM approx for local Kato. 5} once more,
  we obtain $\Sigma_n^{\delta_n}(H) \to \Sigma^*(H)$ in probability,
  which completes the proof.
\end{proof}

Finally, we can prove the continuity of joint laws, using the above result,
Proposition~\ref{prop: R STOM approx in local Kato}, and the spatial tightness condition.

\begin{thm} \label{thm: R map STOM approx for local Kato}
  Suppose that Assumptions~\ref{assum: weak L^0 continuity} and \ref{assum: spatial tightness} hold.
  Let $\mu$ be a Radon measure in the local Kato class and let $A = (A_t)_{t \geq 0}$ be the associated PCAF.
  For each $R > 1$, set 
  \begin{equation} \label{thm eq: R map STOM approx for local Kato. 1}
    \Pi^{(*, R)} \coloneqq \apSTOM^{(*, R)}(\Pi).
  \end{equation}
  Then the following statements hold.
  \begin{enumerate} [label = \textup{(\roman*)}]
    \item The map $\ProcLaw_{(X, \Pi)}$ is continuous.
    \item  It holds that, for any $r > 0$, $N \in \NN$, and $R > 1$,
      \begin{align}
        &\hatCMet{S}{\Prob(L^0(\RNp, S) \times \STOMMeas(S \times \RNp))}\!\left(\ProcLaw_{(X, \Pi^{(*, R)})}, \ProcLaw_{(X, \Pi)}\right)\\
        &\leq
        e^{-r}
        +
        e^{-N} + 
        \sup_{x \in S^{(r)}}
        P^x\!\left( \exitTime_{S^{(R-1)}} < N \right).
        \label{thm eq: R map STOM approx for local Kato. 2}
      \end{align}
    \item It holds that 
      \begin{equation} \label{thm eq: R map STOM approx for local Kato. 3}
        \ProcLaw_{(X, \Pi^{(*, R)})} \xlongrightarrow[R \to \infty]{} \ProcLaw_{(X, \Pi)}
        \quad \text{in}\quad 
        C\bigl(S, \Prob\bigl( D_{L^0}(\RNp, S) \times \STOMMeas(S \times \RNp) \bigr)\bigr).
      \end{equation}
  \end{enumerate}
\end{thm}

\begin{proof}
  For each $N \in \NN$ and $R > 1$, we define an event
  \begin{equation}
    E_{N, R} \coloneqq \left\{ \exitTime_{S^{(R-1)}} > N \right\}.
  \end{equation}
  On this event, by \eqref{prop pr: R STOM approx in local Kato. 1}, we have
  \begin{equation}
    \STOMMet{S \times \RNp}(\Pi^{(*,R)}, \Pi) \leq e^{-N}.
  \end{equation}
  Lemma~\ref{lem: Prohorov estimate} then yields that, for each $x \in S$, 
  \begin{align}
    &\ProhMet{L^0(\RNp, S) \times \STOMMeas(S \times \RNp)}\! 
    \left(\ProcLaw_{(X, \Pi^{(*,R)})}(x), \ProcLaw_{(X, \Pi)}(x)\right)\\
    &\leq
    e^{-N} + P^x\!\left( \exitTime_{S^{(R-1)}} > N \right).
  \end{align}
  Now, we can follow the same argument as in the proof of Theorem~\ref{thm: R map approx for local Kato} 
  to complete the proof.
\end{proof}

By using Proposition~\ref{prop: R STOM vague approx in local Kato} 
and working with the vague topology instead of the topology on $\STOMMeas(S \times \RNp)$,
we can derive the continuity of $\ProcLaw_{(X, \Pi)}$ 
without assuming the spatial tightness condition,
although the continuity is then obtained only with respect to the vague topology.

\begin{thm} \label{thm: R map STOM vague approx for local Kato}
  Suppose that Assumption~\ref{assum: weak L^0 continuity} holds.
  Let $\mu$ be a Radon measure in the local Kato class and let $A = (A_t)_{t \geq 0}$ be the associated PCAF.
  For each $R > 1$, set 
  \begin{equation} \label{thm eq: R map STOM vague approx for local Kato. 1}
    \Pi^{(*, R)} \coloneqq \apSTOM^{(*, R)}(\Pi).
  \end{equation}
  Then the following statements hold.
  \begin{enumerate} [label = \textup{(\roman*)}]
    \item The map 
      \begin{equation}
            \ProcLaw_{(X, \Pi)} \colon S \to \Prob\bigl( D_{L^0}(\RNp, S) \times \Meas(S \times \RNp) \bigr)
      \end{equation}
     is continuous.
     (NB.\ In the codomain, we have $\Meas(S \times \RNp)$ equipped with the vague topology, not $\STOMMeas(S \times \RNp)$.)
    \item  For any $R > 1$,
      \begin{align}
        \hatCMet{S}{\Prob(L^0(\RNp, S) \times \Meas(S \times \RNp))}\!\left(\ProcLaw_{(X, \Pi^{(*, R)})}, \ProcLaw_{(X, \Pi)}\right)
        \leq
        e^{-R+1}
        \label{thm eq: R map STOM vague approx for local Kato. 2}
      \end{align}
    \item It holds that 
      \begin{equation} \label{thm eq: R map STOM vague approx for local Kato. 3}
        \ProcLaw_{(X, \Pi^{(*, R)})} \xlongrightarrow[R \to \infty]{} \ProcLaw_{(X, \Pi)}
        \quad \text{in}\quad 
        C\bigl(S, \Prob\bigl( D_{L^0}(\RNp, S) \times \Meas(S \times \RNp) \bigr)\bigr).
      \end{equation}
  \end{enumerate}
\end{thm}

\begin{proof}
  For any $x \in S$, we have from Proposition~\ref{prop: R STOM vague approx in local Kato} that 
  \begin{equation}
    \ProhMet{L^0(\RNp, S) \times \Meas(S \times \RNp)}\! 
    \left(\ProcLaw_{(X, \Pi^{(*,R)})}(x), \ProcLaw_{(X, \Pi)}(x)\right)
    \leq
    e^{-R+1}.
  \end{equation}
  Thus, we deduce the results in the same way as in the proof of Theorem~\ref{thm: R map STOM approx for local Kato}.
\end{proof}


\section{Convergence of PCAFs and STOMs} \label{sec: conv of PCAFs and STOMs}

In this section, 
we present the main results of this paper,
which establish the convergence of PCAFs and STOMs from the convergence of processes
within the framework of Gromov--Hausdorff topologies introduced in Section~\ref{sec: GH-type topologies}.
In Section~\ref{sec: STOM conv for dtm spaces},
we provide results in the setting where the underlying spaces are deterministic.
This is extended to random spaces in the following subsection, 
with applications to random graph models in mind.

Throughout this section, we fix a Polish structure $\tau$ and a space transformation $\Psi$.

\subsection{Deterministic spaces} \label{sec: STOM conv for dtm spaces}

The main result of this subsection is Theorem~\ref{thm: PCAF/STOM dtm result}.
It asserts that if processes and their smooth measures converge in a Gromov--Hausdorff-type topology,
and if the local Kato class condition is satisfied uniformly by these measures,
then the associated STOMs converge jointly with the processes.

We first clarify the setting for the result.
For each $n \in \NN$, we suppose the following:
\begin{itemize} 
  \item $(S_n, d_{S_n}, \rho_n, \mu_n, a_n)$ is an element of $\rootBCM(\MeasSt(\Psi) \times \tau)$, i.e., 
      $(S_n, d_{S_n})$ is a $\bcmAB$ space, $\rho_n$ is the root of $S_n$, $\mu_n$ is a Radon measure on $\Psi(S_n)$, and $a_n \in \tau(S_n)$;
  \item $X_n = ((X_n(t))_{t \geq 0}, (P_n^x)_{x \in \Psi(S_n)})$ is a standard process on $\Psi(S_n)$ 
    satisfying the DAC and $L^0$-weak continuity conditions (Assumptions~\ref{assum: dual hypothesis} and \ref{assum: weak L^0 continuity}),
    with heat kernel $p_n$;
  \item the Radon measure $\mu_n$ is in the local Kato class of $X_n$.
\end{itemize}
We also suppose the following:
\begin{itemize} 
  \item $(S, d_S, \rho, \mu, a)$ is an element of $\rootBCM(\MeasSt(\Psi) \times \tau)$;
  \item $X = ((X(t))_{t \geq 0}, (P^x)_{x \in \Psi(S)})$ is a standard process on $\Psi(S)$ 
    satisfying the DAC and $L^0$-weak continuity conditions (Assumptions~\ref{assum: dual hypothesis} and \ref{assum: weak L^0 continuity});
    with heat kernel $p_n$,
  \item the Radon measure $\mu$ charges no sets semipolar for $X$.
\end{itemize}

The presence of $\Psi$ in our setting is motivated by applications to collisions of stochastic processes, 
which will be studied from Section~\ref{sec: collision measure} onward. 
The incorporation of $\tau$ is also important, as it broadens the applicability of our main results. 
See Remarks~\ref{rem: tau is preserved} and \ref{rem: tau for canonical embedding} at the end of this subsection for further details.

Recall from Assumption~\ref{assum: weak L^0 continuity}\ref{assum item: weak L^0 continuity. 1} that
the law maps $\ProcLaw_{X_n}$ and $\ProcLaw_X$ are continuous
with respect to the weak topology induced by the $L^0$ topology.
Our fundamental assumption for the main result is the convergence of the underlying spaces 
together with the law maps and smooth measures.
To put this in the framework introduced in Section~\ref{sec: GH-type topologies},
we define a structure by 
\begin{equation} \label{eq: st for L^0 SPM}
  \LzeroSPMSt \coloneqq \hatCSt(\Psi_{\id}, \ProbSt(\LzeroSt)),
\end{equation}
where ``SP'' stands for ``stochastic process''.
In particular, for each $\bcmAB$ space $M$,
\begin{equation}
  \LzeroSPMSt(\Psi)(M) =  \hatC\bigl( \Psi(M), \Prob(L^0(\RNp, \Psi(M))) \bigr),
\end{equation}
where we recall from Definition~\ref{dfn: composition} that $\LzeroSPMSt(\Psi)$ denotes the composition $\LzeroSPMSt \circ \Psi$.
To discuss convergence of heat kernels,
we define another structure by 
\begin{equation} \label{eq: st for hk}
  \HKSt \coloneqq \hatCSt( \Psi_{\RNpp \times \id^2}, \tau_{\RNp}).
\end{equation}
In particular, for each $\bcmAB$ space $M$,
\begin{equation}
  \HKSt(\Psi)(M) =  \hatC\bigl( \RNpp \times \Psi(M) \times \Psi(M), \RNp \bigr).
\end{equation}

\begin{assum} \label{assum: PCAF/STOM dtm assumption} \leavevmode
  \begin{enumerate} [label = \textup{(\roman*)}, series = STOM dtm assumption]
    \item \label{assum item: 1. PCAF dtm assumption}
      It holds that 
      \begin{equation}
        (S_n, d_{S_n}, \rho_n, \mu_n, a_n, p_n, \ProcLaw_{X_n}) \to (S, d_S, \rho, \mu, a, p, \ProcLaw_{X})
      \end{equation}
      in 
      $\rootBCM \bigl(\MeasSt(\Psi) \times \tau \times \HKSt(\Psi) \times  \LzeroSPMSt(\Psi) \bigr)$.
    \item \label{assum item: 2. PCAF dtm assumption}
      For all $r > 0$, 
      \begin{equation}
        \lim_{\alpha \to \infty}
        \limsup_{n \to \infty}
        \bigl\| \Potential_{p_n}^\alpha \mu_n^{(r)} \bigr\|_\infty
        = 0.
      \end{equation}
  \end{enumerate}
\end{assum}

\begin{rem} \label{rem: assum for STOM dtm}
  By Proposition~\ref{prop: Potential and hk behavior},
  condition~\ref{assum item: 2. PCAF dtm assumption} of Assumption~\ref{assum: PCAF/STOM dtm assumption} 
  is equivalent to the following:
  \begin{enumerate}[resume* = STOM dtm assumption]
    \item \label{assum item: 4. PCAF dtm assumption}
    For all $r > 0$, 
    \begin{equation}
      \lim_{\delta \to 0} 
      \limsup_{n \to \infty} 
      \sup_{x \in \Psi(S_n^{(r)})} 
      \int_0^\delta \int_{\Psi(S_n^{(r)})} p_n(t, x, y)\, \mu_n(dy)\, dt 
      = 0.
    \end{equation}
  \end{enumerate}
\end{rem}

As observed in Section~\ref{sec: Continuity of joint laws with STOMs},
the following (uniform) spatial tightness condition
is instrumental in obtaining a stronger mode of convergence for STOMs 
than that given by the vague topology.

\begin{assum} \label{assum: PCAF dtm assumption. spatial tight}
  Each process $X_n$, $n \geq 1$, as well as the limiting process $X$, 
  satisfies the spatial tightness condition (Assumption~\ref{assum: spatial tightness}).
  Moreover, for all $r > 0$ and $T > 0$, it holds that
  \begin{equation} \label{assum eq: PCAF dtm assumption. spatial tight}
    \lim_{R \to \infty}
    \limsup_{n \to \infty}
    \sup_{x \in S_n^{(r)}}
    P_n^x\!\left( \exitTime_{S_n^{(R)}} < T \right)
    = 0.
  \end{equation}
\end{assum}

\begin{rem} \label{rem: unif spatial tightness}
  By the same argument as in the proof of Lemma~\ref{lem: J_1 tight implies spatial tight},
  if the convergence of law maps in 
  Assumption~\ref{assum: PCAF/STOM dtm assumption}\ref{assum item: 1. PCAF dtm assumption} 
  takes place under a topology finer than the $L^0$ topology, 
  such as the $J_1$-Skorohod topology,
  then \eqref{assum eq: PCAF dtm assumption. spatial tight} 
  can be automatically satisfied.
  For instance, set 
  \begin{equation} \label{eq: st for SPM}
    \SPMSt \coloneqq \hatCSt(\Psi_{\id}, \ProbSt(\SkorohodSt)).
  \end{equation}
  If the convergence in 
  Assumption~\ref{assum: PCAF/STOM dtm assumption}\ref{assum item: 1. PCAF dtm assumption} 
  holds with $\LzeroSPMSt$ replaced by $\SPMSt$,
  then Assumption~\ref{assum: PCAF dtm assumption. spatial tight} is satisfied.
\end{rem}

Under the above assumptions, it follows from Fatou's lemma that 
the limiting measure $\mu$ also belongs to the local Kato class.

\begin{lem} \label{lem: STOM dtm limiting meas is smooth}
  Under Assumption~\ref{assum: PCAF/STOM dtm assumption},
  the Radon measure $\mu$ is in the local Kato class of $X$.
\end{lem}

\begin{proof}
  It is enough to show that, for all but countably many $r > 0$,
  \begin{equation} \label{pr eq: 1. STOM dtm limiting meas is smooth}
    \lim_{\alpha \to \infty}
    \bigl\| \Potential_p^\alpha \mu^{(r)} \bigr\|_\infty
    = 0,
  \end{equation}
  By Assumption~\ref{assum: PCAF/STOM dtm assumption}\ref{assum item: 1. PCAF dtm assumption} and Theorem~\ref{thm: conv in M(tau)},
  all the rooted metric spaces $(S_n, \rho_n)$ and $(S, \rho)$ are embedded isometrically into a common rooted $\bcmAB$ space $(M, \rho_M)$
  in such a way that $S_n \to S$ in the Fell topology as closed subsets of $M$,
  $\mu_n \to \mu$ vaguely as measures on $\Psi(M)$, and $p_n \to p$ in $\hatC(\RNpp \times \Psi(M) \times \Psi(M), \RNp)$.
  Fix $r > 0$ such that $\mu_n^{(r)} \to \mu^{(r)}$ weakly as measures on $\Psi(M)$, and fix $x \in S$.
  By Lemma~\ref{lem: ST preserves Fell conv},
  we can find elements $x_n \in \Psi(S_n)$ converging to $x$ in $\Psi(M)$.
  It is then the case that, for each $t > 0$,
  \begin{equation}
    p_n(t, x_n, \cdot) \to p(t, x, \cdot) 
    \quad \text{in} \quad \hatC(\Psi(M), \RNp).
  \end{equation}
  We deduce from Lemma~\ref{lem: vague convergence and hatC topology} that, for each $t > 0$,
  \begin{equation}
    \lim_{n \to \infty}
    \int_{\Psi(M)} p_n(t,x_n,y)\, \mu_n^{(r)}(dy) 
    = 
    \int_{\Psi(M)} p(t,x,y)\, \mu^{(r)}(dy).
  \end{equation}
  Fatou's lemma then yields that, for each $\alpha > 0$,
  \begin{align}
    \int_0^\infty  e^{-\alpha t} \int_{\Psi(M)} p(t,x,y)\, \mu^{(r)}(dy)\, dt
    &\leq 
    \liminf_{n \to \infty}
    \int_0^\infty e^{-\alpha t} \int_{\Psi(M)} p_n(t,x_n,y)\, \mu_n^{(r)}(dy)\, dt\\
    &\leq     
    \liminf_{n \to \infty}
    \bigl\| \Potential_{p_n}^\alpha \mu_n^{(r)} \bigr\|_\infty.
  \end{align}
  Taking the supremum over $x \in \Psi(S)$,
  it follows that 
  \begin{equation} \label{pr eq: 2. STOM dtm limiting meas is smooth}
    \bigl\| \Potential_p^\alpha \mu^{(r)} \bigr\|_\infty
    \leq     
    \liminf_{n \to \infty}
    \bigl\| \Potential_{p_n}^\alpha \mu_n^{(r)} \bigr\|_\infty.
  \end{equation}
  This, combined with Assumption~\ref{assum: PCAF/STOM dtm assumption}\ref{assum item: 2. PCAF dtm assumption}, yields \eqref{pr eq: 1. STOM dtm limiting meas is smooth}.
\end{proof}

Let $A_n$ and $\Pi_n$ (resp.\ $A$ and $\Pi$) denote the PCAF and the STOM of $X_n$ (resp.\ $X$) associated with $\mu_n$ (resp.\ $\mu$).
The main result below asserts that, under the above assumptions,
the joint laws of the processes together with their associated PCAFs/STOMs converge, 
locally uniformly in the starting point. 
For notational convenience, set
\begin{align} 
  &\LzeroSPMPCAFSt \coloneqq \hatCSt\!\left( \Psi_{\id}, \ProbSt\bigl( \LzeroSt \times \Psi_{\upC(\RNp, \RNp)} \bigr) \right),
  \label{eq: dfn of PCAF st}\\
  &\LzeroSPMSTOMSt \coloneqq 
  \hatCSt\!\left( \Psi_{\id}, \ProbSt\bigl( \LzeroSt \times \STOMSt \bigr) \right),
  \label{eq: dfn of STOM st}\\
  &\LzeroSPMvSTOMSt \coloneqq 
  \hatCSt\!\left( \Psi_{\id}, \ProbSt\bigl( \LzeroSt \times \MeasSt(\Psi_{\id} \times \Psi_{\RNp}) \bigr) \right),
  \label{eq: dfn of v-STOM st}
\end{align}
Here, the first structure is used for describing the convergence of PCAFs,
the second for that of STOMs,
and the third also for STOMs,
but with respect to the vague topology (which is weaker than the topology introduced in Section~\ref{sec: The space for STOMs}).
In particular, for each $\bcmAB$ space $M$,
\begin{align}
  \LzeroSPMPCAFSt(\Psi)(M) 
  &= 
  \hatC\!\left( \Psi(M), \Prob\bigl( L^0(\RNp, \Psi(M)) \times \upC(\RNp, \RNp) \bigr)\right),\\
  \LzeroSPMSTOMSt(\Psi)(M) 
  &=  
  \hatC\!\left( \Psi(M), \Prob\bigl( L^0(\RNp, \Psi(M)) \times \STOMMeas(\Psi(M) \times \RNp) \bigr) \right),\\
  \LzeroSPMvSTOMSt(\Psi)(M) 
  &=  
  \hatC\!\left( \Psi(M), \Prob\bigl( L^0(\RNp, \Psi(M)) \times \Meas(\Psi(M) \times \RNp) \bigr) \right).
\end{align}

\begin{thm} \label{thm: PCAF/STOM dtm result} \leavevmode
  \begin{enumerate} [label = \textup{(\roman*)}]
    \item \label{thm item: PCAF/STOM dtm result. 1}
      Under Assumptions~\ref{assum: PCAF/STOM dtm assumption} and \ref{assum: PCAF dtm assumption. spatial tight}, it holds that   
      \begin{equation} \label{thm eq: PCAF dtm result. 1}
        \left( S_n, d_{S_n}, \rho_n, \mu_n, a_n, p_n, \ProcLaw_{(X_n, A_n)} \right) 
        \to \left( S, d_S, \rho, \mu, a, p, \ProcLaw_{(X, A)} \right)
      \end{equation}
      in $\rootBCM \bigl(\MeasSt(\Psi) \times \tau \times \HKSt(\Psi) \times  \LzeroSPMPCAFSt(\Psi) \bigr)$.

    \item \label{thm item: PCAF/STOM dtm result. 2}
      Under Assumptions~\ref{assum: PCAF/STOM dtm assumption} and \ref{assum: PCAF dtm assumption. spatial tight}, it holds that
      \begin{equation}
        \left( S_n, d_{S_n}, \rho_n, \mu_n, a_n, p_n, \ProcLaw_{(X_n, \Pi_n)} \right) 
        \to \left( S, d_S, \rho, \mu, a, p, \ProcLaw_{(X, \Pi)} \right)
      \end{equation}
      in $\rootBCM \bigl(\MeasSt(\Psi) \times \tau \times \HKSt(\Psi) \times  \LzeroSPMSTOMSt(\Psi) \bigr)$.

    \item \label{thm item: PCAF/STOM dtm result. 3}
      Under Assumption~\ref{assum: PCAF/STOM dtm assumption},
      it holds that
      \begin{equation}
        \left( S_n, d_{S_n}, \rho_n, \mu_n, a_n, p_n, \ProcLaw_{(X_n, \Pi_n)} \right) 
        \to \left( S, d_S, \rho, \mu, a, p, \ProcLaw_{(X, \Pi)} \right)
      \end{equation}
      in $\rootBCM \bigl(\MeasSt(\Psi) \times \tau \times \HKSt(\Psi) \times  \LzeroSPMvSTOMSt(\Psi) \bigr)$.
  \end{enumerate}
\end{thm}

Since convergence of STOMs implies convergence of PCAFs (see Remark~\ref{rem: PCAF to STOM map}),
if we establish \ref{thm item: PCAF/STOM dtm result. 2}, 
then \ref{thm item: PCAF/STOM dtm result. 1} follows immediately from Proposition~\ref{prop: PCAF to STOM map}. 
Nevertheless, we provide separate proofs here, 
since several intermediate results obtained in their proofs
will be used to extend the results to random spaces in the next subsection.

The idea of the proofs is common to all three assertions. 
First, we employ the continuity of the approximation maps introduced in 
Definitions~\ref{dfn: restriction of PCAF} and \ref{dfn: apPCAF}
to lift the convergence of the laws of the processes 
to that of the joint laws of the processes and the approximated PCAFs. 
Then, using Assumptions~\ref{assum: PCAF/STOM dtm assumption}\ref{assum item: 2. PCAF dtm assumption} 
and \ref{assum: PCAF dtm assumption. spatial tight},
we verify the uniformity of the approximation 
and finally deduce the desired convergence.
We note that in \ref{thm item: PCAF/STOM dtm result. 3}, 
the lack of Assumption~\ref{assum: PCAF dtm assumption. spatial tight} 
leads to convergence in a weaker topology than that in 
\ref{thm item: PCAF/STOM dtm result. 2}.
This situation has already appeared in 
Proposition~\ref{prop: R STOM vague approx in local Kato} 
and Theorem~\ref{thm: R map STOM vague approx for local Kato}.

\begin{proof} [{Proof of Theorem~\ref{thm: PCAF/STOM dtm result}\ref{thm item: PCAF/STOM dtm result. 1}}]
  Before proceeding to the detailed argument,
  we introduce some notation.
  For each $\delta > 0$ and $R > 1$,
  we set
  \begin{gather}
    A^{(\delta, R)} \coloneqq \apPCAF^{(\delta, R)}(p, \mu, X),
    \qquad
    A^{(*, R)} \coloneqq \apPCAF^{(*,R)}(X, A),\\
    A_n^{(\delta, R)} \coloneqq \apPCAF^{(\delta, R)}(p_n, \mu_n, X_n),
    \qquad
    A_n^{(*, R)} \coloneqq \apPCAF^{(*,R)}(X_n, A_n),
  \end{gather}
  where the approximation maps are recalled from
  Definitions~\ref{dfn: restriction of PCAF} and \ref{dfn: apPCAF}.
  We define
  \begin{align}
    \cX_n &\coloneqq \left( S_n, d_{S_n}, \rho_n, \mu_n, a_n, p_n, \ProcLaw_{(X_n, A_n)} \right),\\
    \cX_n^{(*, R)} &\coloneqq \left( S_n, d_{S_n}, \rho_n, \mu_n, a_n, p_n, \ProcLaw_{(X_n, A_n^{(*, R)})} \right),\\
    \cX_n^{(\delta, R)} &\coloneqq \left( S_n, d_{S_n}, \rho_n, \mu_n, a_n, p_n, \ProcLaw_{(X_n, A_n^{(\delta, R)})} \right),
  \end{align}
  and similarly define $\cX$, $\cX^{(*,R)}$, and $\cX^{(\delta, R)}$ by removing the subscript $n$
  from the corresponding symbols above.
  For notational convenience, we write 
  \begin{equation}
    \tau_1 \coloneqq \MeasSt(\Psi) \times \tau \times \HKSt(\Psi) \times  \LzeroSPMPCAFSt(\Psi).
  \end{equation}
  By Theorems~\ref{prop: delta map approx for local Kato} and \ref{thm: R map approx for local Kato}, 
  we have
  \begin{align}
    &\cX^{(\delta, R)} \xrightarrow[\delta \to 0]{} \cX^{(*, R)}
    \quad \text{in}\quad \rootBCM(\tau_1),
    \label{thm pr: PCAF dtm result. 5}\\
    &\cX^{(*, R)} \xrightarrow[R \to \infty]{} \cX
    \quad \text{in}\quad \rootBCM(\tau_1).
    \label{thm pr: PCAF dtm result. 6}
  \end{align}

  By Assumption~\ref{assum: PCAF/STOM dtm assumption}\ref{assum item: 1. PCAF dtm assumption} and Theorem~\ref{thm: conv in M(tau)},
  there exists a rooted $\bcmAB$space $(M, \rho_M)$
  into which all the metric spaces $S_n$ and $S$ can be isometrically embedded 
  so that the following hold:
  \begin{enumerate} [label = \textup{(A\arabic*)}, leftmargin = *, series = PCAF embedding]
    \item \label{cond: 1. PCAF embedding}
      $S_n \to S$ in the Fell topology as closed subsets of $M$,
    \item \label{cond: 2. PCAF embedding} 
      $\rho_n = \rho = \rho_M$ as elements of $M$,
    \item \label{cond: 3. PCAF embedding} 
      $\mu_n \to \mu$ vaguely as measures on $\Psi(M)$,
    \item \label{cond: 3.5. PCAF embedding} 
      $a_n \to a$ in $\tau(M)$,
    \item \label{cond: 4. PCAF embedding} 
      $p_n \to p$ in $\hatC(\RNpp \times \Psi(M) \times \Psi(M), \RNp)$,
    \item \label{cond: 5. PCAF embedding} 
      $\ProcLaw_{X_n} \to \ProcLaw_X$ in $\hatC\bigl( \Psi(M), \Prob(L^0(\RNp, \Psi(M))) \bigr)$.
  \end{enumerate}

  By \ref{cond: 4. PCAF embedding} and Lemma~\ref{lem: conv in hatC},
  we may assume that the domains of $p_n$ and $p$ are extended continuously to $\RNpp \times \Psi(M) \times \Psi(M)$
  so that 
  \begin{equation} \label{thm pr: PCAF dtm result. 1}
    p_n \to p \quad \text{in} \quad C(\RNpp \times \Psi(M) \times \Psi(M), \RNp).
  \end{equation}
  Fix points $x_n \in \Psi(S_n)$ converging to some $x \in \Psi(S)$ in $\Psi(M)$.
  Then, by \ref{cond: 5. PCAF embedding}, 
  \begin{equation}
    \ProcLaw_{X_n}(x_n) = P_n^{x_n}(X_n \in \cdot) \xrightarrow[n \to \infty]{\mathrm{d}} \ProcLaw_X(x) = P^x(X \in \cdot).
  \end{equation}
  Hence, by the Skorohod representation theorem, 
  we may assume that $X_n$ under $P_n^{x_n}$ and $X$ under $P^x$
  are coupled so that $X_n \to X$ in $L^0(\RNp, \Psi(M))$ almost surely.
  By \ref{cond: 3. PCAF embedding}, \eqref{thm pr: PCAF dtm result. 1}, and this coupling, 
  \begin{equation}
    (p_n, \mu_n, X_n) \xrightarrow[n \to \infty]{\mathrm{a.s.}} (p, \mu, X).
  \end{equation}
  Moreover, from \ref{assum item: 2. PCAF dtm assumption},
  it is clear that 
  \begin{equation}  \label{thm pr: PCAF dtm result. 12}
    \sup_{n \geq 1} \bigl\| \Potential_{p_n}^1 \mu_n^{(r)} \bigr\|_\infty < \infty,
  \end{equation}
  which, combined with Proposition~\ref{prop: Potential and hk behavior}, implies 
  \begin{equation}  \label{thm pr: PCAF dtm result. 7}
    \sup_{n \geq 1} \sup_{x \in \Psi(S_n)} 
    \int_\delta^{2\delta} \int_S p_n(t, x, y)\, \mu_n^{(R)}(dy)\, dt < \infty.
  \end{equation} 
  Hence, by Lemma~\ref{lem: apPCAF continuity}\ref{lem item: apPCAF continuity. 2},
  we obtain 
  \begin{equation}  \label{thm pr: PCAF dtm result. 14}
      A_n^{(\delta, R)} \xrightarrow[n \to \infty]{\mathrm{a.s.}} A^{(\delta, R)}.
  \end{equation}
  In particular,
  \begin{equation}
    \ProcLaw_{(X_n, A_n^{(\delta, R)})}(x_n) \xrightarrow[n \to \infty]{\mathrm{d}} \ProcLaw_{(X, A^{(\delta, R)})}(x).
  \end{equation}
  Since the choice of $(x_n)_{n \geq 1}$ was arbitrary, we deduce that, for each $\delta > 0$ and $R > 1$,
  \begin{equation} \label{thm pr: PCAF dtm result. 13}
    \ProcLaw_{(X_n, A_n^{(\delta, R)})} \xrightarrow[n \to \infty]{} \ProcLaw_{(X, A^{(\delta, R)})}
    \quad \text{in}\quad 
    \hatC\!\left( \Psi(M), \Prob\bigl( L^0(\RNp, \Psi(M)) \times \upC(\RNp, \RNp) \bigr)\right).
  \end{equation}
  Therefore, for each $\delta > 0$ and $R > 1$,
  \begin{equation}  \label{thm pr: PCAF dtm result. 3}
    \cX_n^{(\delta, R)} \xrightarrow[n \to \infty]{} \cX^{(\delta, R)}
    \quad \text{in}\quad \rootBCM(\tau_1).
  \end{equation}
  (Note that so far we have used only 
  Assumption~\ref{assum: PCAF/STOM dtm assumption}\ref{assum item: 1. PCAF dtm assumption}
  and \eqref{thm pr: PCAF dtm result. 12}.)

  By Lemma~\ref{lem: simple estimate of GH distance} and Proposition~\ref{prop: delta map approx for local Kato},
  for any $n \geq 1$, $N \in \NN$, $\alpha \in (0,1)$, and $\delta \in (0,1)$,
  \begin{align}
    \GFMet^{\tau_1}\!\left(\cX_n^{(\delta, R)}, \cX_n^{(*,R)}\right)
    &\leq 
    \hatCMet{S}{\Prob(L^0(\RNp, \Psi(S_n)) \times \upC(\RNp, \RNp))}\!
    \left(\ProcLaw_{(X_n, A_n^{(\delta, R)})}, \ProcLaw_{(X_n, A_n^{(*,R)})}\right)\\
    &\leq
    2^{-N+1} 
    +
    2^{2N}
    C_1 \bigl\| \Potential^\alpha \mu_n^{(R)} \bigr\|_\infty
    \Bigl(
      \delta \alpha 
      \bigl\| \Potential^\alpha \mu_n^{(R)} \bigr\|_\infty
      + 
      \bigl\| \Potential^{1/\delta}\!\mu_n^{(R)} \bigr\|_\infty
    \Bigr)\\
    &\quad
    +
    C_2 (2e)^{2N} (1 - e^{-\alpha N}) \bigl\| \Potential^1 \mu_n^{(R)} \bigr\|_\infty^2.
    \label{thm pr: PCAF dtm result. 8}
  \end{align}
  By Assumption~\ref{assum: PCAF/STOM dtm assumption}\ref{assum item: 2. PCAF dtm assumption},
  letting $n \to \infty$, $\delta \to 0$, then $\alpha \to 0$, and finally $N \to \infty$ in the above inequality yields 
  \begin{equation} \label{thm pr: PCAF dtm result. 9}
    \lim_{\delta \to 0}
    \limsup_{n \to \infty}
    \GFMet^{\tau_1}\!\left(\cX_n^{(\delta, R)}, \cX_n^{(*,R)}\right)
    = 0.
  \end{equation}
  Combining this with \eqref{thm pr: PCAF dtm result. 5} and \eqref{thm pr: PCAF dtm result. 3} gives 
  \begin{equation} \label{thm pr: PCAF dtm result. 4}
    \cX_n^{(*, R)} \xrightarrow[n \to \infty]{} \cX^{(*, R)}
    \quad \text{in}\quad \rootBCM(\tau_1).
  \end{equation}

  Similarly, by Lemma~\ref{lem: simple estimate of GH distance} and Theorem~\ref{thm: R map approx for local Kato},
  for any $n \geq 1$, $N \in \NN$, and $r > 0$,
  \begin{align}
    \GFMet^{\tau_1}\!\left(\cX_n^{(*, R)}, \cX_n\right)
    &\leq 
    \hatCMet{S}{\Prob(L^0(\RNp, \Psi(S_n)) \times \upC(\RNp, \RNp))}\!
    \left(\ProcLaw_{(X_n, A_n^{(*, R)})}, \ProcLaw_{(X_n, A_n)}\right)\\
    &\leq
    e^{-r}
    +
    2^{-N+1} + 
    2^N
    \sup_{x \in \Psi(S_n^{(r)})}
    P^x\!\left( \exitTime_{\Psi(S_n^{(R-1)})} < N \right).
    \label{thm pr: PCAF dtm result. 10}
  \end{align}
  By Assumption~\ref{assum: PCAF dtm assumption. spatial tight},
  letting $n \to \infty$, $R \to \infty$, then $N \to \infty$, and finally $r \to \infty$ in the above inequality gives
  \begin{equation}  \label{thm pr: PCAF dtm result. 11}
    \lim_{R \to \infty}
    \limsup_{n \to \infty}
    \GFMet^{\tau_1}\!\left(\cX_n^{(*, R)}, \cX_n\right)
    = 0.
  \end{equation}
  Combining this with \eqref{thm pr: PCAF dtm result. 6} and \eqref{thm pr: PCAF dtm result. 4},
  we obtain
  \begin{equation} 
    \cX_n \xrightarrow[n \to \infty]{} \cX
    \quad \text{in}\quad \rootBCM(\tau_1),
  \end{equation}
  which completes the proof.
\end{proof}

\begin{proof}  [{Proof of Theorem~\ref{thm: PCAF/STOM dtm result}\ref{thm item: PCAF/STOM dtm result. 2}}]
  We use the same notation as in the proof of Theorem~\ref{thm: PCAF/STOM dtm result}\ref{thm item: PCAF/STOM dtm result. 1},
  and introduce some new notation.
  For each $\delta > 0$ and $R > 1$,
  we set
  \begin{gather}
    \Pi^{(\delta, R)} \coloneqq \apSTOM^{(\delta, R)}(p, \mu, X),
    \qquad
    \Pi^{(*, R)} \coloneqq \apSTOM^{(*,R)}(X, A),\\
    \Pi_n^{(\delta, R)} \coloneqq \apSTOM^{(\delta, R)}(p_n, \mu_n, X_n),
    \qquad
    \Pi_n^{(*, R)} \coloneqq \apSTOM^{(*,R)}(X_n, A_n),
  \end{gather}
  where the approximation maps are recalled from
  Definitions~\ref{dfn: truncation of STOM} and \ref{dfn: apSTOM}.
  We define
  \begin{align}
    \cY_n &\coloneqq \left( S_n, d_{S_n}, \rho_n, \mu_n, a_n, p_n, \ProcLaw_{(X_n, \Pi_n)} \right),\\
    \cY_n^{(*, R)} &\coloneqq \left( S_n, d_{S_n}, \rho_n, \mu_n, a_n, p_n, \ProcLaw_{(X_n, \Pi_n^{(*, R)})} \right),
  \end{align}
  and similarly define $\cY$ and $\cY^{(*,R)}$ by removing the subscript $n$
  from the corresponding symbols above.
  We write
  \begin{equation}
    \tau_2 \coloneqq \MeasSt(\Psi) \times \tau \times \HKSt(\Psi) \times \LzeroSPMSTOMSt(\Psi).
  \end{equation}
  By Theorem~\ref{thm: R map STOM approx for local Kato}, 
  we have
  \begin{equation}
    \cY^{(*, R)} \xrightarrow[R \to \infty]{} \cY
    \quad \text{in}\quad \rootBCM(\tau_2).
    \label{thm pr: STOM dtm result. 2}
  \end{equation}

  Fix $R > 1$.
  We first show that 
  \begin{equation} \label{thm pr: STOM dtm result. 2.1}
    \cY_n^{(*, R)} \xrightarrow[n \to \infty]{} \cY^{(*,R)}
    \quad \text{in} \quad \rootBCM(\tau_2)
  \end{equation}
  by using only 
  Assumption~\ref{assum: PCAF/STOM dtm assumption}\ref{assum item: 1. PCAF dtm assumption}
  and the following condition:
  \begin{equation}  \label{thm pr: STOM dtm result. 2.2}
    \lim_{t \to 0} 
    \limsup_{n \to \infty} 
    \sup_{y \in \Psi(S_n)}
    \int_0^t \int_{\Psi(S_n)} 
      p_n(s, y, z)\, \mu^{(R)}(dz)\, ds 
    = 0,
  \end{equation}
  which is equivalent to 
  \begin{equation}  \label{thm pr: STOM dtm result. 2.3}
    \lim_{\alpha \to \infty}
    \limsup_{n \to \infty}
    \bigl\| \Potential_{p_n}^\alpha \mu_n^{(R)} \bigr\|_\infty
    = 0
  \end{equation}
  by Proposition~\ref{prop: Potential and hk behavior}.
  (Note that this condition is indeed satisfied 
  in the present setting by 
  Assumption~\ref{assum: PCAF/STOM dtm assumption}\ref{assum item: 2. PCAF dtm assumption}.)
  By Assumption~\ref{assum: PCAF/STOM dtm assumption}\ref{assum item: 1. PCAF dtm assumption},
  there exists a rooted $\bcmAB$ space $(M, \rho_M)$
  into which all the metric spaces $S_n$ and $S$ can be isometrically embedded 
  so that \ref{cond: 1. PCAF embedding}--\ref{cond: 5. PCAF embedding} hold.
  The desired convergence follows once we show that 
  \begin{equation} \label{thm pr: STOM dtm result. 3}
    \ProcLaw_{(X_n, \Pi_n^{(*, R)})} \to \ProcLaw_{(X, \Pi^{(*, R)})}
    \quad \text{in} \quad 
    \hatC\!\left( \Psi(M), \Prob\bigl( L^0(\RNp, \Psi(M)) \times \STOMMeas(\Psi(M) \times \RNp) \bigr) \right).
  \end{equation}
  Fix elements $x_n \in \Psi(S_n)$ converging to $x \in \Psi(S)$ in $\Psi(M)$.
  The above convergence is obtained by showing that
  \begin{equation} \label{thm pr: STOM dtm result. 4}
    \ProcLaw_{(X_n, \Pi_n^{(*, R)})}(x_n) \to \ProcLaw_{(X, \Pi^{(*, R)})}(x)
    \quad \text{in} \quad 
    \Prob\bigl( L^0(\RNp, \Psi(M)) \times \STOMMeas(\Psi(M) \times \RNp) \bigr).
  \end{equation}
  As before, by \ref{cond: 5. PCAF embedding} and the Skorohod representation theorem, 
  we may assume that $X_n$ under $P_n^{x_n}$ and $X$ under $P^x$
  are coupled so that $X_n \to X$ in $L^0(\RNp, \Psi(M))$ almost surely.
  We denote by $P$ the underlying probability measure for this coupling.
  We verify the above convergence by showing that 
  \begin{equation} \label{thm pr: STOM dtm result. 5}
    \Pi_n^{(*, R)} \xrightarrow[n \to \infty]{\mathrm{p}} \Pi^{(*, R)}
    \quad \text{in $\STOMMeas(\Psi(M) \times \RNp)$ under $P$}.
  \end{equation}

  Similarly to \eqref{thm pr: PCAF dtm result. 14},
  one can verify that, for each $\delta > 0$,
  \begin{equation}  \label{thm pr: STOM dtm result. 6}
    \Pi_n^{(\delta, R)} \xrightarrow[n \to \infty]{\mathrm{a.s.}} \Pi^{(\delta, R)}
    \quad \text{in}\quad 
    \STOMMeas(\Psi(M) \times \RNp).
  \end{equation}
  Indeed, the only change is that one uses
  Lemma~\ref{lem: apSTOM continuity}\ref{lem item: apSTOM continuity. 2} 
  instead of Lemma~\ref{lem: apPCAF continuity}\ref{lem item: apPCAF continuity. 2}.

  We are now able to follow the proof of Proposition~\ref{prop: delta map STOM approx for local Kato}.
  Define a dissecting semi-ring $\mathcal{I}$ by setting 
  \begin{equation} \label{thm pr: STOM dtm result. 7}
    \mathcal{I} 
    \coloneqq 
    \{E \times [s, t) 
      \mid 
      0 \leq s < t < \infty,\  
      E \in \Borel(\Psi(M))\ \text{such that}\ \mu(\partial_{\Psi(M)} E) = 0\}.
  \end{equation}
  Fix a set $H = E \times [s, t) \in \mathcal{I}$.
  It follows from \eqref{thm pr: delta map STOM approx for local Kato. 5} that 
  \begin{equation} \label{thm pr: STOM dtm result. 8}
    \lim_{\delta \to 0}
    \sup_{y \in \Psi(S)}
    E^y\Bigl[
      \bigl| \Pi^{(*, R)}(H) - \Pi^{(\delta, R)}(H) \bigr| \wedge 1
    \Bigr]
    = 0.
  \end{equation}
  In the same manner as we established \eqref{thm pr: delta map STOM approx for local Kato. 5},
  using \eqref{prop eq: delta STOM approx for local Kato. 2} 
  and \eqref{thm pr: STOM dtm result. 2.3},
  we deduce that 
  \begin{equation}  \label{thm pr: STOM dtm result. 9}
    \lim_{\delta \to 0}
    \limsup_{n \to \infty}
    \sup_{y \in \Psi(S_n)}
    E_n^y\Bigl[
      \bigl| \Pi_n^{(*, R)}(H) - \Pi_n^{(\delta, R)}(H) \bigr| \wedge 1
    \Bigr]
    = 0.
  \end{equation}
  The triangle inequality yields that 
  \begin{align}
    E\bigl[
      |\Pi_n^{(*, R)}(H) - \Pi^{(*, R)}(H)| \wedge 1
    \bigr]
    &\leq 
    E_n^{x_n}\Bigl[
      \bigl| \Pi_n^{(*, R)}(H) - \Pi_n^{(\delta, R)}(H) \bigr| \wedge 1
    \Bigr]\\
    &\quad
    + 
    E\Bigl[
      \bigl| \Pi_n^{(\delta, R)}(H) - \Pi^{(\delta, R)}(H) \bigr| \wedge 1
    \Bigr]\\
    &\quad
    +
    E^x\Bigl[
      \bigl| \Pi^{(\delta, R)}(H) - \Pi^{(*,R)}(H) \bigr| \wedge 1
    \Bigr].
    \label{thm pr: STOM dtm result. 10}
  \end{align}
  Let $n \to \infty$ and then $\delta \to 0$ in the above inequality.
  By \eqref{thm pr: STOM dtm result. 8} and \eqref{thm pr: STOM dtm result. 9},
  the first and the third terms converge to $0$.
  The second term also converges to $0$ by \eqref{thm pr: STOM dtm result. 6}.
  Hence, $\Pi_n^{(*,R)}(H) \xrightarrow{\mathrm{p}} \Pi^{(*,R)}(H)$ under $P$,
  which implies \eqref{thm pr: STOM dtm result. 5}.
  We thus obtain \eqref{thm pr: STOM dtm result. 2.1}.

  By Lemma~\ref{lem: simple estimate of GH distance} and Theorem~\ref{thm: R map STOM approx for local Kato},
  for any $n \geq 1$, $N \in \NN$, $r > 0$, and $R > 1$,
  \begin{align}
    \GFMet^{\tau_2}\!\left(\cY_n^{(*, R)}, \cY_n\right)
    &\leq 
    \hatCMet{\Psi(S_n)}{\Prob(L^0(\RNp, \Psi(S_n)) \times \STOMMeas(\Psi(S_n) \times \RNp))}\!
    \left(\ProcLaw_{(X_n, \Pi_n^{(*, R)})}, \ProcLaw_{(X_n, \Pi_n)}\right)\\
    &\leq
    e^{-r}
    +
    e^{-N} + 
    \sup_{x \in \Psi(S_n^{(r)})}
    P_n^x\!\left( \exitTime_{\Psi(S_n^{(R-1)})} < N \right).
    \label{thm pr: STOM dtm result. 11}
  \end{align}
  Thus, the same argument that led to \eqref{thm pr: PCAF dtm result. 11} yields that 
  \begin{equation}   \label{thm pr: STOM dtm result. 12}
    \lim_{R \to \infty}
    \limsup_{n \to \infty}
    \GFMet^{\tau_2}\!\left(\cY_n^{(*, R)}, \cY_n\right)
    = 0.
  \end{equation}
  Combining this with \eqref{thm pr: STOM dtm result. 2} and \eqref{thm pr: STOM dtm result. 2.1},
  we obtain
  \begin{equation} 
    \cY_n \xrightarrow[n \to \infty]{} \cY
    \quad \text{in}\quad \rootBCM(\tau_2),
  \end{equation}
  which completes the proof.
\end{proof}

\begin{proof}[{Proof of Theorem~\ref{thm: PCAF/STOM dtm result}\ref{thm item: PCAF/STOM dtm result. 3}}]
  We use the same notation as before,
  and additionally introduce
  \begin{equation} \label{thm pr: v-STOM dtm result. 1}
    \tau_3 \coloneqq \MeasSt(\Psi) \times \tau \times \HKSt(\Psi) \times \LzeroSPMvSTOMSt(\Psi).
  \end{equation}
  By Theorem~\ref{thm: R map STOM vague approx for local Kato}, 
  we obtain
  \begin{equation} \label{thm pr: v-STOM dtm result. 2}
    \cY^{(*, R)} \xrightarrow[R \to \infty]{} \cY
    \quad \text{in} \quad \rootBCM(\tau_3).
  \end{equation}
  Note that in the proof above, 
  \eqref{thm pr: STOM dtm result. 2.1} is derived solely from 
  Assumption~\ref{assum: PCAF/STOM dtm assumption}.
  Since the topology on $\rootBCM(\tau_3)$ is weaker than that on $\rootBCM(\tau_2)$,
  it follows that 
  \begin{equation} \label{thm pr: v-STOM dtm result. 3}
    \cY_n^{(*, R)} \xrightarrow[n \to \infty]{} \cY^{(*,R)}
    \quad \text{in} \quad \rootBCM(\tau_3),
    \quad \forall R > 1.
  \end{equation}
  By Lemma~\ref{lem: simple estimate of GH distance} 
  and Theorem~\ref{thm: R map STOM vague approx for local Kato},
  for any $n \geq 1$ and $R > 1$ we have
  \begin{align} 
    \GFMet^{\tau_3}\!\left(\cY_n^{(*, R)}, \cY_n\right)
    &\leq 
    \hatCMet{\Psi(S_n)}{\Prob(L^0(\RNp, \Psi(S_n)) \times \Meas(\Psi(S_n) \times \RNp))}\!
    \left(\ProcLaw_{(X_n, \Pi_n^{(*, R)})}, \ProcLaw_{(X_n, \Pi_n)}\right) \\
    &\leq e^{-R+1}.
    \label{thm pr: v-STOM dtm result. 3.1}
  \end{align}
  This immediately yields
  \begin{equation}  \label{thm pr: v-STOM dtm result. 4}
    \lim_{R \to \infty}
    \limsup_{n \to \infty}
    \GFMet^{\tau_3}\!\left(\cY_n^{(*, R)}, \cY_n\right)
    = 0.
  \end{equation}
  Therefore, the desired convergence follows from 
  \eqref{thm pr: v-STOM dtm result. 2}, 
  \eqref{thm pr: v-STOM dtm result. 3}, 
  and \eqref{thm pr: v-STOM dtm result. 4}.
\end{proof}

\begin{rem} \label{rem: replace L^0 by J^1 for dtm main result}
  When the convergence of processes holds in a stronger Polish topology, such as the $J_1$-Skorohod topology,
  the joint convergence established in Theorem~\ref{thm: PCAF/STOM dtm result} holds accordingly in the corresponding stronger topology.
  For example, as discussed in Remark~\ref{rem: unif spatial tightness},
  if the convergence in Assumption~\ref{assum: PCAF/STOM dtm assumption}\ref{assum item: 1. PCAF dtm assumption} 
  holds with $\LzeroSPMSt$ replaced by $\SPMSt$,
  then the convergences in Theorem~\ref{thm: PCAF/STOM dtm result}\ref{thm item: PCAF/STOM dtm result. 1} and \ref{thm item: PCAF/STOM dtm result. 2},
  holds with $\LzeroSPMPCAFSt(\Psi)$ and $\LzeroSPMSTOMSt$ replaced by 
  the following structures, respectively:
  \begin{align}
    \SPMPCAFSt &\coloneqq \hatC(\Psi_{\id}, \ProbSt(\SkorohodSt \times \Psi_{\upC(\RNp, \RNp)})),\\
    \SPMSTOMSt &\coloneqq \hatC(\Psi_{\id}, \ProbSt(\SkorohodSt \times \STOMSt)).
  \end{align}
  This replacement is readily verified via a standard tightness argument, 
  see Appendix~\ref{appendix: Replacement of the L^0 topology by stronger topologies}.
  As a consequence of this replacement argument,
  Theorem~\ref{thm: intro STOM dtm} follows directly from 
  Theorem~\ref{thm: PCAF/STOM dtm result}\ref{thm item: PCAF/STOM dtm result. 2}.
\end{rem}

\begin{rem} \label{rem: tau is preserved}
We note that the convergence of the additional structures $a_n$ assumed in Assumption~\ref{assum: PCAF/STOM dtm assumption}
is preserved in Theorem~\ref{thm: PCAF/STOM dtm result}. 
Put differently, the convergence of PCAFs and STOMs is established 
while retaining the other structures of the underlying spaces. 
In this sense, our formulation shares a similar philosophy with that of stable convergence \cite{Aldous_Eagleson_78_MixingStability}.
For a concrete application of this formulation, see Remark~\ref{rem: tau for canonical embedding} below.
\end{rem}

\begin{rem} \label{rem: tau for canonical embedding}
In some examples, there exists a canonical space $M$ into which both $S_n$ and $S$ can be embedded. 
For instance, when $S_n$ are the standard graph approximations converging to the two-dimensional Sierpiński gasket $S$, 
the Euclidean space $\RN^2$ serves as such a space $M$. 
In this case, we set $\tau \coloneqq \hatCSt(\Psi_{\id}, \tau_M)$,
where we recall the fixed structure $\tau_M$ from \ref{item: fixed structure},
and let $a_n \in \tau(S_n)$ denote the canonical embedding $a_n \colon S_n \to M$. 
Then, by Theorems~\ref{thm: conv in M(tau)} and \ref{thm: PCAF/STOM dtm result}, together with the continuous mapping theorem, 
we obtain convergence of the processes and their STOMs within the fixed space $M$.
\end{rem}


\subsection{Random spaces}  \label{sec: PCAF conv for rdm spaces}

The main result of this subsection is Theorem~\ref{thm: PCAF/STOM rdm result} below,
which is a version of Theorem~\ref{thm: PCAF/STOM dtm result} for random spaces.

We first clarify the setting for the result.
For each $n \in \NN$, we let $(\Omega_n, \mathcal{G}_n, \mathbf{P}_n)$ be a complete probability space,
and assume that, for $\mathbf{P}_n$-a.s.\ $\omega \in \Omega_n$,
\begin{itemize} 
  \item $(S_n^\omega, d_{S_n}^\omega, \rho_n^\omega, \mu_n^\omega, a_n^\omega) \in \rootBCM(\MeasSt(\Psi) \times \tau)$;
  \item $X_n^\omega$ is a standard process on $\Psi(S_n^\omega)$ 
    satisfying the DAC and $L^0$-weak continuity conditions (Assumptions~\ref{assum: dual hypothesis} and \ref{assum: weak L^0 continuity}), 
    with heat kernel $p_n^\omega$;
  \item the Radon measure $\mu_n^\omega$ belongs to the local Kato class of $X_n^\omega$.
\end{itemize}
Similarly, we let $(\Omega, \mathcal{G}, \mathbf{P})$ be a complete probability space, 
and assume that, for $\mathbf{P}$-a.s.\ $\omega \in \Omega$,
\begin{itemize} 
  \item $(S^\omega, d_S^\omega, \rho^\omega, \mu^\omega, a^\omega) \in \rootBCM(\MeasSt(\Psi) \times \tau)$;
  \item $X^\omega$ is a standard process on $\Psi(S^\omega)$ 
    satisfying the DAC and $L^0$-weak continuity conditions (Assumptions~\ref{assum: dual hypothesis} and \ref{assum: weak L^0 continuity}), 
    with heat kernel $p^\omega$;
  \item the Radon measure $\mu^\omega$ charges no sets semipolar for $X^\omega$.
\end{itemize}
To discuss the laws of these objects,
we assume that the maps
\begin{gather}
  (\Omega_n, \mathcal{G}_n) \ni \omega \mapsto 
  (S_n^\omega, d_{S_n}^\omega, \rho_n^\omega, \mu_n^\omega, a_n^\omega, p_n^\omega, \ProcLaw_{X_n^\omega}) 
  \in \rootBCM \bigl(\MeasSt(\Psi) \times \tau \times \HKSt(\Psi) \times \LzeroSPMSt(\Psi) \bigr),\\
  (\Omega, \mathcal{G}) \ni \omega \mapsto 
  (S^\omega, d_S^\omega, \rho^\omega, \mu^\omega, a^\omega, p^\omega, \ProcLaw_{X^\omega}) 
  \in \rootBCM \bigl(\MeasSt(\Psi) \times \tau \times \HKSt(\Psi) \times \LzeroSPMSt(\Psi) \bigr)
  \label{eq: measurability assumption for SPM}
\end{gather}
are measurable, where the codomain is equipped with the Borel $\sigma$-algebra.
(Strictly speaking, since the objects are defined only for 
$\mathbf{P}_n$- and $\mathbf{P}$-a.s.\ $\omega$,
we extend them arbitrarily on null sets.
As each probability space is complete, 
this does not affect measurability of the maps.)

We now consider the following versions of Assumptions~\ref{assum: PCAF/STOM dtm assumption} and \ref{assum: PCAF dtm assumption. spatial tight}.
As usual, we suppress the dependence on $\omega$ in the notation.

\begin{assum} \label{assum: PCAF rdm assumption} \leavevmode
  \begin{enumerate} [label = \textup{(\roman*)}, series = PCAF rdm assumption]
    \item  \label{assum item: 1. PCAF rdm assumption}
      It holds that 
      \begin{equation}  \label{assum item eq: 1. PCAF rdm assumption}
        (S_n, d_{S_n}, \rho_n, \mu_n, a_n, p_n, \ProcLaw_{X_n}) \xrightarrow{\mathrm{d}} (S, d_S, \rho, \mu, a, p, \ProcLaw_{X})
      \end{equation}
      in 
      $\rootBCM \bigl(\MeasSt(\Psi) \times \tau \times \HKSt(\Psi) \times  \SPMSt(\Psi) \bigr)$.
    \item  \label{assum item: 2. PCAF rdm assumption}
      For all $r > 0$, 
      \begin{equation}
        \lim_{\alpha \to \infty}
        \limsup_{n \to \infty}
        \mathbf{E}_n
        \left[
          \bigl\| \Potential_{p_n}^\alpha \mu_n^{(r)} \bigr\|_\infty \wedge 1
        \right]
        = 0.
      \end{equation}
  \end{enumerate}
\end{assum}

\begin{rem} \label{rem: assum for PCAF rdm}
  By Proposition~\ref{prop: Potential and hk behavior},
  condition~\ref{assum item: 2. PCAF rdm assumption} of Assumption~\ref{assum: PCAF rdm assumption} 
  is equivalent to the following:
  \begin{enumerate}[resume* = PCAF rdm assumption]
    \item For all $r > 0$, 
    \begin{equation}
      \lim_{\delta \to 0} 
      \limsup_{n \to \infty} 
      \mathbf{E}_n\! 
      \left[
        \left(
          \sup_{x \in \Psi(S_n^{(r)})} 
          \int_0^\delta \int_{\Psi(S_n^{(r)})} p_n(t, x, y)\, \mu_n(dy)\, dt
        \right)
        \wedge 
        1
      \right]
      = 0.
    \end{equation}
  \end{enumerate}
\end{rem}

\begin{assum} \label{assum: PCAF rdm assumption. spatial tight}
  For almost every realization $\omega_n$ and $\omega$ of the random spaces,
  the processes $X_n^{\omega_n}$, $n \ge 1$, and $X^\omega$ 
  satisfy the spatial tightness condition (Assumption~\ref{assum: spatial tightness}).
  Moreover, for all $r > 0$ and $T > 0$, it holds that
  \begin{equation}
    \lim_{R \to \infty}
    \limsup_{n \to \infty}
    \mathbf{E}_n\!\left[
      \sup_{x \in \Psi(S_n^{(r)})}
      P_n^x\!\left( \exitTime_{\Psi(S_n^{(R)})} < T \right)
    \right]
    = 0.
  \end{equation}
\end{assum}

Similarly to Lemma~\ref{lem: STOM dtm limiting meas is smooth},
we first verify that the limiting measure $\mu$ is in the local Kato class, almost surely.

\begin{lem} \label{lem: PCAF/STOM rdm limiting meas is smooth}
  Under Assumption~\ref{assum: PCAF rdm assumption}, $\mathbf{P}$-a.s., the Radon measure $\mu$ is in the local Kato class of $X$.
\end{lem}

\begin{proof}
  By the Skorohod representation theorem,
  we may assume that the convergence in \eqref{assum item eq: 1. PCAF rdm assumption} takes place almost surely 
  in some probability space with probability measure $\mathbb{P}$.
  Fix $r > 0$.
  By Assumption~\ref{assum: PCAF rdm assumption}\ref{assum item: 2. PCAF rdm assumption} and Fatou's lemma,
  \begin{align}
    0
    =
    \lim_{\alpha \to \infty} 
    \limsup_{n \to \infty}
    \mathbb{E}\!
    \left[
      \bigl\| \Potential_{p_n}^\alpha \mu_n^{(r)} \bigr\|_\infty \wedge 1
    \right]
    \geq
    \mathbb{E}
    \left[
      \lim_{\alpha \to \infty} 
      \liminf_{n \to \infty}
      \left(
        \bigl\| \Potential_{p_n}^\alpha \mu_n^{(r)} \bigr\|_\infty \wedge 1
      \right)
    \right].
  \end{align}
  It follows that 
  \begin{equation} \label{pr eq: 1. PCAF/STOM rdm limiting meas is smooth}
    \lim_{\alpha \to \infty} 
    \liminf_{n \to \infty}
    \bigl\| \Potential_{p_n}^\alpha \mu_n^{(r)} \bigr\|_\infty
    = 0,
    \quad 
    \mathbb{P}\text{-a.s.}
  \end{equation}
  This, combined with \eqref{pr eq: 2. STOM dtm limiting meas is smooth}, implies that, $\mathbf{P}$-a.s.,
  \begin{equation}
    \lim_{\alpha \to \infty}
    \bigl\| \Potential_p^\alpha \mu^{(r)} \bigr\|_\infty = 0.
  \end{equation}
  By the monotonicity with respect to $r$,
  the above assertion holds for all $r > 0$, $\mathbf{P}$-a.s.
  This completes the proof.
\end{proof}

For each $n \geq 1$, we let $A_n$ and $\Pi_n$ be the PCAF and STOM of $X_n$ associated with $\mu_n$, respectively.
By Lemma~\ref{lem: PCAF/STOM rdm limiting meas is smooth},
$\mathbf{P}$-a.s., 
there exists a PCAF $A$ and STOM $\Pi$ of $X$ associated with $\mu$.
We now state the main result.

\begin{thm} \label{thm: PCAF/STOM rdm result} \leavevmode
  \begin{enumerate} [label = \textup{(\roman*)}]
    \item \label{thm item: PCAF/STOM rdm result. 1}
      Under Assumptions~\ref{assum: PCAF rdm assumption} and \ref{assum: PCAF rdm assumption. spatial tight}, it holds that   
      \begin{equation} \label{thm eq: PCAF rdm result. 1}
        \left( S_n, d_{S_n}, \rho_n, \mu_n, a_n, p_n, \ProcLaw_{(X_n, A_n)} \right) 
          \xrightarrow{\mathrm{d}} \left( S, d_S, \rho, \mu, a, p, \ProcLaw_{(X, A)} \right)
      \end{equation}
      as random elements of  $\rootBCM \bigl(\MeasSt(\Psi) \times \tau \times \HKSt(\Psi) \times  \LzeroSPMPCAFSt(\Psi) \bigr)$.
    \item \label{thm item: PCAF/STOM rdm result. 2}
       Under Assumptions~\ref{assum: PCAF rdm assumption} and \ref{assum: PCAF rdm assumption. spatial tight}, it holds that   
      \begin{equation}
        \left( S_n, d_{S_n}, \rho_n, \mu_n, a_n, p_n, \ProcLaw_{(X_n, \Pi_n)} \right) 
          \xrightarrow{\mathrm{d}} \left( S, d_S, \rho, \mu, a, p, \ProcLaw_{(X, \Pi)} \right)
      \end{equation}
      as random elements of $\rootBCM \bigl(\MeasSt(\Psi) \times \tau \times \HKSt(\Psi) \times  \LzeroSPMSTOMSt(\Psi) \bigr)$.
    \item \label{thm item: PCAF/STOM rdm result. 3}
      Under Assumption~\ref{assum: PCAF rdm assumption}, it holds that  
      \begin{equation}
        \left( S_n, d_{S_n}, \rho_n, \mu_n, a_n, p_n, \ProcLaw_{(X_n, \Pi_n)} \right) 
          \xrightarrow{\mathrm{d}} \left( S, d_S, \rho, \mu, a, p, \ProcLaw_{(X, \Pi)} \right)
      \end{equation}
      as random elements of $\rootBCM \bigl(\MeasSt(\Psi) \times \tau \times \HKSt(\Psi) \times  \LzeroSPMvSTOMSt(\Psi) \bigr)$.
  \end{enumerate}
\end{thm}

Before proving the above theorem, we justify the measurability of the random objects appearing in its statements.

\begin{lem} \label{lem: PCAF rdm measurability}
  All the tuples appearing in the assertions of Theorem~\ref{thm: PCAF/STOM rdm result} 
  are random elements (that is, measurable).
\end{lem}

\begin{proof}
  We prove only the measurability of 
  \begin{equation}
    \left( S, d_S, \rho, \mu, a, p, \ProcLaw_{(X, A)} \right),
  \end{equation}
  as an element of $\rootBCM \bigl(\MeasSt(\Psi) \times \tau \times \HKSt(\Psi) \times \LzeroSPMPCAFSt(\Psi) \bigr)$,
  since the measurability of the remaining tuples can be shown in the same way.
  The idea is, again, to use the approximation results established in 
  Section~\ref{sec: approx of PCAFs and STOMs}.

  For each $\delta > 0$ and $R > 1$, we set 
  \begin{equation}
    A^{(\delta, R)} \coloneqq \apPCAF^{(\delta, R)}(p, \mu, X),
    \qquad 
    A^{(*, R)} \coloneqq \apPCAF^{(*, R)}(X, A),
  \end{equation}
  where the approximation maps are recalled from 
  Definitions~\ref{dfn: restriction of PCAF} and \ref{dfn: apPCAF}.
  By the measurability of the map in \eqref{eq: measurability assumption for SPM}
  and the measurability of $\apPCAF^{(\delta, R)}$ 
  established in Lemma~\ref{lem: apPCAF continuity},
  we deduce that 
  \begin{equation}
    \left( S, d_S, \rho, \mu, a, p, \ProcLaw_{(X, A^{(\delta, R)})} \right)
  \end{equation}
  is a random element.
  By Theorems~\ref{prop: delta map approx for local Kato} 
  and \ref{thm: R map approx for local Kato}, we have
  \begin{gather}
    \left( S, d_S, \rho, \mu, a, p, \ProcLaw_{(X, A^{(\delta, R)})} \right)
    \xrightarrow[\delta \to 0]{\mathrm{a.s.}}
    \left( S, d_S, \rho, \mu, a, p, \ProcLaw_{(X, A^{(*, R)})} \right),\\
    \left( S, d_S, \rho, \mu, a, p, \ProcLaw_{(X, A^{(*, R)})} \right)
    \xrightarrow[R \to \infty]{\mathrm{a.s.}}
    \left( S, d_S, \rho, \mu, a, p, \ProcLaw_{(X, A)} \right).
  \end{gather}
  Hence $\left( S, d_S, \rho, \mu, a, p, \ProcLaw_{(X, A)} \right)$ is a random element.
\end{proof}

In the proof of Theorem~\ref{thm: PCAF/STOM rdm result},
the following tightness of potentials plays a crucial role.

\begin{lem} \label{lem: PCAF rdm tightness of potentials}
  Under Assumption~\ref{assum: PCAF rdm assumption}\ref{assum item: 2. PCAF rdm assumption},
  it holds that 
  \begin{equation}
    \lim_{\alpha \to \infty}
    \limsup_{n \to \infty}
    \mathbf{P}_n\!
    \left(
      \bigl\| \Potential_{p_n}^1 \mu_n^{(r)} \bigr\|_\infty
      > \alpha
    \right)
    = 0,
    \quad \forall r > 0.
  \end{equation}
\end{lem}

\begin{proof}
  By \cite[Equation~(2.7)]{Mori_21_Kato_Sobolev}
  (with the parameters $(p, \alpha, \beta)$ therein set to $(1, 1, \alpha)$), 
  we deduce that 
  \begin{equation}
    \bigl\| \Potential_{p_n}^1 \mu_n^{(r)} \bigr\|_\infty 
    \leq 
    \alpha \bigl\| \Potential_{p_n}^\alpha \mu_n^{(r)} \bigr\|_\infty,
    \quad 
    \forall \alpha > 1.
  \end{equation}
  It follows that 
  \begin{equation}
    \mathbf{P}_n\!
    \left(
      \bigl\| \Potential_{p_n}^1 \mu_n^{(r)} \bigr\|_\infty
      > \alpha/2
    \right)
    \leq 
    \mathbf{P}_n\!
    \left(
      \bigl\| \Potential_{p_n}^\alpha \mu_n^{(r)} \bigr\|_\infty
      > 1/2
    \right)
    \leq   
    2\mathbf{E}_n\!
    \left[
      \bigl\| \Potential_{p_n}^\alpha \mu_n^{(r)} \bigr\|_\infty \wedge 1
    \right],
  \end{equation}
  where we use Markov's inequality to obtain the last inequality.
  Hence, Assumption~\ref{assum: PCAF rdm assumption}\ref{assum item: 2. PCAF rdm assumption} yields the desired result.
\end{proof}

We are now ready to prove Theorem~\ref{thm: PCAF/STOM rdm result}\ref{thm item: PCAF/STOM rdm result. 1}.
Before presenting the proof, we briefly outline the main idea.
The basic strategy is the same as in the deterministic case discussed in the previous subsection: 
we propagate the convergence of the processes 
to that of the joint laws with the approximated STOMs 
by using the continuity of the approximation maps 
(Lemma~\ref{lem: apSTOM continuity}\ref{lem item: apSTOM continuity. 2}).
However, the continuity of these maps requires the uniform boundedness 
of the $1$-potentials of the smooth measures.
This difficulty is overcome by Lemma~\ref{lem: PCAF rdm tightness of potentials}.
Specifically, by applying the Skorohod representation theorem,
we show that for any subsequence, one can further extract a subsubsequence
along which both the convergence of the processes 
and the uniform boundedness of the $1$-potentials 
hold almost surely on a suitably constructed probability space.
This allows us to apply the continuity of the approximation maps appropriately,
so that the approximation argument proceeds in the same manner 
as in the deterministic case.

\begin{proof} [{Proof of Theorem~\ref{thm: PCAF/STOM rdm result}}\ref{thm item: PCAF/STOM rdm result. 1}]
  We use the notation introduced in the proof of Theorem~\ref{thm: PCAF/STOM dtm result}.
  We first note that 
  Lemma~\ref{lem: PCAF rdm tightness of potentials} implies that $\{\|\Potential^1_{p_n} \mu_n^{(r)}\|_\infty\}_{n \geq 1}$ is tight for each $r > 0$.
  From this and Assumption~\ref{assum: PCAF rdm assumption}\ref{assum item: 1. PCAF rdm assumption},
  it follows that the family of random elements
  \begin{equation} \label{thm pr: PCAF rdm result. 1}
    \left(
      S_n, \rho_n, \mu_n, a_n, p_n,
      \mathcal{L}_{X_n}, \bigl(\|\Potential^1_{p_n} \mu_n^{(r)}\|_\infty\bigr)_{r \in \NN}
    \right),
    \quad 
    n \geq 1,
  \end{equation}
  is tight in $\rootBCM \bigl(\MeasSt(\Psi) \times \tau \times \HKSt(\Psi) \times  \LzeroSPMSt(\Psi) \bigr) \times \RNp^\NN$. 
  Fix an arbitrary subsequence $(n_k)_{k \geq 1}$ of $\NN$. 
  By taking a further subsequence and relabeling for simplicity, 
  we can assume that
  \begin{equation}  \label{thm pr: PCAF rdm result. 2}
    \left(
      S_{n_k}, \rho_{n_k}, \mu_{n_k}, a_{n_k}, p_{n_k},
      \mathcal{L}_{X_{n_k}}, \bigl(\|\Potential^1_{p_{n_k}} \mu_{n_k}^{(r)}\|_\infty\bigr)_{r \in \NN}
    \right)
    \xlongrightarrow[n \to \infty]{\mathrm{d}}
    \left(
      S, \rho, \mu, a, p,
      \mathcal{L}_X, ( \xi_r )_{r \in \NN}
    \right),
  \end{equation}
  where $(\xi_r)_{r \in \NN}$ is a certain collection of random variables.
  By the Skorohod representation theorem, we may assume that the above convergence holds almost surely
  in some probability space with probability measure $\mathbb{P}$.
  Now, we fix a realization such that the above convergence holds.
  It is then the case that 
  \begin{equation}
    \sup_{k \geq 1} \|\Potential^1_{p_{n_k}} \mu_{n_k}^{(r)}\|_\infty < \infty,
    \quad \forall r > 0.
  \end{equation}
  Thus, similarly to \eqref{thm pr: PCAF dtm result. 3},
  we deduce that, for each $\delta > 0$ and $R > 1$.
  \begin{equation}  \label{thm pr: PCAF rdm result. 3}
    \cX_{n_k}^{(\delta, R)} \xrightarrow[n \to \infty]{\mathrm{a.s.}} \cX^{(\delta, R)}
    \quad \text{in}\quad \rootBCM(\tau_1).
  \end{equation}
  In particular, the above convergence holds in distribution.
  Since the sequence $(n_k)_{k \geq 1}$ was arbitrary,
  it follows that 
  \begin{equation}  \label{thm pr: PCAF rdm result. 4}
    \cX_n^{(\delta, R)} \xrightarrow[n \to \infty]{\mathrm{d}} \cX^{(\delta, R)}
    \quad \text{in}\quad \rootBCM(\tau_1).
  \end{equation}

  In a similar manner as we checked \eqref{thm pr: PCAF dtm result. 9},
  using \eqref{thm pr: PCAF dtm result. 8} and Assumption~\ref{assum: PCAF rdm assumption}\ref{assum item: 2. PCAF rdm assumption},
  we obtain that 
  \begin{equation}
    \lim_{\delta \to 0}
    \limsup_{n \to \infty}
    \mathbf{E}_n\!
    \left[
          \GFMet^{\tau_1}\!\left(\cX_n^{(\delta, R)}, \cX_n^{(*,R)}\right) \wedge 1
    \right]
    = 0.
  \end{equation}
  This, combined with \eqref{thm pr: PCAF rdm result. 4}, yields that 
  \begin{equation}  \label{thm pr: PCAF rdm result. 5}
    \cX_n^{(*, R)} \xrightarrow[n \to \infty]{\mathrm{d}} \cX^{(*, R)}
    \quad \text{in}\quad \rootBCM(\tau_1).
  \end{equation}
  Moreover, similarly to \eqref{thm pr: PCAF dtm result. 11},
  using \eqref{thm pr: PCAF dtm result. 10} and Assumption~\ref{assum: PCAF dtm assumption. spatial tight},
  we obtain that 
  \begin{equation}
    \lim_{R \to \infty}
    \limsup_{n \to \infty}
    \mathbf{E}_n\!
    \left[
          \GFMet^{\tau_1}\!\left(\cX_n^{(*, R)}, \cX_n \right) \wedge 1
    \right]
    = 0.
  \end{equation}
  This, combined with \eqref{thm pr: PCAF dtm result. 6} and \eqref{thm pr: PCAF rdm result. 5}, yields the desired result.
\end{proof}

We next prove Theorem~\ref{thm: PCAF/STOM rdm result}\ref{thm item: PCAF/STOM rdm result. 2}, 
which is the STOM counterpart of the above result.
We begin by taking a coupling under which the joint laws of the processes and the PCAFs converge 
and the corresponding $1$-potentials are uniformly bounded,
as guaranteed by the previous result.
Under this coupling, the distributional convergence of the PCAFs 
implies the convergence of their moments.
Consequently, the uniform local Kato-type condition required in the deterministic case 
(Assumption~\ref{assum: PCAF/STOM dtm assumption}\ref{assum item: 2. PCAF dtm assumption})
holds almost surely.
This allows us to apply exactly the same argument as in the deterministic setting,
which yields the desired convergence of the STOMs.

\begin{proof} [{Proof of Theorem~\ref{thm: PCAF/STOM rdm result}\ref{thm item: PCAF/STOM rdm result. 2}}]
  We use the same notation as in the proof of Theorem~\ref{thm: PCAF/STOM dtm result}.
  Note that Theorem~\ref{thm: PCAF/STOM rdm result}\ref{thm item: PCAF/STOM rdm result. 1} has already been established,
  so that all the results obtained in its proof are available here as well.
  Fix $R > 1$.
  By the same argument as at the beginning of the previous proof,
  we deduce from Lemma~\ref{lem: PCAF rdm tightness of potentials} and \eqref{thm pr: PCAF rdm result. 5}
  that the family of random elements
  \begin{equation} \label{thm pr: STOM rdm result. 1}
    \left(
      S_n, \rho_n, \mu_n, a_n, p_n,
      \mathcal{L}_{(X_n, A_n^{(R+1)})},
      \bigl(\|\Potential^1_{p_n} \mu_n^{(r)}\|_\infty\bigr)_{r \in \NN}
    \right),
    \quad 
    n \geq 1,
  \end{equation}
  is tight in 
  $\rootBCM \bigl(\MeasSt(\Psi) \times \tau \times \HKSt(\Psi) \times  \LzeroSPMPCAFSt(\Psi) \bigr) \times \RNp^\NN$. 
  Fix an arbitrary subsequence $(n_k)_{k \geq 1}$ of $\NN$. 
  By taking a further subsequence and relabeling it for simplicity, 
  we may assume that
  \begin{align}  \label{thm pr: STOM rdm result. 2}
    &\left(
      S_{n_k}, \rho_{n_k}, \mu_{n_k}, a_{n_k}, p_{n_k},
      \mathcal{L}_{(X_{n_k}, A_{n_k}^{(R+1)})},
      \bigl(\|\Potential^1_{p_{n_k}} \mu_{n_k}^{(r)}\|_\infty\bigr)_{r \in \NN}
    \right)\\
    &\xlongrightarrow[k \to \infty]{\mathrm{d}}
    \left(
      S, \rho, \mu, a, p,
      \mathcal{L}_{(X, A^{(R+1)})}, (\xi_r)_{r \in \NN}
    \right),
  \end{align}
  where $(\xi_r)_{r \in \NN}$ is a certain collection of random variables.
  By the Skorohod representation theorem, we may further assume that the above convergence holds almost surely
  on some probability space with probability measure $\mathbb{P}$.
  We hence fix a realization such that this convergence holds.
  By the convergence of the $1$-potentials, we have 
  \begin{equation} \label{thm pr: STOM rdm result. 2.5}
    \sup_{k \geq 1} \|\Potential^1_{p_{n_k}} \mu_{n_k}^{(r)}\|_\infty < \infty.
  \end{equation}
  We will show that, on the fixed realization,
  \begin{equation}  \label{thm pr: STOM rdm result. 3}
    \lim_{t \to 0}
    \limsup_{k \to \infty}
    \sup_{y \in \Psi(S_{n_k})}
    \int_0^t \int_{\Psi(S_{n_k})} p_{n_k}(s, y, z)\, \mu_{n_k}^{(R)}(dz)\, ds
    = 0,
  \end{equation}
  which corresponds to \eqref{thm pr: STOM dtm result. 2.2} along the subsequence $(n_k)_{k \geq 1}$.

  By Theorem~\ref{thm: conv in M(tau)},
  there exists a rooted $\bcmAB$ space $(M, \rho_M)$
  into which all the metric spaces $S_n$ and $S$ can be isometrically embedded 
  so that \ref{cond: 1. PCAF embedding}--\ref{cond: 5. PCAF embedding} and the following condition hold:
  \begin{enumerate} [resume* = PCAF embedding]
    \item \label{cond: 6. PCAF embedding} 
      $\ProcLaw_{(X_{n_k}, A_{n_k}^{(*, R+1)})} \to \ProcLaw_{(X, A^{(*,R+1)})}$ 
      in $\hatC\!\left( \Psi(M), \Prob\bigl( L^0(\RNp, \Psi(M)) \times \upC(\RNp, \RNp) \bigr)\right)$.
  \end{enumerate}
  Fix $t \geq 0$,
  and take elements $x_n \in \Psi(S_n)$ converging to $x \in \Psi(S)$ in $\Psi(M)$.
  By \ref{cond: 6. PCAF embedding}, 
  \begin{equation} \label{thm pr: STOM rdm result. 4}
    P_{n_k}^{x_{n_k}}(A_{n_k}^{(*, R+1)}(t) \in \cdot) 
    \xrightarrow[k \to \infty]{\mathrm{d}} 
    P^x(A^{(*,R+1)}(t) \in \cdot).
  \end{equation}
  Since $\tilde{\mu}_{n_k}^{(R+1)}$ is the smooth measure associated with $A_{n_k}^{(*,R+1)}$,
  as in the proof of \cite[Theorem~1.6]{Noda_pre_Continuity},
  one readily verifies that 
  \begin{equation} \label{thm pr: STOM rdm result. 5}
    E_{n_k}^{x_{n_k}}\!\left[ A_{n_{k}}^{(*, R+1)}(t)^2 \right] 
    \leq 2 e^{2t} \bigl\| \Potential_{p_{n_k}}^1 \tilde{\mu}_{n_k}^{(R+1)} \bigr\|_\infty.
  \end{equation}
  The right-hand side is bounded uniformly in $k$ by \eqref{thm pr: STOM rdm result. 2.5}.
  Hence, the distributional convergence in \eqref{thm pr: STOM rdm result. 4}
  implies the convergence of the corresponding first moments.
  In particular, we obtain 
  \begin{equation} \label{thm pr: STOM rdm result. 6}
    \lim_{k \to \infty}
    E_{n_k}^{x_{n_k}}\!\left[ A_{n_{k}}^{(*, R+1)}(t) \right] 
    = 
    E^x[A^{(*,R+1)}(t)].
  \end{equation}
  Since the sequence $(x_{n_k})_{k \geq 1}$ was arbitrary,
  it follows that 
  \begin{equation} \label{thm pr: STOM rdm result. 7}
    E_{n_k}^{\cdot}\!\left[ A_{n_{k}}^{(*, R+1)}(t) \right]
    \xrightarrow[k \to \infty]{}
    E^{\cdot}[A^{(*,R+1)}(t)]
    \quad \text{in} \quad  
    \hatC\!\left(\Psi(M), \RNp\right).
  \end{equation}
  In particular, 
  \begin{equation} \label{thm pr: STOM rdm result. 8}
    \lim_{k \to \infty}
    \sup_{y \in \Psi(S_{n_k}^{(R+1)})}
    E_{n_k}^y\!\left[ A_{n_{k}}^{(*, R+1)}(t) \right] 
    = 
    \sup_{y \in \Psi(S^{R+1})}
    E^y[A^{(*,R+1)}(t)].
  \end{equation}
  By Proposition~\ref{prop: moment formula for STOM}, this is equivalent to 
  \begin{align}
    &\lim_{k \to \infty}
    \sup_{y \in \Psi(S_{n_k}^{(R+1)})}
    \int_0^t \int_{\Psi(S_{n_k})} p_{n_k}(s, y, z)\, \tilde{\mu}_{n_k}^{(R+1)}(dz)\, ds\\
    &= 
    \sup_{y \in \Psi(S^{R+1})}
    \int_0^t \int_{\Psi(S)} p(s, y, z)\, \tilde{\mu}^{(R+1)}(dz)\, ds.
    \label{thm pr: STOM rdm result. 9}
  \end{align}
  Since the limiting measure $\mu$ belongs to the local Kato class of $X$ by Lemma~\ref{lem: PCAF/STOM rdm limiting meas is smooth},
  letting $t \to 0$ in the above equation yields 
  \begin{equation} \label{thm pr: STOM rdm result. 10}
    \lim_{t \to 0}
    \lim_{k \to \infty}
    \sup_{y \in \Psi(S_{n_k}^{(R+1)})}
    \int_0^t \int_{\Psi(S_{n_k})} p_{n_k}(s, y, z)\, \tilde{\mu}_{n_k}^{(R+1)}(dz)\, ds
    = 0.
  \end{equation}
  From this, Proposition~\ref{prop: Potential and hk behavior} and the inequality $\mu^{(R)} \leq \tilde{\mu}^{(R+1)}$
  together yield \eqref{thm pr: STOM rdm result. 3}.

  Using \eqref{thm pr: STOM rdm result. 3},
  we can follow the argument that led to \eqref{thm pr: STOM dtm result. 2.1},
  and conclude that, on the fixed realization,
  \begin{equation} \label{thm pr: STOM rdm result. 11}
    \cY_{n_k}^{(*, R)} \xrightarrow[k \to \infty]{} \cY^{(*,R)}
    \quad \text{in} \quad \rootBCM(\tau_2).
  \end{equation}
  In particular, the above convergence holds in distribution.
  Since the subsequence $(n_k)_{k \geq 1}$ was arbitrary,
  we obtain 
  \begin{equation} \label{thm pr: STOM rdm result. 12}
    \cY_n^{(*, R)} \xrightarrow[n \to \infty]{\mathrm{d}} \cY^{(*,R)}
    \quad \text{in} \quad \rootBCM(\tau_2).
  \end{equation}

  Finally, in the same way as in \eqref{thm pr: STOM dtm result. 12},
  using \eqref{thm pr: PCAF dtm result. 11} and 
  Assumption~\ref{assum: PCAF rdm assumption. spatial tight},
  we obtain 
  \begin{equation}
    \lim_{R \to \infty}
    \limsup_{n \to \infty}
    \mathbf{E}_n\!
    \left[
      \GFMet^{\tau_2}\!\left(\cY_n^{(*, R)}, \cY_n \right) \wedge 1
    \right]
    = 0.
  \end{equation}
  Combining this with \eqref{thm pr: STOM dtm result. 2} and 
  \eqref{thm pr: STOM rdm result. 12} completes the proof.
\end{proof}

We complete the proof of Theorem~\ref{thm: PCAF/STOM rdm result} by establishing \ref{thm item: PCAF/STOM rdm result. 3}.

\begin{proof} [{Proof of Theorem~\ref{thm: PCAF/STOM rdm result}\ref{thm item: PCAF/STOM rdm result. 3}}]
  This is established in the same way as that in the deterministic case.
  By \eqref{thm pr: STOM rdm result. 12}, we have 
  \begin{equation} \label{thm pr: STOM rdm result. 18}
    \cY_n^{(*, R)} \xrightarrow[n \to \infty]{\mathrm{d}} \cY^{(*,R)}
    \quad \text{in} \quad \rootBCM(\tau_3),
  \end{equation}
  which corresponds to \eqref{thm pr: v-STOM dtm result. 2}.
  From \eqref{thm pr: v-STOM dtm result. 3.1}, it follows that 
  \begin{equation}
    \lim_{R \to \infty}
    \limsup_{n \to \infty}
    \mathbf{E}_n\!
    \left[
      \GFMet^{\tau_3}\!\left(\cY_n^{(*, R)}, \cY_n \right) \wedge 1
    \right]
    = 0.
  \end{equation}
  Now the result is immediate.
\end{proof}


\section{Collision measures and their convergence} \label{sec: Collision measures and their convergence}

In this section, we apply our main results on STOMs to the study of collisions between stochastic processes.
In Section~\ref{sec: collision measure}, we introduce random measures that capture the collision times and locations 
of two given stochastic processes.
We refer to these random measures as \emph{collision measures}.
In Section~\ref{sec: conv of col meas}, 
we then discuss convergence of collision measures of stochastic processes living on different spaces.

\subsection{Collision measures} \label{sec: collision measure}

Throughout this subsection, we fix a locally compact separable metrizable topological space $S$.
For two given processes $X^1$ and $X^2$ on $S$,
we aim to study the locations and times of their collisions,
that is, the set of pairs $(x,t)$ such that $X^1_t = X^2_t = x$.
In this subsection, we introduce random measures that capture such collisions.
They are defined using the idea of STOMs; see Definition~\ref{dfn: col meas} below.
Although we restrict attention to the case where $X^1$ and $X^2$ are independent,
the same construction of STOMs can be extended to non-independent processes,
as briefly discussed in Remark~\ref{rem: collision of correlated proc} below.

For $i \in \{1, 2\}$, 
we let 
\begin{equation}
  X^i = (\Omega_i, \sigalg_i, (X^i_{t})_{t \in [0, \infty]}, (P^x_i)_{x \in S_{\Delta}}, (\theta^i_t)_{t \in [0,\infty]})
\end{equation}
be a standard process satisfying the DAC condition (Assumption~\ref{assum: dual hypothesis}).
Let $(S \times S) \cup \{\Delta\} $ be the one-point compactification of $S \times S$.
For each $\bm{x}=(x_{1}, x_{2}) \in S \times S$,
we define $\hat{P}^{\bm{x}} \coloneqq P_{1}^{x_{1}} \otimes P_{2}^{x_{2}}$,
which is a probability measure on the product measurable space 
$(\hat{\Omega}, \hat{\sigalg}) \coloneqq (\Omega_1 \times \Omega_2, \sigalg_1 \otimes \sigalg_2)$.
Also,
we define $\hat{P}^{\Delta} \coloneqq P_{1}^{\Delta} \otimes P_{2}^{\Delta}$.
For each $\omega = (\omega_{1}, \omega_{2}) \in \hat{\Omega}$ and $t \in [0, \infty]$,
we set 
\begin{equation}
  \hat{X}_t(\omega) 
  \coloneqq 
  \begin{cases}
    (X^1_t(\omega_{1}), X^2_t(\omega_{2})), & t < \zeta_1(\omega_1) \wedge \zeta_2(\omega_2),\\
    \Delta, & t \geq \zeta_1(\omega_1) \wedge \zeta_2(\omega_2),
  \end{cases}
\end{equation}
where $\zeta_i$ denotes the lifetime of $X^i$ for each $i =1,2$.
Finally, we define, for each $\omega = (\omega_{1}, \omega_{2}) \in \hat{\Omega}$ and $t \in [0,\infty]$,
\begin{equation}
  \hat{\theta}_t(\omega) \coloneqq (\theta^1_t(\omega_1), \theta^2_t(\omega_2)).
\end{equation}
Then by Theorem~\ref{thm: product is standard} and Proposition~\ref{prop: product satisfies dual hypo}, 
\begin{equation}
  \hat{X} 
  \coloneqq 
  \left( \hat{\Omega}, \hat{\sigalg}, (\hat{X}_{t})_{t \in [0, \infty]}, 
    (\hat{P})_{x \in (S \times S) \cup \{\Delta\}}, (\hat{\theta}_t)_{t \in [0,\infty]} \right)
\end{equation}
is a standard process on $S \times S$ satisfying the DAC condition.
Let $m^i$ and $p^i$ denote the reference measure and the heat kernel of $X^i$, respectively, for $i \in \{1,2\}$. 
Then a reference measure of $\hat{X}$ is given by $m^1 \otimes m^2$, 
and the corresponding heat kernel is
\begin{equation} \label{eq: hk of product process}
  \hat{p}(t, (x_1, x_2), (y_1, y_2)) = p^1(t, x_1, y_1)\, p^2(t, x_2, y_2),
\end{equation}
with the convention $0 \cdot \infty \coloneqq 0$.
(See Proposition~\ref{prop: product satisfies dual hypo} for these facts.)

\begin{dfn} \label{dfn: product proc}
  We refer to the above-defined process $\hat{X}$ as the \emph{product process} of $X^1$ and $X^2$.
\end{dfn}

To study the collisions of $X^1$ and $X^2$ in $S$,
we embed $S$ into the diagonal line in $S \times S$ by the following diagonal map.

\begin{dfn}[{The diagonal map}]
  We define the \emph{diagonal map} $\diag_S: S \to S \times S$ by $\diag(x) \coloneqq (x, x)$. 
  When the context is clear, we simply write $\diag = \diag_S$.  
\end{dfn}
 
The collision of $X^1$ and $X^2$ can be rephrased as the event of $\hat{X}$ hitting the diagonal line.  
Specifically, for any $t \geq 0$ and a subset $B \subseteq S$, we have that
\begin{equation}
  X^1_t = X^2_t \in B  
  \quad \Leftrightarrow \quad  
  \hat{X}_t \in \diag(B).
\end{equation}
With this background,
we introduce collision measures below.

Fix a Radon measure on $\mu$ on $S$.
We write $\diagMeas{\mu} \coloneqq \mu \circ \diag^{-1}$,  
i.e., $\diagMeas{\mu}$ is the pushforward measure of $\mu$ by $\diag$.
Assume that $\diagMeas{\mu}$ is a smooth measure of $\hat{X}$.
Let $(C_t)_{t \geq 0}$ and $\hat{\Pi}$ be the associated PCAF and STOMs of $\hat{X}$, respectively.
By Corollary~\ref{cor: support of STOM}, $\hat{\Pi}$ is supported on $\diag(S) \times \RNp$
and formally speaking, for each Borel subset $E \subseteq S$ and $t \geq 0$,
\begin{equation}
  \hat{\Pi}(\diag(E) \times [0,t]) = \text{``the total collision time of $X^1$ and $X^2$ in $E$ up to $t$''.}
\end{equation}
We thus define a random measure $\Pi$ on $S \times [0, \infty)$ by 
\begin{equation}
  \Pi \coloneqq \hat{\Pi} \circ (\diag \times \id_{\RNp}).
\end{equation}
(NB. Since $\diag \times \id_{\RNp}$ is injective, $\Pi$ is indeed a measure on $S \times \RNp$.)
By Corollary~\ref{cor: support of STOM}, we have 
\begin{equation}
  \Pi(\cdot) = \int_0^\infty \mathbf{1}_{(X^1_t, t)}(\cdot)\, dC_t = \int_0^\infty \mathbf{1}_{(X^2_t, t)}(\cdot)\, dC_t
\end{equation}
as random measures on $S \times [0, \infty)$, $\hat{P}^{\bm{x}}$-a.s., for all $\bm{x} \in S \times S$.

\begin{dfn} [{Collision measure}]\label{dfn: col meas}
  We refer to the random measure $\Pi$ as the \emph{collision measure of $X^1$ and $X^2$ 
  associated with $\mu$},
  and we call $\mu$ the \emph{weighting measure}.
\end{dfn}

\begin{rem} \label{rem: coincidence with Nguyen}
  As mentioned in Section~\ref{sec: Application to the study of collisions},
  Nguyen~\cite{Nguyen_23_Collision} obtained the scaling limit 
  of the point process consisting of collision sites and collision times 
  of two independent simple random walks on $\mathbb{Z}$.
  Although the limiting random measure in his construction is characterized only through moment identities,
  it can now be realized as a collision measure within our framework.
  Assume that $X^1$ is a standard one-dimensional Brownian motion on $\RN$, 
  and $X^2$ is an independent copy.
  Then, the collision measure of $X^1$ and $X^2$ associated with $2 \cdot \Leb$
  coincides with the limiting random measure obtained by Nguyen~\cite{Nguyen_23_Collision}. 
  Indeed, using his characterization~\cite[Equation~(1.2)]{Nguyen_23_Collision} 
  together with our moment formula for STOMs (Proposition~\ref{prop: moment formula for STOM}),
  one readily verifies that the two random measures have the same distribution
  (cf.\ \cite[Theorem~2.2]{Kallenberg_17_Random}).
\end{rem}

\begin{rem} \label{rem: more comment about Nguyen}
  We also note that Nguyen~\cite{Nguyen_23_Collision} treated the general case of $k \ge 3$ 
  independent simple random walks $S^1, \ldots, S^k$ on $\mathbb{Z}$, 
  and considered the point process consists of all $(t, x) \in \mathbb{N} \times \mathbb{Z}$ 
  such that $S^i_t = S^j_t = x$ for some distinct $i, j$.
  The scaling limit of this point process obtained in his work 
  can also be interpreted within our framework.
  Let $X^1, \ldots, X^k$ be independent one-dimensional Brownian motions on $\RN$.
  Then the limiting random measure should coincide with the sum of the pairwise collision measures 
  $\sum_{1 \le i < j \le k}\Pi_{i,j}$ of $X^i$ and $X^j$, associated with a suitable multiple of the Lebesgue measure.
  Indeed, as simultaneous collisions of three or more walks vanish in the scaling limit~\cite[Proof of Theorem~1.3]{Nguyen_23_Collision}, 
  only binary collisions of Brownian motions contribute in the limit.
\end{rem}

By Theorem~\ref{thm: delta map STOM approx for Kato}, 
under the weak $J_1$-continuity condition (Assumption~\ref{assum: weak J_1 continuity}) on $X$, 
we have the continuity of the joint law $\ProcLaw_{(X, \hat{\Pi})}$ with respect to the starting point.
It is not difficult to verify that this continuity is also inherited by the collision measure $\Pi$.
However, since $\Pi$ is not a pushforward of $\hat{\Pi}$, 
this does not follow directly from the usual continuous mapping theorem.
The following lemma fills this minor gap.
A more general version of the result can be found in Appendix~\ref{appendix: An inverse version of the continuous mapping theorem}.

\begin{lem} \label{lem: continuity from STOM to col}
  Let $(\Sigma_n)_{n \geq 1}$ be a sequence converging to $\Sigma$ in $\STOMMeas(S \times S \times \RNp)$.
  If $\supp(\Sigma_n) \subseteq \diag(S) \times \RNp$, 
  then $\Sigma_n \circ (\diag \times \id_{\RNp}) \to \Sigma \circ (\diag \times \id_{\RNp})$ 
  in $\STOMMeas(S \times \RNp)$.
\end{lem}

\begin{proof}
  By Proposition~\ref{ap. prop: inverse conti map thm}, we obtain that 
  $\Sigma_n \circ (\diag \times \id_{\RNp}) \to \Sigma \circ (\diag \times \id_{\RNp})$ vaguely.
  Moreover, by Proposition~\ref{prop: precompact in stom meas sp}, we deduce that 
  the family $\{\Sigma_n \circ (\diag \times \id_{\RNp})\}_{n \geq 1}$ is precompact in $\STOMMeas(S \times \RNp)$.
  Using the vague convergence, one can readily verify that the limit of any convergent subsequence is unique
  and coincides with $\Sigma \circ (\diag \times \id_{\RNp})$,
  which yields the desired result. 
\end{proof}

Using the above lemma,
we obtain the following result,
which is a counterpart of the continuity assertion in Theorem~\ref{thm: delta map STOM approx for Kato}.

\begin{prop} \label{prop: continuity of Col map}
  Assume that the following conditions hold.
  \begin{enumerate} [label = \textup{(\roman*)}]
    \item \label{prop item: 1. continuity of Col map}
      The measure $\diagMeas{\mu}$ on $S \times S$ belongs to the local Kato class of $\hat{X}$.
    \item \label{prop item: 2. continuity of Col map}
      For each $i \in \{1, 2\}$, the process $X^i$ satisfies the weak $J_1$-continuity condition 
      (Assumption~\ref{assum: weak J_1 continuity}).
  \end{enumerate}
  Then the following map is continuous:
  \begin{equation}
    \ProcLaw_{(\hat{X}, \Pi)} 
    \colon 
    S \times S 
    \longrightarrow 
    \Prob\bigl( D_{J_1}(\RNp, S) \times D_{J_1}(\RNp, S) \times \STOMMeas(S \times \RNp) \bigr).
  \end{equation}
\end{prop}

\begin{proof}
  By Theorem~\ref{thm: R map STOM approx for local Kato},
  the map 
  \begin{equation}
    \ProcLaw_{(\hat{X}, \hat{\Pi})} 
    \colon 
    S \times S 
    \longrightarrow 
    \Prob\bigl( D_{L^0}(\RNp, S) \times D_{L^0}(\RNp, S) \times \STOMMeas(S \times S \times \RNp) \bigr)
  \end{equation}
  is continuous.
  By Lemma~\ref{lem: continuity from STOM to col},
  we then deduce that the following map is also continuous:
  \begin{equation}
    \ProcLaw_{(\hat{X}, \Pi)} 
    \colon 
    S \times S 
    \longrightarrow 
    \Prob\bigl( D_{L^0}(\RNp, S) \times D_{L^0}(\RNp, S) \times \STOMMeas(S \times \RNp) \bigr).
  \end{equation}
  It remains to upgrade the continuity to the $J_1$-Skorohod topology.
  By condition~\ref{prop item: 2. continuity of Col map} 
  and the independence of $X^1$ and $X^2$,
  the map 
  \begin{equation}
    \ProcLaw_{\hat{X}} 
    \colon 
    S \times S 
    \longrightarrow 
    \Prob\bigl( D_{J_1}(\RNp, S) \times D_{J_1}(\RNp, S) \bigr)
  \end{equation}
  is continuous. 
  Combining this with a tightness argument,
  we can upgrade the $L^0$-continuity established above 
  to the $J_1$-continuity.
  This completes the proof.
  (See Appendix~\ref{appendix: Replacement of the L^0 topology by stronger topologies} for details.)
\end{proof}

The following is immediate from Propositions~\ref{prop: Kato meas for discrete space} and \ref{prop: representation of STOM for discrete space}.
 
\begin{prop} \label{prop: representation of collision measure for discrete space}
  Suppose that $S$ is finite or countable, endowed with the discrete topology.
  Then $\diagMeas{\mu}$ belongs to the local Kato class of $X$,
  and the collision measure $\Pi$ associated with $\mu$ can be represented as follows:
  \begin{equation} \label{prop eq: collision process for discrete space}
    \Pi(dx\, dt) 
    = 
    \frac{1}{m^1(\{x\}) m^2(\{x\})} 
    \mathbf{1}_{\{X^1_t = X^2_t = x\}}\,
    \mu(dx)\, dt.
  \end{equation}
\end{prop}

\begin{rem} \label{rem: uniform col meas}
  As seen in Proposition \ref{prop: representation of collision measure for discrete space}, 
  the weighting measure $\mu(dx)$ represents the contribution of collisions at each point $x$. 
  In the setting of that proposition, 
  if we take the pointwise product $m^1 \cdot m^2$ as the weighting measure $\mu$, i.e., $\mu(\{x\}) \coloneqq m^1(\{x\}) m^2(\{x\})$,
  then the associated collision measure assigns equal weight to the contribution of collisions at each point. 
  We refer to this collision measure as the \emph{uniform collision measure}, i.e., it is defined as follows:
  \begin{equation} \label{rem eq: uniform col meas}
    \Pi(dx\, dt) 
    = 
    \sum_{y \in S} 
    \mathbf{1}_{\{X^1_t = X^2_t = y\}}\, \delta_y(dx)\, dt.
  \end{equation}
  In particular, for each $x \in S$ and $t \geq 0$,
  \begin{equation}
    \Pi(\{x\} \times [0,t]) = \int_0^t \mathbf{1}_{\{X^1_s = X^2_s = x\}}\, ds,
  \end{equation}
  which represents the total amount of time up to $t$ 
  during which $X^1$ and $X^2$ collide at $x$.
  We call the associated weighting measure $\mu = m^1 \cdot m^2$ the \emph{canonical weighting measure}.
\end{rem}

\begin{rem} \label{rem: collision of correlated proc}
  Although $X^1$ and $X^2$ are independent in the above framework,
  one can also study collisions of correlated processes.
  Indeed, if a standard process $X = (X^1, X^2)$ on $S \times S$ satisfies the DAC condition (Assumption~\ref{assum: dual hypothesis}),
  then a collision measure of $X^1$ and $X^2$ can be defined through a STOM of $X$.
  However, since research on collisions has primarily focused on independent stochastic processes, 
  we restrict our attention in this paper to this case.
\end{rem}

\subsection{Convergence of collision measures} \label{sec: conv of col meas}

In this subsection,
by applying Theorems~\ref{thm: PCAF/STOM dtm result} and \ref{thm: PCAF/STOM rdm result},
we prove convergence of collision measures of independent stochastic processes living on spaces that vary along a sequence.
The main results of this subsection are found in Theorems~\ref{thm: col dtm result} and \ref{thm: col rdm result} below.
Throughout this subsection, we fix a Polish structure $\tau$.

\medskip  
We first state a result for deterministic spaces,
obtained as an application of Theorem~\ref{thm: PCAF/STOM dtm result}.
For each $n \in \NN$, we suppose that 
\begin{itemize} 
  \item $(S_n, d_{S_n}, \rho_n, \mu_n, a_n)$ is an element of $\rootBCM(\MeasSt \times \tau)$;
  \item $X^1_n$ and $X^2_n$ are standard processes on $S_n$ 
    satisfying the DAC and weak $J_1$-continuity conditions (Assumptions~\ref{assum: dual hypothesis} and \ref{assum: weak J_1 continuity})
    with heat kernels $p^1_n$ and $p^2_n$, respectively;
  \item the Radon measure $\diagMeas{\mu_n}$ on $S_n \times S_n$ is in the local Kato class of the product process $\hat{X}_n$
    of $X^1_n$ and $X^2_n$.
\end{itemize}
We also suppose that 
\begin{itemize} 
  \item $(S, d_S, \rho, \mu, a)$ is an element of $\rootBCM(\MeasSt \times \tau)$;
  \item $X^1$ and $X^2$ are standard processes on $S$ 
    satisfying the DAC and weak $J_1$-continuity conditions, with heat kernels $p^1$ and $p^2$, respectively;
  \item the Radon measure $\diagMeas{\mu}$ on $S \times S$ charges no sets semipolar for the product process $\hat{X}$ of $X^1$ and $X^2$.
\end{itemize}

\begin{rem}
  Recall that the convergence results in Theorems~\ref{thm: PCAF/STOM dtm result} 
  and \ref{thm: PCAF/STOM rdm result} 
  were formulated under the convergence of processes with respect to the $L^0$ topology. 
  Here we instead work with the stronger $J_1$–Skorohod topology. 
  This choice is not essential in itself, 
  but it is made in view of the framework of resistance metric spaces considered later, 
  where convergence in the $J_1$–Skorohod topology arises naturally 
  from the convergence of the underlying spaces.
\end{rem}

Since we employ the $J_1$–Skorohod topology instead of the $L^0$ topology,
we accordingly use the structure $\SPMSt$ defined in \eqref{eq: st for SPM}:
\begin{equation} \label{eq: st for SPM. revisited}
  \SPMSt \coloneqq \hatCSt(\Psi_{\id}, \ProbSt(\SkorohodSt)).
\end{equation}

We now consider the following version of Assumption~\ref{assum: PCAF/STOM dtm assumption}
(see also Remark~\ref{rem: assum for STOM dtm}).

\begin{assum} \label{assum: col dtm assumption} \leavevmode
  \begin{enumerate} [label = \textup{(\roman*)}]
    \item \label{assum item: 1. col dtm assumption}
      It holds that 
      \begin{equation}
        \left( S_n, d_{S_n}, \rho_n, \mu_n, a_n, p^1_n, p^2_n, \ProcLaw_{X^1_n}, \ProcLaw_{X^2_n} \right) 
        \to 
        \left( S, d_S, \rho, \mu, a, p^1, p^2, \ProcLaw_{X^1}, \ProcLaw_{X^2} \right)
      \end{equation}
      in 
      $\rootBCM \bigl(\MeasSt \times \tau \times \HKSt^{\otimes 2} \times \SPMSt^{\otimes 2} \bigr)$.
    \item \label{assum item: 2. col dtm assumption}
      For all $r > 0$, 
      \begin{equation}
        \lim_{\delta \to 0} 
        \limsup_{n \to \infty} 
        \sup_{x_1, x_2 \in S_n^{(r)}}
        \int_0^\delta \int_{S_n^{(r)}} p^1(t, x_1, y)\, p^2(t, x_2, y)\, \mu_n(dy)\, dt 
        = 0.
      \end{equation}
  \end{enumerate}
\end{assum}

Under the above assumptions, 
Lemma~\ref{lem: STOM dtm limiting meas is smooth} implies that 
$\diagMeas{\mu}$ belongs to the local Kato class of $\hat{X}$.
Let $\Pi_n$ (resp.\ $\Pi$) denote the collision measure of $X^1_n$ and $X^2_n$ (resp.\ $X^1$ and $X^2$) 
associated with $\mu_n$ (resp.\ $\mu$).
Below, by using Theorem~\ref{thm: PCAF/STOM dtm result}, 
we establish the joint convergence of the processes and the collision measures.
We use the following structure to formulate the result precisely:
\begin{align} 
  \ColSt 
  &\coloneqq 
    \hatCSt\!\left( \Psi_{\id^2}, \ProbSt\bigl( \SkorohodSt \times \SkorohodSt \times \STOMSt \bigr) \right).
  \label{eq: st for col meas}
\end{align}
In particular, for each $\bcmAB$ space $M$,
\begin{equation}
  \ColSt(M) 
  =
  \hatC\bigl( M \times M, \Prob\bigl( D_{J_1}(\RNp, M) \times D_{J_1}(\RNp, M) \times \STOMMeas(M \times \RNp) \bigr) \bigr),
\end{equation}
where the domain $M \times M$ corresponds to the starting points of the two processes,
and the codomain represents the joint law of the processes together with their associated collision measure.

\begin{thm} \label{thm: col dtm result}
  Under Assumption~\ref{assum: col dtm assumption}, we have
  \begin{equation}
    \left( S_n, d_{S_n}, \rho_n, \mu_n, a_n, p^1_n, p^2_n, \ProcLaw_{(\hat{X}_n, \Pi_n)} \right) 
    \to 
    \left( S, d_S, \rho, \mu, a, p^1, p^2, \ProcLaw_{(\hat{X}, \Pi)} \right)
  \end{equation}
  in 
  $\rootBCM \bigl(\MeasSt \times \tau \times \HKSt^{\otimes 2} \times \ColSt \bigr)$.
\end{thm}

\begin{proof}
  Let $\hat{p}_n$ and $\hat{p}$ be the heat kernels of the product processes $\hat{X}_n$ and $\hat{X}$, respectively, 
  defined in \eqref{eq: hk of product process}.
  By Lemma~\ref{lem: L^0 and product}, Assumption~\ref{assum: col dtm assumption}\ref{assum item: 1. col dtm assumption},
  and the independence of $X^1_n$ and $X^2_n$, we have
  \begin{align}
    &\left( S_n, d_{S_n}, \rho_n, \mu_n, a_n, p^1_n, p^2_n, 
      \ProcLaw_{X^1_n}, \ProcLaw_{X^2_n}, \diagMeas{\mu_n}, \hat{p}_n, \ProcLaw_{\hat{X}_n} \right) \\
    &\to 
    \left( S, d_S, \rho, \mu, a, p^1, p^2, 
      \ProcLaw_{X^1}, \ProcLaw_{X^2}, \diagMeas{\mu}, \hat{p}, \ProcLaw_{\hat{X}} \right)
  \end{align}
  in 
  $\rootBCM \bigl(\MeasSt \times \tau \times \HKSt^{\otimes 2} \times \SPMSt^{\otimes 2} 
    \times \MeasSt(\Psi_{\id^2}) \times \HKSt(\Psi_{\id^2}) \times \LzeroSPMSt(\Psi_{\id^2}) \bigr)$.
  Since the product processes converge with respect to the product $J_1$ topology, 
  it follows immediately, as noted in Remark~\ref{rem: unif spatial tightness}, 
  that the uniform spatial tightness condition (Assumption~\ref{assum: PCAF dtm assumption. spatial tight}) is satisfied for the product processes. 
  Thus, applying 
  Theorem~\ref{thm: PCAF/STOM dtm result}\ref{thm item: PCAF/STOM dtm result. 2} 
  to the sequence of product processes with the spatial transformation $\Psi = \Psi_{\id^2}$, 
  we obtain
  \begin{equation}
    \left( S_n, d_{S_n}, \rho_n, \mu_n, a_n, p^1_n, p^2_n, 
      \ProcLaw_{X^1_n}, \ProcLaw_{X^2_n}, \ProcLaw_{(\hat{X}_n, \hat{\Pi}_n)} \right)
    \to 
    \left( S, d_S, \rho, \mu, a, p^1, p^2, 
      \ProcLaw_{X^1}, \ProcLaw_{X^2}, \ProcLaw_{(\hat{X}, \hat{\Pi})} \right)
  \end{equation}
  in
  $\rootBCM \bigl(\MeasSt \times \tau \times \HKSt^{\otimes 2} \times \SPMSt^{\otimes 2} 
    \times \LzeroSPMSTOMSt(\Psi_{\id^2}) \bigr)$,
  where $\hat{\Pi}_n$ (resp.\ $\hat{\Pi}$) denotes the STOM of $\hat{X}_n$ (resp.\ $\hat{X}$) 
  associated with $\diagMeas{\mu_n}$ (resp.\ $\diagMeas{\mu}$).

  By Theorem~\ref{thm: conv in M(tau)}, we may assume that all rooted $\bcmAB$ spaces $(S_n, \rho_n)$ and $(S, \rho)$ 
  are isometrically embedded into a common $\bcmAB$ space $(M, \rho_M)$ 
  in such a way that 
  \begin{enumerate} [label = \textup{(B\arabic*)}, leftmargin = *, series = col embedding]
    \item \label{cond: 1. col embedding}
      $S_n \to S$ in the Fell topology as closed subsets of $M$;
    \item \label{cond: 2. col embedding} 
      $\rho_n = \rho = \rho_M$ when viewed as elements on $M$;
    \item \label{cond: 3. col embedding} 
      $\mu_n \to \mu$ vaguely as measures on $M$;
    \item \label{cond: 4. col embedding} 
      $a_n \to a$ in $\tau(M)$;
    \item \label{cond: 5. col embedding} 
      $p^i_n \to p^i$ in $\hatC(\RNpp \times M \times M, \RNp)$ for each $i \in \{1, 2\}$;
    \item \label{cond: 6. col embedding} 
      $\ProcLaw_{X^i_n} \to \ProcLaw_{X^i}$ in $\hatC\bigl( M, \Prob(D_{J_1}(\RNp, M)) \bigr)$ for each $i \in \{1, 2\}$;
    \item \label{cond: 7. col embedding} 
      $\ProcLaw_{(\hat{X}_n, \hat{\Pi}_n)} \to \ProcLaw_{(\hat{X}, \hat{\Pi})}$ 
      in $\hatC\bigl( M \times M, \Prob\bigl( L^0(\RNp, M \times M) \times \STOMMeas(M \times M \times \RNp) \bigr) \bigr)$.
  \end{enumerate}
  Applying Lemma~\ref{lem: continuity from STOM to col} to \ref{cond: 7. col embedding},
  we deduce that 
  \begin{equation}
    \ProcLaw_{(\hat{X}_n, \Pi_n)} \to \ProcLaw_{(\hat{X}, \Pi)}
    \quad \text{in} \quad 
    \hatC\bigl( M \times M, \Prob\bigl( L^0(\RNp, M \times M) \times \STOMMeas(M \times \RNp) \bigr) \bigr).
  \end{equation}
  It remains to upgrade the $L^0$ topology to the $J_1$–Skorohod topology.
  As in the proof of Proposition~\ref{prop: continuity of Col map},
  this follows directly from \ref{cond: 6. col embedding} together with a tightness argument 
  (see Appendix~\ref{appendix: Replacement of the L^0 topology by stronger topologies} for details).
  Consequently, we obtain 
  \begin{equation}
    \ProcLaw_{(\hat{X}_n, \Pi_n)} \to \ProcLaw_{(\hat{X}, \Pi)}
    \quad \text{in} \quad 
    \hatC\bigl( M \times M, \Prob\bigl( D_{J_1}(\RNp, M) \times D_{J_1}(\RNp, M) \times \STOMMeas(M \times \RNp) \bigr) \bigr),
  \end{equation}
  which completes the proof.
\end{proof}

\begin{rem} \label{rem: L^0 is good for col}
  As the above proof illustrates, 
  the adoption of the $L^0$ topology is particularly convenient in the study of collision measures.
  In the proof, we applied the convergence theorem for STOMs 
  (Theorem~\ref{thm: PCAF/STOM dtm result})
  to the product processes $\hat{X}_n = (X^1_n, X^2_n)$.
  These processes converge only in the product of the $J_1$ topologies, that is, in $D_{J_1}(\RNp, M) \times D_{J_1}(\RNp, M)$,
  rather than in $D_{J_1}(\RNp, M \times M)$ itself.
  Therefore, if the main STOM convergence theorems were formulated under the $J_1$-Skorohod topology, 
  they could not be applied to this product setting in a straightforward manner. 
  Of course, one could still modify the proofs to make them work under the $J_1$ topology, 
  but this would require additional technical effort.
  The $L^0$ topology, on the other hand, 
  is stable under product constructions and thus provides a cleaner and more flexible framework 
  for treating collision measures.
\end{rem}

\medskip
We next provide a version of Theorem~\ref{thm: col dtm result} for random spaces.
For each $n \in \NN$, we let $(\Omega_n, \mathcal{G}_n, \mathbf{P}_n)$ be a complete probability space,
and assume that, for $\mathbf{P}_n$-a.s.\ $\omega \in \Omega_n$,
\begin{itemize} 
  \item $(S_n^\omega, d_{S_n}^\omega, \rho_n^\omega, \mu_n^\omega, a_n^\omega) \in \rootBCM(\MeasSt \times \tau)$;
  \item $X_n^{1, \omega}$ and $X_n^{2, \omega}$ are standard processes on $S_n^\omega$  
    satisfying the DAC and weak $J_1$-continuity conditions, 
    with heat kernels $p^{1, \omega}_n$ and $p^{2, \omega}_n$, respectively;
  \item the Radon measure $\diagMeas{(\mu_n^\omega)}$ on $S_n^\omega \times S_n^\omega$ belongs to the local Kato class of 
  the product process $\hat{X}^{\omega}_n$ of $X^{1, \omega}_n$ and $X^{2, \omega}_n$.
\end{itemize}
Similarly, we let $(\Omega, \mathcal{G}, \mathbf{P})$ be a complete probability space, 
and assume that, for each $\omega \in \Omega$,
\begin{itemize} 
  \item $(S^\omega, d_S^\omega, \rho^\omega, \mu^\omega, a^\omega) \in \rootBCM(\MeasSt \times \tau)$;
  \item $X^{1, \omega}$ and $X^{2, \omega}$ are standard processes on $S^\omega$  
    satisfying the DAC and weak $J_1$-continuity conditions,
    with heat kernels $p^{1, \omega}$ and $p^{2, \omega}$, respectively;
  \item the Radon measure $\diagMeas{(\mu^\omega)}$ on $S^\omega \times S^\omega$ charges no sets semipolar 
    for the product process $\hat{X}^{\omega}$ of $X^{1, \omega}$ and $X^{2, \omega}$.
\end{itemize}
To discuss the laws of these objects,
we assume that the maps
\begin{gather}
  (\Omega_n, \mathcal{G}_n) \ni \omega \mapsto 
  (S_n^\omega, d_{S_n}^\omega, \rho_n^\omega, \mu_n^\omega, a_n^\omega, p_n^{1, \omega}, p_n^{2, \omega}, \ProcLaw_{X_n^{1,\omega}}, \ProcLaw_{X_n^{2, \omega}}) 
  \in \rootBCM \bigl(\MeasSt \times \tau \times \HKSt^{\otimes 2} \times \SPMSt^{\otimes 2} \bigr),\\
  (\Omega, \mathcal{G}) \ni \omega \mapsto 
  (S^\omega, d_S^\omega, \rho^\omega, \mu^\omega, a^\omega, p^{1, \omega}, p^{2, \omega}, \ProcLaw_{X^{1,\omega}}, \ProcLaw_{X^{2, \omega}}) 
  \in \rootBCM \bigl(\MeasSt \times \tau \times \HKSt^{\otimes 2} \times \SPMSt^{\otimes 2} \bigr)
\end{gather}
are measurable, where the codomain is equipped with the Borel $\sigma$-algebra.
(The measurability is to be understood in the same sense as in Section~\ref{sec: PCAF conv for rdm spaces},
by extending the objects arbitrarily on null sets if necessary.)

We now consider the following version of Assumption~\ref{assum: col dtm assumption}.
As usual, we suppress the dependence on $\omega$ in the notation.

\begin{assum} \label{assum: col rdm assumption} \leavevmode
  \begin{enumerate} [label = \textup{(\roman*)}, series = STOM rdm assumption]
    \item  \label{assum item: 1. col rdm assumption}
      It holds that 
      \begin{equation}
        \left( S_n, d_{S_n}, \rho_n, \mu_n, a_n, p^1_n, p^2_n, \ProcLaw_{X^1_n}, \ProcLaw_{X^2_n} \right) 
        \xrightarrow{\mathrm{d}} 
        \left( S, d_S, \rho, \mu, a, p^1, p^2, \ProcLaw_{X^1}, \ProcLaw_{X^2} \right)
      \end{equation}
      in 
      $\rootBCM \bigl(\MeasSt \times \tau \times \HKSt^{\otimes 2} \times \SPMSt^{\otimes 2} \bigr)$.
    \item  \label{assum item: 2. col rdm assumption}
      For all $r > 0$ , 
      \begin{equation}
        \lim_{\delta \to 0}
        \limsup_{n \to \infty}
        \mathbf{E}_n\!
        \left[
          \left(
            \sup_{x_1, x_2 \in S_n^{(r)}}
            \int_0^\delta \int_{S_n^{(r)}} p^1(t, x_1, y)\, p^2(t, x_2, y)\, \mu_n(dy)\, dt 
          \right)
          \wedge 1
        \right]
        = 0.
      \end{equation}
  \end{enumerate}
\end{assum}

Under the above assumptions,
by Lemma~\ref{lem: PCAF/STOM rdm limiting meas is smooth},
$\diagMeas{\mu}$ is in the local Kato class of $\hat{X}$, $\mathbf{P}$-a.s.
Denote by $\Pi_n$ (resp.\ $\Pi$) the collision measure of $X^1_n$ and $X^2_n$ (resp.\ $X^1$ and $X^2$) associated with $\mu_n$ (resp.\ $\mu$).
We then obtain the following analogue of Theorem~\ref{thm: col dtm result}.
The proof is similar to that of the deterministic case, 
with Theorem~\ref{thm: PCAF/STOM rdm result}\ref{thm item: PCAF/STOM rdm result. 2} applied 
in place of Theorem~\ref{thm: PCAF/STOM dtm result}\ref{thm item: PCAF/STOM dtm result. 2}, 
and is therefore omitted.

\begin{thm} \label{thm: col rdm result}
  Under Assumption~\ref{assum: col rdm assumption}, it holds that
  \begin{equation}
    \left( S_n, d_{S_n}, \rho_n, \mu_n, a_n, p^1_n, p^2_n, \ProcLaw_{(\hat{X}_n, \Pi_n)} \right) 
    \xrightarrow{\mathrm{d}} 
    \left( S, d_S, \rho, \mu, a, p^1, p^2, \ProcLaw_{(\hat{X}, \Pi)} \right)
  \end{equation}
  in 
  $\rootBCM(\MeasSt \times \tau \times \HKSt^{\otimes 2} \times \ColSt)$.
\end{thm}

\begin{rem}
  As noted in Remark~\ref{rem: collision of correlated proc}, 
  collision measures can also be defined for correlated processes.
  Analogous results for such processes, corresponding to Theorems~\ref{thm: col dtm result} and \ref{thm: col rdm result}, 
  can be established under minor modifications of the assumptions.
  Indeed, the convergence of the product processes and their heat kernels plays a crucial role in the proofs.
  In the above discussions, these convergences are obtained directly from the individual convergences of $X^1_n$ and $X^2_n$ together with their independence,
  and independence is not required beyond that point.
\end{rem}

\section{Preliminary results on stochastic processes on resistance metric spaces} \label{sec: Preliminary results on stoch proc on resis sp}

Typically, collisions between two independent and identically distributed (i.i.d.) stochastic processes occur only in low dimensions.
The purpose of this section is to introduce a class of metric spaces called \emph{resistance metric spaces}, 
which provides a natural framework for a unified study of stochastic processes in low-dimensional spaces. 
This framework was established by Kigami \cite{Kigami_01_Analysis,Kigami_12_Resistance} in the development of analysis on low-dimensional fractals, 
and recent research 
(cf.\ \cite{Croydon_18_Scaling,Croydon_Hambly_Kumagai_17_Time-changes,Noda_pre_Convergence,Noda_pre_Aging,Noda_pre_Scaling}) 
has demonstrated its usefulness, beyond deterministic fractals, in studying the scaling limits of random walks on critical random graphs.

We divide this section into four subsections.  
In Section~\ref{sec: resistance preliminary}, we review the theory of resistance metric spaces and the stochastic processes defined on them.  
In Section~\ref{sec: sp of recurrent resis spaces}, we introduce a Gromov--Hausdorff-type space consisting of resistance metric spaces,  
which will serve as a foundation for stating our main results in Section~\ref{sec: col of proc on resis sp}.  
Sections~\ref{sec: Convergence of law maps} and \ref{sec: Local limit theorem} 
are devoted to consequences of the previous works \cite{Croydon_18_Scaling,Noda_pre_Aging,Noda_pre_Scaling},
concerning the convergence of stochastic processes and their heat kernels 
on sequences of resistance metric spaces.

\subsection{Preliminaries} \label{sec: resistance preliminary}

In this subsection we recall the theory of resistance metric spaces and stochastic processes on them.
This can be regarded as an extension of the theory of electrical networks  
and continuous-time Markov chains on them
(e.g.\ \cite{Levin_Peres_17_Markov}).   
The reader is referred to \cite{Kigami_12_Resistance} for further background.

\begin{dfn} [{Resistance form, \cite[Definition 3.1]{Kigami_12_Resistance}}] 
  \label{dfn: resistance forms}
  Let $F$ be a non-empty set.
  A pair $(\form, \rdomain)$ is called a \emph{resistance form} on $F$ if it satisfies the following conditions.
  \begin{enumerate} [label=\textup{(RF\arabic*)}, leftmargin = *]
    \item \label{dfn cond: the domain of resistance form}
          The symbol $\rdomain$ is a linear subspace of the collection of functions $\{ f : F \to \RN \}$ 
          containing constants,
          and $\form$ is a non-negative symmetric bilinear form on $\rdomain$
          such that $\form(f,f)=0$ if and only if $f$ is constant on $F$.
    \item \label{dfn cond: the quotient space of resistance form is Hilbert}
          Let $\sim$ be the equivalence relation on $\rdomain$ defined by saying $f \sim g$ if and only if $f-g$ is constant on $F$.
          Then $(\rdomain/\sim, \form)$ is a Hilbert space.
    \item
          If $x \neq y$, then there exists a function $f \in \rdomain$ such that $f(x) \neq f(y)$.
    \item \label{dfn cond: resistance forms condition 4}
          For any $x, y \in F$,
          \begin{equation} \label{eq: variational formula of resistance metric}
            R_{(\form, \rdomain)}(x,y)
            \coloneqq
            \sup
            \left\{
            \frac{|f(x) - f(y)|^{2}}{\form(f,f)}
            :
            f \in \rdomain,\
            \form(f,f) > 0
            \right\}
            < \infty.
          \end{equation}     
    \item
          If $\bar{f}\coloneqq  (f \wedge 1) \vee 0$,
          then $\bar{f} \in \rdomain$ and $\form(\bar{f}, \bar{f}) \leq \form(f,f)$ for any $f \in \rdomain$.
  \end{enumerate}
\end{dfn}

For the following definition,
recall the effective resistance on an electrical network with a finite vertex set 
from \cite[Section 9.4]{Levin_Peres_17_Markov}
(see also \cite[Section 2.1]{Kigami_01_Analysis}).

\begin{dfn} [{Resistance metric, \cite[Definition 2.3.2]{Kigami_01_Analysis}}]
  \label{3. dfn: resistance metrics}
  A metric $R$ on a non-empty set $F$ is called a \emph{resistance metric}
  if and only if,
  for any non-empty finite subset $V \subseteq F$,
  there exists an electrical network $G$ (see Definition~\ref{dfn: electrical network}) with vertex set $V$ 
  such that the effective resistance on $G$ coincides with $R|_{V \times V}$.
\end{dfn}

\begin{thm} [{\cite[Theorem 2.3.6]{Kigami_01_Analysis}}]  \label{3. thm: one-to-one correspondence of forms and metrics}
  Fix a non-empty subset $F$.
  There exists a one-to-one correspondence between resistance forms $(\form, \rdomain)$ on $F$ 
  and resistance metrics $R$ on $F$ via $R = R_{(\form, \rdomain)}$.
  In other words,
  a resistance form $(\form, \rdomain)$ is characterized by $R_{(\form, \rdomain)}$ 
  given in \ref{dfn cond: resistance forms condition 4}.
\end{thm}

Motivated by the variational formula \eqref{eq: variational formula of resistance metric},
we define the resistance between sets as follows.

\begin{dfn} \label{dfn: resistance between sets}
  Let $(\form, \rdomain)$ be a resistance form on a set $F$
  and let $R$ be the corresponding resistance metric.
  For subsets $A, B \subseteq F$,
  we define
  \begin{equation}
    R(A,B)
    \coloneqq
    \left(
    \inf\{
    \form(f,f):
    f \in \rdomain,\,
    f|_{A} \equiv 1,\,
    f|_{B} \equiv 0
    \}
    \right)^{-1},
  \end{equation}
  with the convention that it is zero if the infimum is taken over the empty set.
  Note that, by \ref{dfn cond: resistance forms condition 4}, 
  we have $R(\{x\}, \{y\}) = R(x,y)$.
\end{dfn}

We will henceforth assume that we have a resistance metric $R$ on a non-empty set $F$
with corresponding resistance form $(\form,\rdomain)$.
Furthermore,
we assume that $(F, R)$ is boundedly compact
and recurrent, as described by the following definition.

\begin{dfn} [{Recurrent resistance metric, \cite[Definition~3.6]{Noda_pre_Scaling}}] \label{dfn: recurrent resistance metric}
  We say that $R$ is \emph{recurrent} if and only if
  \begin{equation}
    \lim_{r \to \infty} R(\rho, B_{R}(\rho, r)^{c}) = \infty
  \end{equation}
  for some (or, equivalently, any) $\rho \in F$.
\end{dfn}

We next introduce related Dirichlet forms and stochastic processes.
First, suppose that we have a Radon measure $m$ of full support on $(F,R)$.
We define a bilinear form $\form_{1}$ on $\rdomain \cap L^{2}(F, m)$ by setting
$\form_{1}(f,\, g) \coloneqq \form(f,\, g) + \int_{F}fg\, dm$.
Then $(\rdomain \cap L^{2}(F, m), \form_{1})$ is a Hilbert space (see \cite[Theorem~2.4.1]{Kigami_01_Analysis}).
Let $\domain$ denote the closure of $\rdomain \cap C_c(F, \RN)$ with respect to $\form_{1}$,
where $C_c(F, \RN)$ stands for the space of compactly supported continuous functions on $F$, 
equipped with the compact-convergence topology.
The recurrence of the resistance metric $R$ ensures that 
the resistance form $(\form, \rdomain)$ is regular in the sense of \cite[Definition~6.2]{Kigami_12_Resistance}
(see \cite[Corollary~3.22]{Noda_pre_Scaling} for this fact).
In particular, we can apply \cite[Theorem~9.4]{Kigami_12_Resistance},
which implies that $(\form, \domain)$ is a regular Dirichlet form on $L^{2}(F, m)$.
Moreover, by \cite[Lemma~2.3]{Croydon_18_Scaling}, this Dirichlet form is recurrent.
(See \cite{Fukushima_Oshima_Takeda_11_Dirichlet} for the definitions of regular Dirichlet forms and recurrence.)
Write $X$ for the associated Hunt process.
We refer to this Hunt process as the \emph{(Hunt) process associated with $(R, m)$}.
Note that such a process is, in general, only specified uniquely for starting points outside a set of zero capacity.
However, in this setting,
every point has strictly positive capacity (see \cite[Theorem 9.9]{Kigami_12_Resistance}),
and so the process is defined uniquely everywhere.
Since the Dirichlet form is recurrent,
$X$ is conservative (see \cite[Exercise~4.5.1]{Fukushima_Oshima_Takeda_11_Dirichlet}).
Moreover, $X$ satisfies Hunt's hypothesis~\ref{item: Hunt hypo} by \cite[Theorem~4.1.3]{Fukushima_Oshima_Takeda_11_Dirichlet}.

From \cite[Theorem 10.4]{Kigami_12_Resistance},
the Hunt process $X$ admits a unique jointly continuous heat kernel
$p_{(R, m)} \colon (0, \infty) \times F \times F \to [0,\infty)$ with respect to $m$,
which we refer to as the \emph{heat kernel associated with $(R, m)$}.
When the context is clear, we simply write $p = p_{(R, m)}$.
The heat kernel $p$ is symmetric, i.e., $p(t,x,y) = p(t,y,x)$.
In particular, the process $X$ satisfies the DAC condition (Assumption~\ref{assum: dual hypothesis}).
The following is useful for obtaining an upper bound on the heat kernel from lower volume estimates.

\begin{lem} [{\cite[Theorem~10.4]{Kigami_12_Resistance}}] 
  \label{lem: hk estimate from volume}  \leavevmode
  For any $x \in F$, $t>0$, and $s>0$, 
  \begin{equation}
    p(t, x, x) 
    \leq 
    \frac{2s}{t}
    + 
    \frac{\sqrt{2}}{m(D_{R}(x,s))}.
  \end{equation}
\end{lem}

To illustrate the above ideas more concretely,
we now apply the resistance form framework to electrical networks,

\begin{dfn} [{Electrical network}] \label{dfn: electrical network}
  We say that a tuple $(V, E, \mu)$ is an \emph{electrical network} if it satisfies the following conditions.
  \begin{enumerate} [label = \textup{(\roman*)}]
    \item The pair $(V, E)$ is a connected, simple, undirected graph with finitely or countably many vertices,
      where $V$ denotes the vertex set and $E$ denotes the edge set.
      (NB.\ A graph being simple means that it has no loops and no multiple edges.)
    \item The symbol $\mu$ denotes the \emph{conductance}, that is, 
      it is a function $\mu \colon V \times V \to \RNp$ such that $\mu(x,y) = \mu(y, x)$ for all $x, y \in V$,
      $\mu(x, y) > 0$ if and only if $\{x,y\} \in E$, and 
      \begin{equation}
        \mu(x) \coloneqq \sum_{y \in V} \mu(x,y) < \infty,
        \quad \forall x \in V.
        \vspace*{-12pt}
      \end{equation}
  \end{enumerate}
\end{dfn}

Let $(V, E, \mu)$ be an electrical network.
For functions $f, g \colon V \to \RN$, we define 
\begin{equation}
  \form(f, g) 
  \coloneqq 
  \frac{1}{2} 
  \sum_{x, y \in V} 
  \mu(x,y) (f(x) - f(y)) (g(x) - g(y)),
\end{equation}
whenever the right-hand side is well-defined.
We set $\rdomain \coloneqq \{ f \colon V \to \RN \mid \form(f, f) < \infty \}$.
Then the pair $(\form, \rdomain)$ is a regular resistance form on $V$.
The associated resistance metric $R$ induces the discrete topology on $V$.
(See \cite[Theorem~4.1]{Noda_pre_Scaling})
If $R$ is recurrent in the sense of Definition~\ref{dfn: recurrent resistance metric},
then we say that the electrical network is \emph{recurrent}.

Let $m$ be a Radon measure on $V$ with full support.
The Hunt process $X$ associated with $(R, m)$ is the minimal continuous-time Markov chain on $V$
with jump rates $w(x,y) = \mu(x,y)/m(\{x\})$.
(See \cite[Theorem~4.2]{Noda_pre_Scaling}).
Thus, we have the following consequences.
\begin{enumerate} [label = \textup{(RW\arabic*)}, leftmargin = *]
  \item If $m$ is the \emph{counting measure} on $V$, i.e., $m(\{x\}) = 1$ for each $x \in V$,
    then the associated process $X$ is the \emph{variable speed random walk (VSRW)}.
  \item If $m$ is the \emph{conductance measure} on $V$, i.e., $m(\{x\}) = \mu(x)$ for each $x \in V$,
    then the associated process $X$ is the \emph{constant speed random walk (CSRW)}.
\end{enumerate}

\subsection{Spaces of recurrent resistance metric spaces} \label{sec: sp of recurrent resis spaces}

In this subsection,  
we introduce a space of recurrent resistance metric spaces (Definition \ref{dfn: space of resistance spaces})
and verify their measurability with respect to a Gromov-Hausdorff-type topology (Proposition \ref{prop: sp of recurrent resis sp is Borel}).  
This serves as a fundamental result for studying stochastic processes  
on random resistance metric spaces, such as the continuum random tree.  
 
\begin{dfn} \label{dfn: space of resistance spaces}
  Given a Polish structure $\tau$,
  we define 
  \begin{equation}
    \rootResisSp(\tau) 
    \coloneqq
    \left\{ (F, R, \rho, m, a) \in \rootBCM(\MeasSt \times \tau) \,\middle|\,
    \begin{gathered}
      R \text{ is a recurrent resistance metric on } F,\\
      m \text{ has full support}
    \end{gathered}
    \right\}.
  \end{equation} 
  Similarly, $\rootResisSp$ denotes the set of quadruples $(F,R,\rho,m)$
  obtained from the above definition by omitting $a$.
  We always equip $\rootResisSp(\tau)$ with the relative topology 
  induced from $\rootBCM(\MeasSt \times \tau)$.
\end{dfn}

Henceforth, we fix a Polish structure $\tau$.
The following is the main result of this subsection.

\begin{prop} \label{prop: sp of recurrent resis sp is Borel}
  The set $\rootResisSp(\tau)$ is a Borel subset of $\rootBCM(\MeasSt \times \tau)$.
\end{prop}

By definition,
$\rootResisSp(\tau)$ is the intersection of 
the collection of recurrent resistance metric spaces 
and 
the collection of metric spaces equipped with a fully-supported measure.
Thus, by showing that each collection is Borel,
we establish the above proposition.

\begin{lem} \label{lem: sp of full supp meas met sp is Borel}
  Define 
  \begin{equation}
    \mathfrak{N}_1 
    \coloneqq 
    \bigl\{ (F, R, \rho, m, a) \in \rootBCM(\MeasSt \times \tau) \mid
        m\ \text{is of full support}
    \bigr\}.
  \end{equation}
  Then the set $\mathfrak{N}_1$ is a Borel subset of $\rootBCM(\MeasSt \times \tau)$.
\end{lem}

\begin{proof}
  For each $N, k, l \in \NN$,
  we define $\mathfrak{N}_{(N, k, l)}$
  to be the collection of $(S, d, \rho, m, a) \in \rootBCM(\MeasSt \times \tau)$ such that,
  for all but countably many $r \in [N, N+1]$,
  \begin{equation}  \label{pr eq: 1, sp of full supported measured met sp is Borel}
    \inf_{x \in S^{(r)}} m\bigl( D_S(x, 1/k) \bigr) \geq 1/l.
  \end{equation}
  We then have that 
  \begin{equation}
    \mathfrak{N}_1
    = 
    \bigcap_{N \geq 1} 
    \bigcap_{k \geq 1} 
    \bigcup_{l \geq 1} 
    \mathfrak{N}_{(N, k, l)}.
  \end{equation}
  Thus, it is enough to show that $\mathfrak{N}_{(N, k, l)}$ is a Borel subset of $\rootBCM(\MeasSt \times \tau)$.
  We do this by showing that $\mathfrak{N}_{(N, k, l)}$ is closed. 
  Let $G_n = (S_n, d_n, \rho_n, m_n, a_n) \in \mathfrak{N}_{(N, k, l)}$, $n \in \NN$, 
  be such that 
  \begin{equation}
    G_n = (S_n, d_n, \rho_n, m_n, a_n) \to (S, d, \rho, m, a) = G
  \end{equation}
  in $\rootBCM(\MeasSt \times \tau)$.
  By Theorem~\ref{thm: conv in M(tau)},
  we may assume that 
  $(S_n, \rho_n)$, $n \in \NN$, and $(S, \rho)$ are embedded isometrically into
  a common rooted $\bcmAB$ space $(M, \rho_M)$ 
  in such a way that 
  $\rho_n = \rho = \rho_M$ as elements of $M$,
  $S_n \to S$ in the Fell topology as closed subsets of $M$,
  $m_n \to m$ vaguely as measures on $M$,
  and $a_n \to a$ in $\tau(M)$. 
  Let $r \in [N, N+1]$ be such that each $G_n$ satisfies \eqref{pr eq: 1, sp of full supported measured met sp is Borel}
  and $S_n^{(r)} \to S^{(r)}$ in the Hausdorff topology.
  For an arbitrarily fixed $x \in S^{(r)}$,
  we choose elements $x_n \in S_n^{(r)}$ converging to $x$ in $M$.
  Then, for any $\delta > 0$,
  we have that, for all sufficiently large $n$,
  $D_M(x, 1/k + \delta) \supseteq D_M(x_n, 1/k)$.
  This, combined with the Portmanteau theorem (cf.\ \cite[Lemma~4.1]{Kallenberg_17_Random}), yields that, for any $\delta > 0$,
  \begin{equation}
    m\bigl(D_S(x, 1/k + \delta)\bigr)
    \geq    
    \limsup_{n \to \infty} 
    m_n\bigl(D_S(x, 1/k + \delta)\bigr)
    \geq     
    \limsup_{n \to \infty} 
    m_n\bigl(D_S(x_n, 1/k)\bigr)
    \geq    
    1/l.
  \end{equation}
  By letting $\delta \to 0$,
  we obtain that $m\bigl(D_S(x, 1/k)\bigr) \geq 1/l$,
  which implies $G \in \mathfrak{N}_{(N, k, l)}$.
  This completes the proof.
\end{proof}

\begin{lem} \label{lem: sp of recurrent resis met sp is Borel}
  Define 
  \begin{equation}
    \mathfrak{N}_2 
    \coloneqq 
    \bigl\{ (F, R, \rho, m, a) \in \rootBCM(\MeasSt \times \tau) \mid
        R\ \text{is a recurrent resistance metric on}\ F
    \bigr\}.
  \end{equation}
  Then the set $\mathfrak{N}_2$ is a Borel subset of $\rootBCM(\MeasSt \times \tau)$.
\end{lem}

\begin{proof}
  We prove the result similarly to Lemma~\ref{lem: sp of full supp meas met sp is Borel}.
  For each $N, l \in \NN$,
  we define $\mathfrak{N}_{(N, l)}$ 
  to be the collection of $(F, R, \rho, m, a) \in \rootBCM(\tau)$ such that 
  $R$ is a resistance metric on $F$
  and, for all but countably many $r \in [N, N+1]$,
  it holds that $R(\rho, B_R(\rho, r)^c) \geq l$.
  We then have that 
  \begin{equation}
    \mathfrak{N}_2
    = 
    \bigcap_{l \geq 1} 
    \bigcup_{N \geq 1} 
    \mathfrak{N}_{(N, l)}.
  \end{equation}
  Thus, it suffices to show that each $\mathfrak{N}_{(N, l)}$ is closed in $\rootBCM(\tau)$.
  Let $G_n = (F_n, R_n, \rho_n, m_n, a_n) \in \mathfrak{N}_{(N, l)}$, $n \in \NN$, 
  be such that 
  \begin{equation}
    G_n = (F_n, R_n, \rho_n, m_n, a_n) \to (F, R, \rho, m, a) = G
  \end{equation}
  in $\rootBCM(\MeasSt \times \tau)$.
  It then follows from \cite[Theorem 5.1]{Noda_pre_Scaling} that 
  $R$ is a resistance metric and 
  \begin{equation}
    R(\rho, B_R(\rho, r)^c)
    \geq 
    \limsup_{n \to \infty} 
    R_n(\rho_n, B_{R_n}(\rho_n, r)^c).
  \end{equation}
  The right-hand side of the above inequality is larger than or equal to $l$ as $G_n \in \mathfrak{N}_{(N, l)}$ for all $n$.
  It follows that $G \in \mathfrak{N}_{(N, l)}$, which completes the proof.
\end{proof}

\begin{proof} [{Proof of Proposition \ref{prop: sp of recurrent resis sp is Borel}}]
  Since $\rootResisSp(\tau) = \mathfrak{N}_1 \cap \mathfrak{N}_2$,
  the desired result follows from Lemmas \ref{lem: sp of full supp meas met sp is Borel} and \ref{lem: sp of recurrent resis met sp is Borel}.
\end{proof}


\subsection{Convergence of law map of stochastic process} \label{sec: Convergence of law maps}

In \cite{Croydon_18_Scaling},
Croydon proved that, when rooted-and-measured resistance metric spaces converge,
the laws of associated stochastic processes started at the root also converge.
In this subsection,  
we strengthen the result to convergence of law maps of the stochastic processes
(see Theorems \ref{thm: dtm conv of SPM} and \ref{thm: rdm conv for SPM} below). 
This result will be used in the next subsection  
to establish the convergence of heat kernels.  

\begin{dfn} \label{dfn: law map}
  Let $(F, R, m)$ be a boundedly-compact recurrent resistance metric space equipped with a fully-supported Radon measure $m$.
  Write $((X_t)_{t \geq 0},\allowbreak (P^x)_{x \in F})$ for the process associated with $(R, m)$.
  We then define a map $\SPM\colon F \to \Prob(D_{J_1}(\RNp, F))$ by setting 
  \begin{equation}
    \SPM_{(R, m)}(x) 
    \coloneqq 
    \ProcLaw_X(x)
    =
    P^x \bigl((X_t)_{t \geq 0} \in \cdot\bigr),
    \quad 
    x \in F.
  \end{equation}
  We refer to $\SPM_{(R,m)}$ as the \emph{law map associated with $(R, m)$}.
  When the context is clear, we simply write $\SPM \coloneqq \SPM_{(R,m)}$.
\end{dfn}

We recall the main result of \cite{Croydon_18_Scaling}.

\begin{thm} [{\cite[Proof of Theorem~1.2]{Croydon_18_Scaling}}] \label{thm: Croydon result}
  Let $G_n = (F_n, R_n, \rho_n, m_n)$, $n \in \NN$, and $G = (F, R, \rho, m)$ be elements of $\rootResisSp(\tau)$.
  Assume that 
  \begin{equation} \label{thm eq: non-explosion. Croydon result}
    \lim_{r \to \infty} \liminf_{n \to \infty} R_{n}(\rho_{n}, B_{R_{n}}(\rho_{n}, r)^{c}) = \infty.
  \end{equation}
  Assume further that all the space $F_n$, $n \geq 1$, and $F$ are isometrically embedded into a common $\bcmAB$ space $M$ in such a way that 
  $F_n \to F$ in the Fell topology as closed subsets of $M$,
  $\rho_n \to \rho$ as elements of $M$,
  and $m_n \to m$ vaguely as measures on $M$.
  It then holds that $\SPM_{(R_n, m_n)}(\rho_n) \to \SPM_{(R, m)}(\rho)$ as probability measures on $D_{J_1}(\RNp, M)$.
\end{thm}

\begin{rem}
  The condition \eqref{thm eq: non-explosion. Croydon result} is called the \emph{non-explosion condition} in \cite{Croydon_18_Scaling}.
  Note that there is a typographical error in \cite[Assumption~1.1(b)]{Croydon_18_Scaling},
  where ``$\limsup_{n \to \infty}$'' should read ``$\liminf_{n \to \infty}$'',
  as stated in the above theorem.
\end{rem}

Using Theorem~\ref{thm: Croydon result},
we first prove the continuity of $\SPM_{(R,m)}$ with respect to the starting point.
Recall that $\Prob(D_{J_1}(\RNp, F))$ is equipped with the weak topology induced from the usual $J_{1}$-Skorohod topology.

\begin{lem} \label{lem: continuity of stoc map on resis sp}
  The law map $\SPM\colon F \to \Prob(D_{J_1}(\RNp, F))$ is continuous.
\end{lem}

\begin{proof}
  Suppose that elements $x_{n} \in F$ converge to an element $x \in F$.
  By the recurrence of $R$,
  we can use Theorem~\ref{thm: Croydon result} to obtain that 
  $\SPM(x_n) \to \SPM(x)$
  as probability measures on $D_{J_1}(\RNp, F)$.
  This completes the proof.
\end{proof}

We are then interested in the continuity of $\SPM_{(R, m)}$ with respect to $(R, m)$.
To discuss this within the framework of Gromov-Hausdorff-type topologies,
we recall the structure $\SPMSt$ from \eqref{eq: st for SPM. revisited} .
Henceforth, we fix a Polish structure $\tau$.

\begin{dfn} \label{dfn: GHSPM}
  For each $G = (F, R, \rho, m, a) \in \rootResisSp(\tau)$,
  we set 
  \begin{equation}
    \GHSPM_\tau(G)
    \coloneqq 
    \bigl( F, R, \rho, m, a, \SPM_{(R, m)} \bigr),
  \end{equation}
  which is an element of $\rootResisSp(\tau \times \SPMSt)$.
\end{dfn}

It is straightforward to verify that Theorem~\ref{thm: Croydon result} 
implies the convergence of law maps of stochastic processes, as follows.

\begin{thm} \label{thm: dtm conv of SPM}
  Let $G_n = (F_n, R_n, \rho_n, m_n, a_n)$, $n \in \NN$, be elements of $\rootResisSp(\tau)$,
  and $G = (F, R, \rho, m, a)$ be an element of $\rootBCM(\MeasSt \times \tau)$ such that $m$ is of full support.
  Assume that the following conditions are satisfied:
  \begin{enumerate} [label = \textup{(\roman*)}]
    \item \label{thm item: 1. dtm conv of SPM}
      $G_n \to G$ in $\rootBCM(\MeasSt \times \tau)$;
    \item \label{thm item: 2. dtm conv of SPM}
      $\displaystyle \lim_{r \to \infty} \liminf_{n \to \infty} R_{n}(\rho_{n}, B_{R_{n}}(\rho_{n}, r)^{c}) = \infty$.
  \end{enumerate}
  Then $G \in \rootResisSp(\tau)$ and $\GHSPM_\tau(G_n) \to \GHSPM_\tau(G)$ in $\rootResisSp(\tau \times \SPMSt)$.
\end{thm}

\begin{proof}
  By \cite[Theorem~5.1]{Noda_pre_Scaling},
  we have $G \in \rootResisSp(\tau)$.
  We write $\SPM_n \coloneqq \SPM_{(R_n, m_n)}$ and $\SPM \coloneqq \SPM_{(R, m)}$.
  By \ref{thm item: 1. dtm conv of SPM} and Theorem~\ref{thm: conv in M(tau)},
  we may assume that $(F_{n}, R_{n}, \rho_n)$ and $(F, R, \rho)$ are embedded 
  isometrically into a common rooted $\bcmAB$ space $(M, d_M, \rho_M)$
  in such a way that 
  $F_n \to F$ in the Fell topology as closed subsets of $M$,
  $\rho_n = \rho = \rho_M$ as elements of $M$,
  $m_{n} \to m $ vaguely as measures on $M$,
  and $a_n \to a$ in $\tau(M)$.
  Suppose that elements $x_{n} \in F_{n}$ converge to an element $x \in F$ in $M$.
  Then, using \ref{thm item: 2. dtm conv of SPM},
  we can follow the proof of \cite[Lemma~5.8]{Noda_pre_Scaling} to obtain that  
  \begin{equation}
    \lim_{r \to \infty} \liminf_{n \to \infty} R_{n}(x_{n}, B_{R_{n}}(x_{n}, r)^{c}) = \infty.
  \end{equation}
  By Theorem~\ref{thm: Croydon result}, we obtain that $\Upsilon_n(x_n) \to \Upsilon(x)$
  as probability measures on $D_{J_1}(\RNp, M)$.
  Therefore, $\Upsilon_{n}$ converges to $\Upsilon$ in $\hatC(M, \Prob(D_{J_1}(\RNp, M)))$,
  which completes the proof.
\end{proof}

To discuss a version of Theorem \ref{thm: dtm conv of SPM} for random resistance metric spaces,
we prove the following result, which ensures that, if $G$ is a random element of $\rootResisSp(\tau)$,
then $\GHSPM_{\tau}(G)$ is a random element of $\rootResisSp(\tau \times \SPMSt)$.

\begin{prop}  \label{prop: measurability of SPM wrt GH topology}
  The map $\GHSPM_{\tau}\colon \rootResisSp(\tau) \to \rootResisSp(\tau \times \SPMSt)$ is Borel measurable.
  Moreover, its image, that is, 
  \begin{equation}
    \left\{ \bigl( F, R, \rho, m, a, \SPM_{(R,m)} \bigr)\, \middle|\, (F, R, \rho, m, a) \in \rootResisSp(\tau) \right\}
  \end{equation}
  is a Borel subset of $\rootBCM(\MeasSt \times \tau \times \SPMSt)$.
\end{prop}

We follow \cite[Proof of Proposition 6.1]{Noda_pre_Scaling} to show Proposition~\ref{prop: measurability of SPM wrt GH topology}.
For each $(F, R, \rho, m) \in \rootResisSp$ and $r>0$,
we define $F^{[r]}$ to be the closure of $B_F(\rho, r)$ in $F$, 
and denote by $R^{[r]}$ and $m^{[r]}$ the restrictions of $R$ and $m$ 
to $F^{[r]}$ and $B_R(\rho, r)$, respectively. 
We note that $(F^{[r]}, R^{[r]}, \rho, m^{[r]}) \in \rootResisSp$ 
by \cite[Theorem~8.4]{Kigami_12_Resistance}.
We write 
\begin{equation}
  \SPM^{[r]}(x) 
  = 
  \SPM^{[r]}_{(R, m)}(x)
  \coloneqq 
  \SPM_{(R^{[r]}, m^{[r]})}(x),
  \quad 
  x \in F^{[r]}.
\end{equation}
Since $F^{[r]}$ is a subspace of $F$,
we naturally regard $\SPM^{[r]}_{(R, m)}$ as an element of $\hatC(F, \Prob(D_{J_1}(\RNp, F)))$.
Below, we present two lemmas corresponding to \cite[Lemmas 6.2 and 6.3]{Noda_pre_Scaling}.

\begin{lem} \label{lem: left continuity of restriction of SPM}
  For each $(F, R, \rho, m) \in \rootResisSp$,
  the following map 
  \begin{equation}
    (0, \infty) \ni r \longmapsto \SPM^{[r]} \in \hatC(F, \Prob(D_{J_1}(\RNp, F)))
  \end{equation}
  is left-continuous.
\end{lem}

\begin{proof}
  Fix $r > 0$ and a sequence $(r_n)_{n \geq 1}$ such that $r_n \uparrow r$.
  We regard each $(F^{[r_n]}, R^{[r_n]})$ and $(F^{[r]}, R^{[r]})$
  as a subspace of $(F, R)$ in the obvious way.
  It is then straightforward to check that $F^{[r_n]} \to F^{[r]}$ in the Fell topology on $F$
  and $m^{[r_n]} \to m^{[r]}$ vaguely as measures on $F$.
  Thus, by Theorem~\ref{thm: Croydon result},
  we obtain that $\SPM^{[r_n]} \to \SPM^{[r]}$,
  which completes the proof.
\end{proof}

\begin{rem}
  In \cite[Lemmas~2.14 and 6.2]{Noda_pre_Convergence}, 
  one finds a similar left-continuity result.
  In that paper,
  the author considered the restriction of $m$ to $F^{[r]}$.
  However, such a choice generally destroys the left-continuity in $r$ 
  of the map $r \mapsto m^{[r]} \in \finMeas(F)$.
  (For example, consider the case where $F = [0,1]$ and $m = \Leb + \delta_{\{1\}}$,
  which breaks the left-continuity at $r = 1$.)
  Thus, the restriction used in that paper should be replaced by the current definition. 
  We note that this does not affect any of the results in that paper.
\end{rem}

For each $r > 0$,
we define a map $\GHSPM^{[r]}_\tau \colon \rootResisSp(\tau) \to \rootResisSp(\tau \times \SPMSt)$ as follows:
for each $G = (F, R, \rho, m, a) \in \rootResisSp(\tau)$,
\begin{equation}
  \GHSPM^{[r]}_\tau(G) 
  \coloneqq 
  \bigl( F, R, \rho, m, a, \SPM^{[r]}_{(R, m)} \bigr).
\end{equation}

\begin{lem} \label{lem: continuity of restriction of SPM}
  For each $r > 0$,
  the map $\GHSPM^{[r]}_\tau \colon \rootResisSp(\tau) \to \rootResisSp(\tau \times \SPMSt)$ is measurable.
\end{lem}

To prove the above lemma, we prepare an auxiliary result,
which will also be used in the proof of Theorem~\ref{thm: rdm conv for SPM} below.

\begin{lem} \label{lem: rough continuity of GHSPM}
  Assume that elements $G_n \in \rootResisSp(\tau)$ converge to an element $G \in \rootResisSp(\tau)$.
  Then, for all but countably many $r > 0$,
  \begin{equation} \label{lem eq: GHSPM rough cont}
    \GHSPM^{[r]}_\tau(G_n)
    \xrightarrow[n \to \infty]{}
    \GHSPM^{[r]}_\tau(G)
    \quad \text{in } \rootResisSp(\tau \times \SPMSt).
  \end{equation}
\end{lem}

\begin{proof}
  Write $G_n = (F_n, R_n, \rho_n, m_n, a_n)$ for $n \in \NN$, and $G =(F, R, \rho, m, a)$.
  By Theorem~\ref{thm: conv in M(tau)}, 
  we may assume that $(F_n, R_n, \rho_n)$ and $(F, R, \rho)$ are embedded 
  isometrically into a common rooted $\bcmAB$ space $(M, d^{M}, \rho_{M})$
  in such a way that 
  $\rho_{n} = \rho = \rho_{M}$ as elements of $M$,
  $F_{n} \to F$ in the Fell topology as closed subsets of $M$,
  $m_{n} \to m$ vaguely as measures on $M$,
  and $a_n \to a$ in $\tau(M)$.
  It then follows that, for all but countably many $r>0$,
  $F_n^{[r]} \to F^{[r]}$ in the Hausdorff topology,
  and $m_n^{[r]} \to m^{[r]}$ weakly.
  Following the proof of Theorem~\ref{thm: dtm conv of SPM},
  one can verify that, for such radii $r>0$,
  $\SPM^{[r]}_n \to \SPM^{[r]}$ 
  in $\hatC(M, \Prob(D_{J_1}(\RNp, M)))$.
  Combining these convergences yields \eqref{lem eq: GHSPM rough cont},
  which completes the proof.
\end{proof}

\begin{proof} [{Proof of Lemma~\ref{lem: continuity of restriction of SPM}}]
  Fix $r>0$.
  Let $f \colon \rootResisSp(\tau \times \SPMSt) \to [0,1]$ be a continuous function.
  It suffices to show that $g \coloneqq f \circ \GHSPM^{[r]}_\tau$ is measurable.
  For each $\varepsilon \in (0, r)$, 
  define $g_\varepsilon \colon \rootResisSp(\tau \times \SPMSt) \to [0,1]$ by
  \begin{equation}
    g_\varepsilon(G)
    \coloneqq 
    \frac{1}{\varepsilon} \int_{r - \varepsilon}^r 
      f\bigl(\GHSPM^{[s]}_\tau(G)\bigr)\, ds,
    \qquad G = (F, R, \rho, m, a) \in \rootResisSp(\tau).
  \end{equation}
  The above integral is well-defined since the integrand is measurable,
  which follows from Lemma~\ref{lem: left continuity of restriction of SPM}.
  Moreover, by that lemma, $g_\varepsilon(G) \to f\bigl(\GHSPM^{[r]}_\tau(G)\bigr)$ for each $G$.
  Hence, it suffices to show that $g_\varepsilon$ is measurable.
  We verify this by showing that $g_\varepsilon$ is continuous.

  Indeed, let $G_n \to G$ in $\rootResisSp(\tau)$.
  Then Lemma~\ref{lem: rough continuity of GHSPM} implies that 
  $\GHSPM^{[s]}_\tau(G_n) \to \GHSPM^{[s]}_\tau(G)$
  for all but countably many $s>0$.
  The continuity of $f$ then gives 
  $f(\GHSPM^{[s]}_\tau(G_n)) \to f(\GHSPM^{[s]}_\tau(G))$
  for such $s$.
  Since $f$ takes values in $[0,1]$, the dominated convergence theorem yields
  $g_\varepsilon(G_n) \to g_\varepsilon(G)$.
  Hence $g_\varepsilon$ is continuous, and therefore measurable.
  This completes the proof.
\end{proof}

We next show that $\GHSPM_\tau$ is obtain as a (pointwise) limit of $\GHSPM^{[r]}_\tau$ as $r \to \infty$,
which, combined with the above two lemmas, implies Proposition~\ref{prop: measurability of SPM wrt GH topology}.

\begin{lem} \label{lem: conv of traces}
  For any $G \in \rootResisSp(\tau)$,
  it holds that 
  $\GHSPM^{[r]}_\tau(G) \to \GHSPM_\tau(G)$ in $\rootResisSp(\tau \times \SPMSt)$ as $r \to \infty$.
\end{lem}

\begin{proof}
  Fix $G = (F, R, \rho, m, a) \in \rootResisSp(\tau)$.
  It is easy to check that, as $r \to \infty$, $F^{[r]} \to F$ in the Fell topology as closed subsets of $F$ 
  and $m^{[r]} \to m$ vaguely as measures on $F$.
  Since $R$ is a recurrent resistance metric,
  we can follow the proof of Theorem~\ref{thm: dtm conv of SPM}
  to deduce that $\Upsilon^{[r]} \to \Upsilon$ as $r \to \infty$ in $\hatC(F, \Prob(D_{J_1}(\RNp, F)))$.
  This immediately implies the desired result.
\end{proof}

Now, it is straightforward to show Proposition~\ref{prop: measurability of SPM wrt GH topology}.

\begin{proof} [{Proof of Proposition~\ref{prop: measurability of SPM wrt GH topology}}]
  The first assertion regarding the measurability of $\GHSPM_\tau$ is an immediate consequence of Lemmas~\ref{lem: continuity of restriction of SPM} and \ref{lem: conv of traces}.
  Clearly, the map $\GHSPM_\tau$ is injective.
  We recall from Proposition~\ref{prop: sp of recurrent resis sp is Borel} that
  its domain $\rootResisSp(\tau)$ is a Borel subset of the Polish space $\rootResisSp(\MeasSt \times \tau)$.
  Thus, we can apply Lemma~\ref{lem: Lousin--Souslin}
  to deduce that the image of $\GHSPM_\tau$ is a Borel subset of $\rootBCM(\MeasSt \times \tau \times \SPMSt)$,
  which completes the proof.
\end{proof}

Thanks to Proposition~\ref{prop: measurability of SPM wrt GH topology},  
if $G$ is a random element of $\rootResisSp(\tau)$,  
then $\GHSPM_\tau(G)$ is a random element of $\rootResisSp(\tau \times \SPMSt)$.  
Below, we provide a version of Theorem \ref{thm: dtm conv of SPM} for random resistance metric spaces.  

\begin{thm} \label{thm: rdm conv for SPM}
  Let $G_n = (F_n, R_n, \rho_n, m_n, a_n)$, $n \in \NN$, be random elements of $\rootResisSp(\tau)$,
  and $G = (F, R, \rho, m, a)$ be a random element of $\rootBCM(\MeasSt \times \tau)$
  such that $m$ is of full support, almost surely.
  Assume that the following conditions are satisfied.
  \begin{enumerate} [label = \textup{(\roman*)}]
    \item \label{thm item: rdm conv for SPM. 1}
      It holds that $G_n \xrightarrow{\mathrm{d}} G$ in $\rootResisSp(\tau)$;
    \item \label{thm item: rdm conv for SPM. 2}
      If we write $\mathbf{P}_n$ for the underlying probability measure of $G_n$, 
      then, for all $\lambda > 0$,
      \begin{equation}
        \lim_{r \to \infty} \liminf_{n \to \infty} 
        \mathbf{P}_{n}\bigl( R_{n}(\rho_{n}, B_{R_{n}}(\rho_{n}, r)^{c}) > \lambda \bigr) = 1.
      \end{equation}
  \end{enumerate}
  Then $G \in \rootResisSp(\tau)$ almost surely, 
  and $\GHSPM_\tau(G_n) \xrightarrow{\mathrm{d}} \GHSPM_\tau(G)$ as random elements of $\rootResisSp(\tau \times \SPMSt)$.
\end{thm}

\begin{proof}
  By \cite[Proposition~6.1]{Noda_pre_Scaling},
  we have $G \in \rootResisSp(\tau)$ almost surely.
  For the second assertion,
  by \cite[Theorem~2.1]{Billingsley_99_Convergence},
  it suffices to show that 
  for all bounded and uniformly continuous functions 
  $f \colon \rootResisSp(\tau \times \SPMSt) \to \RN$,
  \begin{equation} \label{pr eq: rdm conv for SPM 0}
    \lim_{n \to \infty} 
    \mathbf{E}_n\!\left[f(\GHSPM_\tau(G_n))\right] 
    = 
    \mathbf{E}\!\left[f(\GHSPM_\tau(G))\right],
  \end{equation}
  where $\mathbf{E}$ denotes the expectation with respect to the randomness of $G$.
  Fix such a function $f$.
  
  By Lemma~\ref{lem: rough continuity of GHSPM} and 
  Theorem~\ref{thm: rdm conv for SPM}~(i),
  we have that, for each $r > 1$,
  \begin{equation}
    \lim_{n \to \infty} 
    \mathbf{E}_n\!\left[ 
      \int_{r-1}^r 
        f\bigl(\GHSPM^{[s]}_\tau(G_n)\bigr)\, ds
    \right] 
    =
    \mathbf{E}\!\left[
      \int_{r-1}^r 
        f\bigl(\GHSPM^{[s]}_\tau(G)\bigr)\, ds
    \right].
  \end{equation}
  (Note that the above integrals are well-defined by 
  Lemma~\ref{lem: left continuity of restriction of SPM}.)
  Moreover, Lemma~\ref{lem: conv of traces} yields that
  \begin{equation}
    \lim_{r \to \infty} 
    \mathbf{E}\!\left[
      \int_{r-1}^r 
        f\bigl(\GHSPM^{[s]}_\tau(G)\bigr)\, ds
    \right] 
    =
    \mathbf{E}\!\left[
      f(\GHSPM_\tau(G))
    \right].
  \end{equation}
  Hence, it remains to show that 
  \begin{equation} \label{pr eq: 1. rdm conv for SPM}
    \lim_{r \to \infty} 
    \limsup_{n \to \infty}
    \mathbf{E}_n
    \!\left[
      \int_{r-1}^r 
        \bigl| 
          f\bigl(\GHSPM^{[s]}_\tau(G_n)\bigr)
          - 
          f\bigl(\GHSPM_\tau(G_n)\bigr)
        \bigr|
      ds
    \right]
    = 0.
  \end{equation}

  For each $\varepsilon > 0$, define the modulus of continuity of $f$ by
  \begin{equation}
    w_f(\varepsilon) 
    \coloneqq 
    \sup\!\left\{
      |f(\mathcal{Y}_1) - f(\mathcal{Y}_2)| 
      \,\middle|\,
      \mathcal{Y}_1, \mathcal{Y}_2 \in \rootResisSp(\tau \times \SPMSt),
      \ \GFMet^{\tau \times \SPMSt}(\mathcal{Y}_1, \mathcal{Y}_2) \leq \varepsilon
    \right\}.
  \end{equation}
  Then it follows that 
  \begin{align}
    &\mathbf{E}_n
    \!\left[
      \int_{r-1}^r 
        \bigl| 
          f\bigl(\GHSPM^{[s]}_\tau(G_n)\bigr)
          - 
          f\bigl(\GHSPM_\tau(G_n)\bigr)
        \bigr|
      ds
    \right] \notag\\
    &\quad \leq     
    w_f(\varepsilon)
    + 
    \|f\|_\infty   
    \mathbf{P}_n
    \!\left(
      \sup_{r-1 < s < r} 
      \GFMet^{\tau \times \SPMSt}\bigl(
        \GHSPM^{[s]}_\tau(G_n), 
        \GHSPM_\tau(G_n)
      \bigr) 
      > \varepsilon
    \right).
  \end{align}
  The uniform continuity of $f$ implies $w_f(\varepsilon) \to 0$ as $\varepsilon \to 0$.
  Therefore, it remains to prove that
  \begin{equation} \label{pr eq: 2. rdm conv for SPM}
    \lim_{r \to \infty} 
    \limsup_{n \to \infty}
    \mathbf{P}_n
    \!\left(
      \sup_{r-1 < s < r} 
      \GFMet^{\tau \times \SPMSt}\bigl(
        \GHSPM^{[s]}_\tau(G_n), 
        \GHSPM_\tau(G_n)
      \bigr)
      > \varepsilon
    \right)
    = 0,
    \quad 
    \forall \varepsilon > 0.
  \end{equation}

  We simply write $\SPM_n \coloneqq \SPM_{(R_n, m_n)}$.
  Fix $r > 1$ and $r_0 > 0$ with $r-1 > r_0$,
  and take $s \in (r-1, r)$.
  By Lemma~\ref{lem: simple estimate of GH distance}, we have  
  \begin{equation} \label{pr eq: 3. rdm conv for SPM}
    \GFMet^{\tau \times \SPMSt} 
    \bigl(
      \GHSPM_\tau(G_n), 
      \GHSPM^{[s]}_\tau(G_n)
    \bigr)
    \leq    
    \hatCMet{(F_n, \rho_n)}{\Prob(D_{J_1}(\RNp, F_n))}
    \bigl(
      \SPM_n, 
      \SPM_n^{[s]}
    \bigr).
  \end{equation}
  Since $s > r_0$, it holds that 
  \begin{equation}
    \dom(\SPM_n)^{(r_0)}
    =
    \dom(\SPM_n^{[s]})^{(r_0)} 
    = 
    F_n^{(r_0)}.
  \end{equation}
  Then Lemma~\ref{lem: hatC metric simple estimate} yields 
  \begin{equation}
    \hatCMet{(F_n, \rho_n)}{\Prob(D_{J_1}(\RNp, F_n))}
    \bigl(
      \SPM_n, 
      \SPM_n^{[s]}
    \bigr)
    \leq
    e^{-r_0}  
    + 
    \sup_{x \in F_n^{(r_0)}} 
    \ProhMet{D_{J_1}(\RNp, F_n)}
    \bigl(
      \SPM_n(x), 
      \SPM_n^{[s]}(x)
    \bigr).
  \end{equation}
  By the proof of \cite[Lemma~6.7]{Noda_pre_Aging},
  for any $x \in F_n^{(r_0)}$ and $T > 0$,
  \begin{align}
    \ProhMet{D_{J_1}(\RNp, F_n)}
    \bigl( \SPM_n(x), \SPM_n^{[s]}(x) \bigr) 
    &\leq 
    e^{-T} 
    + 
    P_n^x\!\left( \exitTime_{B_{R_n}(x, s/2)} \leq T \right) \notag\\
    &\leq    
    e^{-T} 
    + 
    P_n^x\!\left( \exitTime_{B_{R_n}(x, (r-1)/2)} \leq T \right),
  \end{align}
  where $((X_n(t))_{t \geq 0}, (P_n^x)_{x \in F_n})$ denotes the process associated with $(R_n, m_n)$,
  and $\exitTime_\cdot$ is the first exit time defined in \eqref{eq: dfn of exit time}.
  Taking the supremum over $x \in F_n^{(r_0)}$ gives
  \begin{equation}
    \sup_{x \in F_n^{(r_0)}}
    \ProhMet{D_{J_1}(\RNp, F_n)}
    \bigl( \SPM_n(x), \SPM_n^{[s]}(x) \bigr) 
    \leq 
    e^{-T} 
    + 
    \sup_{x \in F_n^{(r_0)}}
    P_n^x\!\left( \breve{\sigma}_{B_{R_n}(x, (r-1)/2)} \leq T \right).
  \end{equation}
  Combining this with \eqref{pr eq: 3. rdm conv for SPM}, we have,
  for any $r-1 > r_0$, $T > 0$, and $s \in (r-1, r)$,
  \begin{equation}
    \GFMet^{\tau \times \SPMSt} 
    \bigl( \GHSPM_\tau(G_n), \GHSPM^{[s]}_\tau(G_n) \bigr)
    \leq 
    e^{-r_0}
    + e^{-T} 
    + 
    \sup_{x \in F_n^{(r_0)}}
    P_n^x\!\left( \breve{\sigma}_{B_{R_n}(x, (r-1)/2)} \leq T \right).
  \end{equation}
  Taking the supremum over $s \in (r-1, r)$ yields
  \begin{equation}  \label{pr eq: 4. rdm conv for SPM}
    \sup_{r-1 < s < r}
    \GFMet^{\tau \times \SPMSt} 
    \bigl( \GHSPM_\tau(G_n), \GHSPM^{[s]}_\tau(G_n) \bigr)
    \leq 
    e^{-r_0}
    + e^{-T} 
    + 
    \sup_{x \in F_n^{(r_0)}}
    P_n^x\!\left( \breve{\sigma}_{B_{R_n}(x, (r-1)/2)} \leq T \right).
  \end{equation}

  Finally, under the non-explosion condition 
  \ref{thm item: rdm conv for SPM. 2},
  by following \cite[Proof of Lemma~6.3]{Noda_pre_Scaling},
  one can verify that for each $r_0 > 0$ and $T > 0$,
  \begin{equation}
    \lim_{r \to \infty} 
    \limsup_{n \to \infty} 
    \mathbf{P}_n  
    \!\left(
      \sup_{x \in F_n^{(r_0)}} 
      P_n^x
      \!\left(
        \breve{\sigma}_{B_{R_n}(x, (r-1)/2)} \leq T
      \right) 
      > \varepsilon
    \right)
    = 0,
    \quad 
    \forall \varepsilon > 0.
  \end{equation}
  Combining this with \eqref{pr eq: 4. rdm conv for SPM} proves 
  \eqref{pr eq: 2. rdm conv for SPM},
  and hence \eqref{pr eq: 1. rdm conv for SPM}.
  This completes the proof.
\end{proof}


\subsection{Local limit theorem} \label{sec: Local limit theorem}

In this subsection, we establish a local limit theorem (see Theorem \ref{cor: dtm conv of SPM and heat kernels} below),
that is, the convergence of heat kernels,
under the assumption that measured resistance metric spaces 
and the law maps of the associated stochastic processes converge.
Throughout this subsection, we fix a Polish structure $\tau$.

Below, we modify the map $\GHSPM_\tau$ introduced in Definition~\ref{dfn: GHSPM}
to incorporate the heat kernel.
Recall the structure $\HKSt$ from \eqref{eq: st for hk}.

\begin{dfn} \label{dfn: GHHKSPM}
  For each $G = (F, R, \rho, m, a) \in \rootResisSp(\tau)$,
  we set  
  \begin{equation}
    \GHHKSPM_\tau(G) 
    \coloneqq 
    \bigl( F, R, \rho, m, a, p_{(R, m)}, \SPM_{(R, m)} \bigr),
  \end{equation}
  which is an element of $\rootResisSp(\tau \times \HKSt \times \SPMSt)$.
\end{dfn}

We recall from \cite{Noda_pre_Aging} that the convergence of measured resistance metric spaces 
implies the precompactness of the associated heat kernels.

\begin{lem} [{\cite[Proposition~4.13]{Noda_pre_Aging}}]\label{lem: precompactness of heat kernels}
  Assume that $G_n = (F_n, R_n, \rho_n, m_n)$, $n \in \NN$, and $G = (F, R, \rho, m)$ are elements of $\rootResisSp$
  such that $G_n \to G$ in $\rootResisSp$.
  We simply write $p_n = p_{(R_n, m_n)}$.
  Then, for any $r > 0$ and $T > 0$,
  \begin{equation}
    \sup_{n \geq 1}
    \sup_{T \leq t < \infty} 
    \sup_{x,y \in F_{n}^{(r)}}
    p_{n}(t, x, y) < \infty,
    \qquad
    \lim_{\delta \to 0} 
    \sup_{n \geq 1}
    w(p_{n}, T, r, \delta) 
    = 0,
  \end{equation}
  where $w(p_{n}, T, r, \delta)$ is defined to be 
  \begin{equation}
    \sup\left\{
      |p_{n}(t,x,y) - p_{n}(t', x', y')|\,
      \middle | 
      \begin{array}{l}
        (t,x,y), (t', x', y') \in [T, \infty) \times F_{n}^{(r)} \times F_{n}^{(r)}\\
        |t-t'| \vee R_{n}(x,x') \vee R_{n}(y, y') \leq \delta
      \end{array}
    \right\}.
  \end{equation}
\end{lem}

We now present the main result of this subsection,  
which states that the precompactness of heat kernels  
and the convergence of law maps of stochastic processes
together imply the convergence of the heat kernels.  
A similar result for simple random walks on graphs  
can be found in \cite{Croydon_Hambly_08_local}.

\begin{thm} \label{thm: local limit theorem}
  Let $G_n = (F_n, R_n, \rho_n, m_n, a_n),\, n \in \NN$ and $G = (F, R, \rho, m, a)$ are elements of $\rootResisSp(\tau)$.
  If $\GHSPM_\tau(G_n) \to \GHSPM_\tau(G)$ in $\rootResisSp(\tau \times \SPMSt)$ as $n \to \infty$,
  then, as $n \to \infty$,
  \begin{equation}
    \GHHKSPM_\tau(G_n) \to \GHHKSPM_\tau(G) \quad \text{in}\quad \rootResisSp(\tau \times \HKSt \times \SPMSt).
  \end{equation}
  
\end{thm}

\begin{proof}
  We simply write $p_n = p_{(R_n, m_n)}$ and $\SPM_n = \SPM_{(R_n, m_n)}$ for $n \in \NN$,
  and $p = p_{(R, m)}$ and $\SPM = \SPM_{(R, m)}$.
  We let $((X_n(t))_{t \geq 0}, (P_n^x)_{x \in F_n})$ 
  and $((X(t))_{t \geq 0}, (P^x)_{x \in F})$ 
  be the processes associated with $(R_n, m_n)$ and $(R, m)$, respectively.
  By Theorem~\ref{thm: conv in M(tau)},
  we may think that $(F_n, R_n, \rho_n)$ and $(F, R, \rho)$ are embedded 
  isometrically into a common rooted $\bcmAB$ space $(M, d_M, \rho_M)$
  in such a way that 
  $\rho_{n} = \rho = \rho_{M}$ as elements of $M$,
  $F_n \to F$ in the Fell topology as closed subsets of $M$,
  $m_n \to m $ vaguely as measures on $M$,
  $a_n \to a$ in $\tau(M)$,
  and
  $\Upsilon_n \to \Upsilon$ in $\hatC(M, \Prob(D_{J_1}(\RNp, M)))$.
  We deduce from Lemma~\ref{lem: precompactness of heat kernels} and \cite[Theorem~3.31]{Noda_pre_Metrization}
  that the collection $\{p_{n}\}_{n \geq 1}$ is precompact in $\hatC(\RNpp \times M \times M, \RNp)$.
  Thus, it is enough to show that the limit of any convergent subsequence in $\{p_{n}\}_{n \geq 1}$ coincides with $p$.
  To avoid notational complexity, we assume that the full sequence $p_{n}$ converges to $q$ in $\hatC(\RNpp \times M \times M, \RNp)$.
  Since the domain $\dom(p_{n}) = \RNpp \times F_n \times F_n$ converges to $\RNpp \times F \times F$ in the Fell topology,
  it follows from Lemma~\ref{lem: conv in hatC} that $\dom(q) = \RNpp \times F \times F$.
  Fix $t \in (0, \infty)$ and $x \in F$.
  By the convergence of $F_n $ to $F$,
  we can choose $x_n \in F_n$ such that $x_n \to x$ in $M$.
  It then follows from the convergence of $\Upsilon_n$ to $\Upsilon$ that  
  \begin{equation}
    \Upsilon_n(x_n) = P_n^{x_n}((X_{n}(t))_{t \geq 0} \in \cdot) 
    \to
    \Upsilon(x) = P^x((X(t))_{t \geq 0} \in \cdot)
  \end{equation} 
  weakly as probability measures on $D_{J_1}(\RNp, M)$.
  By its quasi-left continuity, 
  $X$ is continuous at each deterministic time $t \geq 0$, $P^x$-a.s.
  Thus, the above distributional convergence in $D_{J_1}(\RNp, M)$ implies that 
  $P_n^{x_n}(X_{n}(t) \in \cdot) \xrightarrow{\mathrm{d}} P^x(X(t) \in \cdot)$ for each $t \geq 0$.
  This yields that, for any compactly-supported continuous function $f$ on $M$, 
  \begin{equation}
    \lim_{n \to \infty} 
    \int_{F_n} p_{n}(t,x_n,y) f(y)\, m_n(dy) 
    = 
    \lim_{n \to \infty}
    E_{n}^{x_n}[f(X_{t})] 
    = 
    E^{x}[f(X_{t})]
    = 
    \int_{F} p(t, x, y) f(y)\, m(dy).
  \end{equation}
  On the other hand,
  since $p_{n} \to q$ in $\hatC(\RNpp \times M \times M, \RNp)$ and $f(y)\,m_n(dy) \to f(y)\,m(fy)$ weakly,
  we deduce from Lemma~\ref{lem: vague convergence and hatC topology} that 
  \begin{equation}
    \lim_{n \to \infty} 
    \int_{F_n} p_{n}(t,x_n,y) f(y)\, m_n(dy)  
    = 
    \int_{F} q(t,x,y) f(y)\, m(dy).
  \end{equation}
  Therefore,
  for all compactly-supported continuous functions $f$ on $M$,
  \begin{equation}
    \int_{F} p(t, x, y) f(y)\, m(dy) 
    = 
    \int_{F} q(t,x,y) f(y)\, m(dy).
  \end{equation}
  In particular,
  $p(t, x, \cdot) = q(t, x, \cdot)$, $m$-a.e.
  Since $m$ is of full support on $F$ and both $p(t, x, \cdot)$ and $q(t, x, \cdot)$ are continuous functions on $F$,
  it follows that $p(t, x, y) = q(t, x, y)$ for all $y \in F$.
  Hence, we conclude that $p = q$, which completes the proof.
\end{proof}

Below, we present a result analogous to Proposition~\ref{prop: measurability of SPM wrt GH topology}.  
This enables us to discuss the convergence of heat kernels of stochastic processes on random resistance metric spaces.  

\begin{prop} \label{prop: measurability of gh hk map}
  The map $\GHHKSPM_\tau \colon \rootResisSp(\tau) \to \rootResisSp(\tau \times \HKSt \times \SPMSt)$ is Borel measurable.
  Moreover, its image, that is,
  \begin{equation}
    \left\{ \bigl( F, R, \rho, m, a, p_{(R,m)}, \SPM_{(R,m)} \bigr)\, \middle|\, (F, R, \rho, m, a) \in \rootResisSp(\tau) \right\}
  \end{equation}
  is a Borel subset of $\rootBCM(\MeasSt \times \tau \times \HKSt \times \SPMSt)$.
\end{prop}

\begin{proof}
  If we define a map $\tilde{\GHSPM}^{\mathrm{HK}}_\tau \colon \GHSPM_\tau(\rootResisSp(\tau)) \to \rootResisSp(\tau \times \HKSt \times \SPMSt)$
  by setting 
  $\tilde{\GHSPM}^{\mathrm{HK}}_\tau( \GHSPM_{\tau}(G) ) \coloneqq \GHHKSPM_\tau(G)$,
  then $\tilde{\GHSPM}^{\mathrm{HK}}$ is continuous by Theorem \ref{thm: local limit theorem}.
  Moreover, by Proposition~\ref{prop: measurability of SPM wrt GH topology},
  the map $\GHSPM_\tau \colon \rootResisSp(\tau) \to \GHSPM_\tau(\rootResisSp(\tau))$ is Borel measurable 
  (note that the codomain is restricted to its image).
  Since we have that $\GHHKSPM_\tau = \tilde{\GHSPM}^{\mathrm{HK}}_\tau \circ \GHSPM_\tau$,
  we obtain the first assertion.
  The second claim can be verified in the same way as in Proposition~\ref{prop: measurability of SPM wrt GH topology}.
\end{proof}
 
\begin{cor} \label{cor: resis sp with hk is Borel}
  The set 
  \begin{equation}
    \left\{ \bigl( F, R, \rho, m, a, p_{(R,m)} \bigr)\, \middle|\, (F, R, \rho, m, a) \in \rootResisSp(\tau) \right\}
  \end{equation}
  is a Borel subset of $\rootBCM(\MeasSt \times \tau \times \HKSt)$.
\end{cor}

\begin{proof}
  This is an immediate consequence of Lemma~\ref{lem: Lousin--Souslin} and Proposition~\ref{prop: measurability of SPM wrt GH topology}.
\end{proof}

By Theorem~\ref{thm: local limit theorem},
the results of Theorems~\ref{thm: dtm conv of SPM} and \ref{thm: rdm conv for SPM} are extended
to the convergence of not only the law maps of stochastic processes but also the associated heat kernels,
as shown below.

\begin{cor} \label{cor: dtm conv of SPM and heat kernels}
  Under the same conditions as Theorem~\ref{thm: dtm conv of SPM},
  it holds that $\GHHKSPM_\tau(G_n) \to \GHHKSPM_\tau(G)$ in $\rootResisSp(\tau \times \HKSt \times \SPMSt)$.
\end{cor}

\begin{proof}
  This is immediate from Theorems~\ref{thm: dtm conv of SPM} and \ref{thm: local limit theorem}.
\end{proof}

\begin{cor} \label{cor: rdm conv of SPM and heat kernels}
  Under the same conditions as Theorem~\ref{thm: rdm conv for SPM},
  it holds that $\GHHKSPM_\tau(G_n) \xrightarrow{\mathrm{d}} \GHHKSPM_\tau(G)$ as random elements of $\rootResisSp(\tau \times \HKSt \times \SPMSt)$.
\end{cor}

\begin{proof}
  Recall the continuous map $\tilde{\GHSPM}^{\mathrm{HK}}_\tau \colon \GHSPM_\tau(\rootResisSp(\tau)) \to \rootResisSp(\tau \times \HKSt \times \SPMSt)$
  from the proof of Proposition~\ref{prop: measurability of gh hk map}.
  By Theorem~\ref{thm: rdm conv for SPM}, we have that 
  $\GHSPM_\tau(G_n) \xrightarrow{\mathrm{d}} \GHSPM_\tau(G)$ as random elements of $\rootResisSp(\tau \times \SPMSt)$.
  Thus, we deduce from the continuity of $\tilde{\GHSPM}^{\mathrm{HK}}_\tau$ that 
  $\tilde{\GHSPM}^{\mathrm{HK}}_\tau( \GHSPM_{\tau}(G_n) ) = \GHHKSPM_\tau(G_n)$
  converges in distribution to $\tilde{\GHSPM}^{\mathrm{HK}}_\tau( \GHSPM_{\tau}(G) ) = \GHHKSPM_\tau(G)$,
  which completes the proof.
\end{proof}

\section{Collisions of stochastic processes on resistance metric spaces} \label{sec: col of proc on resis sp}

In Section~\ref{sec: Local limit theorem}, 
we showed that if measured resistance metric spaces converge and satisfy the non-explosion condition,
then the law maps of the associated stochastic processes and heat kernels also converge. 
Building on this, and by applying Theorems~\ref{thm: col dtm result} and \ref{thm: col rdm result}, 
we establish convergence results for collision measures of stochastic processes on resistance metric spaces in this section.

\subsection{Resistance metric spaces for the study of collisions} \label{sec: sp for collision}

In this subsection, we introduce a space of resistance metric spaces and verify its measurability.
This space will be used in the main results of the next subsection.

\begin{dfn} \label{dfn: resis col sp}
  We define $\ColResisSp$ to be the collection of $(F, R, \rho, m, \mu) \in \rootResisSp(\MeasSt)$ 
  such that 
  \begin{equation} \label{dfn eq: resis col sp}
    \lim_{\delta \to 0} 
    \sup_{x_1, x_2 \in F^{(r)}} 
    \int_0^\delta \int_{F^{(r)}} p_{(R,m)}(t, x_1, y) p_{(R,m)}(t, x_2, y)\, \mu(dy)\, dt 
    = 0,
    \quad 
    \forall r > 0.
  \end{equation}
  More generally,
  given a Polish structure $\tau$,
  we write $\ColResisSp(\tau)$ for the collection of $(F, R, \rho, m,\allowbreak \mu,\allowbreak a) \in \rootResisSp(\MeasSt \times \tau)$
  such that $(F, R, \rho, m, \mu) \in \ColResisSp$.
  We always equip $\ColResisSp(\tau)$ with the relative topology 
  induced from $\rootBCM(\MeasSt^{\otimes 2} \times \tau)$.
\end{dfn}

Let $(F, R, \rho, m, \mu) \in \ColResisSp$, $X^1$ be the associated process, and $X^2$ be an independent copy of $X^1$.
Write $\hat{X} = (X^1, X^2)$ for their product process.
Condition~\eqref{dfn eq: resis col sp} implies that the Radon measure $\diagMeas{\mu}$ on $F \times F$
belongs to the local Kato class of $\hat{X}$ 
(see Definition~\ref{dfn: Kato class}).
Hence, the collision measure of $X^1$ and $X^2$ associated with $\mu$ is well defined
(see Definition~\ref{dfn: col meas}).
We write 
\begin{equation}
  \hatSPMcol(\bm{x}) 
  = 
  \hatSPMcol_{(R, m, \mu)}(\bm{x})
  \coloneqq 
  \ProcLaw_{(\hat{X}, \Pi)}(\bm{x}),
  \quad 
  \bm{x} \in F \times F.
\end{equation}
The following is immediate from Proposition~\ref{prop: continuity of Col map}.

\begin{lem} \label{lem: continuity of col map for resis sp}
  For each $(F, R, \rho, m, \mu) \in \rootResisSp$, 
  the map 
  \begin{equation}
      \hatSPMcol_{(R, m, \mu)} \colon F \times F \to \Prob\bigl( D_{J_1}(\RNp, F) \times D_{J_1}(\RNp, F) \times \Meas(F \times \RNp)\bigr)
  \end{equation}
  is continuous.
\end{lem}

Thus, the space $\ColResisSp(\tau)$ forms a suitable class of resistance metric spaces for the study of collisions.
Below, we verify its measurability.

\begin{prop} \label{prop: col resis sp is Borel subset}
  The set $\ColResisSp(\tau)$ is a Borel subset of $\rootBCM(\MeasSt^{\otimes 2} \times \tau)$.
\end{prop}

\begin{proof}
  Define 
  \begin{gather}
    \mathfrak{G}_1
    \coloneqq 
    \bigl\{ (F,R,\rho,m,\mu,a,p_{(R,m)}) \bigm| (F,R,\rho,m,\mu,a) \in \rootResisSp(\MeasSt \times \tau) \bigr\},\\
    \mathfrak{G}_2
    \coloneqq 
    \bigl\{ (F,R,\rho,m,\mu,a,p_{(R,m)}) \bigm| (F,R,\rho,m,\mu,a) \in \ColResisSp(\tau) \bigr\}.
  \end{gather}
  Since, given a measured resistance metric space $(F, R, m)$, the associated heat kernel $p_{(R,m)}$ is unique,
  the following map is injective:
  \begin{equation}
    \mathfrak{G}_2 \ni (F,R,\rho,m,\mu,a,p_{(R,m)}) \longmapsto (F,R,\rho,m,\mu,a) \in \rootBCM(\MeasSt^{\otimes 2} \times \tau).
  \end{equation}
  Moreover, by Theorem~\ref{thm: conv in M(tau)}, the above map is continuous.
  Thus, once we establish that $\mathfrak{G}_2$ is a Borel subset of $\rootBCM(\MeasSt^{\otimes 2} \times \tau \times \HKSt)$,
  we can use Lemma~\ref{lem: Lousin--Souslin} to obtain the desired result.
  Since $\mathfrak{G}_1$ is a Borel subset of $\rootBCM(\MeasSt^{\otimes 2} \times \tau \times \HKSt)$ by Corollary~\ref{cor: resis sp with hk is Borel},
  it is enough to show that $\mathfrak{G}_2$ is a Borel subset of $\mathfrak{G}_1$.

  For each $k, l, N \in \NN$, we define $\mathfrak{G}^{(k,l,N)}(\tau)$ to be 
  the collection of $(F, R, \rho, m, \mu, a, p_{(R,m)}) \in \mathfrak{G}_1$
  such that,
  for all but countably many $r \in [k, k+1]$,
  \begin{equation} \label{pr eq: 1. col resis sp is Borel subset}
    \sup_{x_1, x_2 \in F^{(r)}} 
    \int_0^{1/N} \int_{F^{(r)}} p_{(R,m)}(t, x_1, y) p_{(R,m)}(t, x_2, y)\, \mu(dy)\, dt 
    \leq 1/l.
  \end{equation}
  Since it holds that 
  \begin{equation}  \label{pr eq: 2. col resis sp is Borel subset}
    \mathfrak{G}_2 = \bigcap_{k \geq 1} \bigcap_{l \geq 1} \bigcup_{N \geq 1} \mathfrak{G}^{(k,l,N)}(\tau).
  \end{equation}
  it is enough to show that each $\mathfrak{G}^{(k,l,N)}(\tau)$ is a Borel subset of $\mathfrak{G}_1$.
  We do this by proving that $\mathfrak{G}^{(k,l,N)}(\tau)$ is closed in $\mathfrak{G}_1$.
  Fix $G_n = (F_n, R_n, \rho_n, m_n, \mu_n, a_n, p_n) \in \mathfrak{G}^{(k,l,N)}(\tau)$, $n \in \NN$, such that 
  \begin{equation}
    G_n = (F_n, R_n, \rho_n, m_n, \mu_n, a_n, p_n)
    \to 
    G = (F, R, \rho, m, \mu, a, p)
  \end{equation}
  in $\mathfrak{G}_1$.
  Here, we simply write $p_n = p_{(R_n, m_n)}$, $n \in \NN$, and $p = p_{(R, m)}$.
  By Theorem~\ref{thm: conv in M(tau)},
  we may think that $(F_n, R_n, \rho_n)$ and $(F, R, \rho)$ are embedded 
  isometrically into a common rooted $\bcmAB$ space $(M, d_M, \rho_M)$
  in such a way that 
  $\rho_n = \rho = \rho_M$ as elements of $M$,
  $F_n \to F$ in the Fell topology as closed subsets of $M$,
  $m_n \to m$ and $\mu_n \to \mu$ vaguely as measures on $M$,
  $a_n \to a$ in $\tau(M)$,
  and
  $p_n \to p$ in $\hatC(\RNpp \times M \times M, \RNp)$.
  Let $r \in [k, k+1]$ be such that 
  $F_{n}^{(r)} \to F^{(r)}$ in the Hausdorff topology in $M$,
  $\mu_{n}^{(r)} \to \mu^{(r)}$ weakly in $M$,
  and all the pairs $(\mu_n, p_n)$ satisfy \eqref{pr eq: 1. col resis sp is Borel subset} with radius $r$.
  Fix $x, y \in F^{(r)}$.
  By the convergence of $F_{n}^{(r)}$ to $F^{(r)}$,
  we can find $x_{n}, y_{n} \in F_{n}^{(r)}$ such that 
  $x_{n} \to x$ and $y_{n} \to y$ in $M$.
  Since $p_{n}$ converges to $p$, it holds that 
  \begin{equation}
    p_n(t, x_n, z) p_{n}(t,y_{n},z)
    \to
    p(t, x, z) p(t,y,z)
  \end{equation}
  in $\hatC(\RNpp \times M, \RNp)$ with respect to $(t,z) \in \RNpp \times M$.
  This leads to the following vague convergence as measures on $\RNpp \times M$:
  \begin{equation}
    p_n(t, x_n, z) p_{n}(t,y_{n},z)\, \mu_{n}^{(r)}(dz)\, dt
    \to
    p(t, x, z) p(t,y,z)\, \mu^{(r)}(dz)\, dt.
  \end{equation}
  By the Portmanteau theorem for vague convergence (see \cite[Lemma~4.1]{Kallenberg_17_Random}),
  it follows that, for any $\varepsilon \in (0, 1/N)$,
  \begin{align}
    \int_{\varepsilon}^{1/N} \int_F  p(t, x, z) p(t,y,z)\, \mu^{(r)}(dz)\, dt
    &\leq    
    \liminf_{n \to \infty}     
    \int_{\varepsilon}^{1/N} \int_{F_n} p_n(t, x_n, z) p_n(t,y_{n},z)\, \mu_n^{(r)}\,(dz) dt\\
    &\leq 
    1/l,
    \label{prop pr eq: 1. col resis sp is Borel subset}
  \end{align}
  where the last inequality follows from \eqref{pr eq: 1. col resis sp is Borel subset}.
  Letting $\varepsilon \to 0$ and taking the supremum over $x,y \in F^{(r)}$,
  we obtain that 
  \begin{equation}  
    \sup_{x,y \in F^{(r)}}
    \int_0^{1/N} \int_{F^{(r)}}  p(t, x, z) p(t,y,z)\, \mu(dz)\, dt
    \leq 1/l.
  \end{equation}
  This implies that $G \in \mathfrak{G}^{(k,l,N)}(\tau)$,
  which completes the proof.
\end{proof}

\subsection{Convergence of collision measures of stochastic processes on resistance metric spaces} 
\label{sec: conv of col meas in resis sp}

The main results of this subsection are stated in Theorems~\ref{thm: col dtm resis sp} and \ref{thm: col rdm resis sp}, 
which establish the convergence of collision measures of stochastic processes on resistance metric spaces
varying along a sequence,
as an application of Theorems~\ref{thm: col dtm result} and \ref{thm: col rdm result}.
Throughout this subsection, we fix a Polish structure~$\tau$.

We first introduce the map of interest.
Recall the structure $\ColSt$ defined in \eqref{eq: st for col meas}.

\begin{dfn}
  For each $G = (F, R, \rho, m, \mu, a) \in \ColResisSp(\tau)$,
  we set  
  \begin{equation}
    \GHColSPM_\tau(G) 
    \coloneqq 
    \Bigl( F, R, \rho, m, \mu, a, p_{(R, m)}, \hatSPMcol_{(R, m, \mu)} \Bigr),
  \end{equation}
  which is an element of $\ColResisSp(\tau \times \HKSt \times \ColSt)$.
\end{dfn}

We first state a convergence result for deterministic resistance metric spaces.

\begin{assum} \label{assum: col dtm resis sp}
  Let $G_n = (F_n, R_n, \rho_n, m_n, \mu_n, a_n)$, $n \in \NN$, be elements of $\ColResisSp(\tau)$,
  and $G = (F, R, \rho, m, \mu, a)$ be an element of $\rootBCM(\MeasSt^{\otimes 2} \times \tau)$
  such that $m$ is of full support.
  Assume that the following conditions hold.
  \begin{enumerate} [label = \textup{(\roman*)}]
    \item \label{assum item: 1. col dtm resis sp}
      The sequence satisfies $G_n \to G$ in $\rootBCM(\MeasSt^{\otimes 2} \times \tau)$.
    \item \label{assum item: 2. col dtm resis sp}
      We have $\displaystyle \lim_{r \to \infty} \liminf_{n \to \infty} R_{n}(\rho_n, B_{R_n}(\rho_n, r)^{c}) = \infty$.
    \item \label{assum item: 3. col dtm resis sp}
      For all $r > 0$,
      \begin{equation}
        \lim_{\delta \to 0} 
        \limsup_{n \to \infty} 
        \sup_{x_1, x_2 \in F_n^{(r)}}
        \int_0^\delta \int_{F_n^{(r)}} p_n(t, x_1, y)\, p_n(t, x_2, y)\, \mu_n(dy)\, dt 
        = 0,
      \end{equation}
      where $p_n \coloneqq p_{(R_n, m_n)}$.
  \end{enumerate}
\end{assum}
 
\begin{thm} \label{thm: col dtm resis sp}
  Under Assumption~\ref{assum: col dtm resis sp}, it holds that $G \in \ColResisSp(\tau)$,
  and $\GHColSPM_\tau(G_n) \to \GHColSPM_\tau(G)$ in $\ColResisSp(\tau \times \HKSt \times \ColSt)$.
\end{thm}

\begin{proof}
  By Theorem~\ref{thm: dtm conv of SPM},
  $G \in \rootResisSp(\MeasSt \times \tau)$.
  Moreover, we deduce from \eqref{prop pr eq: 1. col resis sp is Borel subset} and Assumption~\ref{assum: col dtm resis sp}\ref{assum item: 3. col dtm resis sp} 
  that, for all $r > 0$,
  \begin{equation}
    \lim_{\delta \to 0} 
    \sup_{x,y \in F^{(r)}}
    \int_0^\delta \int_{F^{(r)}} p(t,x,z)p(t,y,z)\, \mu(dz)\, dt = 0.
  \end{equation}
  Thus, $G \in \ColResisSp(\tau)$.
  Now the desired convergence follows from Theorem~\ref{thm: col dtm result} and Corollary~\ref{cor: dtm conv of SPM and heat kernels}.
\end{proof}

We next provide a version of Theorem~\ref{thm: col dtm result} for random resistance metric spaces.

\begin{assum} \label{assum: col rdm resis sp}
  Let $G_n = (F_n, R_n, \rho_n, m_n, \mu_n, a_n)$, $n \in \NN$, be random elements of $\ColResisSp(\tau)$,
  and $G = (F, R, \rho, m, \mu, a)$ be an random element of $\rootBCM(\MeasSt^{\otimes 2} \times \tau)$
  such that $m$ is of full support, almost surely.
  Assume that the following conditions are satisfied.
  \begin{enumerate} [label = \textup{(\roman*)}]
    \item \label{assum item: 1. col rdm resis sp}
      It holds that $G_n \xrightarrow{\mathrm{d}} G$ in $\rootBCM(\MeasSt^{\otimes 2} \times \tau)$.
    \item \label{assum item: 2. col rdm resis sp}
      If we write $\mathbf{P}_n$ for the underlying probability measure of $G_n$, 
      then, for all $\lambda > 0$,
      \begin{equation}
        \lim_{r \to \infty} \liminf_{n \to \infty} 
        \mathbf{P}_{n}\bigl( R_{n}(\rho_{n}, B_{R_{n}}(\rho_{n}, r)^{c}) > \lambda \bigr) = 1.
      \end{equation}
    \item \label{assum item: 3. col rdm resis sp}
      For all $r > 0$,
      \begin{equation}
        \lim_{\delta \to 0} 
        \limsup_{n \to \infty} 
        \mathbf{E}_n\!\left[
          \left(
            \sup_{x_1, x_2 \in F_n^{(r)}}
            \int_0^\delta \int_{F_n^{(r)}} p_n(t, x_1, y)\, p_n(t, x_2, y)\, \mu_n(dy)\, dt 
          \right)
          \wedge 1
        \right]
        = 0,
      \end{equation}
      where we write $p_n \coloneqq p_{(R_n, m_n)}$.
  \end{enumerate}
\end{assum}

\begin{thm} \label{thm: col rdm resis sp}
  Under Assumption~\ref{assum: col rdm resis sp}, it holds that $G \in \ColResisSp(\tau)$ almost surely,
  and $\GHColSPM_\tau(G_n) \to \GHColSPM_\tau(G)$ in $\ColResisSp(\tau \times \HKSt \times \ColSt)$.
\end{thm}

\begin{proof}
  We have from Theorem~\ref{thm: rdm conv for SPM} that 
  $G \in \rootResisSp(\MeasSt \times \tau)$ almost surely.
  By the Skorohod representation theorem,
  we may assume that $G_n \to G$ almost surely on some probability space.
  Denote by $\mathbb{P}$ the underlying probability measure.
  From Fatou's lemma and Assumption~\ref{assum: col rdm resis sp}\ref{assum item: 3. col rdm resis sp},
  we deduce that, for all $r > 0$ and $\varepsilon > 0$,
  \begin{equation}
    \mathbb{E}\!
    \left[\lim_{\delta \to 0} 
      \liminf_{n \to \infty} 
      \left(
        \sup_{x_1, x_2 \in F_n^{(r)}}
        \int_0^\delta \int_{F_n^{(r)}} 
        p_n(t, x_1, y)\, p_n(t, x_2, y)\, \mu_n(dy)\, dt 
      \right)
      \wedge 1
    \right]
    = 0.
  \end{equation}
  This implies that 
  \begin{equation}
    \liminf_{n \to \infty} 
    \sup_{x_1, x_2 \in F_n^{(r)}}
    \int_0^\delta \int_{F_n^{(r)}} 
      p_n(t, x_1, y)\, p_n(t, x_2, y)\, \mu_n(dy)\, dt 
    = 0,
    \quad 
    \mathbb{P}\text{-a.s.}
  \end{equation}
  Therefore, by the same argument as in
  \eqref{prop pr eq: 1. col resis sp is Borel subset},
  we obtain
  \begin{equation}
    \lim_{\delta \to 0} 
    \sup_{x,y \in F^{(r)}}
    \int_0^\delta \int_{F^{(r)}} 
      p(t,x,z)\, p(t,y,z)\, \mu(dz)\, dt 
    = 0,
    \quad 
    \mathbb{P}\text{-a.s.}
  \end{equation}
  Hence, $G \in \ColResisSp(\tau)$ almost surely.
  Now the desired result follows from
  Theorem~\ref{thm: col rdm result}
  and Corollary~\ref{cor: rdm conv of SPM and heat kernels}.
\end{proof}

\subsection{Scaling limits of collision measures of VSRWs} 
\label{sec: conv of col meas of VSRW}

The heat kernel conditions in 
Assumption~\ref{assum: col dtm resis sp}\ref{assum item: 3. col dtm resis sp} 
and Assumption~\ref{assum: col rdm resis sp}\ref{assum item: 3. col rdm resis sp} 
can be verified through uniform lower bounds on the volumes of balls.
In this subsection, we focus on uniform collision measures of VSRWs on electrical networks
and provide sufficient lower-volume conditions for these heat kernel assumptions;
see Theorems~\ref{thm: dtm col meas for VSRW} and \ref{thm: rdm col meas for VSRW} below.
These results apply to a broad class of low-dimensional fractals and critical random graphs,
as discussed in Examples~\ref{exm: dtm VSRW} and \ref{exm: rdm VSRW} below.

We first clarify the setting for the main results.
For each $n \geq 1$, let $(V_n, E_n, \mu_n)$ be a recurrent electrical network,
as recalled from Definition~\ref{dfn: electrical network}.
Write $R_n$ for the associated resistance metric and $m_n$ for the counting measure on $V_n$,
i.e., $m(\{x\}) = 1$ for each $x \in V_n$.
Let $a_n, b_n > 0$ be constants, serving as scaling factors for the metric and measure, respectively.
Let $X_n^1$ be the process associated with $(a_n^{-1} R_n, b_n^{-1}m_n)$ and $X_n^2$ be an independent copy of $X_n^1$.
In other words, $X_n^1$ and $X_n^2$ are i.i.d.\ VSRWs on the electrical network,
rescaled in time by the factor $a_n b_n$.
We write $\hat{X}_n = (X_n^1, X_n^2)$ for the product process.
Let $\Pi_n$ denote the collision measure of $X_n^1$ and $X_n^2$ associated with $b_n^{-1}m_n$, i.e.,
\begin{equation} \label{eq: dfn of unif col meas for VSRW}
  \Pi_n(dx\, dt)
  = 
  b_n \sum_{y \in V_n} 
  \mathbf{1}_{\{X_n^1(t) = X_n^2(t) = y\}}\,
  \delta_y(dx)\, dt.
\end{equation}
This identity follows from Proposition~\ref{prop: representation of collision measure for discrete space}.
In particular, the scaling factor $b_n$ arises from the following computation:
\begin{equation}
  \frac{b_n^{-1}m_n(\{x\})}{(b_n^{-1} m_n(\{x\}))^2}
  = 
  \frac{b_n^{-1}}{b_n^{-2}}
  = 
  b_n,
  \quad \forall\, x \in V_n.
\end{equation}

\begin{assum}  \label{assum: dtm col meas for VSRW}\leavevmode
  \begin{enumerate} [label = \textup{(\roman*)}, series = dtm col meas for VSRW]
    \item \label{assum item: 0. dtm col meas for VSRW}
      As $n \to \infty$, $b_n \to \infty$ .
    \item \label{assum item: 1. dtm col meas for VSRW}
      There exists an element $G = (F, R, \rho, m) \in \rootBCM(\MeasSt)$ such that 
      \begin{equation}
        \left(V_n, a_n^{-1}R_n, \rho_n, b_n^{-1}m_n\right) \to (F, R, \rho, m)\quad \text{in}\quad \rootBCM(\MeasSt).
      \end{equation}
    \item \label{assum item: 2. dtm col meas for VSRW}
      It holds that 
      \begin{equation}
        \lim_{r \to \infty} \liminf_{n \to \infty} a_n^{-1}R_n\bigl( \rho_n, B_{R_n}(\rho_n, a_nr)^c \bigr) = \infty.
      \end{equation}
    \item \label{assum item: 3. dtm col meas for VSRW}
      For each $r > 0$, there exist constants $\alpha_r, \beta_r, C_r > 0$ such that, 
      for all sufficiently large $n \geq 1$, 
      \begin{equation}  \label{prop eq: dtm col meas for VSRW}
        \inf_{x \in B_{R_n}(\rho_n, a_n r)} b_n^{-1} m_n\!\left( B_{R_n}(x, a_n s) \right) \geq C_r\, s^{\alpha_r},
        \quad 
        \forall s \in \left( \frac{1}{(\log b_n) (\log \log b_n)^{(1+\beta_r)}}, 1 \right).
      \end{equation}
      (NB.\ By condition~\ref{assum item: 0. dtm col meas for VSRW}, 
        for all sufficiently large $n$,
        $b_n > e$ and thus $\log \log b_n > 0$.)
  \end{enumerate}
\end{assum}

Under Assumption~\ref{assum: dtm col meas for VSRW},
by Theorem~\ref{thm: col dtm resis sp},
we have $(F, R, \rho, m) \in \rootResisSp$.
Let $X^1$ be the process associated with $(R, m)$ and $X^2$ be an independent copy of $X^1$.
Write $\hat{X} = (X^1, X^2)$ for their product process.

\begin{thm} \label{thm: dtm col meas for VSRW}
  Suppose that Assumption~\ref{assum: dtm col meas for VSRW} is satisfied.
  Then the measure $\diagMeas{m}$ belongs to the local Kato class of $\hat{X}$.
  Moreover, if we write $\Pi$ for the collision measure of $X^1$ and $X^2$ associated with $m$,
  then  
  \begin{equation}
    \left(V_n, a_n^{-1}R_n, \rho_n, b_n^{-1}m_n, \ProcLaw_{(\hat{X}_n, \Pi_n)}\right)
    \to 
    \left(F, R, \rho, m, \ProcLaw_{(\hat{X}, \Pi)}\right)
  \end{equation}
  in $\rootBCM(\MeasSt \times \ColSt)$.
\end{thm}

In order to prove the above theorem,
we define, for each $\beta > 0$, a bijective function 
$\trf_\beta \colon (0,1/e) \to (0,\infty)$ by
\begin{equation}  \label{eq: dfn of trf}
  \trf_\beta(t) \coloneqq 
  \frac{1}{\log (1/t)\, \bigl(\log \log (1/t)\bigr)^{1+\beta}} .
\end{equation}
We denote by $\trf_\beta^{-1}$ its inverse.
The properties of $\trf_\beta$ used in the proof below 
are collected in Appendix~\ref{sec: trf function}.

\begin{proof}
  It is enough to check that 
  Assumption~\ref{assum: col dtm resis sp}\ref{assum item: 3. col dtm resis sp} is satisfied.
  Let $p_n$ denote the heat kernel of $X_n^1$ with respect to $b_n^{-1} m_n$.
  Fix $r > 0$ for the remainder of this proof.
  By condition~\ref{assum item: 0. dtm col meas for VSRW},
  we have $b_n > e$ for all sufficiently large $n$.
  Henceforth, we restrict our attention to such large values of $n$,
  since the subsequent arguments concern the limit as $n \to \infty$.
  Set 
  \begin{equation}
    \delta_n^{(r)} \coloneqq \frac{1}{(\log b_n) (\log \log b_n)^{(1+\beta_r)}}.
  \end{equation}
  For each $n \geq 1$, define
  \begin{equation}  \label{thm pr eq: 6. dtm col meas for VSRW}
    v_n^{(r)}(s) \coloneqq 
    \begin{cases}
      C_r\, s^{\alpha_r}, & s \in (\delta_n^{(r)}, 1),\\
      b_n^{-1}, & s \in (0, \delta_n^{(r)}).
    \end{cases}
  \end{equation}
  Since $m_n$ is the counting measure,
  Assumption~\ref{assum: dtm col meas for VSRW}\ref{assum item: 3. dtm col meas for VSRW} yields
  \begin{equation}
    \inf_{x \in B_{R_n}(\rho_n, a_n r)} b_n^{-1} m_n\!\left( B_{R_n}(x, a_n s) \right) \geq v_n^{(r)}(s),
    \quad 
    \forall s \in (0, 1).
  \end{equation}
  Since $\trf_{\beta_r/2}$ is increasing,
  we can find $t_0^{(r)} \in (0,1/e)$ such that $\trf_{\beta_r/2}(t) < 1$ for all $t \in (0, t_0)$.
  By Lemma~\ref{lem: hk estimate from volume}, 
  for any $t \in (0, t_0^{(r)})$ and $x \in B_{R_n}(\rho_n, a_n r)$, 
  \begin{equation}
    p_n(t, x, x) \leq \frac{2 \trf_{\beta_r/2}(t)}{t} 
      + \frac{\sqrt{2}}{v_n^{(r)}(\trf_{\beta_r/2}(t))}.
  \end{equation}
  The Chapman--Kolmogorov equation \eqref{assum item eq: C-K equation} 
  and the Cauchy--Schwarz inequality yield
  \begin{equation} \label{thm pr eq: 3. dtm col meas for VSRW}
    p_n(t, x, y) 
    \leq p_n(t, x, x)^{1/2} p_n(t, y, y)^{1/2},
    \quad \forall x,y \in V_n
  \end{equation}
  (see \cite[Lemma~4.4(i)]{Noda_pre_Continuity}).
  Thus, for any $x,y \in B_{R_n}(\rho_n, a_n r)$ and $t \in (0, t_0^{(r)})$,
  \begin{equation} \label{thm pr eq: 2. dtm col meas for VSRW}
    p_n(t, x, y) \leq \frac{2 \trf_{\beta_r/2}(t)}{t} 
      + \frac{\sqrt{2}}{v_n^{(r)}(\trf_{\beta_r/2}(t))}.
  \end{equation}
  By the Chapman--Kolmogorov equation again, for any 
  $x_1, x_2 \in B_{R_n}(\rho_n, a_n r)$ and $\delta \in (0, t_0^{(r)})$,
  \begin{align} 
    &\int_0^\delta \int_{B_{R_n}(\rho_n, a_n r)} 
      p_n(t, x_1, y)\, p_n(t, x_2, y)\, b_n^{-1} m_n(dy)\, dt \\
    &\leq \int_0^\delta p_n(2t, x_1, x_2)\, dt \\
    &\leq \int_0^\delta 
      \left(\frac{\trf_{\beta_r/2}(2t)}{t} 
      + \frac{\sqrt{2}}{v_n^{(r)}(\trf_{\beta_r/2}(2t))}\right)\, dt.
    \label{thm pr eq: 1. dtm col meas for VSRW}
  \end{align}
  By the definition of $v_n^{(r)}$,
  the last expression is bounded by
  \begin{align}
    &\int_0^\delta \frac{\trf_{\beta_r/2}(2t)}{t}\, dt
    +
    \int_0^\delta \frac{\sqrt{2}}{C_r \trf_{\beta_r/2}(2t)^{\alpha_r}}\, dt
    +
    \int_0^{\trf_{\beta_r/2}^{-1}(\delta_n^{(r)})/2} \sqrt{2} b_n\, dt\\
    &\leq 
    \int_0^\delta \frac{\trf_{\beta_r/2}(2t)}{t}\, dt
    +
    \int_0^\delta \frac{\sqrt{2}}{C_r \trf_{\beta_r/2}(2t)^{\alpha_r}}\, dt
    +
    \frac{\sqrt{2}}{2} b_n \trf_{\beta_r/2}^{-1}(\delta_n^{(r)}).
    \label{thm pr eq: 5. dtm col meas for VSRW}
  \end{align}
  By Lemma~\ref{lem: property of trf}\ref{lem item: 1. property of trf} and \ref{lem item: 2. property of trf},
  the first and second terms in the last expression converge to $0$ as $\delta \to 0$.
  Thus, it remains to show that the third term vanishes as $n \to \infty$.
  Using Lemma~\ref{lem: property of trf}\ref{lem item: 3. property of trf},
  we obtain
  \begin{align}
    \log \bigl( b_n \trf_{\beta_r/2}^{-1}(\delta_n^{(r)}) \bigr)
    &= 
    \log b_n - \frac{1 + o(1)}{\delta_n^{(r)} (\log (\delta_n^{(r)})^{-1})^{1+\beta_r/2}}\\
    &= 
    \log b_n \left( 1 - \frac{(1 + o(1)) (\log \log b_n)^{1+\beta_r}}{(\log (\log b_n)(\log \log b_n)^{1+\beta_r})^{1+ \beta_r/2}}\right),
  \end{align}
  which tends to $-\infty$ as $n \to \infty$ by Assumption~\ref{assum: dtm col meas for VSRW}\ref{assum item: 0. dtm col meas for VSRW}.
  This implies that $b_n \trf_{\beta_r/2}^{-1}(\delta_n^{(r)}) \to 0$ as $n \to \infty$.
  Therefore,
  we conclude that 
  Assumption~\ref{assum: col dtm resis sp}\ref{assum item: 3. col dtm resis sp} 
  is satisfied.
\end{proof}

\begin{exm} \label{exm: dtm VSRW}
  In \cite{Croydon_Hambly_Kumagai_17_Time-changes},
  it is shown that, under Assumption~\ref{assum: dtm col meas for VSRW}\ref{assum item: 1. dtm col meas for VSRW} 
  and the uniform volume doubling (UVD) condition 
  (see \cite[Definition~1.1]{Croydon_Hambly_Kumagai_17_Time-changes}),
  the associated processes and their local times converge.
  The UVD condition implies the non-explosion condition,
  Assumption~\ref{assum: dtm col meas for VSRW}\ref{assum item: 2. dtm col meas for VSRW},
  see \cite[Remark~1.3(b)]{Croydon_18_Scaling}.
  Moreover, it can be readily verified that the UVD condition implies Assumption~\ref{assum: dtm col meas for VSRW}\ref{assum item: 3. dtm col meas for VSRW}.
  Hence, Theorem~\ref{thm: dtm col meas for VSRW} applies to all the examples treated in \cite{Croydon_Hambly_Kumagai_17_Time-changes}.
  The following is a list of such applicable examples. 
  \begin{enumerate} [label = \textup{(\alph*)}]
    \item \label{exm item: one-dimensional lattice}
      A sequence of scaled one-dimensional lattices converging to $\RN$.
    \item \label{exm item: ufr fractals}
      A sequence of graphs approximating a uniform finitely ramified fractal, 
      such as the two-dimensional Sierpiński gasket (see \cite[Section~6.2]{Croydon_Hambly_Kumagai_17_Time-changes} for details).
    \item \label{exm item: carpet}
      A (sub)sequence of graphical approximations of the two-dimensional Sierpiński carpet
      (see \cite[Example~4.5(iv)]{Croydon_Hambly_Kumagai_17_Time-changes} for details).
  \end{enumerate}
  In particular, by applying Theorem~\ref{thm: dtm col meas for VSRW} to \ref{exm item: one-dimensional lattice},
  we obtain the scaling limit of uniform collision measures of two i.i.d.\ VSRWs on $\mathbb{Z}$,
  which provides a continuous analogue of Nguyen's result~\cite{Nguyen_23_Collision}
  (recall that Nguyen considered discrete-time random walks, 
  see Remarks~\ref{rem: coincidence with Nguyen} and \ref{rem: more comment about Nguyen}).
\end{exm}

We next state a version of Theorem~\ref{thm: dtm col meas for VSRW} for random electrical networks.
We use the same notation as before, but now assume that the recurrent electrical network $(V_n, E_n, \mu_n)$ is random,
so that $(V_n, R_n, \rho_n, m_n)$ is a random element of $\rootResisSp$.
We also assume that the scaling factors $a_n$ and $b_n$ are random variables 
defined on the same probability space as $(V_n, E_n, \mu_n)$.
(They may be taken to be deterministic.)
We write $\mathbf{P}_n$ for the underlying complete probability measure for the random electrical network $(V_n, E_n, \mu_n)$.

\begin{assum}  \label{assum: rdm col meas for VSRW}\leavevmode
  \begin{enumerate} [label = \textup{(\roman*)}, series = rdm col meas for VSRW]
    \item \label{assum item: 0. rdm col meas for VSRW}
      As $n \to \infty$, $b_n \to \infty$ in probability, i.e.,
      \begin{equation}
        \lim_{\lambda \to \infty} \liminf_{n \to \infty} \mathbf{P}_n(b_n > \lambda) = 1.
      \end{equation}
    \item \label{assum item: 1. rdm col meas for VSRW}
      There exists a random element $G = (F, R, \rho, m)$ of $\rootBCM(\MeasSt)$
      with the underlying complete probability measure $\mathbf{P}$ 
      such that 
      \begin{equation}
        \left(V_n, a_n^{-1}R_n, \rho_n, b_n^{-1}m_n\right) 
        \xrightarrow{\mathrm{d}} (F, R, \rho, m)
        \quad \text{in}\quad \rootBCM(\MeasSt).
      \end{equation}
    \item \label{assum item: 2. rdm col meas for VSRW}
      For all $\lambda > 0$,
      \begin{equation}
        \lim_{r \to \infty} \liminf_{n \to \infty} 
        \mathbf{P}_n\!\left( 
          a_n^{-1}R_n\bigl( \rho_n, B_{R_n}(\rho_n, a_n r)^c \bigr) > \lambda
        \right) = 1.
      \end{equation}
    \item \label{assum item: 3. rdm col meas for VSRW}
      For each $r > 0$ and $\varepsilon > 0$, 
      there exist deterministic constants $\alpha_{r,\varepsilon}, \beta_{r,\varepsilon}, C_{r,\varepsilon} > 0$ such that,
      for all sufficiently large $n$, with probability at least $1 - \varepsilon$, 
      we have $b_n > e$ and 
      \begin{equation}  \label{prop eq: rdm col meas for VSRW}
        \inf_{x \in B_{R_n}(\rho_n, a_n r)} 
        b_n^{-1} m_n\!\left( B_{R_n}(x, a_n s) \right) 
        \geq C_{r,\varepsilon}\, s^{\alpha_{r,\varepsilon}},
        \quad 
        \forall s \in \left( \frac{1}{(\log b_n)(\log \log b_n)^{1+\beta_{r,\varepsilon}}}, 1 \right).
      \end{equation}
      (NB.\ By condition~\ref{assum item: 0. rdm col meas for VSRW}, 
        for all sufficiently large $n$, we have $b_n > e$ with probability at least $1-\varepsilon$.)
  \end{enumerate}
\end{assum}

\begin{thm} \label{thm: rdm col meas for VSRW}
  Suppose that Assumption~\ref{assum: rdm col meas for VSRW} is satisfied.
  Then the measure $\diagMeas{m}$ belongs to the local Kato class of $\hat{X}$, $\mathbf{P}$-a.s.
  Moreover, if we write $\Pi$ for the collision measure of $X^1$ and $X^2$ associated with $m$,
  then  
  \begin{equation}
    \left(V_n, a_n^{-1}R_n, \rho_n, b_n^{-1}m_n, 
      \ProcLaw_{(\hat{X}_n, \Pi_n)}\right)
    \xrightarrow{\mathrm{d}}
    \left(F, R, \rho, m, \ProcLaw_{(\hat{X}, \Pi)}\right)
  \end{equation}
  in $\rootBCM(\MeasSt \times \ColSt)$.
\end{thm}

\begin{proof}
  It is sufficient to verify 
  Assumption~\ref{assum: col rdm resis sp}\ref{assum item: 3. col rdm resis sp}.
  The proof proceeds in the same way as that of Theorem~\ref{thm: dtm col meas for VSRW}.
  Fix $r > 0$ for the remainder of this proof.
  For each $n \geq 1$ and $\varepsilon > 0$, define
  \begin{equation}
    \delta_{n, \varepsilon}^{(r)} 
    \coloneqq 
    \begin{cases}
        \frac{1}{(\log b_n) (\log \log b_n)^{(1+\beta_{r, \varepsilon})}}, & \text{if}\quad b_n >e,\\
        0, & \text{if}\quad b_n \leq e.
    \end{cases}
  \end{equation}
  and 
  \begin{equation}
    v_{n, \varepsilon}^{(r)}(s) \coloneqq 
    \begin{cases}
      C_{r,\varepsilon}\, s^{\alpha_{r,\varepsilon}}, & s \in (\delta_{n, \varepsilon}^{(r)}, 1),\\
      b_n^{-1}, & s \in (0, \delta_{n, \varepsilon}^{(r)}).
    \end{cases}
  \end{equation}
  If we write $E_{n, \varepsilon}$ for the event that $b_n > e$ and 
  \begin{equation}
    \inf_{x \in B_{R_n}(\rho_n, a_n r)} 
    b_n^{-1} m_n\!\left( B_{R_n}(x, a_n s) \right) 
    \geq v_{n, \varepsilon}^{(r)}(s),
    \quad 
    \forall s \in (0,1),
  \end{equation}
  then Assumption~\ref{assum: rdm col meas for VSRW}\ref{assum item: 3. rdm col meas for VSRW} implies that 
  \begin{equation}   \label{thm pr eq: 1. rdm col meas for VSRW}
    \liminf_{n \to \infty} \mathbf{P}_n(E_{n, \varepsilon}) \geq 1 - \varepsilon.
  \end{equation}
  Write $B_n(r) \coloneqq B_{R_n}(\rho_n, a_n r)$.
  Using \eqref{thm pr eq: 1. dtm col meas for VSRW} and \eqref{thm pr eq: 5. dtm col meas for VSRW}, we obtain 
  \begin{align}
    &\mathbf{E}_n\!
    \left[
      \left(
        \sup_{x_1, x_2 \in B_n(r/2)}
        \int_0^\delta \int_{B_n(r/2)} 
        p_n(t, x_1, y)\, p_n(t, x_2, y)\, b_n^{-1} m_n(dy)\, dt 
      \right)
      \wedge 1
    \right] \\
    &\quad\leq    
    \mathbf{P}_n(E_{n, \varepsilon}^c)\\
    &\qquad
    + 
    \mathbf{E}_n\!
    \left[
      \left(
        \int_0^\delta 
        \left(
          \frac{\trf_{\beta_r/2}(2t)}{t}
          +
          \frac{\sqrt{2}}{C_{r,\varepsilon} \trf_{\beta_{r,\varepsilon}/2}(2t)^{\alpha_{r, \varepsilon}}}
        \right)\, dt
        +
        \frac{\sqrt{2}}{2} b_n \trf_{\beta_{r, \varepsilon}/2}^{-1}(\delta_{n, \varepsilon}^{(r)})
      \right)
      \wedge 1
    \right].
    \label{thm pr eq: 2. rdm col meas for VSRW}
  \end{align}
  Following the argument in the proof of Theorem~\ref{thm: dtm col meas for VSRW},
  we obtain
  \begin{gather}
    \int_0^\delta 
    \left(
      \frac{\trf_{\beta_r/2}(2t)}{t}
      +
      \frac{\sqrt{2}}{C_{r,\varepsilon} \trf_{\beta_{r,\varepsilon}/2}(2t)^{\alpha_{r, \varepsilon}}}
    \right)\! dt
    \xrightarrow[\delta \to 0]{} 0,\\
    b_n \trf_{\beta_{r, \varepsilon}/2}^{-1}(2\delta_{n, \varepsilon}^{(r)}) 
    \xrightarrow[n \to \infty]{\mathrm{p}} 0.
  \end{gather}
  Combining these with \eqref{thm pr eq: 1. rdm col meas for VSRW} and \eqref{thm pr eq: 2. rdm col meas for VSRW},
  we obtain
  \begin{equation}
    \lim_{\delta \to 0} 
    \limsup_{n \to \infty}
    \mathbf{E}_n\!
    \left[
      \left(
        \sup_{x_1, x_2 \in B_n(r/2)}
        \int_0^\delta \int_{B_n(r/2)} 
        p_n(t, x_1, y)\, p_n(t, x_2, y)\, b_n^{-1} m_n(dy)\, dt 
      \right)
      \wedge 1
    \right]
    \leq \varepsilon.
  \end{equation}
  Letting $\varepsilon \to 0$, we conclude that 
  Assumption~\ref{assum: col rdm resis sp}\ref{assum item: 3. col rdm resis sp} holds.
\end{proof}

\begin{exm}\label{exm: rdm VSRW}
  In \cite{Noda_pre_Convergence}, 
  it is shown that 
  the volume condition~\ref{assum item: 3. rdm col meas for VSRW} of Assumption~\ref{assum: rdm col meas for VSRW}
  is satisfied for many random graph models. 
  In particular, all the examples considered in \cite[Section~8]{Noda_pre_Convergence} 
  satisfy Assumption~\ref{assum: rdm col meas for VSRW}, 
  and hence Theorem~\ref{thm: rdm col meas for VSRW} applies to them,
  see the list below.
  \begin{enumerate} [label = (\alph*)]
    \item \label{rem item: random gasket}
      A random recursive Sierpi\'{n}ski gasket (cf.\ \cite{Hambly_97_Brownian}).
    \item \label{rem item: GW trees}
      Critical Galton--Watson trees (cf.\ \cite{Aldous_93_The_continuum,Andriopoulos_23_Convergence,Duquesne_03_A_limit}).
    \item \label{rem item: UST}
      Uniform spanning trees on $\ZN^d$ with $d = 2,3$, and on high-dimensional tori
      (cf.\ \cite{Angel_Croydon_Hernandez-Torres_Shiraishi_21_Scaling,Archer_Nachmias_Shalev_24_The_GHP,Barlow_Croydon_Kumagai_17_Subsequential}).
    \item \label{rem item: Erdos-renyi random graph}
      The critical Erd\H{o}s--R\'{e}nyi random graph (cf.\ \cite{Berry_Broutin_Goldschmidt_12_The_continuum}).
    \item \label{rem item: configuration models} 
      The critical configuration model (cf.\ \cite{Bhamidi_Sen_20_Geometry}).
  \end{enumerate}
\end{exm}

\subsection{Scaling limits of collision measures of CSRWs} 
\label{sec: conv of col meas of CSRW}

In this subsection, we provide versions of Theorems~\ref{thm: dtm col meas for VSRW} and \ref{thm: rdm col meas for VSRW}
for CSRWs on electrical networks.
In this case,
verifying the assumptions for convergence of collision measures is more difficult than in the VSRW case,
which has been discussed in Section~\ref{sec: conv of col meas of VSRW} above.
This is because the canonical weighting measure (see Remark~\ref{rem: uniform col meas}) of i.i.d.\ CSRWs does not coincide 
with the invariant measure of the walk.
The main results of this subsection are Theorems~\ref{thm: dtm col meas for CSRW} and \ref{thm: rdm col meas for CSRW} below.
In the following subsubsection, Section~\ref{sec: Critical Galton--Watson trees}, 
we apply the results to critical Galton--Watson trees conditioned on their size.

We clarify the setting for the main results.
For each $n \geq 1$, let $(V_n, E_n, \mu_n)$ be a recurrent electrical network.
Write $R_n$ for the associated resistance metric.
We identify $\mu_n$ with the associated conductance measure on $V_n$,
that is,
\begin{equation}
  \mu_n(A) \coloneqq \sum_{x \in A} \mu_n(x),
  \quad A \subseteq V_n.
\end{equation}
More generally, for each $q > 0$, 
we define the \emph{$q$-power conductance measure} $\mu_n^q$ on $V_n$ by 
\begin{equation}
  \mu_n^q(A) \coloneqq \sum_{x \in A} \mu_n(x)^q,
  \quad A \subseteq V_n.
\end{equation}
Let $a_n, b_n > 0$ be constants, serving as scaling factors for the metric and measure, respectively.
Let $X_n^1$ be the process associated with $(a_n^{-1} R_n, b_n^{-1}\mu_n)$ and $X_n^2$ be an independent copy of $X_n^1$.
In other words, $X_n^1$ and $X_n^2$ are i.i.d.\ CSRWs on the electrical network,
rescaled in time by the factor $a_n b_n$.
We write $\hat{X}_n = (X_n^1, X_n^2)$ for their product process.
Let $\Pi_n$ be the collision measure of $X_n^1$ and $X_n^2$ associated with $b_n^{-1}\mu_n^2$, i.e.,
\begin{equation} \label{eq: dfn of unif col meas for CSRW}
  \Pi_n(dx\, dt) = b_n \sum_{y \in V_n} \mathbf{1}_{\{X_n^1(t) = X_n^2(t) = y\}}\, \delta_y(dx)\, dt.
\end{equation}

Below, we state the assumption for the convergence of $\Pi_n$,
which requires stronger volume conditions than that of the VSRW case, Assumption~\ref{assum: dtm col meas for VSRW}.
In particular, we assume the precompactness of the $(1+q)$-power measures for suitable $q > 1$.

\begin{assum}  \label{assum: dtm col meas for CSRW}\leavevmode
  \begin{enumerate} [label = \textup{(\roman*)}, series = dtm col meas for CSRW]
    \item \label{assum item: 1. dtm col meas for CSRW}
      There exists an element $G = (F, R, \rho, \mu, \nu) \in \rootBCM(\MeasSt^{\otimes 2})$ such that 
      \begin{equation}
        \left(V_n, a_n^{-1}R_n, \rho_n, b_n^{-1}\mu_n, b_n^{-1}\mu_n^2\right) 
        \to (F, R, \rho, \mu, \nu)\quad \text{in}\quad \rootBCM(\MeasSt^{\otimes 2}).
      \end{equation}
    \item \label{assum item: 2. dtm col meas for CSRW}
      It holds that 
      \begin{equation}
        \lim_{r \to \infty} \liminf_{n \to \infty} a_n^{-1}R_n\bigl( \rho_n, B_{R_n}(\rho_n, a_nr)^c \bigr) = \infty.
      \end{equation}
    \item \label{assum item: 3. dtm col meas for CSRW}
      For each $r > 0$, there exist constants $\alpha_r, C_r > 0$ 
      such that 
      \begin{equation}  \label{assum item eq: 3. dtm col meas for CSRW}
        \inf_{x \in B_{R_n}(\rho_n, a_n r)} b_n^{-1} \mu_n\!\left( B_{R_n}(x, a_n s) \right) \geq C_r\, s^{\alpha_r},
        \quad 
        \forall s \in (0, 1)
      \end{equation}
      for each $n \geq 1$, and, for some $q_r \in (1, \infty)$ with $q_r > \alpha_r$,
      \begin{equation}   \label{assum item eq: 3.1. dtm col meas for CSRW}
         \limsup_{n \to \infty} b_n^{-1} \mu_n^{1+q_r} \bigl( B_{R_n}(\rho_n, a_n r) \bigr) < \infty.
      \end{equation} 
  \end{enumerate}
\end{assum}

Under Assumption~\ref{assum: dtm col meas for CSRW},
by Theorem~\ref{thm: col dtm resis sp},
we have $(F, R, \rho, m) \in \rootResisSp$.
Let $X^1$ be the process associated with $(R, m)$ and $X^2$ be an independent copy of $X^1$.
Write $\hat{X} = (X^1, X^2)$ for their product process.

\begin{thm} \label{thm: dtm col meas for CSRW}
  Suppose that Assumption~\ref{assum: dtm col meas for CSRW} is satisfied.
  Then the measure $\diagMeas{\nu}$ belongs to the local Kato class of $\hat{X}$.
  Moreover, if we write $\Pi$ for the collision measure of $X^1$ and $X^2$ associated with $\nu$,
  then  
  \begin{equation}
    \left(V_n, a_n^{-1}R_n, \rho_n, b_n^{-1}\mu_n, b_n^{-1}\mu_n^2, \ProcLaw_{(\hat{X}_n, \Pi_n)}\right)
    \to 
    \left(F, R, \rho, \mu, \nu, \ProcLaw_{(\hat{X}, \Pi)}\right)
  \end{equation}
  in $\rootBCM(\MeasSt^{\otimes 2} \times \ColSt)$.
\end{thm}

\begin{proof}
  It is sufficient to verify that 
  Assumption~\ref{assum: col dtm resis sp}\ref{assum item: 3. col dtm resis sp} holds.
  Denote by $p_n$ the heat kernel of the scaled walk $X^1_n$ with respect to $b_n^{-1} \mu_n$.
  Fix $r > 0$ for the remainder of this proof, and set $B_n(r) \coloneqq B_{R_n}(\rho_n, a_n r)$.
  Since $q_r > \alpha_r$ by Assumption~\ref{assum: rdm col meas for CSRW}\ref{assum item: 3. rdm col meas for CSRW}, 
  we can find $\gamma_r \in (0,1)$ such that
  \begin{equation}  \label{thm pr eq: 2. dtm col meas for CSRW}
    \frac{1}{q_r+1} < \gamma_r < \frac{q_r}{\alpha_r(q_r+1)} .
  \end{equation}
  Following the argument that yields \eqref{thm pr eq: 2. dtm col meas for VSRW},
  with $v_n^{(r)}(s)$ in \eqref{thm pr eq: 6. dtm col meas for VSRW} replaced by $C_r\, s^{\gamma_r}$,
  we obtain 
  \begin{equation}  \label{thm pr eq: 3. dtm col meas for CSRW}
    \sup_{x, y \in B_n(r)}
    p_n(t, x, y) 
    \leq 
    \frac{2}{t^{1-\gamma_r}} 
    + \frac{\sqrt{2}}{C_r\, t^{\alpha_r \gamma_r}},
    \quad 
    \forall t \in (0, 1).
  \end{equation}
  Let $p_r > 1$ be such that $p_r^{-1} + q_r^{-1} = 1$.
  By H\"{o}lder's inequality, for any $x_1, x_2 \in B_n(r)$,
  \begin{align} 
    &\int_{B_n(r)} 
      p_n(t, x_1, y)\, p_n(t, x_2, y)\, b_n^{-1} \mu_n^2(dy) \\
    &= \int_{B_n(r)} 
      p_n(t, x_1, y)\, p_n(t, x_2, y) \mu_n(y)\, b_n^{-1} \mu_n(dy) \\
    &\leq 
      \left(
        \int_{B_n(r)} 
        p_n(t, x_1, y)\, p_n(t, x_2, y)\, b_n^{-1} \mu_n(dy)
      \right)^{1/p_r}\\
    &\qquad
      \left(
        \int_{B_n(r)} 
        p_n(t, x_1, y)\, p_n(t, x_2, y) \mu_n(y)^{q_r}\, b_n^{-1} \mu_n(dy)
      \right)^{1/q_r} \\
    &\leq 
      p_n(2t, x_1, x_2)^{1/p_r}
      \sup_{y \in B_n(r)} p_n(t, y, y)^{2/q_r}
      \left\{ b_n^{-1} \mu_n^{1+q_r}\!\bigl( B_{R_n}(\rho_n, a_n r) \bigr) \right\}^{1/q_r},
    \label{thm pr eq: 4. dtm col meas for CSRW}
  \end{align}
  where we have used the Chapman--Kolmogorov equation~\eqref{assum item eq: C-K equation} 
  together with \eqref{thm pr eq: 3. dtm col meas for VSRW}.
  Using \eqref{thm pr eq: 3. dtm col meas for CSRW}
  and the monotonicity of its right-hand side in $t$,
  we deduce that 
  \begin{align}
    p_n(2t, x_1, y)^{1/p_r}
    \sup_{y \in B_n(r)} p_n(t, y, y)^{2/q_r}
    &\leq
    \left(
      \frac{2}{t^{1-\gamma_r}}
      + \frac{\sqrt{2}}{C_r\, t^{\alpha_r \gamma_r}}
    \right)^{p_r^{-1} + 2 q_r^{-1}} \\
    &= 
    \left(
      \frac{2}{t^{1-\gamma_r}} 
      + \frac{\sqrt{2}}{C_r\, t^{\alpha_r \gamma_r}}
    \right)^{1 + q_r^{-1}} .
  \end{align}
  Combining this with \eqref{thm pr eq: 4. dtm col meas for CSRW}, we obtain
  \begin{align}
    &\sup_{x_1, x_2 \in B_n(r)}
      \int_{B_n(r)} 
      p_n(t, x_1, y)\, p_n(t, x_2, y)\, b_n^{-1} \mu_n^2(dy) \\
    &\leq
      \left\{ b_n^{-1} \mu_n^{1+q_r}\!\bigl( B_{R_n}(\rho_n, a_n r) \bigr) \right\}^{1/q_r}
      \left(
        \frac{2}{t^{1-\gamma_r}} 
        + \frac{\sqrt{2}}{C_r\, t^{\alpha_r \gamma_r}}
      \right)^{1 + q_r^{-1}} .
  \end{align}
  Hence, by \eqref{assum item eq: 3.1. dtm col meas for CSRW},
  it suffices to show that 
  \begin{equation}  \label{thm pr eq: 5. dtm col meas for CSRW}
    \lim_{\delta \to 0}
    \int_0^\delta 
      \left(
        \frac{2}{t^{1-\gamma_r}} 
        + \frac{\sqrt{2}}{C_r\, t^{\alpha_r \gamma_r}}
      \right)^{1 + q_r^{-1}} dt 
    = 0 .
  \end{equation}
  Let $c_1 > 0$ be a constant such that 
  \begin{equation}
    (a+b)^{1+q_r^{-1}} \leq c_1\,(a^{1+q_r^{-1}} + b^{1+q_r^{-1}}),
    \quad \forall a, b > 0.
  \end{equation}
  Then
  \begin{equation}
    \left(
      \frac{2}{t^{1-\gamma_r}} 
      + \frac{\sqrt{2}}{C_r\, t^{\alpha_r \gamma_r}}
    \right)^{1 + q_r^{-1}}
    \leq
    \frac{c_1 2^{1+q_r^{-1}}}{t^{(1-\gamma_r)(1+q_r^{-1})}} 
    + 
    \frac{c_1 2^{\frac{1}{2}+\frac{1}{2q_r}}}{C_r^{1+q_r^{-1}} t^{\alpha_r \gamma_r (1+ q_r^{-1})}} .
  \end{equation}
  By \eqref{thm pr eq: 2. dtm col meas for CSRW}, we have 
  \begin{equation}
    (1-\gamma_r)(1+q_r^{-1}) < 1, 
    \quad 
    \alpha_r \gamma_r (1+ q_r^{-1}) < 1 .
  \end{equation}
  Hence \eqref{thm pr eq: 5. dtm col meas for CSRW} follows, which completes the proof.
\end{proof}

When the conductance weights on vertices are uniformly bounded away from $0$ and $\infty$,
the proof of Theorem~\ref{thm: dtm col meas for VSRW} can be applied,
and hence the same volume condition as Assumption~\ref{assum: dtm col meas for VSRW}\ref{assum item: 3. dtm col meas for VSRW} is sufficient, 
as shown below.

\begin{prop} \label{prop: dtm weakend vol cond for CSRW}
  In addition to Assumption~\ref{assum: dtm col meas for CSRW}\ref{assum item: 1. dtm col meas for CSRW} and \ref{assum item: 2. dtm col meas for CSRW},
  assume that the following conditions hold.
  \begin{enumerate} [resume* = dtm col meas for CSRW]
    \item \label{assum item: 4. dtm col meas for CSRW}
      For each $r > 0$, 
      \begin{equation} \label{assum item eq: 4. dtm col meas for CSRW}
        \liminf_{n \to \infty} \inf_{x \in B_{R_n}(\rho_n, a_n r)} \mu_n(x) >0,
        \qquad 
        \limsup_{n \to \infty} \sup_{x \in B_{R_n}(\rho_n, a_n r)} \mu_n(x) < \infty.
      \end{equation}
    \item \label{assum item: 5. dtm col meas for CSRW}
      As $n \to \infty$, $b_n \to \infty$.
    \item \label{assum item: 6. dtm col meas for CSRW}
      For each $r > 0$, there exist constants $\alpha_r, \beta_r, C_r > 0$ such that, for each $n \geq 1$,
      \begin{equation}  \label{assum item eq: 6. dtm col meas for CSRW}
        \inf_{x \in B_{R_n}(\rho_n, a_n r)} b_n^{-1} \mu_n\!\left( B_{R_n}(x, a_n s) \right) \geq C_r\, s^{\alpha_r},
        \quad 
        \forall s \in \left( \frac{1}{(\log b_n) (\log \log b_n)^{(1+\beta_r)}}, 1 \right).
      \end{equation}
  \end{enumerate}
  Then the conclusion of Theorem~\ref{thm: dtm col meas for CSRW} holds.
\end{prop}

\begin{proof}
  We use the same notation as in the proof of Theorem~\ref{thm: dtm col meas for CSRW}.
  Fix $r > 0$.
  By \eqref{assum item eq: 4. dtm col meas for CSRW},
  we can find a constant $M_r \in (0, \infty)$ satisfying
  \begin{equation}
    M_r >
    \limsup_{n \to \infty} \sup_{x \in B_{R_n}(\rho_n, a_n r)} \mu_n(x).
  \end{equation}
  Using the Chapman--Kolmogorov equation,
  for all sufficiently large $n$,
  we obtain that, for any $x_1, x_2 \in B_n(r)$,
  \begin{align}
    \int_{B_n(r)} 
    p_n(t, x_1, y)\, p_n(t, x_2, y)\, b_n^{-1} \mu_n^2(dy)
    &\leq
    M_r \int_{B_n(r)} 
    p_n(t, x_1, y)\, p_n(t, x_2, y)\, b_n^{-1} \mu_n(dy)\\
    &\leq    
    M_r\, p_n(2t, x_1, x_2).
  \end{align}
  Thus, it suffices to show that 
  \begin{equation} \label{prop pr eq: 1. dtm weakend vol cond for CSRW}
    \lim_{\delta \to 0}
    \limsup_{n \to \infty} 
    \sup_{x_1, x_2 \in B_n(r)} 
    \int_0^\delta p_n(2t, x_1, x_2)\, dt 
    = 0.
  \end{equation}
  Henceforth, we consider only sufficiently large $n$ for which $b_n > e$.
  Set 
  \begin{equation}
    \delta_n^{(r)} \coloneqq \frac{1}{(\log b_n) (\log \log b_n)^{(1+\beta_r)}}.
  \end{equation}
  By \eqref{assum item eq: 4. dtm col meas for CSRW},
  we can find a constant $m_r \in (0, \infty)$ satisfying
  \begin{equation}
    m_r <
    \liminf_{n \to \infty} \inf_{x \in B_{R_n}(\rho_n, a_n r)} \mu_n(x).
  \end{equation}
  For each $n$, define
  \begin{equation}  
    v_n^{(r)}(s) \coloneqq 
    \begin{cases}
      C_r\, s^{\alpha_r}, & s \in (\delta_n^{(r)}, 1),\\
      m_r\, b_n^{-1}, & s \in (0, \delta_n^{(r)}).
    \end{cases}
  \end{equation}
  Condition~\ref{assum item: 6. dtm col meas for CSRW} yields that, for all sufficiently large $n$,
  \begin{equation}
    \inf_{x \in B_{R_n}(\rho_n, a_n r)} b_n^{-1} m_n\!\left( B_{R_n}(x, a_n s) \right) \geq v_n^{(r)}(s),
    \quad 
    \forall s \in (0, 1).
  \end{equation}
  Therefore, by the same argument as in the proof of Theorem~\ref{thm: dtm col meas for VSRW},
  we obtain \eqref{prop pr eq: 1. dtm weakend vol cond for CSRW}.
\end{proof}

\begin{exm}
  The examples \ref{exm item: one-dimensional lattice}, \ref{exm item: ufr fractals}, and \ref{exm item: carpet} listed in Example~\ref{exm: dtm VSRW} 
  all have uniformly bounded degrees.
  Hence, the preceding proposition applies.
  In particular, Theorem~\ref{thm: dtm col meas for CSRW} then yields the scaling limits of the (uniform) collision measures of two i.i.d.\ CSRWs 
  on these models.
\end{exm}

We next provide a version of the above theorem for random electrical networks.
Similarly to Theorem~\ref{thm: dtm col meas for VSRW},
we now assume that the recurrent electrical network $(V_n, E_n, \mu_n)$ is random,
so that $(V_n, R_n, \rho_n, \mu_n)$ is a random element of $\rootResisSp$.
We also assume that the scaling factors $a_n$ and $b_n$ are random variables 
defined on the same probability space as $(V_n, E_n, \mu_n)$.
(They may be taken to be deterministic.)
We write $\mathbf{P}_n$ for the underlying complete probability measure for the random electrical network $(V_n, E_n, \mu_n)$.

\begin{assum}  \label{assum: rdm col meas for CSRW}\leavevmode
  \begin{enumerate} [label = \textup{(\roman*)}, leftmargin = *, series = rdm col meas for CSRW]
    \item \label{assum item: 1. rdm col meas for CSRW}
      There exists a random element $G = (F, R, \rho, \mu, \nu)$ of $\rootBCM(\MeasSt^{\otimes 2})$
      with the underlying complete probability measure $\mathbf{P}$ 
      such that 
      \begin{equation}
        \left(V_n, a_n^{-1}R_n, \rho_n, b_n^{-1}\mu_n, b_n^{-1}\mu_n^2\right) 
        \xrightarrow{\mathrm{d}} (F, R, \rho, \mu, \nu)
        \quad \text{in}\quad \rootBCM(\MeasSt^{\otimes 2}).
      \end{equation}
    \item \label{assum item: 2. rdm col meas for CSRW}
      For all $\lambda > 0$,
      \begin{equation}
        \lim_{r \to \infty} \liminf_{n \to \infty} 
        \mathbf{P}_n\!\left( 
          a_n^{-1}R_n\bigl( \rho_n, B_{R_n}(\rho_n, a_n r)^c \bigr) > \lambda
        \right) = 1.
      \end{equation}
    \item \label{assum item: 3. rdm col meas for CSRW}
      For each $r > 0$ and $\varepsilon > 0$, 
      there exist deterministic constants $\alpha_{r,\varepsilon}, \beta_{r,\varepsilon}, C_{r,\varepsilon} > 0$ such that,
      for all sufficiently large $n$, with probability at least $1 - \varepsilon$, 
      \begin{equation}  \label{prop eq: rdm col meas for CSRW}
        \inf_{x \in B_{R_n}(\rho_n, a_n r)} 
        b_n^{-1} \mu_n\!\left( B_{R_n}(x, a_n s) \right) 
        \geq C_{r,\varepsilon}\, s^{\alpha_{r,\varepsilon}},
        \quad 
        \forall s \in ( 0, 1),
      \end{equation}
      for some $q_{r, \varepsilon} > 1 \vee \alpha_{r, \varepsilon}$,
      \begin{equation}   \label{assum item eq: 3.1. rdm col meas for CSRW}
        \limsup_{\lambda \to \infty}
         \limsup_{n \to \infty} 
         \mathbf{P}_n\!
         \left(
          b_n^{-1} \mu_n^{1+q_{r,\varepsilon}} \bigl( B_{R_n}(\rho_n, a_n r) \bigr) > \lambda
         \right)
         = 0.
      \end{equation} 
  \end{enumerate}
\end{assum}

\begin{thm} \label{thm: rdm col meas for CSRW}
  Suppose that Assumption~\ref{assum: rdm col meas for VSRW} is satisfied.
  Then the measure $\diagMeas{\nu}$ belongs to the local Kato class of $\hat{X}$, $\mathbf{P}$-a.s.
  Moreover, if we write $\Pi$ for the collision measure of $X^1$ and $X^2$ associated with $\nu$,
  then  
  \begin{equation}
    \left(V_n, a_n^{-1}R_n, \rho_n, b_n^{-1}\mu_n, b_n^{-1}\mu_n^2, 
      \ProcLaw_{(\hat{X}_n,\, \Pi_n)}\right)
    \xrightarrow{\mathrm{d}}
    \left(F, R, \rho, \mu, \nu, \ProcLaw_{(\hat{X}, \Pi)}\right)
  \end{equation}
  in $\rootBCM(\MeasSt^{\otimes 2} \times \ColSt)$.
\end{thm}

\begin{proof}
  This is verified by following the arguments in the proof of Theorems~\ref{thm: rdm col meas for VSRW} and \ref{thm: dtm col meas for CSRW},
  and therefore we omit the proof.
\end{proof}

The next proposition is a counterpart of Proposition~\ref{prop: dtm weakend vol cond for CSRW} 
in the setting of random electrical networks.
As its proof follows the same idea as in Theorem~\ref{thm: rdm col meas for VSRW},
we state the result without proof.

\begin{prop} \label{prop: rdm weakend vol cond for CSRW}
  In addition to Assumption~\ref{assum: dtm col meas for CSRW}\ref{assum item: 1. dtm col meas for CSRW} and \ref{assum item: 2. dtm col meas for CSRW},
  assume that the following conditions hold.
  \begin{enumerate} [resume* = dtm col meas for CSRW]
    \item \label{assum item: 4. rdm col meas for CSRW}
      For each $r > 0$, 
      \begin{gather} \label{assum item eq: 4. rdm col meas for CSRW}
        \lim_{\varepsilon \to 0} 
        \liminf_{n \to \infty}
        \mathbf{P}_n\!
        \left(
          \inf_{x \in B_{R_n}(\rho_n, a_n r)} \mu_n(x) > \varepsilon
        \right)
        = 1,\\
        \lim_{\lambda \to \infty} 
        \limsup_{n \to \infty}
        \mathbf{P}_n\!
        \left(
          \sup_{x \in B_{R_n}(\rho_n, a_n r)} \mu_n(x) > \lambda
        \right)
        =0.
      \end{gather}
    \item \label{assum item: 5. rdm col meas for CSRW}
      As $n \to \infty$, $b_n \to \infty$ in probability.
    \item \label{assum item: 6. rdm col meas for CSRW}
      For each $r > 0$ and $\varepsilon > 0$, 
      there exist deterministic constants $\alpha_{r,\varepsilon}, \beta_{r,\varepsilon}, C_{r,\varepsilon} > 0$ such that,
      for all sufficiently large $n$, with probability at least $1 - \varepsilon$, we have $b_n > e$ and
      \begin{equation} 
        \inf_{x \in B_{R_n}(\rho_n, a_n r)} 
        b_n^{-1} \mu_n\!\left( B_{R_n}(x, a_n s) \right) 
        \geq C_{r,\varepsilon}\, s^{\alpha_{r,\varepsilon}},
        \quad 
        \forall s \in \left( \frac{1}{(\log b_n)(\log \log b_n)^{1+\beta_{r,\varepsilon}}}, 1 \right).
      \end{equation}
  \end{enumerate}
  Then the conclusion of Theorem~\ref{thm: rdm col meas for CSRW} holds.
\end{prop}

\subsubsection{Critical Galton--Watson trees} \label{sec: Critical Galton--Watson trees}

In this subsubsection, 
as an application of Theorem~\ref{thm: rdm col meas for CSRW},
we establish the convergence of the uniform collision measures of i.i.d.\ CSRWs 
on critical Galton--Watson trees conditioned on their size, see Theorem~\ref{thm: conv of col meas of CSRW on GW trees} below.
For definitions and further details concerning trees, we follow \cite{LeGall_06_Random}.

Fix a probability measure $\bm{\pi} = (\pi_{k})_{k \geq 0}$ on $\ZNp$ such that 
$\sum_{k \geq 1} k \pi_k = 1$, $p_{1} < 1$, and whose variance is $\sigma^2$.
We do not assume that $\bm{\pi}$ is aperiodic.
Let $d$ denote the period of $\bm{\pi}$, that is, 
the greatest common divisor of the set $\{k \ge 1 \mid \pi_k > 0\}$.
Then $d = 1$ corresponds to the aperiodic case,
and when $d \ge 2$, all the statements below should be understood along the subsequence $n \in d\mathbb{N}$.
Let $T^{GW}$ be the Galton--Watson tree with offspring distribution $\bm{\pi}$
(see \cite[Section~3]{LeGall_06_Random} for its definition).
We write $T_n$ for a random plane tree having the same distribution as $T^{GW}$ conditioned on having exactly $n$ vertices, 
which is well defined for all sufficiently large $n$.
We denote by $\mathbf{P}_n$ the underlying probability measure for $T_n$.

We begin by fixing some notation for $T_n$.
Let $\rho_n$ denote the root of $T_n$,
and let $d_n$ denote the graph metric on $T_n$.
For each $x \in T_n$, we denote by $\deg(x)$ its degree,
that is, the number of vertices adjacent to $x$.
We then define the degree measure $\mu_n$ on $T_n$ by
\begin{equation}
  \mu_n(A) \coloneqq \sum_{x \in A} \deg(x),
  \quad 
  A \subseteq T_n.
\end{equation}
More generally, for each $q > 1$, we define the $q$-th power degree measure $\mu_n^q$ on $T_n$ by
\begin{equation}
  \mu_n^q(A) \coloneqq \sum_{x \in A} \deg(x)^q,
  \quad 
  A \subseteq T_n.
\end{equation}

The scaling limit of $T_n$ is the continuum random tree.
To recall it,
we define $\BExc=(\BExc(t))_{t \in [0,1]}$ to be the normalized Brownian excursion,
built on a probability space with probability measure $\mathbf{P}$.
We write $(T_{\BExc}, d_{\BExc})$ for the real tree coded by $\BExc$
and $p_{\BExc} : [0, 1] \to T_{\BExc}$ for the canonical projection
(see \cite[Section~2]{LeGall_06_Random}).
The canonical Radon measure on $T_{\BExc}$ is given by $\mu_{\BExc} \coloneqq  \Leb \circ (p_{\BExc})^{-1}$,
where $\Leb$ stands for the one-dimensional Lebesgue measure.
We define the root $\rho_{\BExc}$ by setting $\rho_{\BExc} \coloneqq  p_{\BExc}(0)$.
It then holds that  
\begin{equation}
  \left(T_n, \frac{\sigma}{2\sqrt{n}} d_n, \rho_n, \frac{1}{n} \mu_n\right)
  \to
  \left(T_{\BExc}, d_{\BExc}, \rho_{\BExc}, \mu_{\BExc}\right)
  \quad 
  \text{in}\ \mathfrak{K}_\bullet(\finMeasSt),
\end{equation}
where we recall that $\rootCM(\cdot)$ was introduced at the end of Section~\ref{sec: GH general framework}.
This is an immediate consequence of \cite[Theorem~23]{Aldous_93_The_continuum} and \cite[Proposition~3.3]{Abraham_Delmas_Hoscheit_13_A_note}.
See also \cite[Corollary~8.7]{Noda_pre_Convergence} for this extension to offspring distribution with infinite variance.

In order to apply Theorem~\ref{thm: rdm col meas for VSRW},
we need to prove the convergence of the squared degree measures $\mu_n^2$.
As a first step, we establish their tightness.
The following proposition in fact provides a stronger statement.

\begin{prop} \label{prop: degree moment conv}
  Fix $q > 1$.
  Assume that the offspring distribution has a finite $q$-th moment, i.e., $\sum_{k \geq 1} k^q\, \pi_k < \infty$.
  Then, for any $k_0 \geq 0$,
  \begin{equation}
    \lim_{n \to \infty}
    \frac{1}{n} \mathbf{E}_n\!\left[ \sum_{x \in T_n} \deg(x)^q\, \mathbf{1}_{\{\deg(x) \geq k_0\}} \right]
    \;=\; 
    \sum_{k=k_0}^\infty (k+1)^q\, \pi_k .
  \end{equation}
\end{prop}

\begin{proof}
  Let $v_{0} = \rho_{T}, v_{1}, \ldots, v_{n-1}$ be the vertices of $T_n$ in lexicographical order.
  Write $\xi_{n,i}$ for the number of children of $v_i$.
  Let $(Y_i)_{i=1}^\infty$ be i.i.d.\ random variables with distribution $\bm{\pi}$,
  with underlying probability measure $P$.
  By \cite[Lemma~17.1]{Janson_12_Simply},
  the multiset $\{\xi_{n,i}\}_{i=0}^{n-1}$ has the same distribution as $\{Y_i\}_{i=1}^n$ conditioned on 
  $\sum_{i=1}^n Y_i = n-1$.
  Thus, writing $S_n = \sum_{i=1}^n Y_i$, we have 
  \begin{align} \label{prop pr eq: 1. degree moment conv}
    \frac{1}{n} \mathbf{E}_n\!\left[\sum_{i = 0}^{n-1} (\xi_{n,i} + 1)^q\, \mathbf{1}_{\{\xi_{n,i} + 1 \geq k_0\}}\right]
    &=
    \frac{1}{n}E\!\left[\sum_{i=1}^n (Y_i + 1)^q\, \mathbf{1}_{\{Y_i + 1 \geq k_0\}} \,\middle|\, S_n = n-1\right]\\
    &= 
    E\!\left[ (Y_1  + 1)^q\, \mathbf{1}_{\{Y_1 + 1 \geq k_0\}} \,\middle|\, S_n = n-1\right].
  \end{align}
  The last conditional expectation equals
  \begin{equation}  \label{prop pr eq: 2. degree moment conv}
    \frac{E\!\left[ (Y_1+1)^q\, \mathbf{1}_{\{Y_1 + 1 \geq k_0\}} \mathbf{1}_{\{S_n = n-1\}} \right]}{P(S_n = n-1)}
    =
    \sum_{k=k_0}^\infty (k+1)^q \pi_k \,\frac{P(S_{n-1} = n-1-k)}{P(S_n = n-1)},
  \end{equation}
  where we used the Markov property of $(S_n)_{n \geq 1}$.
  By the local central limit theorem,
  \begin{equation}  \label{prop pr eq: 3. degree moment conv}
    \sup_{l \in \mathbb{Z}} 
    \left|
      \sqrt{2\pi \sigma^2 n}\, P(S_n - n = l) - \exp\Bigl(- \frac{l^2}{2 \sigma^2 n} \Bigr)
    \right|
    \xrightarrow[n \to \infty]{} 0
  \end{equation}
  (cf.\ \cite[Chapter~VII, Theorem~1]{Petrov_75_Sums}).
  It follows that 
  \begin{align}
    &\sup_{n \geq 1} \sup_{k \geq 0} \frac{P(S_{n-1} = n-1-k)}{P(S_n = n-1)} < \infty,
    \label{prop pr eq: 4. degree moment conv}\\
    &\lim_{n \to \infty} \frac{P(S_{n-1} = n-1-k)}{P(S_n = n-1)} = 1,\quad \forall k \geq 0.
    \label{prop pr eq: 5. degree moment conv}
  \end{align}
  Hence, by the dominated convergence theorem, 
  \begin{equation}  \label{prop pr eq: 6. degree moment conv}
    \lim_{n \to \infty}
    \frac{1}{n} \mathbf{E}_n\!\left[\sum_{i = 0}^{n-1} (\xi_{n,i} + 1)^q\, \mathbf{1}_{\{\xi_{n,i} + 1 \geq k_0\}}\right]
    =
    \sum_{k=k_0}^\infty (k+1)^q \pi_k.
  \end{equation}
  Except for the root, 
  the degree of a vertex equals its number of children plus one.
  Therefore,
  \begin{align}  \label{prop pr eq: 7. degree moment conv}
    &\left| \frac{1}{n} \mathbf{E}_n\!\left[ \sum_{x \in T_n} \deg(x)^q\, \mathbf{1}_{\{\deg(x) \geq k_0\}} \right] 
      - \frac{1}{n} \mathbf{E}_n\!\left[\sum_{i = 0}^{n-1} (\xi_{n,i} + 1)^q\, \mathbf{1}_{\{\xi_{n,i} + 1 \geq k_0\}}\right]
    \right|\\
    &=
    \frac{1}{n} \mathbf{E}_n\!\left[ \bigl( (\xi_{n,0} + 1)^q - \xi_{n,0}^q \bigr)\, \mathbf{1}_{\{\xi_{n,0} + 1 \geq k_0\}}\right]\\
    &\leq
    \frac{q}{n}\, \mathbf{E}_n\!\left[(\xi_{n,0}+1)^{q - 1}\, \mathbf{1}_{\{\xi_{n,0} + 1 \geq k_0\}}\right],
  \end{align}
  where, for the last inequality, we used the elementary bound
  $(x+1)^q - x^q \le q (x+1)^{q-1}$ for all $x \ge 0$.
  By \cite[Lemma~15.7]{Janson_12_Simply},
  \begin{equation}
    \mathbf{P}_n(\xi_{n,0} = k) = \frac{n}{n-1}\, k\, P\!\left(Y_1 = k \,\middle|\, S_n = n-1\right),
    \quad 
    \forall\, k \ge 0.
  \end{equation}
  Hence,
  \begin{align}
    \mathbf{E}_n\!\left[(\xi_{n,0} + 1)^{q-1}\, \mathbf{1}_{\{\xi_{n,i} + 1 \geq k_0\}}\right]
    &=
    \frac{n}{n-1}\sum_{k=k_0}^\infty (k+1)^{q - 1}\, k\, P\!\left(Y_1 = k \,\middle|\, S_n = n-1\right)\\
    &\le
    \frac{n}{n-1}\, E\!\left[(Y_1+1)^q\, \mathbf{1}_{\{Y_1 + 1 \geq k_0\}} \,\middle|\, S_n = n-1\right],
  \end{align}
  which converges to the same limit as \eqref{prop pr eq: 6. degree moment conv}, as $n \to \infty$.
  Combining \eqref{prop pr eq: 6. degree moment conv} and \eqref{prop pr eq: 7. degree moment conv} yields the desired result.
\end{proof}

Combining the above proposition with the results of \cite[Section~7]{Noda_pre_Aging},
we obtain the convergence of the squared degree measures as follows.

\begin{thm} \label{thm: squared measure conv on GW}
  It holds that 
  \begin{equation}
    \left(T_n, \frac{\sigma^2}{2\sqrt{n}} d_n, \rho_n, \frac{1}{n} \mu_n, \frac{1}{n} \mu_n^2\right)
    \to
    \left(T_{\BExc}, d_{\BExc}, \rho_{\BExc}, \mu_{\BExc}, (\sigma^2 + 4) \mu_{\BExc}\right).
  \end{equation}
   in $\mathfrak{K}_\bullet(\MeasSt^{\otimes 2})$.
\end{thm}

\begin{proof}
  Define a Radon measure $\dot{\mu}_n$ on $T_n \times \NN$ by setting
  \begin{equation}
    \dot{\mu}_n(\{x\} \times \{k\}) = \mathbf{1}_{\{\deg(x) = k\}},
    \quad 
    x \in T_n,\ k \in \NN.
  \end{equation}
  Define a probability measure $\tilde{\bm{\pi}} =  (\tilde{\pi}_k)_{k \geq 1}$ on $\NN$ 
  by setting $\tilde{\pi}_k \coloneqq \pi_{k-1}$ for each $k \geq 1$.
  By following the proof of \cite[Corollary~7.4]{Noda_pre_Aging}, 
  we deduce that 
  \begin{equation}  \label{thm pr eq: 1. squared measure conv on GW}
    \left(T_n, \frac{\sigma^2}{2\sqrt{n}} d_n, \rho_n, \frac{1}{n} \mu_n, \frac{1}{n} \dot{\mu}_n\right)
    \xrightarrow{\mathrm{d}}
    \left(T_{\BExc}, d_{\BExc}, \rho_{\BExc}, \mu_{\BExc}, \mu_{\BExc} \otimes \tilde{\bm{\pi}}\right)
  \end{equation}
  in $\rootCM(\finMeasSt \times \finMeasSt(\Psi_{\id \times \RNp}))$.
  For each $M \in \NN$, we define a Radon measure $\mu_{n,M}^2$ on $T_n$ by 
  \begin{equation}   \label{thm pr eq: 2. squared measure conv on GW}
    \mu_{n, M}^2(\{x\}) \coloneqq (\deg(x) \wedge M)^2.
  \end{equation}
  For any function $f$ on $T_n$, we have 
  \begin{equation}   \label{thm pr eq: 3. squared measure conv on GW}
    \int_{T_n} f(x)\, \mu_{n,M}^2(dx) 
    = 
    \sum_{x \in T_n} f(x) (\deg(x) \wedge M)^2
    =
    \int_{T_n \times \NN} f(x) (k \wedge M)^2\, \dot{\mu}_n(dx\, dk).
  \end{equation}
  Thus, if we define a Radon measure $\mu_{\BExc, M}$ on $T_{\BExc}$ by
  \begin{equation}   \label{thm pr eq: 4. squared measure conv on GW}
    \mu_{\BExc,M}(dx) \coloneqq \left( \sum_{k = 1}^\infty (k \wedge M)^2 \tilde{\pi}_k\right)\, \mu_{\BExc}(dx),
  \end{equation}
  then we deduce from \eqref{thm pr eq: 1. squared measure conv on GW} that 
  \begin{equation}  \label{thm pr eq: 5. squared measure conv on GW}
    \left(T_n, \frac{\sigma^2}{2\sqrt{n}} d_n, \rho_n, \frac{1}{n} \mu_n, \frac{1}{n}\mu_{n,M}^2\right)
    \xrightarrow{\mathrm{d}}
    \left(T_{\BExc}, d_{\BExc}, \rho_{\BExc}, \mu_{\BExc}, \mu_{\BExc, M}\right).
  \end{equation}
  For any subset $A \subseteq T_n$, it holds that 
  \begin{equation}  \label{thm pr eq: 6. squared measure conv on GW}
    \frac{1}{n}\left|\mu_{n,M}^2(A) - \mu_n^2(A)\right|
    \leq 
    \frac{1}{n} \sum_{x \in T_n} \deg(x)^2\, \mathbf{1}_{\{\deg(x) > M\}}.
  \end{equation}
  Since 
  \begin{equation} \label{thm pr eq: 7. squared measure conv on GW}
    \sum_{k=1}^\infty k^2 \tilde{\pi}_k 
    = 
    \sum_{k=0}^\infty (k+1)^2\, p_k  
    =
    \sigma^2 + 4,
  \end{equation}
  we have that, for any Borel subset $A$ of $T_{\BExc}$,
  \begin{equation} \label{thm pr eq: 8. squared measure conv on GW}
    |(\sigma^2 +4)\mu_{\BExc}(A) - \mu_{\BExc, M}(A)| \leq \sum_{k > M} k^2\, \tilde{\pi}_k.
  \end{equation}
  Now the result is immediate from Proposition~\ref{prop: degree moment conv}, 
  \eqref{thm pr eq: 5. squared measure conv on GW},
  \eqref{thm pr eq: 6. squared measure conv on GW}, and \eqref{thm pr eq: 8. squared measure conv on GW}.
\end{proof}

We are ready to prove the main result, Theorem~\ref{thm: conv of col meas of CSRW on GW trees} below.
Let $Y_n^1$ be the CSRW on $T_n$ and $Y_n^2$ be an independent copy of $Y_n^1$.
For each $i \in \{1,2\}$, we define the scaled walk $X_n^i$ by 
\begin{equation}
  X_n^i(t) \coloneqq Y_n^i\bigl( 2\sigma^{-1}n^{3/2} t \bigr), \quad t \geq 0.
\end{equation}
We write $\hat{X}_n = (X_n^1, X_n^2)$ for their product process.
Define a collision measue $\Pi_n$ by 
\begin{equation}
  \Pi_n(dx dt) \coloneqq n \sum_{y \in T_n} \mathbf{1}_{\{X_n^1(t) = X_n^2(t) = y\}}\, \delta_y(dx)\, dt.
\end{equation}
By \cite[Corollary~8.11]{Noda_pre_Convergence},
$(T_{\BExc}, d_{\BExc}, \rho_{\BExc}, \mu_{\BExc}) \in \rootResisSp$, $\mathbf{P}$-a.s.
Given a realization of $(T_{\BExc}, d_{\BExc}, \rho_{\BExc}, \mu_{\BExc})$,
we write $X^1$ for the associated process and $X^2$ for an independent copy of $X^1$.
We then set $\hat{X} = (X^1, X^2)$ for their product process.

\begin{thm} \label{thm: conv of col meas of CSRW on GW trees}
  Assume that the offspring distribution has a finite $q$-th moment for some $q > 3$.
  Then $\diagMeas{\mu_{\BExc}}$ belongs to the local Kato class of $\hat{X}$, $\mathbf{P}$-a.s.
  If we write $\Pi$ for the collision measure of $X^1$ and $X^2$ associated with $(\sigma^2 + 4)\mu_{\BExc}$,
  then 
  \begin{equation}
    \left(T_n, \frac{\sigma}{2\sqrt{n}} d_n, \rho_n, \frac{1}{n} \mu_n, \frac{1}{n} \mu_n^2, \ProcLaw_{(\hat{X}_n, \Pi_n)}\right)
    \to
    \left(T_{\BExc}, d_{\BExc}, \rho_{\BExc}, \mu_{\BExc}, (\sigma^2 + 4) \mu_{\BExc}, \ProcLaw_{(\hat{X}, \Pi)}\right).
  \end{equation}
  in $\mathfrak{K}_\bullet(\MeasSt^{\otimes 2} \times \ColSt)$.
\end{thm}

\begin{proof}
  By placing conductance $1$ on each edge of $T_n$,
  we regard $T_n$ as an electrical network.
  Then the graph metric $d_n$ coincides with the resistance metric on $T_n$
  (see \cite[Proposition~5.1]{Kigami_95_Harmonic}).
  Thus, it is enough to verify that Assumption~\ref{assum: rdm col meas for CSRW} is satisfied.

  By Theorem~\ref{thm: squared measure conv on GW}, 
  Assumption~\ref{assum: rdm col meas for CSRW}\ref{assum item: 1. rdm col meas for CSRW} holds.
  Moreover, the same theorem implies that 
  the diameters of the metric spaces $(T_n, \tfrac{\sigma}{2\sqrt{n}} d_n)$ converge 
  in distribution to the diameter of $(T_{\BExc}, d_{\BExc})$.
  (This follows from the continuity of the diameter with respect to the Gromov--Hausdorff topology;
  see \cite[Exercise~7.3.14]{Burago_Burago_Ivanov_01_A_course}.)
  From this, 
  one can verify Assumption~\ref{assum: rdm col meas for CSRW}\ref{assum item: 2. rdm col meas for CSRW}.
  Furthermore, by \cite[Proposition~8.10]{Noda_pre_Convergence},
  we deduce that for every $\varepsilon > 0$ 
  there exists a constant $c_{q, \varepsilon} > 0$ such that 
  \begin{equation}
    \liminf_{n \to \infty} 
    \mathbf{P}_n\!\left(
      \inf_{x \in T_n} n^{-1} \mu_n\Bigl( B_{d_n}\bigl(x, 2\sigma^{-1}\sqrt{n} r\bigr) \Bigr) 
      \geq c_{q, \varepsilon}\, r^{q-1},\ \ r \in (0,1)
    \right)
    \geq 1 - \varepsilon .
  \end{equation}
  Combining this with Proposition~\ref{prop: degree moment conv},
  we obtain Assumption~\ref{assum: rdm col meas for CSRW}\ref{assum item: 3. rdm col meas for CSRW}.
  Therefore, the desired result follows from Theorem~\ref{thm: rdm col meas for CSRW}.
\end{proof}

\begin{rem}
  The above theorem is proved under the finite $(3+\varepsilon)$-moment assumption,
  but the convergence of the collision measures is in fact expected to hold
  under the weaker assumption that the offspring distribution has finite variance.
  Indeed, in this case the squared degree measures still converge
  to $(\sigma^2 + 4)\mu_{\BExc}$ by Theorem~\ref{thm: squared measure conv on GW},
  so the result should follow once the heat-kernel condition 
  (Assumption~\ref{assum: col rdm assumption}\ref{assum item: 2. col rdm assumption}) is verified.
  However, verifying this condition appears to require finer heat-kernel estimates
  beyond the argument based on Hölder's inequality and lower volume bounds used in Theorem~\ref{thm: rdm col meas for CSRW}.
  In the infinite-variance regime, by contrast, 
  the squared degree measures are expected to become singular 
  with respect to the invariant measures,
  suggesting that collisions would concentrate on vertices of exceptionally large degrees,
  which would be an intriguing direction for future study.
\end{rem}

\appendix


\section{Omitted proofs and technical results} \label{appendix: Technical results}
\subsection{Proof of Proposition~\ref{prop: bounded potential and polarity}} \label{appendix: proof of potential result}

Recall the setting of Section~\ref{sec: smooth measure}.
In particular, we consider a standard process $X$ satisfying the DAC condition (Assumption~\ref{assum: dual hypothesis}),
with dual process $\check{X}$, reference measure $m$, and heat kernel $p$.
In this appendix,
we prove Proposition~\ref{prop: bounded potential and polarity}.
To this end, we recall several notions from the potential theory of Markov processes.

Recall that, given a Borel measure $\nu$ on $S$ and $\alpha > 0$,
its $\alpha$-\emph{potential} is given by
\begin{equation}
  \Potential^\alpha \nu(x)
  \coloneqq 
  \int r^\alpha(x,y)\, \nu(dy),
  \quad x \in S.
\end{equation}
When $\nu(dx) = f(x)\, m(dx)$ for some non-negative Borel measurable function $f$,
we simply write $\Potential^\alpha f \coloneqq \Potential^\alpha \nu$.
The mapping $f \mapsto \Potential^\alpha f$ is called the $\alpha$-\emph{potential (or resolvent) operator}.

A measurable function $f \colon S \to [0,\infty)$ is said to be \emph{$\alpha$-excessive}
if it satisfies
\begin{equation}
  e^{-\alpha t} E^x[f(X_t)] \le f(x),
  \quad t \ge 0,\ x \in S,
\end{equation}
and
\begin{equation}
  \lim_{t \to 0} e^{-\alpha t} E^x[f(X_t)] = f(x),
  \quad x \in S,
\end{equation}
(see \cite[Definition~2.1 in Chapter~III]{Blumenthal_Getoor_68_Markov}).
A $\sigma$-finite Borel measure $\nu$ on $S$ is said to be \emph{$\alpha$-excessive}
if and only if it satisfies
\begin{equation}
  \int_S e^{-\alpha t} E^x[f(X_t)]\, \nu(dx)
  \le
  \int_S f(x)\, \nu(dx),
  \quad \forall t \ge 0,
\end{equation}
for all non-negative Borel measurable functions $f$ on $S$
(see \cite[Definition~1.10 in Chapter~VI]{Blumenthal_Getoor_68_Markov}).

Let $m'$ be a Borel measure on $S$.
A Borel subset $A$ of $S$ is called \emph{$m'$-polar} if and only if 
\begin{equation}
  \int_S P^x(\sigma_A < \infty)\, m'(dx) = 0
\end{equation}
(see \cite[Definition~A.2.12]{Chen_Fukushima_12_Symmetric}).

Concepts corresponding to the dual process $\check{X}$ are prefixed with ``co-''; 
for example, a function that is excessive for $\check{X}$ is called a \emph{coexcessive function}.
We also attach the symbol $\check{\cdot}$ to the corresponding notation;
for instance, $\check{\Potential}^\alpha$ denotes the $\alpha$-potential operator associated with $\check{X}$.

We are now ready to prove Proposition~\ref{prop: bounded potential and polarity}.

\begin{proof}[Proof of Proposition~\ref{prop: bounded potential and polarity}]
  The claim follows by applying \cite[Theorem~3.4.1]{Beznea_Boboc_04_Potential} 
  to the shifted resolvent operators $\{\check{\Potential}^{\alpha + \beta}\}_{\beta > 0}$ of the dual process $\check{X}$.
  Our argument below follows the proof of \cite[Theorem~3.4.2]{Beznea_Boboc_04_Potential}.

  Let $g \colon S \to (0,1]$ be a Borel measurable function such that $\int g(x)\, m(dx) < \infty$,
  which exists by the $\sigma$-finiteness of $m$.
  Define Borel measures $m'$ and $\mu'$ on $S$ by
  \begin{equation}
    m' \coloneqq (g \cdot m) \circ \check{\Potential}^\alpha, 
    \qquad 
    \mu' \coloneqq \mu \circ \check{\Potential}^\alpha.
  \end{equation}
  Here, $\check{\Potential}^\alpha$ is regarded as an operator acting on measures, 
  that is, for a Borel measure $\nu$ on $S$, 
  the measure $\nu \circ \check{\Potential}^\alpha$ is defined by
  \begin{equation}
    (\nu \circ \check{\Potential}^\alpha)(f)
    \coloneqq 
    \int_S \check{\Potential}^\alpha f(x)\, \nu(dx),
    \quad 
    f \colon S \to \RNp \text{ Borel measurable.}
  \end{equation}
  By duality, for each non-negative Borel measurable function $f$ on $S$ we have
  \begin{align}
    \int f(x)\, m'(dx) 
    &= 
    \int \check{\Potential}^\alpha f(x)\, g(x)\, m(dx)
    =
    \int f(y)\, \Potential^\alpha g(y)\, m(dy),
    \\
    \int f(x)\, \mu'(dx)  
    &=
    \int \check{\Potential}^\alpha f(x)\, \mu(dx)
    =
    \int f(y)\, \Potential^\alpha \mu(y)\, m(dy).
  \end{align}
  Hence, 
  \begin{equation}
    m'(dx) = \Potential^\alpha g(x)\, m(dx),
    \qquad 
    \mu'(dx) = \Potential^\alpha \mu(x)\, m(dx).
  \end{equation}
  Since $\Potential^\alpha g$ and $\Potential^\alpha \mu$ are $\alpha$-excessive,
  both $m'$ and $\mu'$ are $\alpha$-coexcessive
  (cf.\ \cite[Proposition~1.11 in Chapter~VI]{Blumenthal_Getoor_68_Markov}).

  Noting that $g(x) > 0$ for all $x \in S$,
  we have $\Potential^\alpha g(x) > 0$ for all $x \in S$,
  and hence $m$ is absolutely continuous with respect to $m'$.
  Let $h \coloneqq \tfrac{dm}{dm'}$ denote the Radon--Nikodym derivative.
  For each $n \in \NN$, define a Borel subset $E_n \coloneqq \{x \in S \mid h(x) \leq n\}$,
  and let $\mu'_n \coloneqq \mu'|_{E_n}$.
  Then, for any Borel subset $B \subseteq S$,
  \begin{equation}
    \mu'_n(B) 
    \leq \|\Potential^\alpha \mu\|_\infty\, m(E_n \cap B) 
    = \|\Potential^\alpha \mu\|_\infty \int_{E_n \cap B} h(x)\, m'(dx) 
    \leq n \|\Potential^\alpha \mu\|_\infty\, m'(B).
  \end{equation}
  By assumption, $\|\Potential^\alpha \mu\|_\infty < \infty$.
  Therefore, $\mu'$ is $m'$-quasi bounded in the sense of \cite[p.~88]{Beznea_Boboc_04_Potential}.
  By \cite[Corollary~1.8.6 and Theorem~3.4.1]{Beznea_Boboc_04_Potential},
  it follows that $\mu(E) = 0$ for any Borel set $E$ that is both $m'$-copolar and $m$-copolar.
  Since $m'$ is absolutely continuous with respect to $m$, 
  every $m$-copolar set is also $m'$-copolar.
  Hence, we have $\mu(E) = 0$ for any $m$-copolar Borel set $E$.
  Finally, by \cite[Proposition~1.19 in Chapter~VI]{Blumenthal_Getoor_68_Markov},
  every polar set is copolar (and hence $m$-copolar),
  which yields the desired result.
\end{proof}

\subsection{Proof of Lemmas~\ref{lem: apPCAF continuity} and \ref{lem: apSTOM continuity}} \label{appendix: Lemma of PCAF/STOM approx}

Recall the setting of Section~\ref{sec: Continuity of joint laws with PCAFs},
and in particular the maps $\apPCAF^{(\delta, R)}$ and $\apSTOM^{(\delta, R)}$
introduced in Definitions~\ref{dfn: apPCAF} and \ref{dfn: apSTOM}, respectively.
We provide the proofs of Lemmas~\ref{lem: apPCAF continuity} and \ref{lem: apSTOM continuity},
which are recalled below for convenience.

\begin{lem} \label{ap lem: apPCAF continuity}
  Fix $\delta > 0$ and $R > 1$.
  \begin{enumerate} [label = \textup{(\roman*)}]
    \item \label{ap lem item: apPCAF continuity. 1}
      The map $\apPCAF^{(\delta, R)}$ is Borel measurable.
    \item \label{ap lem item: apPCAF continuity. 2}
      Fix an arbitrary sequence $(q_n, \nu_n, \eta_n)_{n \geq 1}$ converging to $(q, \nu, \eta)$ in the domain of $\apPCAF^{(\delta, R)}$.
      Assume that
      \begin{equation}  \label{ap lem eq: apPCAF continuity}
        \sup_{n \geq 1} 
        \sup_{y \in S} \int_{\delta}^{2\delta} \int_S q_n(u, y, z)\, \nu_n^{(R)}(dz)\, du < \infty.
      \end{equation}
      Then $\apPCAF^{(\delta, R)}(q_n, \nu_n, \eta_n)$ converges to $\apPCAF^{(\delta, R)}(q, \nu, \eta)$ in $\upC(\RNp, \RNp)$.
  \end{enumerate}
\end{lem}

\begin{lem} \label{ap lem: apSTOM continuity}
  Fix $\delta > 0$ and $R > 1$.
  \begin{enumerate} [label = \textup{(\roman*)}]
    \item \label{ap lem item: apSTOM continuity. 1}
      The map $\apSTOM^{(\delta, R)}$ is Borel measurable.
    \item \label{ap lem item: apSTOM continuity. 2}
      Fix an arbitrary sequence $(q_n, \nu_n, \eta_n)_{n \geq 1}$ converging to $(q, \nu, \eta)$ in the domain of $\apSTOM^{(\delta, R)}$.
      Assume that
      \begin{equation}  \label{ap lem eq: apSTOM continuity}
        \sup_{n \geq 1} 
        \sup_{y \in S} \int_{\delta}^{2\delta} \int_S q_n(u, y, z)\, \nu_n^{(R)}(dz)\, du < \infty.
      \end{equation}
      Then $\apSTOM^{(\delta, R)}(q_n, \nu_n, \eta_n)$ converges to $\apSTOM^{(\delta, R)}(q, \nu, \eta)$ in $\STOMMeas(S \times \RNp)$.
  \end{enumerate}
\end{lem}

We first prove Lemma~\ref{ap lem: apSTOM continuity}.

\begin{proof} [{Proof of Lemma~\ref{ap lem: apSTOM continuity}}]
  Since the argument is rather long and technical, we first give an outline. 
  We begin by constructing a family of continuous maps $\{\apSTOM^{(\delta, R, r)}\}_{r > 0}$. 
  Condition~\ref{ap lem item: apPCAF continuity. 1} is then verified by showing that 
  $\apSTOM^{(\delta, R, r)}$ converges pointwise to $\apSTOM^{(\delta, R)}$ as $r \to \infty$. 
  Furthermore, condition~\ref{ap lem item: apPCAF continuity. 2} follows from the fact that, under \eqref{ap lem eq: apPCAF continuity}, 
  the family $\apSTOM^{(\delta, R, r)}$ uniformly approximates $\apSTOM^{(\delta, R)}$.

  To define $\apPCAF^{(\delta, R, r)}$,
  we use the cutoff function introduced in \eqref{lem pr: smooth truncation of meas. 1},
  that is, 
  \begin{equation}  
    \chi^{(r)}(x) \coloneqq \int_{r-1}^r \mathbf{1}_{S^{(v)}}(x)\, dv, 
    \quad x \in S,\ r > 1.
  \end{equation}
  Recall that $\chi^{(r)}$ is continuous.
  Moreover,
  \begin{equation}  \label{lem pr: measurability of apSTOM. 1}
    \chi^{(r)} \nearrow \mathbf{1}_S 
    \quad \text{pointwise as } r \to \infty.
  \end{equation}

  For each $r > 1$, we define a map $\apSTOM^{(\delta, R, r)}$
  \begin{equation}
    \apSTOM^{(\delta, R, r)} \colon 
    C(\RNpp \times S \times S, \RNp) \times \Meas(S) \times D_{L^0}(\RNp, S) 
    \to 
    \upC(\RNp, \RNp)
  \end{equation}
  by 
  \begin{equation}
    (q, \nu, \eta) 
    \mapsto 
    \frac{1}{\delta}
    \int_0^\infty \mathbf{1}_{(\eta_t, t)}(\cdot) \int_S \int_\delta^{2\delta} 
    \chi^{(r)}(\eta_t)\, q(u, \eta_t, x)\, du\, \tilde{\nu}^{(R)}(dx)\, dt.
  \end{equation}
  Using this map $\apSTOM^{(\delta, R, r)}$,
  we prove the results simultaneously.
  Fix an arbitrary sequence $(q_n, \nu_n,\allowbreak \eta_n)$, $n \geq 1$, converging to $(q, \nu, \eta)$ 
  in the domain of $\apSTOM^{(\delta, R, r)}$.
  Fix a bounded continuous function $F$ on $S \times \RNp$
  whose support is contained in $\RNp \times [0,T]$ for some $T \in \RNp$.
  Write 
  \begin{align}
    G_n^{(r)} &\coloneqq 
    \frac{1}{\delta}
    \int_0^T F(\eta_n(t), t) \int_S \int_\delta^{2\delta} 
    \chi^{(r)}(\eta_n(t))\, q_n(u, \eta_n(t), x)\, du\, \tilde{\nu}_n^{(R)}(dx)\, dt,\\
    G^{(r)} &\coloneqq 
    \frac{1}{\delta}
    \int_0^T F(\eta(t), t) \int_S \int_\delta^{2\delta} 
    \chi^{(r)}(\eta(t))\, q(u, \eta(t), x)\, du\, \tilde{\nu}^{(R)}(dx)\, dt,\\
    G_n &\coloneqq 
    \frac{1}{\delta}
    \int_0^T F(\eta_n(t), t) \int_S \int_\delta^{2\delta} 
    q_n(u, \eta_n(t), x)\, du\, \tilde{\nu}_n^{(R)}(dx)\, dt,\\
    G &\coloneqq 
    \frac{1}{\delta}
    \int_0^T F(\eta(t), t) \int_S \int_\delta^{2\delta} 
    q(u, \eta(t), x)\, du\, \tilde{\nu}^{(R)}(dx)\, dt.
  \end{align}

  We first show that $\apSTOM^{(\delta, R, r)}$ is continuous for each $r > 0$.
  Fix $r > 0$.
  It suffices to show that 
  \begin{equation} \label{lem pr: apSTOM continuity. 2}
    G_n^{(r)} \xrightarrow[n \to \infty]{} G^{(r)}.
  \end{equation}
  By the convergence of $q_n$,
  we have  
  \begin{equation}  \label{lem pr: apSTOM continuity. 1}
    \sup_{n \geq 1} 
    \left\{
      q_n(u, y, z)
      \;\middle|\;
      u \in [\delta, 2\delta],\ 
      y \in S^{(r)},\
      x \in S^{(R)}
    \right\}
    < \infty.
  \end{equation}
  For each $n \geq 1$, define 
  \begin{equation}
    f_n(x) \coloneqq 
    \frac{1}{\delta}
    \int_0^T 
    \int_\delta^{2\delta}
      F(\eta_n(t), t)\, \chi^{(r)}(\eta_n(t))\, q_n(u, \eta_n(t), x)\, du\, dt,
    \quad x \in S.
  \end{equation}
  Similarly, define 
  \begin{equation}
    f(x) \coloneqq 
    \frac{1}{\delta}
    \int_0^T
    \int_\delta^{2\delta}
      F(\eta(t), t)\, \chi^{(r)}(\eta(t))\, q(u, \eta(t), x)\, du\,dt,
    \quad x \in S.
  \end{equation}
  By the dominated convergence theorem,
  one readily verifies that all the functions $f_n$ and $f$ are continuous.
  By definition, 
  \begin{equation}
    G_n^{(r)} 
    = \int_S f_n(x)\, \tilde{\nu}_n^{(R)}(dx),
    \quad 
    G^{(r)}
    = \int_S f(x)\, \tilde{\nu}^{(R)}(dx).
  \end{equation}
  Hence, it suffices to show that 
  \begin{equation}
    \lim_{n \to \infty} 
    \int_S f_n(x)\, \tilde{\nu}_n^{(R)}(dx) 
    = 
    \int_S f(x)\, \tilde{\nu}^{(R)}(dx).
  \end{equation}
  Since $\tilde{\nu}_n^{(R)} \to \tilde{\nu}^{(R)}$ weakly,
  Lemma~\ref{lem: vague convergence and hatC topology} implies that 
  the above convergence holds once we show $f_n \to f$ in the compact-convergence topology.

  Let $(x_n)_{n \geq 1}$ be a sequence converging to $x$ in $S$,
  and fix an arbitrary subsequence $(n_k)_{k \geq 1}$.
  By the $L^0$-convergence of $\eta_n$ to $\eta$,
  we can find a further subsequence $(n_{k(l)})_{l \geq 1}$ 
  such that $\eta_{n_{k(l)}}(t) \to \eta(t)$ for Lebesgue-almost every $t \geq 0$.
  For every such $t$, by the convergence of $q_n$ to $q$ and the continuity of $\chi^{(r)}$, 
  we have that, for any $u >0$,
  \begin{equation}
    \chi^{(r)}(\eta_{n_{k(l)}}(t))\, q_{n_{k(l)}}\bigl( u, \eta_{n_{k(l)}}(t), x_{n_{k(l)}} \bigr) 
    \xrightarrow[l \to \infty]{} 
    \chi^{(r)}(\eta(t))\,q(u, \eta(s), x).
  \end{equation}
  This, together with \eqref{lem pr: apSTOM continuity. 1} and the dominated convergence theorem, yields
  \begin{align}
    f_{n_{k(l)}}(x_{n_{k(l)}}) 
    &= 
    \frac{1}{\delta}
    \int_0^T 
    \int_\delta^{2\delta}
      F(\eta_{n_{k(l)}}(t), t)\, \chi^{(r)}(\eta_{n_{k(l)}}(t))\, q_{n_{k(l)}}\bigl( u, \eta_{n_{k(l)}}(t), x_{n_{k(l)}} \bigr)\, du\, dt\\
    &\xrightarrow[l \to \infty]{}
    \frac{1}{\delta}
    \int_0^T
    \int_\delta^{2\delta}
      F(\eta(t), t)\, \chi^{(r)}(\eta(t))\, q(u, \eta(t), x)\, du\,dt\\
    &= f(x).
  \end{align}
  Since the subsequence $(n_k)_{k \geq 1}$ was arbitrary,
  we conclude that $f_n(x_n) \to f(x)$, which implies that $f_n \to f$ in the compact-convergence topology.
  Therefore, the map $\apSTOM^{(\delta, R, r)}$ is continuous for each $r > 1$.

  By \eqref{lem pr: measurability of apSTOM. 1} and the monotone convergence theorem,
  one readily verifies that 
  \begin{equation} \label{lem pr: apSTOM continuity. 3}
    G^{(r)} \xrightarrow[r \to \infty]{} G,
  \end{equation}
  which implies that 
  \begin{equation} 
    \apSTOM^{(\delta, R, r)}(q, \nu, \eta) 
    \xrightarrow[r \to \infty]{} 
    \apSTOM^{(\delta, R)}(q, \nu, \eta)
    \quad \text{in}\ \STOMMeas(S \times \RNp).
  \end{equation}
  In particular, the map $\apSTOM^{(\delta, R)}$ is the pointwise limit of continuous maps, 
  and hence we obtain \ref{ap lem item: apSTOM continuity. 1}.

  It remains to verify \ref{ap lem item: apSTOM continuity. 2}.
  Assume that the above sequence $(q_n, \nu_n, \eta_n)_{n \geq 1}$ satisfies \eqref{ap lem eq: apSTOM continuity},
  that is,
  \begin{equation}
    M \coloneqq 
    \sup_{n \geq 1} 
        \sup_{y \in S} 
        \int_{\delta}^{2\delta} \int_S q_n(u, y, z)\, \nu_n^{(R)}(dz)\, du 
    < \infty.
  \end{equation}
  By definition, 
  \begin{equation}
    |\chi^{(r)}(x) - 1| 
    \leq 
    \mathbf{1}_{S \setminus S^{(r-1)}}(x),
    \quad 
    \forall x \in S,\ r > 1.
  \end{equation}
  This yields 
  \begin{align}
    &|G_n^{(r)} - G_n|\\
    &\leq
    \frac{\|F\|_\infty}{\delta}\int_0^T \int_S \int_\delta^{2\delta} 
      \mathbf{1}_{S \setminus S^{(r-1)}}(\eta_n(t))\, 
      q_n(u, \eta_n(t), x)\, du\, 
      \tilde{\nu}_n^{(R)}(dx)\, dt\\
    &\leq 
    \frac{\|F\|_\infty M}{\delta}\, 
    \Leb\bigl( \{t \in [0, T] \mid \eta_n(t) \notin S^{(r-1)}\} \bigr).
  \end{align}
  By Lemma~\ref{lem: range compactness wrt L^0 topology}, 
  the last expression converges to $0$ in the successive limits 
  $n \to \infty$ and then $r \to \infty$, that is,
  \begin{equation}
    \lim_{r \to \infty} 
    \limsup_{n \to \infty}
    |G_n^{(r)} - G_n|
    = 0.
  \end{equation}
  Combining this with \eqref{lem pr: apSTOM continuity. 2} and \eqref{lem pr: apSTOM continuity. 3},
  we obtain that $G_n \to G$, which yields the conclusion of \ref{ap lem item: apSTOM continuity. 2}.
\end{proof}

Lemma~\ref{ap lem: apPCAF continuity} is then readily verified by Lemma~\ref{ap lem eq: apSTOM continuity}.

\begin{proof}[{Proof of Lemma~\ref{ap lem: apPCAF continuity}}]
  Define a map 
  \begin{equation}
    \mathfrak{T} \colon 
    \STOMMeas(S \times \RNp) \ni \Sigma 
    \longmapsto \bigl(t \mapsto \Sigma(S \times [0,t])\bigr) 
    \in L^0(\RNp, \RNp).
  \end{equation}
  It is readily verified that the map $\mathfrak{T}$ is measurable (in fact, continuous).
  By definition, we have $\apPCAF^{(\delta, R)} = \mathfrak{T} \circ \apSTOM^{(\delta, R)}$.
  Hence the map $\apPCAF^{(\delta, R)}$ is measurable when its codomain is taken to be $L^0(\RNp, \RNp)$.
  Moreover, as in Proposition~\ref{prop: D is Borel in L^0},
  the space $\upC(\RNp, \RNp)$ is a Borel subset of $L^0(\RNp, \RNp)$.
  Therefore, $\apPCAF^{(\delta, R)}$ is Borel measurable 
  with respect to its original codomain $\upC(\RNp, \RNp)$.

  Fix an arbitrary sequence $(q_n, \nu_n, \eta_n)_{n \ge 1}$ 
  converging to $(q, \nu, \eta)$ in the domain of $\apPCAF^{(\delta, R)}$.
  Then, by Lemma~\ref{ap lem: apSTOM continuity}\ref{ap lem item: apSTOM continuity. 2}, we have 
  \begin{equation}
    \apSTOM^{(\delta, R)}(q_n, \nu_n, \eta_n) 
    \xrightarrow[n \to \infty]{} 
    \apSTOM^{(\delta, R)}(q, \nu, \eta)
    \quad \text{in} \quad \STOMMeas(S \times \RNp).
  \end{equation}
  By the same argument as in the proof of 
  Proposition~\ref{prop: PCAF to STOM map}\ref{prop item: PCAF to STOM map. 1},
  we deduce that 
  \begin{equation}
    \apSTOM^{(\delta, R)}(q_n, \nu_n, \eta_n)(S \times [0,\cdot]) 
    \xrightarrow[n \to \infty]{} 
    \apSTOM^{(\delta, R)}(q, \nu, \eta)(S \times [0,\cdot])
    \quad \text{in} \quad \upC(\RNp, \RNp),
  \end{equation}
  which completes the proof.
\end{proof}

\subsection{The function \texorpdfstring{$\trf_\beta$}{trf_beta}} \label{sec: trf function}

Here, we prove some technical results that are used in Section~\ref{sec: conv of col meas of VSRW}.
For each $\beta > 0$, define a bijective function $\trf_\beta \colon (0,1/e) \to (0,\infty)$ by setting
\begin{equation}
  \trf_\beta(t) \coloneqq \frac{1}{\log (1/t)\, \bigl(\log \log (1/t)\bigr)^{1+\beta}} .
\end{equation}
We write $\trf_\beta^{-1} \colon (0, \infty) \to (0,1/e)$ for its inverse.

\begin{lem} \label{lem: property of trf}
  For each $\beta > 0$, the following statements hold.
  \begin{enumerate} [label = \textup{(\roman*)}]
    \item \label{lem item: 1. property of trf}
      It holds that 
      \begin{equation}
        \lim_{\delta \to 0}
        \int_0^\delta \frac{\trf_\beta(t)}{t}\, dt 
        = 0.
      \end{equation}
    \item \label{lem item: 2. property of trf}
      For any $\alpha > 0$, 
      \begin{equation}
        \lim_{\delta \to 0} \int_0^\delta \frac{1}{\bigl(\trf_\beta(t)\bigr)^\alpha}\, dt \;=\; 0 .
      \end{equation}
    \item \label{lem item: 3. property of trf}
      As $t \downarrow 0$, 
      \begin{equation}
        \trf_\beta^{-1}(t) \;=\; 
        \exp\!\left( - \frac{1 + o(1)}{\,t \bigl(\log (1/t)\bigr)^{1 + \beta}}\right).
      \end{equation}
  \end{enumerate}
\end{lem}

\begin{proof}
  \ref{lem item: 1. property of trf}.
  By the change of variables $s = \log(1/t)$, 
  \begin{equation}
     \int_0^\delta \frac{\trf_\beta(t)}{t}\, dt
     = 
     \int_{\log(1/\delta)}^\infty \frac{1}{s \bigl(\log s\bigr)^{1+\beta}}\, ds
     =
     \left[ -\frac{1}{\beta}\, \bigl(\log s\bigr)^{-\beta} \right]_{\log(1/\delta)}^\infty
     = 
     \frac{1}{\beta \bigl(\log \log (1/\delta)\bigr)^{\beta}},
  \end{equation}
  which yields the desired result.

  \ref{lem item: 2. property of trf}.
  Using the same substitution $s=\log(1/t)$,
  \begin{equation}
    \int_0^\delta \frac{1}{\bigl(\trf_\beta(t)\bigr)^\alpha}\, dt
    =
    \int_{\log(1/\delta)}^\infty e^{-s}\, s^{\alpha}\, \bigl(\log s\bigr)^{\alpha(1+\beta)}\, ds .
  \end{equation}
  The integrand is dominated by the exponential factor $e^{-s}$, hence the tail integral tends to $0$ as $\delta \to 0$,
  which yields the claim.

  \ref{lem item: 3. property of trf}.
  Define three bijective functions:
  \begin{equation}
  \begin{alignedat}{3}
    g_1 &\colon\; & (0,1/e) \to (0,\infty), &\quad & t &\mapsto \frac{1}{1+\beta}\,\log\!\bigl(\log(1/t)\bigr),\\
    g_2 &\colon\; & (0,\infty) \to (0,\infty), &\quad & u &\mapsto u\, e^{u},\\
    g_3 &\colon\; & (0,\infty) \to (0,\infty), &\quad & v &\mapsto \frac{1}{\bigl((1+\beta)\, v\bigr)^{1+\beta}} .
  \end{alignedat}
  \end{equation}
  Then $\trf_\beta = g_3 \circ g_2 \circ g_1$, hence $\trf_\beta^{-1} = g_1^{-1} \circ g_2^{-1} \circ g_3^{-1}$.
  A direct computation gives
  \begin{equation}
    g_1^{-1}(y) = \exp\bigl(-e^{(1+\beta)y}\bigr),
    \qquad
    g_3^{-1}(t) = \frac{1}{1+\beta}\, t^{-1/(1+\beta)} .
  \end{equation}
  The inverse of $g_2$ is the Lambert $W$ function, and as $u \to \infty$,
  \begin{equation}
    g_2^{-1}(u) = \log\Bigl(\frac{u}{\log u}\Bigr) + o(1)
  \end{equation}
  (cf.\ \cite[Equation~(4.20)]{Corless_Gonnet_etal_96_Lambert}).
  Therefore, as $t \downarrow 0$,
  \begin{align}
    -\log\bigl( \trf_\beta^{-1}(t) \bigr) 
    &= \exp\Bigl( (1+\beta)\, g_2^{-1}\bigl(g_3^{-1}(t)\bigr) \Bigr) \\
    &= \exp\left((1+\beta)\, \log\Bigl(\frac{g_3^{-1}(t)}{\log g_3^{-1}(t)}\Bigr) + o(1)\right) \\
    &= e^{o(1)} \left( \frac{g_3^{-1}(t)}{\log g_3^{-1}(t)} \right)^{1 + \beta} \\
    &= e^{o(1)}\left( \frac{ \frac{1}{1+\beta}\, t^{-1/(1+\beta)} }{ \log\bigl( \frac{1}{1+\beta}\, t^{-1/(1+\beta)} \bigr)} \right)^{1+\beta} \\
    &= e^{o(1)}\left( \frac{ t^{-1/(1+\beta)} }{ \log(1/t)\,(1+o(1)) } \right)^{1+\beta}\\
    &= \frac{1+o(1)}{\, t \bigl(\log(1/t)\bigr)^{1+\beta}} ,
  \end{align}
  which proves the desired asymptotic.
\end{proof}

\section{An inverse version of the continuous mapping theorem} \label{appendix: An inverse version of the continuous mapping theorem}

Let $S_1$ and $S_2$ be $\bcmAB$ spaces and let $f \colon S_1 \to S_2$ be a continuous map.
By the continuous mapping theorem, 
the pushforward map $\mu \mapsto \mu \circ f^{-1}$ 
from $\Meas(S_1)$ to $\Meas(S_2)$ is continuous.
When $f$ is injective, 
the Lusin--Souslin theorem (Lemma~\ref{lem: Lousin--Souslin}) ensures that 
$f(B)$ is a Borel subset of $S_2$ for every Borel set $B \subseteq S_1$.
Hence, one can define the map in the inverse direction 
$\nu \mapsto \nu \circ f$ from $\Meas(S_2)$ to $\Meas(S_1)$.
Although this map is not continuous in general, 
its continuity can be ensured under certain additional assumptions.
Below, we provide a result on this point, 
which is used in Proposition~\ref{prop: continuity of Col map}.

\begin{prop} \label{ap. prop: inverse conti map thm}
  Let $S_1$ and $S_2$ be $\bcmAB$ spaces and let 
  $\iota \colon S_1 \to S_2$ be an injective Borel measurable map 
  such that its image $\iota(S_1)$ is closed in $S_2$ and its inverse is continuous.
  Let $(\Sigma_n)_{n \ge 1}$ be a sequence of Radon measures on $S_2$.
  Assume that $\supp(\Sigma_n) \subseteq \iota(S_1)$ for all $n \ge 1$,
  and that $\Sigma_n \to \Sigma$ vaguely.
  Then $\Sigma_n \circ \iota \to \Sigma \circ \iota$ vaguely as measures on $S_1$.
\end{prop}

The idea of the proof is to show that 
the restrictions of the measures to $\iota(S_1)$ still converge vaguely.
Once this is obtained, 
the desired result follows from the continuous mapping theorem.
We use the following lemma to verify the convergence of such restrictions.

\begin{lem} \label{lem: restriction of vague conv}
  Let $S$ be a $\bcmAB$ space,
  and let $(\Sigma_n)_{n \ge 1}$ be a sequence of Radon measures on $S$ 
  converging vaguely to $\Sigma$.
  If a closed subset $F$ of $S$ satisfies $\supp(\Sigma_n) \subseteq F$ for all $n \ge 1$,
  then $\Sigma_n|_F \to \Sigma|_F$ vaguely as measures on $F$.
\end{lem}

Note that this is not a consequence of the standard fact 
that vague convergence is preserved under restriction to continuity sets.
Indeed, we allow that the limiting measure $\Sigma$ has positive mass on the boundary of $F$.

\begin{proof}[{Proof of Lemma~\ref{lem: restriction of vague conv}}]
  Fix a compactly supported continuous function $f \colon F \to \mathbb{R}$,
  and set $K \coloneqq \supp(f)$.
  Choose a continuous function $\chi \colon S \to [0,1]$ with compact support 
  such that $\chi \equiv 1$ on $K$.
  Then $\chi(x)\, \Sigma_n(dx) \to \chi(x)\, \Sigma(dx)$ weakly.
  By Lemma~\ref{lem: vague convergence and hatC topology}, it follows that 
  \begin{equation}
    \lim_{n \to \infty} \int_F f(x)\, \chi(x)\, \Sigma_n(dx)
    = 
    \int_F f(x)\, \chi(x)\, \Sigma(dx).
  \end{equation}
  Since $f$ vanishes outside $K$ and $\chi \equiv 1$ on $K$,
  we conclude that 
  \begin{equation}
    \lim_{n \to \infty} \int_F f(x)\, \Sigma_n(dx)
    = 
    \int_F f(x)\, \Sigma(dx).
  \end{equation}
  This completes the proof.
\end{proof}

Now we can prove Proposition~\ref{ap. prop: inverse conti map thm}.

\begin{proof} [{Proof of Proposition~\ref{ap. prop: inverse conti map thm}}]
  By Lemma~\ref{lem: restriction of vague conv},
  we have $\Sigma_n|_{\iota(S_1)} \to \Sigma|_{\iota(S_1)}$ vaguely as measures on $\iota(S_1)$.
  Applying the continuous mapping theorem to the inverse 
  $\iota^{-1} \colon \iota(S_1) \to S_1$, we obtain
  \begin{equation}
    \Sigma_n \circ \iota 
    = \Sigma_n|_{\iota(S_1)} \circ \iota
    \to \Sigma|_{\iota(S_1)} \circ \iota 
    = \Sigma \circ \iota
  \end{equation}
  vaguely as measures on $S_1$.
  This completes the proof.
\end{proof}

\section{Replacement of the \texorpdfstring{$L^0$}{L0} topology by stronger topologies} \label{appendix: Replacement of the L^0 topology by stronger topologies}

Throughout this paper, 
we have primarily worked with the $L^0$ topology on path spaces of stochastic processes,
and most of our main results are formulated and proved under this topology.
This choice allows us to establish the results under minimal assumptions 
and to cover a wide range of applications.
In many concrete situations, however, 
the convergence of stochastic processes is known to hold 
under stronger topologies such as the compact–convergence topology, 
the $J_1$-Skorohod topology, or the $M_1$ topology.
In this appendix, 
we provide a result which ensures that whenever such stronger convergence holds,
the $L^0$ topology appearing in the statements of the main results 
can be replaced with the corresponding stronger topology.

\begin{prop} \label{ap prop: joint law upgrade}
  Let $\mathfrak{X}_1$, $\mathfrak{X}_2$, and $\mathfrak{Y}$ be Polish spaces.
  Assume that the following conditions hold.
  \begin{enumerate} [label = \textup{(\roman*)}]
    \item \label{ap prop item: joint law upgrade. 1}
      The set $\mathfrak{X}_1$ is a Borel subset of $\mathfrak{X}_2$.
    \item \label{ap prop item: joint law upgrade. 2}
      We have $\Borel(\mathfrak{X}_1) = \{B \cap \mathfrak{X}_1 \mid B \in \Borel(\mathfrak{X}_2)\}$.
    \item \label{ap prop item: joint law upgrade. 3}
      The inclusion map $\mathfrak{X}_1 \hookrightarrow \mathfrak{X}_2$ is continuous.
  \end{enumerate}
  Let $(X, Y)$ and $(X_n, Y_n)$, $n \geq 1$, be random elements of $\mathfrak{X}_1 \times \mathfrak{Y}$,
  where each pair $(X_n, Y_n)$ and $(X, Y)$ is defined on its own probability space.
  If $X_n \xrightarrow[n \to \infty]{\mathrm{d}} X$ in $\mathfrak{X}_1$ 
  and $(X_n, Y_n) \xrightarrow[n \to \infty]{\mathrm{d}} (X, Y)$ in $\mathfrak{X}_2 \times \mathfrak{Y}$,
  then $(X_n, Y_n) \xrightarrow[n \to \infty]{\mathrm{d}} (X, Y)$ in $\mathfrak{X}_1 \times \mathfrak{Y}$.
\end{prop}

\begin{proof}
  By the assumption, the family $\{(X_n, Y_n)\}_{n \geq 1}$ is tight as random elements of $\mathfrak{X}_1 \times \mathfrak{Y}$.
  Assume that, along a subsequence $(n_k)_{k \geq 1}$, the law of $(X_{n_k}, Y_{n_k})$ converges to a probability measure $P$
  on $\mathfrak{X}_1 \times \mathfrak{Y}$.
  By \ref{ap prop item: joint law upgrade. 3},
  we can apply the continuous mapping theorem and deduce that 
  the law of $(X_{n_k}, Y_{n_k})$ also converges to $P$ as probability measures on $\mathfrak{X}_2 \times \mathfrak{Y}$.
  Then, by the assumption,
  we deduce that $P$ coincides with the law of $(X, Y)$ as probability measure on $\mathfrak{X}_2 \times \mathfrak{Y}$.
  Conditions~\ref{ap prop item: joint law upgrade. 1} and \ref{ap prop item: joint law upgrade. 2} yield that 
  $P$ coincides with the law of $(X, Y)$ as probability measures on $\mathfrak{X}_1 \times \mathfrak{Y}$.
  This implies the uniqueness of the limit $P$, which completes the proof.
\end{proof}

\begin{rem}
  By the Lusin--Souslin theorem (Lemma~\ref{lem: Lousin--Souslin}),
  condition~\ref{ap prop item: joint law upgrade. 1} 
  follows from the continuity of the inclusion map 
  in condition~\ref{ap prop item: joint law upgrade. 3}.
\end{rem}

We illustrate an application of Proposition~\ref{ap prop: joint law upgrade}.
Consider the setting of Section~\ref{sec: Continuity of joint laws with PCAFs}.
Suppose that the standard process $X$ satisfies the weak $J_1$-continuity condition 
(Assumption~\ref{assum: weak J_1 continuity})
so that the map $\ProcLaw_X$ is continuous with respect to 
the weak topology induced by the $J_1$-Skorohod topology.
Let $A$ be a PCAF in the local Kato class.
Then, by Theorem~\ref{thm: R map approx for local Kato},
the map $\ProcLaw_{(X, A)}$ is continuous 
with respect to the weak topology induced by the product of the $L^0$ topology and the compact-convergence topology.
Applying Proposition~\ref{ap prop: joint law upgrade}, 
we see that this continuity remains valid 
when the $L^0$ topology is replaced by the $J_1$-Skorohod topology.


\section{The product process of independent standard processes} \label{appendix: product process}

In this appendix,
we prove that the pair of two independent standard processes is again a standard process.

Fix two locally compact separable metrizable topological spaces $S_1$ and $S_2$.
For each $i \in \{1,2\}$,
we write $S_1 \cup \{\Delta\}$ for the one-point compactification of $S_i$.
We fix, for each $i \in \{1, 2\}$, a standard process on $S_i$, denoted by 
\begin{equation}
  X^i = (\Omega_i, \sigalg_i, (X^i_{t})_{t \in [0, \infty]}, (P^x_i)_{x \in S_{\Delta}}, (\theta^i_t)_{t \in [0,\infty]}).
\end{equation}
We first recall the product process of $X^1$ and $X^2$ from Section~\ref{sec: collision measure}.
Let $(S_1 \times S_2) \cup \{\Delta\}$ be the one-point compactification of the product space $S_1 \times S_2$.
For each $\bm{x}=(x_{1}, x_{2}) \in S \times S$,
we define $\hat{P}^{\bm{x}} \coloneqq P_{1}^{x_{1}} \otimes P_{2}^{x_{2}}$,
which is a probability measure on the product measurable space 
$(\hat{\Omega}, \hat{\sigalg}) \coloneqq (\Omega_1 \times \Omega_2, \sigalg_1 \otimes \sigalg_2)$.
Also,
we define $\hat{P}^{\Delta} \coloneqq P_{1}^{\Delta} \otimes P_{2}^{\Delta}$.
For each $\omega = (\omega_{1}, \omega_{2}) \in \hat{\Omega}$ and $t \geq 0$,
we set 
\begin{equation}
  \hat{X}_t(\omega) 
  \coloneqq 
  \begin{cases}
    (X^1_t(\omega_{1}), X^2_t(\omega_{2})), & t < \zeta_1(\omega_1) \wedge \zeta_2(\omega_2),\\
    \Delta, & t \geq \zeta_1(\omega_1) \wedge \zeta_2(\omega_2),
  \end{cases}
\end{equation}
where $\zeta_i$ denotes the lifetime of $X^i$ for each $i =1,2$.
Finally, we define, for each $\omega = (\omega_{1}, \omega_{2}) \in \hat{\Omega}$ and $t \geq 0$,
\begin{equation}
  \hat{\theta}_t(\omega) \coloneqq (\theta^1_t(\omega_1), \theta^2_t(\omega_2)).
\end{equation}
The main result is the following.

\begin{thm} \label{thm: product is standard}
  In the above setting, 
  \begin{equation}
    \hat{X} \coloneqq (\hat{\Omega}, \hat{\sigalg}, (\hat{X}_{t})_{t \in [0, \infty]}, (\hat{P}^x)_{x \in (S_1 \times S_2) \cup \{\Delta\}}, (\hat{\theta}_t)_{t \in [0,\infty]})
  \end{equation}
  is a standard process on $S_1 \times S_2$.
\end{thm}

Let $(\filt^{X^1,0}_t)_{t \geq 0}$ be the filtration generated by $X^1$, 
that is,
\begin{equation}
  \filt^{X^1,0}_t \coloneqq \sigma(X^1_s; s \leq t).
\end{equation}
We similarly define $(\filt^{X^2,0}_t)_{t \geq 0}$ and $(\filt^{\hat{X},0}_t)_{t \geq 0}$
to be the filtrations generated by $X^2$ and $\hat{X}$, respectively.
Below, we state a basic result regarding $(\filt^{X^1,0}_t)_{t \geq 0}$.
In the following lemma,
for sets $A$ and $B$, and a subset $C \subseteq A \times B$,
we write $C|_{(a, *)} \coloneqq \{b \in B \mid (a, b) \in C\}$ for each $a \in A$ 
and $C|_{(*, b)} \coloneqq \{a \in A \mid (a, b) \in C\}$ for each $b \in B$.

\begin{lem} \label{lem: projection of generated algebra}
  Fix $t \in [0, \infty)$ and $\Lambda \in \filt^{\hat{X}, 0}_t$,
  $\Lambda|_{(*, \omega_2)} \in \filt^{X^1, 0}_t$ for any $\omega_2 \in \Omega_2$, 
  and $\Lambda|_{(\omega_1, *)} \in \filt^{X^2, 0}_t$for any $\omega_1 \in \Omega_1$.
\end{lem}

\begin{proof}
  By the $\pi$-$\lambda$ theorem, it is enough to consider the case where
  $\Lambda = \bigcap_{k=1}^n \hat{X}_{s_k}^{-1}(A_{k})$, where 
  $n \in \NN$, $s_k \in [0, t]$, and $A_k \in \Borel((S_1 \times S_2) \cup \{\Delta\})$ for each $k$.
  Fix $\omega_2 \in \Omega_2$ and write $B_k \coloneqq A_k|_{(*, X^2_{s_k}(\omega_2))}$, which is a Borel subset of $S$.
  We then have that
  \begin{align}
    \Lambda|_{(*, \omega_2)} 
    &= 
    \bigl\{ 
      \omega_1 \in \Omega_1 \mid (X^1_{s_k}(\omega_1), X^2_{s_k}(\omega_2)) \in A_k,\ \forall k \in \{1, \ldots, n\}
    \bigr\}
    \\
    &= 
    \bigl\{ 
      \omega_1 \in \Omega_1 \mid X^1_{s_k}(\omega_1) \in B_k,\ \forall k \in \{1, \ldots, n\}
    \bigr\}.
  \end{align}
  From the last expression, we see that $\Lambda|_{(*, \omega_2)} \in \filt^{X^1, 0}_t$.
  Similarly, one can prove the second assertion.
\end{proof}

We next verify that $\hat{X}$ is a Markov process.
Recall that, given a filtration $(\mathcal{G}_t)_{t \geq 0}$, 
its right-continous version $(\mathcal{G}_{t+})_{t \geq 0}$ is defined by $\mathcal{G}_{t+} = \bigcap_{s > t} \mathcal{G}_s$.

\begin{lem} \label{lem: product is Markov}
  The process $\hat{X}$ is a normal Markov process,
  and it has the Markov property with respect to $(\filt^{\hat{X}, 0}_{t+})_{t \geq 0}$.
\end{lem}

\begin{proof}
  Recall the definition of a normal Markov process from \cite[p.~385]{Fukushima_Oshima_Takeda_11_Dirichlet}.
  We briefly check conditions (M1)--(M5) in \cite[p.~385]{Fukushima_Oshima_Takeda_11_Dirichlet}.
  All these conditions, except for (M4), are readily verified since each $X^1$ and $X^2$ satisfies the corresponding conditions.
  To verify (M4), we show the Markov property of $\hat{X}$ with respect to $(\filt^{\hat{X},0}_{t+})_{t \geq 0}$.

  Fix $s,t \in [0, \infty)$, $\bm{x} = (x_1, x_2) \in S_1 \times S_2$, $E \in \Borel(S_1 \times S_2)$, and $\Lambda \in \filt^{\hat{X}, 0}_{t+}$.
  It suffices to show that 
  \begin{equation} \label{pr eq: 00, product is hunt process}
    \hat{E}^{\bm{x}}[ \mathbf{1}_E(\hat{X}_{s+t}) \cdot \mathbf{1}_\Lambda ]
    = 
    \hat{E}^{\bm{x}}
    \bigl[ 
      \hat{P}^{\hat{X}_t}(\hat{X}_s \in E) \cdot \mathbf{1}_\Lambda 
    \bigr].
  \end{equation}
  We first consider the case where $E = E_1 \times E_2$ for some $E_1 \in \Borel(S_1)$ and $E_2 \in \Borel(S_2)$.
  By Lemma~\ref{lem: projection of generated algebra},
  we have $\Lambda|_{(*, \omega_2)} \in \filt^{X^1, 0}_{t+}$ for each $\omega_2 \in \Omega_2$.
  Thus, using the Markov property of $X^1$ with respect to $(\filt^{X^1, 0}_{u+})_{u \geq 0}$,
  we deduce that, for each $\omega_2 \in \Omega_2$,
  \begin{equation} 
    \int_{\Omega_1} \mathbf{1}_{E_1}(X^1_{s+t}(\omega_1)) \mathbf{1}_{\Lambda}(\omega_1, \omega_2)\, P_1^{x_1}(d\omega_1)
    =
    \int_{\Omega_1} P_1^{X^1_t(\omega_1)}(X^1_s \in E_1) \mathbf{1}_{\Lambda}(\omega_1, \omega_2)\, P_1^{x_1}(d\omega_1).
  \end{equation}
  Similarly, for each $\omega_1 \in \Omega_1$,
  \begin{equation}  
    \int_{\Omega_2} \mathbf{1}_{E_2}(X^2_{s+t}(\omega_2)) \mathbf{1}_{\Lambda}(\omega_1, \omega_2)\, P_2^{x_2}(d\omega_2)
    =
    \int_{\Omega_2} P_1^{X^1_t(\omega_2)}(X^2_s \in E_2) \mathbf{1}_{\Lambda}(\omega_1, \omega_2)\, P_2^{x_2}(d\omega_2).
  \end{equation}
  Therefore, using Fubini's theorem, we deduce that 
  \begin{align}
    &\hat{E}^{\bm{x}}[ \mathbf{1}_E(\hat{X}_{s+t}) \cdot \mathbf{1}_\Lambda ]\\
    &=
    \int_{\Omega_1} \int_{\Omega_2} 
    \mathbf{1}_{E_1}(X^1_{s+t}(\omega_1)) \mathbf{1}_{E_2}(X^2_{s+t}(\omega_2))  
    \mathbf{1}_{\Lambda}(\omega_1, \omega_2)\, 
    P_1^{x_1}(d\omega_1)\, P_2^{x_2}(d\omega_2)\\
    &=
    \int_{\Omega_1} \int_{\Omega_2} 
    P_1^{X^1_t(\omega_1)}(X^1_s \in E_1) P_1^{X^1_t(\omega_2)}(X^2_s \in E_2) 
    \mathbf{1}_{\Lambda}(\omega_1, \omega_2)\, 
    P_1^{x_1}(d\omega_1)\, P_2^{x_2}(d\omega_2)\\
    &= 
    \int_{\Omega_1 \times \Omega_2} 
    \hat{P}^{\hat{X}_t(\omega_1, \omega_2)}(\hat{X}_s \in E)  
    \mathbf{1}_{\Lambda}(\omega_1, \omega_2)\, 
    \hat{P}^{\bm{x}}(d (\omega_1, \omega_2))\\
    &=
    \hat{E}^{\bm{x}}
    \bigl[ 
      \hat{P}^{\hat{X}_t}(\hat{X}_s \in E) \cdot \mathbf{1}_\Lambda 
    \bigr].
  \end{align}
  Thus, \eqref{pr eq: 00, product is hunt process} holds for $E = E_1 \times E_2$.
  It is now a standard argument to extend the result to general Borel subsets $E \subseteq S_1 \times S_2$.
\end{proof}

We next prove that $\hat{X}$ is a Borel right process 
(see \cite[Definition~A.1.17]{Chen_Fukushima_12_Symmetric} for the definition).
To this end, 
we introduce some notation for filtrations of Markov processes.
Recall that 
$(\filt^{X^1, 0}_t)_{t \geq 0}$ denotes the natural filtration of $X^1$,
and set $\filt^{X^1, 0}_\infty \coloneqq \bigcup_{t \geq 0} \filt^{X^1, 0}_t$.
Given a probability measure $\mu$ on $S_1 \cup \{\Delta\}$,
define 
\begin{equation}
  P^\mu(\cdot) \coloneqq \int_{S_1 \cup \{\Delta\}} P_1^x(\cdot)\, \mu(dx),
\end{equation}
which is a probability measure on $(\Omega_1, \filt^{X^1, 0}_\infty)$.
We then let $\filt^{X^1, \mu}_\infty$ denote the $P^\mu$-completion of $\filt^{X^1,0}_\infty$.
Let $\mathcal{N}$ be the collection of all $P^\mu$-null sets in $\filt^{X^1, \mu}_\infty$,
and set $\filt^{X^1, \mu}_t \coloneqq \sigma(\filt^{X^1, 0}_t,\, \mathcal{N})$ for each $t \in [0, \infty)$.
When $\mu$ is the Dirac measure $\delta_x$ at $x$,
we simply write $\filt^{X^1, x}_t \coloneqq \filt^{X^1, \delta_x}_t$, $t \in [0, \infty]$.
The minimum augmented admissible filtration $(\filt^{X^1}_t)_{t \geq 0}$ is then defined by
\begin{equation}
  \filt^{X^1}_t \coloneqq \bigcap_{\mu \in \Prob(S_1 \cup \{\Delta\})} \filt^{X^1,\mu}_t,
  \quad t \in [0,\infty].
\end{equation} 
The same notation applies to $X^2$ and $\hat{X}$.

\begin{lem} \label{lem: product is Borel right proc}
  The process $\hat{X}$ is a Borel right process.
\end{lem}

\begin{proof}
  Fix a probability measure $\mu$ on $(S \times S) \cup \{\Delta\}$.
  By Lemma~\ref{lem: product is Markov} and \cite[Lemma~A.2.2]{Fukushima_Oshima_Takeda_11_Dirichlet},
  the filtration $(\filt^{\hat{X}, \mu}_t)_{t \geq 0}$ is right-continuous. 
  Fix $s > 0$, an $(\filt^{\hat{X}, \mu}_t)$-stopping time $\sigma$,   
  and $\Lambda \in \filt^{\hat{X}, \mu}_{\sigma}$.
  By \cite[Theorem~A.1.18]{Chen_Fukushima_12_Symmetric},
  it suffices to show that 
  \begin{equation}
    \hat{E}^{\mu}\!\left[\mathbf{1}_E(\hat{X}_{\sigma + s}) \cdot \mathbf{1}_\Lambda \cdot \mathbf{1}_{\{\sigma < \infty\}} \right]
    = 
    \hat{E}^{\mu}\!\left[
      \hat{P}^{\hat{X}_\sigma}(\hat{X}_s \in E) \cdot \mathbf{1}_\Lambda \cdot \mathbf{1}_{\{\sigma < \infty\}}
    \right],
    \quad 
    \forall E \in \Borel(S_1 \times S_2).
  \end{equation}
  By the monotone class theorem, it is enough to prove the following:
  for any bounded and continuous function $f_i \colon S_i \to [0,\infty)$, $i \in \{1,2\}$,
  \begin{equation}  
    \hat{E}^{\mu}\!\left[ 
      f_1(X^1_{\sigma+s}) f_2(X^2_{\sigma+s}) \cdot \mathbf{1}_\Lambda \cdot \mathbf{1}_{\{\sigma < \infty\}}
    \right] 
    = 
    \hat{E}^{\mu}\!\left[
      \hat{E}^{\hat{X}_\sigma}[f_1(X^1_s) f_2(X^2_s)] \cdot \mathbf{1}_\Lambda \cdot \mathbf{1}_{\{\sigma < \infty\}}
    \right].
  \end{equation}

  From \cite[Chapter~II, Exercise~4.14]{Blumenthal_Getoor_68_Markov}, 
  for each $i \in \{1,2\}$ the map $x \mapsto E_i^{x}[f_i(X^i_s)]$ is continuous in the fine topology of $X^i$.  
  By \cite[Chapter~II, Theorem~4.8]{Blumenthal_Getoor_68_Markov},
  for any $x \in S_i$ we then have that $E_i^{X^i_t}[f_i(X^i_s)]$ is right-continuous in $t \geq 0$, $P_i^x$-a.s.
  Using the independence of $X^1$ and $X^2$, we obtain
  \begin{equation} \label{lem pr eq: 5. product is Borel right proc}
    \hat{E}^{\hat{X}_t}[f_1(X^1_s) f_2(X^2_s)]
    = 
    E_1^{X^1_t}[f_1(X^1_s)]\, E_2^{X^2_t}[f_2(X^2_s)] ,
  \end{equation}
  and hence
  \begin{equation}  \label{lem pr eq: 6. product is Borel right proc}
    \hat{E}^{\hat{X}_t}[f_1(X^1_s) f_2(X^2_s)] 
    \quad \text{is right-continuous in $t \geq 0$, $\hat{P}^{\mu}$-a.s.}
  \end{equation}
  For each $n \geq 1$, define
  \begin{equation}
    \sigma_n \coloneqq \frac{\lfloor 2^n \sigma \rfloor + 1}{2^n}.
  \end{equation}
  Then $\sigma_n$ is an $(\filt^{\hat{X},\mu}_t)$-stopping time,
  and $\sigma_n(\omega) \downarrow \sigma(\omega)$ for all $\omega \in \hat{\Omega}$ (see \cite[Proof of Lemma~2.2.5]{Marcus_Rosen_06_Markov}).  
  By the right-continuity of $\hat{X}$ and \eqref{lem pr eq: 6. product is Borel right proc}, 
  \begin{align}
    &\hat{E}^{\mu}\!\left[
      f_1(X^1_{\sigma+s}) f_2(X^2_{\sigma+s}) \cdot \mathbf{1}_\Lambda \cdot \mathbf{1}_{\{\sigma < \infty\}}
    \right] 
    = \lim_{n \to \infty} 
    \hat{E}^{\mu}\!\left[
      f_1(X^1_{\sigma_n+s}) f_2(X^2_{\sigma_n+s}) \cdot \mathbf{1}_\Lambda \cdot \mathbf{1}_{\{\sigma_n < \infty\}}
    \right], \\
    &\hat{E}^{\mu}\!\left[
      \hat{E}^{\hat{X}_\sigma}[f_1(X^1_s) f_2(X^2_s)] \cdot \mathbf{1}_\Lambda \cdot \mathbf{1}_{\{\sigma < \infty\}}
    \right]
    = \lim_{n \to \infty}
    \hat{E}^{\mu}\!\left[
      \hat{E}^{\hat{X}_{\sigma_n}}[f_1(X^1_s) f_2(X^2_s)] \cdot \mathbf{1}_\Lambda \cdot \mathbf{1}_{\{\sigma_n < \infty\}}
    \right].
  \end{align}
  Thus it suffices to prove that, for each fixed $n$,
  \begin{equation} \label{lem pr eq: 1. product is Borel right proc}
    \hat{E}^{\mu}\!\left[
      f_1(X^1_{\sigma_n+s}) f_2(X^2_{\sigma_n+s}) \cdot \mathbf{1}_\Lambda \cdot \mathbf{1}_{\{\sigma_n < \infty\}}
    \right]
    =
    \hat{E}^{\mu}\!\left[
      \hat{E}^{\hat{X}_{\sigma_n}}[f_1(X^1_s) f_2(X^2_s)] \cdot \mathbf{1}_\Lambda \cdot \mathbf{1}_{\{\sigma_n < \infty\}}
    \right].
  \end{equation}

  Fix $n \geq 1$.
  By \cite[Lemma~A.2.3]{Fukushima_Oshima_Takeda_11_Dirichlet},
  there exist an $(\filt^{\hat{X},0}_{t+})$-stopping time $\sigma'_n$
  and $\Lambda' \in \filt^{\hat{X},0}_\infty$ such that 
  \begin{gather}
    \hat{P}^\mu(\sigma_n \neq \sigma'_n) = 0, \quad 
    \hat{E}^\mu[|\mathbf{1}_\Lambda - \mathbf{1}_{\Lambda'}|] = 0, 
    \label{lem pr eq: 2. product is Borel right proc}\\
    \Lambda' \cap \{\sigma'_n \leq t\} \in \filt^{\hat{X},0}_{t+}, \quad \forall t \geq 0 .
    \label{lem pr eq: 3. product is Borel right proc}
  \end{gather}
  Hence it suffices to prove \eqref{lem pr eq: 1. product is Borel right proc}
  with $\sigma_n$ and $\Lambda$ replaced by $\sigma'_n$ and $\Lambda'$, respectively.
  Since $\sigma'_n$ takes values in $\{j 2^{-n} \mid j \in \NN\}$,
  \begin{align}
    &\hat{E}^{\mu}\!\left[
      f_1(X^1_{\sigma'_n+s}) f_2(X^2_{\sigma'_n+s}) \cdot \mathbf{1}_{\Lambda'} \cdot \mathbf{1}_{\{\sigma'_n < \infty\}}
    \right] \\
    &= \sum_{j=1}^\infty    
      \hat{E}^{\mu}\!\left[
        f_1(X^1_{j2^{-n}+s}) f_2(X^2_{j2^{-n}+s}) \cdot 
        \mathbf{1}_{\{\Lambda' \cap \{\sigma'_n = j2^{-n}\}\}}
      \right] \\
    &= \sum_{j=1}^\infty    
      \hat{E}^{\mu}\!\left[
        \mathbf{1}_{\{\Lambda' \cap \{\sigma'_n = j2^{-n}\}\}}
        \hat{E}^{\mu}\!\left[
          f_1(X^1_{j2^{-n}+s}) f_2(X^2_{j2^{-n}+s}) \,\middle|\, \filt^{\hat{X},0}_{j2^{-n}}
        \right]
      \right],
  \end{align}
  where we used \eqref{lem pr eq: 3. product is Borel right proc}.
  Applying the Markov property of $\hat{X}$ with respect to $(\filt^{\hat{X},0}_{t+})$ (Lemma~\ref{lem: product is Markov}),
  this equals
  \begin{align}
    &\sum_{j=1}^\infty    
      \hat{E}^{\mu}\!\left[
        \mathbf{1}_{\{\Lambda' \cap \{\sigma'_n = j2^{-n}\}\}}
        \hat{E}^{\hat{X}_{j2^{-n}}}[f_1(X^1_s) f_2(X^2_s)]
      \right] \\
    &= \hat{E}^{\mu}\!\left[
      \hat{E}^{\hat{X}_{\sigma'_n}}[f_1(X^1_s) f_2(X^2_s)] \cdot 
      \mathbf{1}_{\Lambda'} \cdot \mathbf{1}_{\{\sigma'_n < \infty\}}
    \right].
  \end{align}
  This completes the proof.
\end{proof}

We can now conclude the proof of Theorem~\ref{thm: product is standard}.

\begin{proof} [{Proof of Theorem~\ref{thm: product is standard}}]
  Write $\zeta$ for the lifetime of $\hat{X}$.
  By the definition of the product process $\hat{X}$,
  we have 
  \begin{equation}  \label{thm pr eq: 0. product is standard}
    \zeta = \zeta_1 \wedge \zeta_2,
  \end{equation}
  where we recall that $\zeta_i$ denotes the lifetime of $X^i$ for each $i \in \{1,2\}$.
  By Lemma~\ref{lem: product is Borel right proc},
  it remains to show the quasi-left-continuity of $\hat{X}$ on $(0, \zeta)$.
  We note that, by Lemma~\ref{lem: product is Markov} and \cite[Lemma~A.2.2]{Fukushima_Oshima_Takeda_11_Dirichlet},
  the filtration $(\filt^{\hat{X}}_t)_{t \geq 0}$ is right-continuous. 

  Fix $\bm{x} = \bm{x} \in S_1 \times S_2$,
  and an increasing sequence $(\sigma_n)_{n \geq 1}$ of $(\filt^{\hat{X}}_t)_{t \geq 0}$-stopping times such that 
  $\sigma_n(\omega)$ converges to $\sigma(\omega) \in [0,\infty]$ for each $\omega \in \hat{\Omega}$.
  Note that $\sigma$ is also an $(\filt^{\hat{X}}_t)_{t \geq 0}$-stopping time thanks to the right-continuity of $(\filt^{\hat{X}}_t)_{t \geq 0}$.
  Using \cite[Lemma~A.2.3]{Fukushima_Oshima_Takeda_11_Dirichlet},
  we can find $(\filt^{\hat{X}, 0}_{t+})_{t \geq 0}$-stopping times $\tilde{\sigma}_n$, $n \in \NN$, and $\tilde{\sigma}$
  such that 
  \begin{equation} \label{thm pr eq: 1. product is standard}
    \hat{P}^{\bm{x}}(\Lambda) = 1, 
    \quad 
    \text{where}
    \quad
    \Lambda
    \coloneqq 
    \{\sigma = \tilde{\sigma},\ \text{and}\ \tilde{\sigma}_n = \sigma_n\ \text{for all}\ n \}
    \in \filt^{\hat{X}}_\infty.
  \end{equation}
  By the definition of $\filt^{\hat{X}}_\infty$,
  we have $\Lambda \in \filt^{\hat{X},\bm{x}}_\infty$.
  Thus, we can find $\Lambda', \Lambda'' \in \filt^{\hat{X},0}_\infty$ such that 
  \begin{equation}  \label{thm pr eq: 2. product is standard}
    \Lambda' \subseteq \Lambda \subseteq \Lambda'', \qquad \hat{P}^{\bm{x}}(\Lambda'' \setminus \Lambda') = 0.
  \end{equation}
  It follows from \eqref{thm pr eq: 1. product is standard}, \eqref{thm pr eq: 2. product is standard}, and Fubini's theorem that 
  \begin{equation}
    \hat{P}^{\bm{x}}(\Lambda') = \int_{\Omega_2} P_1^{x_1}(\Lambda'|_{(*, \omega_2)})\, P_2^{x_2}(d\omega_2) = 1.
  \end{equation}
  In particular, there exists a set $\Omega_2^0 \in \filt^{X^2, 0}_\infty$ such that 
  \begin{gather}
     P_2^{x_2}(\Omega_2^0) = 1,
     \label{thm pr eq: 3. product is standard}\\
     P_1^{x_1}(\Lambda'|_{(*, \omega_2)}) = 1, \quad \forall \omega_2 \in \Omega_2^0.
     \label{thm pr eq: 4. product is standard}
  \end{gather}
  The second property yields that 
  \begin{equation}  \label{thm pr eq: 5. product is standard}
    \Lambda'|_{(*, \omega_2)} \in \filt^{X^1, x_1}_0, \quad \forall \omega_2 \in \Omega_2^0.
  \end{equation}
  Fix $\omega_2 \in \Omega_2^0$.
  We define 
  \begin{equation}  \label{thm pr eq: 6. product is standard}
    \tilde{\sigma}^{\omega_2}(\omega_1) 
    \coloneqq 
    \begin{cases}
      \tilde{\sigma}(\omega_1, \omega_2), & \omega_1 \in \Lambda'|_{(*, \omega_2)},\\
      0, & \omega_1 \notin \Lambda'|_{(*, \omega_2)}.
    \end{cases}
  \end{equation}
  Similarly, for each $n \in \NN$, we define $\tilde{\sigma}_n^{\omega_2}$.
  We deduce from Lemma~\ref{lem: projection of generated algebra} and \eqref{thm pr eq: 5. product is standard}
  that $\tilde{\sigma}^{\omega_2}$ is an $(\filt^{X^1, x_1}_{t+})_{t \geq 0}$-stopping time.
  The same holds for $\tilde{\sigma}_n^{\omega_2}$ for each $n \in \NN$.
  By the inclusion $\Lambda' \subseteq \Lambda$ given in \eqref{thm pr eq: 2. product is standard}
  and the definition of $\Lambda$ stated at \eqref{thm pr eq: 1. product is standard},
  we have that $\tilde{\sigma}_n^{\omega_2}(\omega_1) \uparrow \tilde{\sigma}^{\omega_2}(\omega_1)$ for each $\omega \in \Omega_1$.
  Hence, by the quasi-left-continuity of $X^1$ on $(0, \zeta_1)$,
  we obtain that 
  \begin{equation}  \label{thm pr eq: 7. product is standard}
    P_1^{x_1} 
    \Bigl(
      \lim_{n \to \infty} X^1_{\tilde{\sigma}_n^{\omega_2}} = X^1_{\tilde{\sigma}^{\omega_2}},\        
      \tilde{\sigma}^{\omega_2} < \zeta_1
    \Bigr)
    = 
    P_1^{x_1} 
    (\tilde{\sigma}^{\omega_2} < \zeta_1),
    \quad 
    \forall \omega_2 \in \Omega_2^0.
  \end{equation}
  It then follows from \eqref{thm pr eq: 4. product is standard} and \eqref{thm pr eq: 6. product is standard} that 
  \begin{equation}  \label{thm pr eq: 8. product is standard}
    P_1^{x_1} 
    \Bigl(
      \lim_{n \to \infty} X^1_{\tilde{\sigma}_n(\cdot, \omega_2)} = X^1_{\tilde{\sigma}(\cdot, \omega_2)},\        
      \tilde{\sigma}(\cdot, \omega_2) < \zeta_1
    \Bigr)
    = 
    P_1^{x_1} 
    (\tilde{\sigma}(\cdot, \omega_2) < \zeta_1),
    \quad 
    \forall \omega_2 \in \Omega_2^0.
  \end{equation}
  Integrating both sides of the above equation with respect to $P_2^{x_2}(d\omega_2)$
  and using \eqref{thm pr eq: 3. product is standard},
  we obtain that  
  \begin{equation} \label{thm pr eq: 9. product is standard}
    \hat{P}^{\bm{x}}
    \Bigl(
      \lim_{n \to \infty} X^1_{\tilde{\sigma}_n} = X^1_{\tilde{\sigma}},\        
      \tilde{\sigma} < \zeta_1
    \Bigr)
    = 
    \hat{P}^{\bm{x}} 
    (\tilde{\sigma} < \zeta_1).
  \end{equation}
  By the same argument, we also deduce that 
  \begin{equation} \label{thm pr eq: 10. product is standard} 
    \hat{P}^{\bm{x}}
    \Bigl(
      \lim_{n \to \infty} X^2_{\tilde{\sigma}_n} = X^2_{\tilde{\sigma}},\        
      \tilde{\sigma} < \zeta_2
    \Bigr)
    = 
    \hat{P}^{\bm{x}} 
    (\tilde{\sigma} < \zeta_2).
  \end{equation}
  Therefore, we conclude from \eqref{thm pr eq: 0. product is standard}, \eqref{thm pr eq: 1. product is standard},
  \eqref{thm pr eq: 9. product is standard}, and \eqref{thm pr eq: 10. product is standard} 
  that 
  \begin{equation}  \label{pr eq: 13, product is hunt process}
    \hat{P}^{\bm{x}}
    \Bigl(
      \lim_{n \to \infty} \hat{X}_{\sigma_n} = \hat{X}_{\sigma},\        
      \sigma < \zeta
    \Bigr)
    = 
    \hat{P}^{\bm{x}} 
    (\sigma < \zeta),
  \end{equation}
  which implies the quasi-left-continuity of $\hat{X}$ on $(0, \zeta)$.
  This completes the proof.
\end{proof}

\begin{rem}
  By the same argument, one can verify that if both $X^1$ and $X^2$ are Hunt processes,
  then so is $\hat{X}$.
\end{rem}

Below, we state that the DAC condition (Assumption~\ref{assum: dual hypothesis}) is inherited by the product process.

\begin{prop} \label{prop: product satisfies dual hypo}
  In the setting of Theorem~\ref{thm: product is standard},
  if, for each $i \in \{1,2\}$, the process $X^i$ satisfies Assumption~\ref{assum: dual hypothesis} 
  with a reference measure $m^i$ and the heat kernel $p^i \colon \RNpp \times S_i \times S_i \to [0,\infty]$,
  then the product process $\hat{X}$ also satisfies Assumption~\ref{assum: dual hypothesis} 
  with the reference measure $m^1 \otimes m^2$ and the heat kernel $\hat{p}$ given by 
  \begin{equation}
    \hat{p}(t, (x_1, x_2), (y_1, y_2)) 
    = 
    p^1(t, x_1, y_1)\, p^2(t, x_2, y_2),
    \quad 
    t > 0,\ (x_1, y_1), (x_2, y_2) \in S_1 \times S_2,
  \end{equation}
  where we use the convention $0 \cdot \infty = 0$.
\end{prop}

\begin{proof}
  For each $i \in \{1,2\}$, let $Y^i$ denote the dual process of $X^i$.
  Then one can readily check that the product process $\hat{Y}$ of $Y^1$ and $Y^2$ is the dual process of $\hat{X}$.
  Verifying the remaining conditions regarding heat kernel is straightforward, so we omit the details.
\end{proof}

\section*{Acknowledgements}
\addcontentsline{toc}{section}{Acknowledgements}

I would like to express my deepest gratitude to Dr David Croydon, my supervisor, 
for his kind and devoted guidance over the past four years since my master’s course. 
Our weekly discussions were always stimulating and enjoyable, 
and through them I have learned a great deal and deepened my mathematical understanding.

I am also sincerely grateful to Dr Naotaka Kajino, 
a member of the probability group at the Research Institute for Mathematical Sciences (RIMS), Kyoto University, 
for his generous support, many valuable discussions, and helpful advice on various aspects of my research activities.

My sincere thanks go to the administrative staff of RIMS 
for their continual assistance with research grants and administrative procedures.

Finally, I would like to express my heartfelt gratitude to my family 
for their understanding and unwavering support that allowed me to pursue mathematics throughout graduate school. 
I could not have come this far without their constant encouragement.

This work was supported by 
JSPS KAKENHI Grant Number JP24KJ1447 
and by the Research Institute for Mathematical Sciences, 
an International Joint Usage/Research Center located at Kyoto University.

\bibliographystyle{amsplain}
\bibliography{ref_collision}

\providecommand{\bysame}{\leavevmode\hbox to3em{\hrulefill}\thinspace}
\providecommand{\MR}{\relax\ifhmode\unskip\space\fi MR }
\providecommand{\MRhref}[2]{%
  \href{http://www.ams.org/mathscinet-getitem?mr=#1}{#2}
}
\providecommand{\href}[2]{#2}
\begin{thebibliography}{10}

\bibitem{Abraham_Delmas_Hoscheit_13_A_note}
R.~Abraham, J.-F. Delmas, and P.~Hoscheit, \emph{A note on the
  {G}romov-{H}ausdorff-{P}rokhorov distance between (locally) compact metric
  measure spaces}, Electron. J. Probab. \textbf{18} (2013), no. 14, 21.
  \MR{3035742}

\bibitem{Berry_Broutin_Goldschmidt_12_The_continuum}
L.~Addario-Berry, N.~Broutin, and C.~Goldschmidt, \emph{The continuum limit of
  critical random graphs}, Probab. Theory Related Fields \textbf{152} (2012),
  no.~3-4, 367--406. \MR{2892951}

\bibitem{Albeverio_Blanchard_Ma_91_FKsemigroup}
S.~Albeverio, P.~Blanchard, and Z.~M. Ma, \emph{Feynman-{K}ac semigroups in
  terms of signed smooth measures}, Random partial differential equations
  ({O}berwolfach, 1989), Internat. Ser. Numer. Math., vol. 102, Birkh\"auser,
  Basel, 1991, pp.~1--31. \MR{1185735}

\bibitem{Aldous_93_The_continuum}
D.~Aldous, \emph{The continuum random tree. {III}}, Ann. Probab. \textbf{21}
  (1993), no.~1, 248--289. \MR{1207226}

\bibitem{Aldous_Eagleson_78_MixingStability}
D.~J. Aldous and G.~K. Eagleson, \emph{On mixing and stability of limit
  theorems}, Ann. Probability \textbf{6} (1978), no.~2, 325--331. \MR{517416}

\bibitem{Andres_Croydon_Kumagai_24_HKEon1dBTM}
S.~Andres, D.~A. Croydon, and T.~Kumagai, \emph{Heat kernel fluctuations and
  quantitative homogenization for the one-dimensional {B}ouchaud trap model},
  Stochastic Process. Appl. \textbf{172} (2024), Paper No. 104336, 20.
  \MR{4715271}

\bibitem{Andriopoulos_23_Convergence}
G.~Andriopoulos, \emph{Convergence of blanket times for sequences of random
  walks on critical random graphs}, Combin. Probab. Comput. (2023), 38.

\bibitem{Andriopoulos_etc_pre_cover_on_CRT}
G.~Andriopoulos, D.~A. Croydon, V.~Margarint, and L.~Menard, \emph{On the cover
  time of {B}rownian motion on the {B}rownian continuum random tree}, 2024,
  Preprint. {A}vailable at arXiv:2410.03922.

\bibitem{Angel_Croydon_Hernandez-Torres_Shiraishi_21_Scaling}
O.~Angel, D.~A. Croydon, S.~Hernandez-Torres, and D.~Shiraishi, \emph{Scaling
  limits of the three-dimensional uniform spanning tree and associated random
  walk}, Ann. Probab. \textbf{49} (2021), no.~6, 3032--3105. \MR{4348685}

\bibitem{Archer_Nachmias_Shalev_24_The_GHP}
E.~Archer, A.~Nachmias, and M.~Shalev, \emph{The {GHP} scaling limit of uniform
  spanning trees in high dimensions}, Comm. Math. Phys. \textbf{405} (2024),
  no.~3, paper no. 73. \MR{4712854}

\bibitem{Athreya_Lohr_Winter_16_The_gap}
S.~Athreya, W.~L\"{o}hr, and A.~Winter, \emph{The gap between {G}romov-vague
  and {G}romov-{H}ausdorff-vague topology}, Stochastic Process. Appl.
  \textbf{126} (2016), no.~9, 2527--2553. \MR{3522292}

\bibitem{Barlow_Croydon_Kumagai_17_Subsequential}
M.~T. Barlow, D.~A. Croydon, and T.~Kumagai, \emph{Subsequential scaling limits
  of simple random walk on the two-dimensional uniform spanning tree}, Ann.
  Probab. \textbf{45} (2017), no.~1, 4--55. \MR{3601644}

\bibitem{Barlow_Grigoryan_Kumagai_12_PHIandHKE}
M.~T. Barlow, A.~Grigor'yan, and T.~Kumagai, \emph{On the equivalence of
  parabolic {H}arnack inequalities and heat kernel estimates}, J. Math. Soc.
  Japan \textbf{64} (2012), no.~4, 1091--1146. \MR{2998918}

\bibitem{Barlow_Peres_Sousi_12_Collision}
M.~T. Barlow, Y.~Peres, and P.~Sousi, \emph{Collisions of random walks}, Ann.
  Inst. Henri Poincar\'e{} Probab. Stat. \textbf{48} (2012), no.~4, 922--946.
  \MR{3052399}

\bibitem{Arous_Cerny_06_Dynamics}
G.~Ben~Arous and J.~\v{C}ern\'{y}, \emph{Dynamics of trap models}, Mathematical
  statistical physics, Elsevier B. V., Amsterdam, 2006, pp.~331--394.
  \MR{2581889}

\bibitem{Beznea_Boboc_04_Potential}
L.~Beznea and N.~Boboc, \emph{Potential theory and right processes},
  Mathematics and its Applications, vol. 572, Kluwer Academic Publishers,
  Dordrecht, 2004. \MR{2153655}

\bibitem{Bhamidi_Sen_20_Geometry}
S.~Bhamidi and S.~Sen, \emph{Geometry of the vacant set left by random walk on
  random graphs, {W}right's constants, and critical random graphs with
  prescribed degrees}, Random Structures Algorithms \textbf{56} (2020), no.~3,
  676--721. \MR{4084187}

\bibitem{Billingsley_99_Convergence}
P.~Billingsley, \emph{Convergence of probability measures}, second ed., Wiley
  Series in Probability and Statistics: Probability and Statistics, John Wiley
  \& Sons, Inc., New York, 1999, A Wiley-Interscience Publication. \MR{1700749}

\bibitem{Blumenthal_Getoor_68_Markov}
R.~M. Blumenthal and R.~K. Getoor, \emph{Markov processes and potential
  theory}, Pure and Applied Mathematics, Vol. 29, Academic Press, New
  York-London, 1968. \MR{0264757}

\bibitem{Bogachev_07_Measure}
V.~I. Bogachev, \emph{Measure theory. {V}ol. {I}, {II}}, Springer-Verlag,
  Berlin, 2007. \MR{2267655}

\bibitem{Bouchaud_Cugliandolo_etc_97_Out}
J.-P. Bouchaud, L.~F. Cugliandolo, J.~Kurchan, and M.~M\'ezard, \emph{Out of
  equilibrium dynamics in spin-glasses and other glassy systems}, Series on
  Directions in Condensed Matther Physics, vol.~12, pp.~161--223, 1997.

\bibitem{Burago_Burago_Ivanov_01_A_course}
D.~Burago, Y.~Burago, and S.~Ivanov, \emph{A course in metric geometry},
  Graduate Studies in Mathematics, vol.~33, American Mathematical Society,
  Providence, RI, 2001. \MR{1835418}

\bibitem{Chen_Fukushima_12_Symmetric}
Z.-Q. Chen and M.~Fukushima, \emph{Symmetric {M}arkov processes, time change,
  and boundary theory}, London Mathematical Society Monographs Series, vol.~35,
  Princeton University Press, Princeton, NJ, 2012. \MR{2849840}

\bibitem{Corless_Gonnet_etal_96_Lambert}
R.~M. Corless, G.~H. Gonnet, D.~E.~G. Hare, D.~J. Jeffrey, and D.~E. Knuth,
  \emph{On the {L}ambert {$W$} function}, Adv. Comput. Math. \textbf{5} (1996),
  no.~4, 329--359. \MR{1414285}

\bibitem{Croydon_18_Scaling}
D.~A. Croydon, \emph{Scaling limits of stochastic processes associated with
  resistance forms}, Ann. Inst. Henri Poincar\'{e} Probab. Stat. \textbf{54}
  (2018), no.~4, 1939--1968. \MR{3865663}

\bibitem{Croydon_pre_CoverTimeOnBianaryTree}
\bysame, \emph{Scaling limit for the cover time of the $\lambda$-biased random
  walk on a binary tree with $\lambda<1$}, 2025.

\bibitem{Croydon_Ambroggio_24_TripleCollision}
D.~A. Croydon and U.~D. Ambroggio, \emph{Triple collisions on a comb graph},
  2024.

\bibitem{Croydon_Hambly_08_local}
D.~A. Croydon and B.~M. Hambly, \emph{Local limit theorems for sequences of
  simple random walks on graphs}, Potential Anal. \textbf{29} (2008), no.~4,
  351--389. \MR{2453564}

\bibitem{Croydon_Hambly_Kumagai_17_Time-changes}
D.~A. Croydon, B.~M. Hambly, and T.~Kumagai, \emph{Time-changes of stochastic
  processes associated with resistance forms}, Electron. J. Probab. \textbf{22}
  (2017), no.~82, 41. \MR{3718710}

\bibitem{Duquesne_03_A_limit}
T.~Duquesne, \emph{A limit theorem for the contour process of conditioned
  {G}alton-{W}atson trees}, Ann. Probab. \textbf{31} (2003), no.~2, 996--1027.
  \MR{1964956}

\bibitem{Fan_16_Discrete}
W.-T. Fan, \emph{Discrete approximations to local times for reflected
  diffusions}, Electron. Commun. Probab. \textbf{21} (2016), Paper No. 16, 12.
  \MR{3485385}

\bibitem{Fontes_Isopi_Newman_02_Random}
L.~R.~G. Fontes, M.~Isopi, and C.~M. Newman, \emph{Random walks with strongly
  inhomogeneous rates and singular diffusions: convergence, localization and
  aging in one dimension}, Ann. Probab. \textbf{30} (2002), no.~2, 579--604.
  \MR{1905852}

\bibitem{Fukushima_Oshima_Takeda_11_Dirichlet}
M.~Fukushima, Y.~Oshima, and M.~Takeda, \emph{Dirichlet forms and symmetric
  {M}arkov processes}, extended ed., De Gruyter Studies in Mathematics,
  vol.~19, Walter de Gruyter \& Co., Berlin, 2011. \MR{2778606}

\bibitem{Garban_Rhodes_Vargas_16_Liouville}
C.~Garban, R.~Rhodes, and V.~Vargas, \emph{Liouville {B}rownian motion}, Ann.
  Probab. \textbf{44} (2016), no.~4, 3076--3110. \MR{3531686}

\bibitem{Getoor_Sharpe_82_Excursions}
R.~K. Getoor and M.~J. Sharpe, \emph{Excursions of dual processes}, Adv. in
  Math. \textbf{45} (1982), no.~3, 259--309. \MR{673804}

\bibitem{Gradinaru_Haugomat_18_LocalSkorohod}
M.~Gradinaru and T.~Haugomat, \emph{Local {S}korokhod topology on the space of
  cadlag processes}, ALEA Lat. Am. J. Probab. Math. Stat. \textbf{15} (2018),
  no.~2, 1183--1213. \MR{3860820}

\bibitem{Gromov_07_Metric}
M.~Gromov, \emph{Metric structures for {R}iemannian and non-{R}iemannian
  spaces}, english ed., Modern Birkh\"{a}user Classics, Birkh\"{a}user Boston,
  Inc., Boston, MA, 2007, Based on the 1981 French original, With appendices by
  M. Katz, P. Pansu and S. Semmes, Translated from the French by Sean Michael
  Bates. \MR{2307192}

\bibitem{Halberstam_Hutchcroft_22_Collision}
N.~Halberstam and T.~Hutchcroft, \emph{Collisions of random walks in dynamic
  random environments}, Electron. J. Probab. \textbf{27} (2022), Paper No. 8,
  18. \MR{4364738}

\bibitem{Hambly_97_Brownian}
B.~M. Hambly, \emph{Brownian motion on a random recursive {S}ierpinski gasket},
  Ann. Probab. \textbf{25} (1997), no.~3, 1059--1102. \MR{1457612}

\bibitem{Hutchcroft_Peres_15_Collision}
T.~Hutchcroft and Y.~Peres, \emph{Collisions of random walks in reversible
  random graphs}, Electron. Commun. Probab. \textbf{20} (2015), no. 63, 6.
  \MR{3399814}

\bibitem{Jain_Pruitt_69_Collisions}
N.~Jain and W.~E. Pruitt, \emph{Collisions of stable processes}, Illinois J.
  Math. \textbf{13} (1969), 241--248. \MR{240873}

\bibitem{Janson_12_Simply}
S.~Janson, \emph{Simply generated trees, conditioned {G}alton-{W}atson trees,
  random allocations and condensation}, Probab. Surv. \textbf{9} (2012),
  103--252. \MR{2908619}

\bibitem{Kajino_Noda_pre_Generalized}
N.~Kajino and R.~Noda, \emph{Generalized {K}ac's moment formula for positive
  continuous additive functionals of symmetric {M}arkov processes}, 2025,
  Preprint. {A}vailable at arXiv:2503.04210.

\bibitem{Kallenberg_17_Random}
O.~Kallenberg, \emph{Random measures, theory and applications}, Probability
  Theory and Stochastic Modelling, vol.~77, Springer, Cham, 2017. \MR{3642325}

\bibitem{Kallenberg_21_Foundations}
\bysame, \emph{Foundations of modern probability}, third ed., Probability
  Theory and Stochastic Modelling, vol.~99, Springer, Cham, [2021] \copyright
  2021. \MR{4226142}

\bibitem{Kechris_95_Classical}
A.~S. Kechris, \emph{Classical descriptive set theory}, Graduate Texts in
  Mathematics, vol. 156, Springer-Verlag, New York, 1995. \MR{1321597}

\bibitem{Khezeli_20_Metrization}
A.~Khezeli, \emph{Metrization of the {G}romov-{H}ausdorff (-{P}rokhorov)
  topology for boundedly-compact metric spaces}, Stochastic Process. Appl.
  \textbf{130} (2020), no.~6, 3842--3864. \MR{4092421}

\bibitem{Khezeli_23_A_unified}
\bysame, \emph{A unified framework for generalizing the {G}romov-{H}ausdorff
  metric}, Probab. Surv. \textbf{20} (2023), 837--896. \MR{4671147}

\bibitem{Kigami_95_Harmonic}
J.~Kigami, \emph{Harmonic calculus on limits of networks and its application to
  dendrites}, J. Funct. Anal. \textbf{128} (1995), no.~1, 48--86. \MR{1317710}

\bibitem{Kigami_01_Analysis}
\bysame, \emph{Analysis on fractals}, Cambridge Tracts in Mathematics, vol.
  143, Cambridge University Press, Cambridge, 2001. \MR{1840042}

\bibitem{Kigami_12_Resistance}
\bysame, \emph{Resistance forms, quasisymmetric maps and heat kernel
  estimates}, Mem. Amer. Math. Soc. \textbf{216} (2012), no.~1015, vi+132.
  \MR{2919892}

\bibitem{Knopova_Schilling_15_CollisionofFeller}
V.~Knopova and R.~L. Schilling, \emph{On level and collision sets of some
  {F}eller processes}, ALEA Lat. Am. J. Probab. Math. Stat. \textbf{12} (2015),
  no.~2, 1001--1029. \MR{3457549}

\bibitem{Krishnapur_Peres_04_Comb}
M.~Krishnapur and Y.~Peres, \emph{Recurrent graphs where two independent random
  walks collide finitely often}, Electron. Comm. Probab. \textbf{9} (2004),
  72--81. \MR{2081461}

\bibitem{LeGall_06_Random}
J.-F. Le~Gall, \emph{Random real trees}, Ann. Fac. Sci. Toulouse Math. (6)
  \textbf{15} (2006), no.~1, 35--62. \MR{2225746}

\bibitem{Levin_Peres_17_Markov}
D.~A. Levin and Y.~Peres, \emph{Markov chains and mixing times}, second ed.,
  American Mathematical Society, Providence, RI, 2017, With contributions by
  Elizabeth L. Wilmer, With a chapter on ``Coupling from the past'' by James G.
  Propp and David B. Wilson. \MR{3726904}

\bibitem{Marcus_Rosen_06_Markov}
M.~B. Marcus and J.~Rosen, \emph{Markov processes, {G}aussian processes, and
  local times}, Cambridge Studies in Advanced Mathematics, vol. 100, Cambridge
  University Press, Cambridge, 2006. \MR{2250510}

\bibitem{Molchanov_17_Theory}
I.~Molchanov, \emph{Theory of random sets}, second ed., Probability Theory and
  Stochastic Modelling, vol.~87, Springer-Verlag, London, 2017. \MR{3751326}

\bibitem{Mori_20_LargeDeviations}
T.~Mori, \emph{Large deviations for intersection measures of some {M}arkov
  processes}, Math. Nachr. \textbf{293} (2020), no.~3, 533--553. \MR{4075013}

\bibitem{Mori_21_Kato_Sobolev}
\bysame, \emph{{$L^p$}-{K}ato class measures and their relations with {S}obolev
  embedding theorems for {D}irichlet spaces}, J. Funct. Anal. \textbf{281}
  (2021), no.~3, Paper No. 109034, 32. \MR{4247031}

\bibitem{Nguyen_23_Collision}
D.-T. Nguyen, \emph{Scaling limit of the collision measures of multiple random
  walks}, ALEA Lat. Am. J. Probab. Math. Stat. \textbf{20} (2023), no.~2,
  1385--1410. \MR{4683379}

\bibitem{Nishimori_Tomisaki_Tsuchida_Uemura_pre_On}
Y.~Nishimori, M.~Tomisaki, K.~Tsuchida, and T.~Uemura, \emph{On a convergence
  of positive continuous additive functionals in terms of their smooth
  measures}, 2024, Preprint. {A}vailable at arXiv:2405.03937.

\bibitem{Noda_pre_Metrization}
R.~Noda, \emph{Metrization of {G}romov--{H}ausdorff-type topologies on
  boundedly-compact metric spaces}, Preprint. {A}vailable at arXiv:2404.19681.

\bibitem{Noda_pre_Convergence}
\bysame, \emph{Convergence of local times of stochastic processes associated
  with resistance forms}, 2023, Preprint. {A}vailable at arXiv:2305.13224.

\bibitem{Noda_pre_Aging}
\bysame, \emph{Aging and sub-aging for {B}ouchaud trap models on resistance
  metric spaces}, 2024, Preprint. {A}vailable at arXiv:2412.08236.

\bibitem{Noda_pre_Scaling}
\bysame, \emph{Scaling limits of discrete-time {M}arkov chains and their local
  times on electrical networks}, 2024, Preprint. {A}vailable at
  arXiv:2405.01871.

\bibitem{Noda_pre_Continuity}
\bysame, \emph{Continuity of the {R}evuz correspondence under the absolute
  continuity condition}, 2025, Preprint. {A}vailable at arXiv:2501.10994.

\bibitem{Ooi_25_Convergence}
T.~Ooi, \emph{Convergence of processes time-changed by {G}aussian
  multiplicative chaos}, Potential Anal., Published online.

\bibitem{Ooi_pre_Homeo}
T.~Ooi, \emph{Homeomorphism of the {R}evuz correspondence for finite energy
  integrals}, 2025, Preprint. {A}vailable at arXiv:2502.01234.

\bibitem{Ooi_Tsuchida_Uemura_25_Nest}
T.~Ooi, K.~Tsuchida, and T.~Uemura, \emph{Smooth measures and positive
  continuous additive functionals attached to a compact nest}, 2025, Preprint.
  {A}vailable at arXiv:2509.23060.

\bibitem{Petrov_75_Sums}
V.~V. Petrov, \emph{Sums of independent random variables}, Ergebnisse der
  Mathematik und ihrer Grenzgebiete [Results in Mathematics and Related Areas],
  vol. Band 82, Springer-Verlag, New York-Heidelberg, 1975, Translated from the
  Russian by A. A. Brown. \MR{388499}

\bibitem{Polya_1921_recurrence}
G.~P\'olya, \emph{\"uber eine {A}ufgabe der {W}ahrscheinlichkeitsrechnung
  betreffend die {I}rrfahrt im {S}tra\ss ennetz}, Math. Ann. \textbf{84}
  (1921), no.~1-2, 149--160. \MR{1512028}

\bibitem{Resnick_08_Extrem}
S.~I. Resnick, \emph{Extreme values, regular variation and point processes},
  Springer Series in Operations Research and Financial Engineering, Springer,
  New York, 2008, Reprint of the 1987 original. \MR{2364939}

\bibitem{Revuz_70_Mesures}
D.~Revuz, \emph{Mesures associ\'ees aux fonctionnelles additives de {M}arkov.
  {I}}, Trans. Amer. Math. Soc. \textbf{148} (1970), 501--531. \MR{279890}

\bibitem{Sato_99_Levy}
K.~Sato, \emph{L\'evy processes and infinitely divisible distributions},
  Cambridge Studies in Advanced Mathematics, vol.~68, Cambridge University
  Press, Cambridge, 1999, Translated from the 1990 Japanese original, Revised
  by the author. \MR{1739520}

\bibitem{Shiozawa_Wang_24_HausdorffDim}
Y.~Shiozawa and J.~Wang, \emph{Hausdorff dimensions of inverse images and
  collision time sets for symmetric {M}arkov processes}, Electron. J. Probab.
  \textbf{29} (2024), Paper No. 6, 56. \MR{4688683}

\bibitem{Srivastava_98_A_Course}
S.~M. Srivastava, \emph{A course on {B}orel sets}, Graduate Texts in
  Mathematics, vol. 180, Springer-Verlag, New York, 1998. \MR{1619545}

\bibitem{Cech_69_Point}
E.~\v{C}ech, \emph{Point sets}, Academic Press, New York-London; Academia
  [Publishing House of the Czechoslovak Academy of Sciences], Prague, 1969,
  Translated from the Czech by Ale\v s{} Pultr. \MR{256344}

\bibitem{Watanabe_23_ICPfor3dUST}
S.~Watanabe, \emph{Infinite collision property for the three-dimensional
  uniform spanning tree}, International Journal of Mathematics for Industry
  \textbf{15} (2023), no.~01, 2350005.

\bibitem{Whitt_80_Some}
W.~Whitt, \emph{Some useful functions for functional limit theorems}, Math.
  Oper. Res. \textbf{5} (1980), no.~1, 67--85. \MR{561155}

\bibitem{Whitt_02_Stochastic}
\bysame, \emph{Stochastic-process limits}, Springer Series in Operations
  Research, Springer-Verlag, New York, 2002, An introduction to
  stochastic-process limits and their application to queues. \MR{1876437}

\bibitem{Williamson_Janos_87_Construction}
R.~Williamson and L.~Janos, \emph{Constructing metrics with the {H}eine-{B}orel
  property}, Proc. Amer. Math. Soc. \textbf{100} (1987), no.~3, 567--573.
  \MR{891165}

\bibitem{Yan_88_A_formula}
J.-A. Yan, \emph{A formula for densities of transition functions}, S\'eminaire
  de {P}robabilit\'es, {XXII}, Lecture Notes in Math., vol. 1321, Springer,
  Berlin, 1988, pp.~92--100. \MR{960514}

\end{thebibliography}

\end{document}